\documentclass[10pt]{article}
\usepackage[utf8]{inputenc}

\usepackage{amsfonts}
\usepackage{amsmath}
\usepackage{amssymb}
\usepackage{afterpage}
\usepackage{amsthm}
\usepackage{mathtools}
\usepackage{esint}
\usepackage{mathrsfs}
\usepackage{bbm}
\usepackage{comment, todonotes}
\usepackage{protosem}
\usepackage{longtable}

\usepackage{mathabx}
\usepackage{graphicx}
\usepackage[font=footnotesize, labelfont=bf]{caption}
\usepackage[font=footnotesize, labelfont=bf]{subcaption}

\definecolor{labelkey}{rgb}{0,0,1}

\usepackage{fullpage}
\usepackage{enumitem}
\usepackage{hyperref,xcolor,fullpage}

\usepackage{todonotes}

\allowdisplaybreaks 

\renewcommand*{\div}{\ensuremath{\mathrm{div}}}
\newcommand{\eps}{\varepsilon}
\newcommand{\I}{\infty}

\newcommand{\R}{\mathbb{R}}
\newcommand{\N}{\mathbb{N}}

\newcommand{\p}{\partial}

\newcommand{\ga}{\gamma}
\newcommand{\al}{\alpha}

\newcommand{\mc}{\mathcal}
\newcommand{\mf}{\mathfrak}
\newcommand{\mt}{\widetilde}

\newcommand{\mw}{\widebar}
\newcommand{\stirling}[2]{\genfrac{[}{]}{0pt}{}{#1}{#2}}
\def\les{\lesssim}
\newcommand{\eye}{\textproto{o}}
\newcommand{\Peye}{P_\eye}

\newcommand{\abs}[1]{\left|#1\right|}
\newcommand{\norm}[1]{\left\|#1\right\|}
\DeclarePairedDelimiterX{\ip}[2]{\langle}{\rangle}{#1, #2}

\newenvironment{detailssth}[1]{\begin{trivlist} \item[] {\textbf{Details of #1:}}}{
                       \end{trivlist}}

\newtheorem{theorem}{Theorem}[section]
\newtheorem{lemma}[theorem]{Lemma}
\newtheorem{definition}[theorem]{Definition}

\newtheorem{corollary}[theorem]{Corollary}

\newtheorem{proposition}[theorem]{Proposition}
\newtheorem{remark}[theorem]{Remark}
\numberwithin{equation}{section}

\title{Smooth imploding solutions for 3D compressible fluids}
\date{}
\author{Tristan Buckmaster\thanks{\footnotesize Department of Mathematics, University of Maryland, College Park, MD \& School of Mathematics, Institute for Advanced Study
Princeton, NJ \& Department of Mathematics, 
Princeton University, Princeton, NJ 
 \href{tristanb@umd.edu}{{tristanb@umd.edu}}} , Gonzalo Cao-Labora\thanks{\footnotesize Department of Mathematics, 
Massachusetts Institute of Technology, MIT, MA \href{gcaol@mit.edu}{gcaol@mit.edu}}\, and 
Javier G\'omez-Serrano\thanks{\footnotesize Department of Mathematics,
Brown University,
Providence RI \&         	
Departament de Matem\`atiques i Inform\`atica,
Universitat de Barcelona,
 Barcelona 08007 \& Centre de Recerca Matem\`atica, Edifici C, Campus Bellaterra, 08193 Bellaterra, Spain
 \href{javier\_gomez\_serrano@brown.edu}{javier\_gomez\_serrano@brown.edu} / \href{jgomezserrano@ub.edu}{jgomezserrano@ub.edu}}
}

\begin{document}
\maketitle

\begin{abstract}
Building upon the pioneering work of Merle, Rapha\"el, Rodnianski and Szeftel \cite{MeRaRoSz19a,MeRaRoSz19c,MeRaRoSz19b}, we construct exact, smooth self-similar imploding solutions to the 3D isentropic compressible Euler equations for ideal gases for \emph{all} adiabatic exponents $\gamma>1$.  	For the particular case $\gamma=\frac75$ (corresponding to a diatomic gas, e.g.\ oxygen, hydrogen, nitrogen), akin to the result \cite{MeRaRoSz19a}, we show the existence of a sequence of smooth, self-similar imploding solutions. In addition, we provide simplified proofs of  linear stability \cite{MeRaRoSz19b} and  non-linear stability \cite{MeRaRoSz19c}, which allow us to construct asymptotically self-similar imploding solutions to the compressible Navier-Stokes equations with density independent viscosity for the case $\gamma=\frac75$. Moreover, unlike \cite{MeRaRoSz19c}, the solutions constructed have density bounded away from zero and converge to a constant at infinity, representing the first example of singularity formation in such a setting.
\end{abstract}

\setcounter{tocdepth}{1}
\tableofcontents


\section{Introduction} \label{sec:intro}

In this paper, we construct  self-similar imploding solutions to the 3D isentropic compressible Euler equations
\begin{equation}\label{eq:Euler}
 \begin{split} 
\partial_t (\rho u) + \div (\rho u \otimes u) + \nabla p(\rho) &= 0 \,, \\
\partial_t \rho  +  \div (\rho u)&=0 \,, 
\end{split}
\end{equation}
where here $u$ is the velocity, $\rho$ the density and we will assume the
 ideal gas law $p(\rho) = \tfrac{1}{\gamma} \rho^\gamma$ for $\gamma >1$.
 Additionally,  these self-similar solutions to Euler will be used as a basis to construct asymptotically self-similar solutions to the 3D isentropic compressible Navier-Stokes equations with density independent viscosity
\begin{equation}\label{eq:NS}
 \begin{split} 
\partial_t (\rho u) + \div (\rho u \otimes u) + \nabla p(\rho)-\mu_1\Delta u-(\mu_1+\mu_2)\nabla \div u &= 0 \,, \\
\partial_t \rho  +  \div (\rho u)&=0 \,, 
\end{split}
\end{equation}
for Lam\'e viscosity coefficients $(\mu_1,\mu_2)$ satisfying $\mu_1>0$ and $2\mu_1+\mu_2> 0$. In the case of the Navier-Stokes equations, we will assume the initial density to be constant at infinity in order to rule out the possibility that the singularities are an artifact of vacuum. Local well-posedness for the compressible Euler and Navier-Stokes equations \eqref{eq:NS} is classical (cf.\ \cite{MR390516,MR0350216,Ma1984,MR1076448,BSMF_1962__90__487_0,197160,Da07}).

\subsection{Background}

\subsubsection{Shock wave singularities}

 The prototypical singularity for the Euler equations is a shock wave, occurring when the speed of a disturbance exceeds the local speed of sound. Mathematically, one would like to provide a detailed description of both the formation of a shock and its development past the first singularity. 

The first rigorous result in this direction is due to Lax \cite{Lax1964}. Lax showed that in 1D, when writing the equation in terms of its \emph{Riemann invariants}, one can use the method of characteristics to prove finite time singularity formation. The existence of finite time singularities in 2D and 3D was demonstrated by Sideris in \cite{Si1997} via a virial type argument. Lebaud, in her seminal thesis work \cite{MR1309163}, provided the first detailed description of shock formation, in the context of one-dimensional p-systems, as well as proving development (see \cite{MR1860832,MR1897395} for generalizations of Lebaud's result).

In higher dimensions,  Alinhac \cite{Al1999a,Al1999b} was the first to provide a detailed description of shock formation for a class of quasilinear wave equations.  Yin in \cite{MR2085314} was able to adapt the work of Lebaud in order to prove shock formation and development in 3D under spherical symmetry (cf.\ \cite{MR3489205}). Within the sub-class of irrotational solutions, Christodoulou and Miao \cite{ChMi2014} gave the first proof of shock formation in higher dimensions in the absence of symmetry (cf.\ \cite{Ch2007}). The work was extended by  Luk and Speck to the 2D setting with non-trivial vorticity in   \cite{LuSp2018}. 

In the work \cite{BuShVi2019}, the first author, Shkoller and Vicol  developed a new self-similar framework in order to prove the existence and stability of shock wave formation for the Euler equations under \emph{azimuthal symmetry}. This new framework provided the foundation for the works \cite{buckmaster2020formation,buckmaster2020shock} by the same authors, which provided the first full detailed description of 3D shock formation in the presence of non-trivial vorticity and non-constant entropy (see \cite{2021arXiv210703426L} for a recent related work of Luk and Speck in the framework of Christodoulou). As described above, the shock formation problem has been studied up to the time of  the first singularity.  The problem of maximal development has been very recently  studied by Abbrescia and Speck \cite{2022arXiv220707107A} and by Shkoller and Vicol \cite{ShVi22} using two very different  approaches, in which solutions of the Euler equations are constructed  for times that are much larger than the first blow-up time.   In particular, the hypersurface of pre-shocks (or first singularities) is classified, and this is  precisely the data required for the development problem.

In 2D, under azimuthal symmetry, the first author together with Drivas, Shkoller and Vicol were able to develop the singularity considered in the earlier work \cite{BuShVi2019} in order to give the first full description of shock development, including the first description of \emph{weak discontinuities} conjectured by Landau and Lifschitz. 

\subsubsection{Imploding solutions}

While shock waves are the prototypical and possibly the sole stable form of singularity for the Euler equations, they are not the only form of singularity that can form from smooth initial data. It is a fundamentally interesting problem, both from a mathematical perspective and a physical perspective, to classify other forms of singularities resulting from smooth initial data. 

Motivated by the classical work of Guderley \cite{MR8522} (cf.\ \cite{chisnell_1998,MeyerterVehn1982}) on \emph{imploding} solutions, Merle, Rapha\"el, Rodnianski and Szeftel, in the breakthrough work \cite{MeRaRoSz19a}, rigorously proved the existence of smooth radially symmetric imploding solutions to the isentropic compressible Euler equations for which the velocity and density become infinite at the time of singularity (cf.\ \cite{doi:10.1137/20M1340241,2022arXiv220515876J}). The work \cite{MeRaRoSz19a} differs from the prior work of Guderley \cite{MR8522} in a significant way, the solutions \cite{MeRaRoSz19a} are smooth up until blow up; whereas, the solutions \cite{MR8522} represent solutions for which a shock has already formed. It should be noted that the solutions described in \cite{MeRaRoSz19a} are highly unstable, which would make observing such solutions in numerical simulations or physical experiments extremely difficult. However, given that the structure of the solutions is now known, these solutions can be numerically computed as was done by Biasi in \cite{Biasi21}, which provides a detailed numerical survey of the Merle et al.\ solutions.

In the companion works \cite{MeRaRoSz19b,MeRaRoSz19c}, the solutions constructed in  \cite{MeRaRoSz19a} have been used to construct asymptotically self-similar solutions to both the  compressible Navier-Stokes equation \eqref{eq:NS}  and  the  defocusing  nonlinear Schr\"odinger equation; the later result resolving a major open problem in the field. 

To describe the solutions of \cite{MeRaRoSz19a}, one must rewrite \eqref{eq:Euler} in isentropic, radial form:
\begin{equation}\label{eq:wombat}
 \p_t u + u \p_R u + \frac{1}{\gamma\rho}\p_R( \rho^{\gamma})=0\quad\mbox{and}\quad
 \p_t \rho + \frac{1}{R^2}\p_R(R^2 \rho u)=0\,,
\end{equation}
 where for matters of simplicity, we restricted the problem to 3 dimensions. 
 Letting  $\sigma=\frac{1}{\alpha}\rho^{\alpha}$, for $\alpha=\frac{\gamma-1}{2}$, denote the rescaled sound speed, one makes the following self-similar anzatz
\begin{equation}\label{eq:ansatz:intro}
  u(R,t)=r^{-1}\tfrac{R}{T-t} U\left(\log(\tfrac{R}{(T-t)^{\frac1r}})\right)\quad\mbox{and}\quad
 \sigma(R,t)=\alpha^{-\tfrac12}r^{-1} \tfrac{R}{T-t} S\left(\log(\tfrac{R}{(T-t)^{\frac1r}})\right)\,,
\end{equation}
 where here $r$ is a self-similar scaling parameter to be determined. Defining  a new  self-similar variable $\xi=\log(\tfrac{R}{(T-t)^{\frac1r}})$, then \eqref{eq:wombat} reduces to an autonomous system of the form
 \begin{align}\label{eq:DS}
\tfrac{dU}{d\xi}  = \tfrac{N_U (U, S)}{D (U, S)},\quad\mbox{and}\quad
\tfrac{dS}{d\xi}  = \tfrac{N_S (U, S)}{D (U, S)}\,.
\end{align}
The phase portrait for the case $\gamma=\frac75$, $r=1.079404$ is represented in Figure \ref{fig:US}. The red, green and black curves represent the vanishing of $D$, $N_U$ and $N_S$ respectively.  $P_0$ is a point in the compactified phase portrait, with finite value of $U$ but $S = + \infty$, and it will correspond to the values of $(\bar U / R, \bar S / R)$ at the origin for our profiles. $P_\infty$ is the point $(0, 0)$ and it will correspond to values of the profiles at $R = \infty$ (both profiles decay). $P_s$ is a regular singular point of the dynamical system \eqref{eq:DS} and hence one can construct integral curves which cross $P_s$. There exist two smooth integral curves that cross $P_s$, one tangent to the direction $\nu_-$ and the other one tangent to $\nu_+$. The curve tangent to $\nu_+$ corresponds to the Guderley solution, whereas the curve tangent to $\nu_-$ corresponds to solution found in \cite{MeRaRoSz19a}. In order to create a globally defined self-similar solution, one must find an integral curve connecting $P_0$ to $P_{\infty}$ via $P_s$. It is impossible to achieve this with the Guderley solution with a continuous integral curve; however, by adding a shock discontinuity, one may jump from one point in the phase portrait to another and hence describe a globally defined self-similar solution. The major difficulty faced in \cite{MeRaRoSz19a} is that the alternate smooth integral curve in general also does not connect $P_0$ to $P_\infty$,  rather it intersects the sonic line $D=0$ at a point other than $P_s$ leading to a solution that is not globally defined.\footnote{The problem of constructing non-smooth global solutions is however comparatively simple, involving gluing a curve from $P_\infty$ to $P_s$ and the unique curve connecting $P_s$ to $P_0$. It is unclear what the physical significance of such solutions is as they have essentially no stability properties even modulo a finite dimensional space.} The authors however showed that for almost every $\gamma>1$, there exists an infinite sequence $\{r_j\}$, depending on $\gamma$, converging to some $r^{\ast}$, such that there exists a smooth curve connecting $P_0$ to $P_\infty$. The condition on $\gamma$ for which the result holds is described in terms of the non-vanishing of an analytic function. This condition is not proven for any specific $\gamma$; however it may be checked numerically. The analysis in  \cite{MeRaRoSz19a} becomes singular at $\gamma=\frac53$, and so this specific, physically important case, corresponding to monatomic gas (helium), is not included in their theorem.

\begin{figure}
\centering
  \includegraphics[width=.5\linewidth]{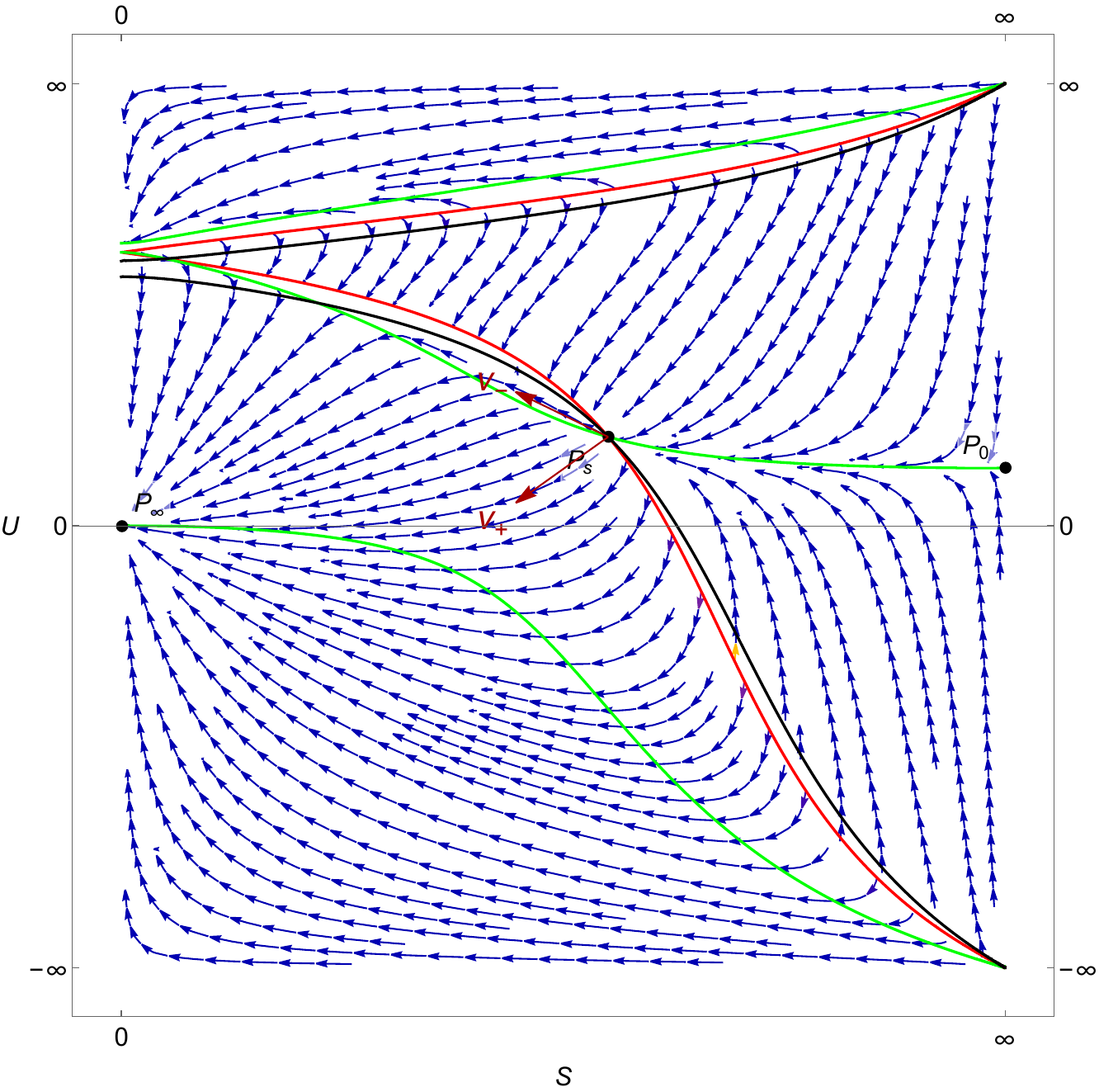}
  \captionof{figure}{\footnotesize Imploding solutions in $(U,S)$ variables. Note that a singular coordinate change has been made in order to compactify the $(U,S)$ coordinates.}
  \label{fig:US}
\end{figure}

In the  work \cite{MeRaRoSz19c}, the authors used the solutions of \cite{MeRaRoSz19a} in order to show that for almost every $1<\gamma< \tfrac{2 + \sqrt 3}{\sqrt 3}$ there exists an asymptotically self-similar solution to the compressible Navier-Stokes equation \eqref{eq:NS} that blows up in finite time. Existence of finite-time blow up for compressible Navier-Stokes was known previously for the case of compactly supported density \cite{Xin} and rapidly decaying density \cite{ROZANOVA20081762}. Neither works however give a precise description of the singularity formation. Within this range of $\gamma$, there exist self-similar solutions to the Euler equations for which the dissipation terms for the corresponding self-similar Navier-Stokes problem can be  treated as exponentially decaying forcing due to the specific self-similar scaling. Applying stability analysis borrowed from \cite{MeRaRoSz19b}, the authors then use the solutions of \cite{MeRaRoSz19c} to construct asymptotically self-similar solutions to \eqref{eq:NS} via a  Brouwer fixed point argument. One caveat of the work \cite{MeRaRoSz19c} is that the initial density of solutions is required to decay at infinity. Ideally, one would like to remove this condition in order to rule out the importance of the solution at infinity in the singularity formation process.

The works \cite{MeRaRoSz19a} and \cite{MeRaRoSz19c} leave open two important questions:  
\begin{enumerate}
\item
Do imploding solutions for Euler exist for all $\gamma>1$?
\item Can one construct imploding solutions to the Navier-Stokes equation with initial density constant at infinity?
\end{enumerate}

\subsection{Main results}

\begin{theorem} \label{th:mainlarge} Let $\ga \in (1, +\I )$. There exists $r^{(3)}(\gamma) \in (r_3(\gamma), r_{4}(\gamma))$, such that there exists a smooth solution to \eqref{eq:DS} starting at $P_0$ and ending at $P_\infty = (0, 0)$, where $ (r_3(\gamma), r_{4}(\gamma))$ are defined in Section \ref{ss:imploding:solution}. This gives a smooth and radially symmetric self-similar solution to \eqref{eq:wombat} of the form \eqref{eq:ansatz:intro}.
\end{theorem}
 
\begin{theorem} \label{th:mainr3} Let $\ga = 7/5$ and $n \in \N$ be an odd number large enough. There exists $r^{(n)}(\gamma) \in (r_n(\gamma), r_{n+1}(\gamma))$ such that there exists a smooth solution to \eqref{eq:DS} starting at $P_0$ and ending at $P_\infty$, where $ r_j(\gamma)$ is defined in Section \ref{ss:imploding:solution}. This gives a smooth and radially symmetric self-similar solution  to \eqref{eq:wombat} of the form \eqref{eq:ansatz:intro}.
\end{theorem}

\begin{theorem} \label{th:stability} Let $\ga = 7/5$ and $n \in \N$ be an odd number large enough. Let $(U^E, S^E)$ be the profiles of Theorem \ref{th:mainr3}, solving \eqref{eq:DS}. Then, for sufficiently small $T>0$, there exists a radially symmetric initial data $(u_0,\rho_0)$ such that we have the following
\begin{enumerate}
\item \label{item:Hagoromo1} The initial density $\rho_0$ is constant at infinity:
\[\lim_{\abs{x}\rightarrow\infty } \rho_0(x)=\rho_c\,. \]
\item \label{item:Hagoromo2} The initial data $(u_0,\rho_0)$ is smooth and has finite energy:
\[\frac12 \int \rho_0 \abs{u_0}^2+\frac{1}{\gamma(\gamma-1)}\int(\rho_0-\rho_c)^\gamma<\infty\,.\]
\item \label{item:Hagoromo3}At time $T$, the solution $(u,\rho)$ becomes singular at the origin: for any $\eps>0$
\[ \lim_{t\rightarrow T}\sup_{R\in[0,\eps)}\abs{u(R,t)}=\infty\quad\mbox{and}\quad \lim_{t\rightarrow T}\rho(0,t)=\infty\,.\]

\item\label{item:Hagoromo4} The solution $(u,\rho)$ blows up in an asymptotically self-similar manner: for any fixed $\xi\geq 0$
\[  \lim_{t\rightarrow T}r\tfrac{T-t}{R}u\left((T-t)^{\frac1r}\exp(\xi),t\right)= U^E(\xi)\quad\mbox{and}\quad
\lim_{t\rightarrow T} \alpha^{\tfrac12}r \tfrac{T-t}{R}\sigma\left((T-t)^{\frac1r}\exp(\xi),t\right)= S^E(\xi)\,.
\]
\end{enumerate}
Moreover, there exists a finite codimensional manifold of radially symmetric initial data satisfying the above conclusions (see Remark \ref{rem:Manifold} for more details).
\end{theorem}

\begin{remark}
For simplicity, we will only prove Theorem \ref{th:stability} for the case  $\mu_1=1$ and $\mu_2=-1$. The general case follows analogously with minor changes to the energy estimates in Section \ref{sec:nonlinarstability}.
\end{remark}
\begin{remark}
We note that as a corollary of the proof of Theorem  \ref{th:stability}, the statement of Theorem  \ref{th:stability} holds with the Navier-Stokes equations \eqref{eq:NS} replaced by the Euler equations \eqref{eq:Euler} for $\gamma=7/5$. With some minor work, Theorem  \ref{th:stability} can be extended to all $\gamma>1$ in the case of Euler by making use of the self-similar profiles of Theorem \ref{th:mainlarge}.
\end{remark}

\subsection{Self-similar implosion in terms of Riemann invariants}

Motivated by the works \cite{BuShVi2019,buckmaster2020formation,buckmaster2020shock}, we introduce the Riemann invariants
 \begin{equation}\label{eq:Riemann:invariants}
 w=u+\sigma\quad\mbox{and}\quad z=u-\sigma \,,
 \end{equation}
so that
\[u=\frac12(w+z)\quad\mbox{and}\quad\sigma=\frac{w-z}{2}\,.\]
One can now diagonalize \eqref{eq:wombat} in terms of $w$ and $z$, in order to rewrite \eqref{eq:wombat}  as a nonlinear transport equation
 \begin{align}\label{eq:Euler:Riemann}\begin{split}
 \p_t w + \frac12(w+z+\alpha(w-z)) \p_R w + \frac{\alpha}{2R}(w^2-z^2)&=0\,,\\
  \p_t z + \frac12(w+z-\alpha(w-z)) \p_R z - \frac{\alpha}{2R}(w^2-z^2)&=0\,.
  \end{split}
 \end{align}
Employing the self-similar ansatz
  \begin{align}\label{eq:ansatz}
  \begin{split}
 w(R,t)=\frac{1}{r} \cdot\frac{R}{T-t} W( \xi)\quad\mbox{and}\quad
 z(R,t)=\frac{1}{r} \cdot \frac{R}{T-t} Z( \xi)\,,
 \end{split}
 \end{align}
 where we recall $\xi=\log(\tfrac{R}{(T-t)^{\frac1r}})$, then we obtain \begin{align} \begin{split} \label{eq:mainother}
(r+\frac12((1+2\alpha)W+(1-\alpha)Z))W+(1+\frac12(W+Z+\alpha(W-Z)))\partial_{\xi}  W  - \frac{\alpha}{2}Z^2&=0\,,\\
(r+\frac12((1-\alpha)W+(1+2\alpha)Z))Z+(1+\frac12(W+Z-\alpha(W-Z)))\partial_{\xi}  Z  - \frac{\alpha}{2}W^2&=0\,.
 \end{split} \end{align}
 
Rearranging, we obtain the autonomous system
\begin{equation}\label{eq:wz:ODE}
\begin{split}
\partial_{ \xi }  W&= \frac{-(r+\frac12((1+2\alpha)W+(1-\alpha)Z))W+ \frac{\alpha}{2}Z^2}{1+\frac12(W+Z+\alpha(W-Z))}=\frac{N_W}{D_W},\\
\partial_{ \xi } Z&=\frac{-(r+\frac12((1-\alpha)W+(1+2\alpha)Z))Z+\frac{\alpha}{2}W^2}{1+\frac12(W+Z-\alpha(W-Z))}=\frac{N_Z}{D_Z}.
\end{split}
\end{equation}

\begin{figure}
\centering
  \includegraphics[width=.5\linewidth]{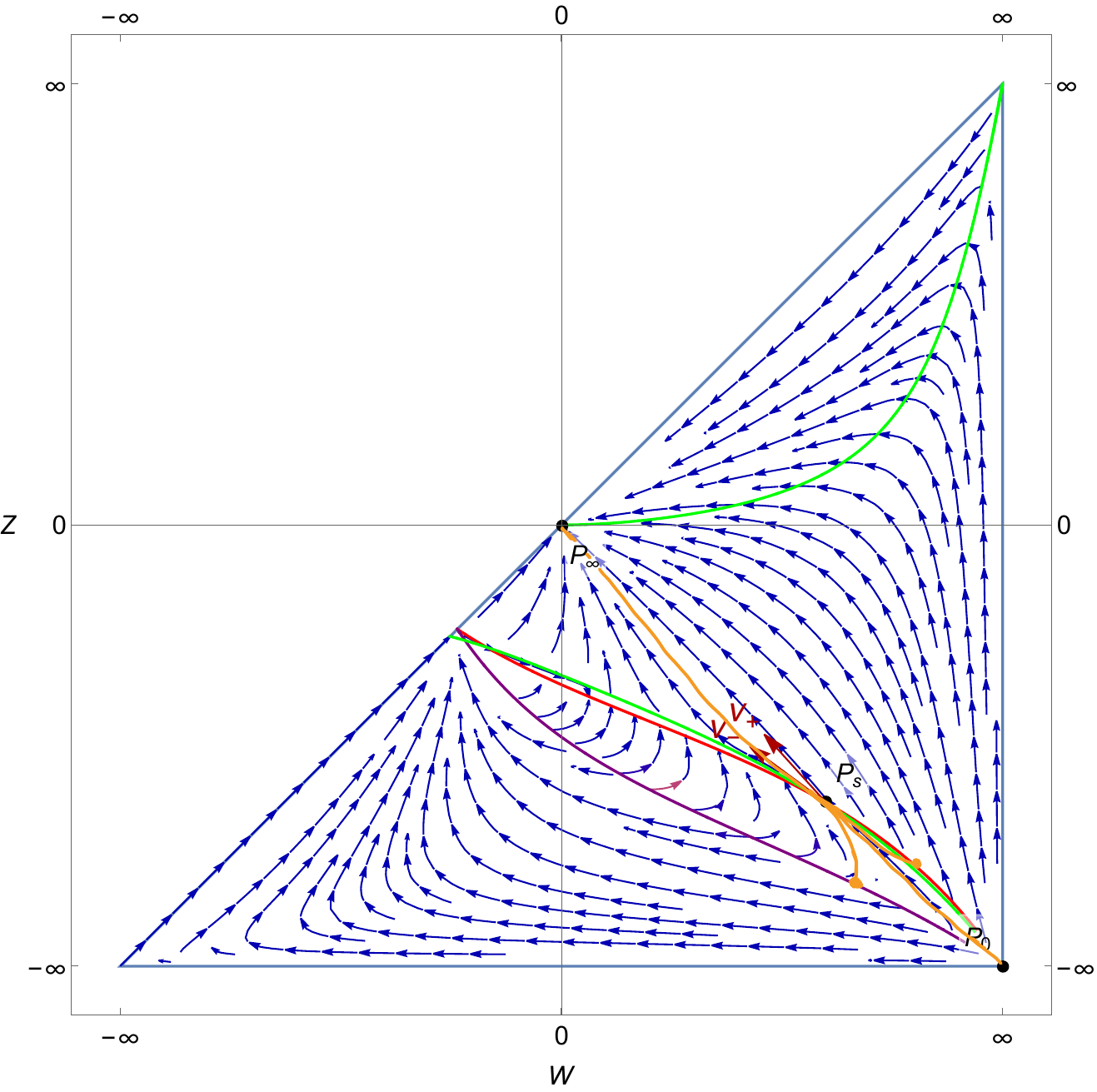}
  \captionof{figure}{\footnotesize Imploding solutions in $(W,Z)$ variables. Note that a singular coordinate change has been made in order to compactify the $(W,Z)$ coordinates. We have indicated in orange the type of smooth solutions we will find, crossing through $P_s$ with direction $v_-$. On the left of $P_s$ the solution converges to $P_\infty$, while on the right, we show three possibilities for its behaviour (it can start at $D_W = 0$, at $P_0$ or at $D_Z = 0$).}
  \label{fig:WZ}
\end{figure}

Figure \ref{fig:WZ} represents the phase portrait for the region $W-Z>0$ for which the density is positive. The red, purple and green lines correspond to $D_Z=0$, $D_W=0$ and $N_Z=0$ respectively. A key difference to the system \eqref{eq:DS} is that the denominator $D_W$ does not vanish at $P_s$, which simplifies the analysis in a neighborhood of $P_s$. Unlike the self-similar variables $(U,S)$, the variables $(W,Z)$ satisfy transport equations, which leads to the possibility of employing transport arguments in order to simplify the stability analysis. In particular, the $(W,Z)$ variables give rise to a very geometric understanding of the imploding solution in terms of the trajectories of the $W$ and $Z$ waves: $P_s$ is an unstable fixed point for the trajectories of $Z$-waves. Let $P_s$ divide space into an interior (backward acoustic cone emanating from the singular point) and exterior region.  $Z$-waves in the exterior region cannot cross into the interior region; whereas, $Z$-waves in the interior region cross the origin to become $W$-waves, whereupon they cross $P_s$ and travel to the exterior region. Since the system \eqref{eq:wz:ODE} is autonomous, we are free to fix the location $\xi$ for which the solutions crosses $P_s$. As such, we make the choice that $P_s$ is located at $\xi=0$.

Due to the singular nature of the coordinate transformation $R\mapsto \xi$ for $R$ near $0$, it is also helpful to introduce alternate self-similar coordinates. If we write
\[ \zeta=\frac{R}{(T-t)^{\frac{1}{r}}}=\exp(\xi)\,,\]
and write
\begin{align}
\label{eq:ansatz2}
\begin{split}
  w(R,t)&=r^{-1}(T-t)^{r^{-1}-1} \mathcal W(\zeta)=r^{-1} \frac{R}{T-t} W( \xi)\,,\\
 z(R,t)&=r^{-1} (T-t)^{r^{-1}-1} \mathcal Z(\zeta)=r^{-1}\frac{R}{T-t} Z( \xi)\,,
 \end{split}
 \end{align}
 then \eqref{eq:Euler:Riemann} becomes
  \begin{align}\label{eq:Euler:SS:alt}
   \begin{split} 
(r-1)\mc W+(\zeta+\frac12(\mc W+\mc Z+\alpha(\mc W-\mc Z)))\p_{\zeta}  \mc W  +\frac{\alpha}{2\zeta}(\mc W^2-\mc Z^2)&=0\,,\\
 (r-1)\mc Z+(\zeta+\frac12(\mc W+\mc Z-\alpha(\mc W-\mc Z)))\p_{\zeta}  \mc Z  - \frac{\alpha}{2\zeta}(\mc W^2-\mc Z^2)&=0\,.
 \end{split}
  \end{align} 
This form of the equation will be useful in studying the solution at the origin $\zeta=0$. A time dependent version of these equations will also be used to study  stability. Since we will be looking for solutions that are smooth at the origin, we can extend the solution to all of $\zeta\in \mathbb R$ by requiring that $\mc Z(\zeta)=-\mc W(-\zeta)$, The equations reduce to a single equation in $\mc W$: 
  \begin{align}\label{eq:Euler:SS:alt2}
   \begin{split} 
(r-1)\mc W(\zeta)+(\zeta+\frac12(\mc W(\zeta)-\mc W(-\zeta)+\alpha(\mc W(\zeta)+\mc W(-\zeta))))\p_{\zeta}  \mc W(\zeta)  +\frac{\alpha}{2\zeta}(\mc W^2(\zeta)-\mc W^2(-\zeta)&=0\,.
 \end{split}
  \end{align}

We first begin with a result for the maximal time of existence of solutions to the ODE \eqref{eq:mainother}. We also show that the system does not have periodic orbits.

\begin{proposition} \label{prop:existence} Let $(W_\star, Z_\star )\in \R^2$ such that $D_W(W_\star , Z_\star ) \neq 0$ and $D_Z(W_\star , Z_\star ) \neq 0$ and let $\xi_\star \in \R$. There exists a smooth solution $W(\xi ), Z(\xi ):(\xi_1, \xi_2 ) \rightarrow \R$ to \eqref{eq:mainother} such that $W(\xi_\star ) = W_\star $, $Z (\xi_\star ) = Z_\star$ and: \begin{itemize}
\item Either $(W(\xi ), Z(\xi ))$ tends to a point of $\{ D_W = 0 \} \cup \{ D_Z = 0 \}$ as $\xi \rightarrow \xi_1^+$, or to infinity or to an equilibrium point. Moreover, $\xi_1 = -\infty$ in the latter case.
\item Either $(W(\xi ), Z(\xi ))$ tends to a point of $\{ D_W = 0 \} \cup \{ D_Z = 0 \}$ as $\xi \rightarrow \xi_2^-$, or to infinity or to an equilibrium point. Moreover, $\xi_2 = +\infty$ in the latter case.
\end{itemize}
\end{proposition}
The proof of Proposition \ref{prop:existence} is given in Appendix \ref{sec:rapid:antigen}.

\begin{remark}  \label{rem:Omega12} By local existence and uniqueness we can divide the phase portrait in disjoint orbits ending either at the nullsets of $D_W$, $D_Z$ or at infinity. Let $\Omega$ be the region where $D_W > 0, D_Z < 0$. Let $\Omega_1^{(r)}$ be the subset of points whose trajectories emanate from the half-line of $D_Z = 0$ located to the right of $P_s$ and $\Omega_2^{(r)}$ be the points for whose trajectory emanates from $D_W=0$.\end{remark}

\subsection{Smooth self-similar imploding solution}\label{ss:imploding:solution}

In order to further analyze $P_s$, it is helpful to consider the dynamical system under the change of variables $\xi\mapsto \psi$ where $\p_\psi=-D_WD_Z\p_\xi$. The equation \eqref{eq:wz:ODE} becomes
\begin{equation}\label{eq:wz:ODE2} 
\partial_{ \psi}  W=-N_WD_Z \quad\mbox{and}\quad
\partial_{ \psi } Z=-N_ZD_W\,,
\end{equation}
and $P_s$ becomes a stable stationary point of \eqref{eq:wz:ODE2}. The smooth integral curves of \eqref{eq:wz:ODE} correspond to slope matching smooth curves with limit $P_s$. 

It is illustrative to consider the following simple system
\[\dot x=\lambda_+ x\,,\quad\dot y =\lambda_-  y\,,\]
for some $\lambda_-<\lambda_+<0$, which correspond to the eigenvalues of the system's Jacobian. So long as $k=\frac{\lambda_-}{\lambda_+}\notin \mathbb{N}$, the only smooth integral curves are along $x=0$ and $y=0$. There do however exist non-smooth solutions of the form $y=Cx^k$ which are $C^k$ regular and whose Taylor series agrees with the solution $y=0$ up to order $\lfloor k \rfloor$, i.e.\ the largest integer smaller or equal to $k$.

Returning to our ODE  \eqref{eq:wz:ODE2}, define $\lambda_-<\lambda_+<0$ to be the eigenvalues  of the Jacobian of \eqref{eq:wz:ODE2} at $P_s$, and define
\begin{equation}\label{eq:k:def}
k=\frac{\lambda_-}{\lambda_+}\,.
\end{equation}
If $\nu_-$, $\nu_+$ are the eigenvectors of the Jacobian of \eqref{eq:wz:ODE2} associated with the eigenvalues  $\lambda_-,\lambda_+$, then we will be considering the smooth solutions of  \eqref{eq:wz:ODE} with tangent parallel to $\nu_-$ -- the Guderley solutions correspond to the direction $\nu_+$. These two directions are illustrated in Figure \ref{fig:WZ}.

We restrict the self-similar parameter $r$ to $1<r<r^\ast$ where
\begin{equation}\label{eq:rstar}
 r^\ast (\gamma ) = \begin{cases}
      \tfrac{2}{\left(\sqrt{2} \sqrt{\frac{1}{\gamma-1}}+1\right)^2}+1  & 1 < \gamma < \tfrac{5}{3} \,,\\
      \tfrac{3\gamma - 1 }{2 + \sqrt{3}(\gamma - 1)}    & \gamma \geq \tfrac{5}{3}\,.
\end{cases}
\end{equation}
In this regime, $k$ will be a monotonically increasing function of $r$, converging to $\infty$ as $r\rightarrow r^\ast$ (see Lemma \ref{lemma:k}).

To study the behavior of the smooth solution corresponding to the direction $\nu_-$ around $P_s$, which we denote $(W^{(r)},Z^{(r)})$, we apply a Frobenius-like series expansion of the solution. 
As will be shown in Section~\ref{sec:taylor}, letting $(W_n,Z_n)$ denote the $n$-th Taylor coefficient of $(W^{(r)},Z^{(r)})$ expanded at $P_s$, then for $n\geq 2$
\begin{equation}\label{eq:Quoll}
W_n=F_W(r,\gamma,W_0,\dots,W_{n-1}, Z_0,\dots, Z_{n-1})\quad\mbox{and}\quad Z_n=\frac{F_Z(r,\gamma,W_0,\dots,W_{n}, Z_0, \dots,Z_{n-1})}{n-k(r)}\,,
\end{equation}
where $(F_W, F_Z)$ are given in Section~\ref{sec:taylor}.
For $j\in\mathbb N$, we define $r_j$ such that $j=k(r_j)$. Note that the denominator in \eqref{eq:Quoll} becomes singular as $k(r)$ approaches $n$ and switches sign at $k(r)=n$. This has a \emph{wiggling} effect on the integral curve of the smooth solution, which in turn allows us to show that for a subset of $\gamma>1$ and odd $n\geq 3$ :
\begin{enumerate}
\item\label{pt:1} For $r\in(r_n,r_{n+1})$ the solution to the left of $P_s$ converges to $P_{\infty}$ as $\xi\rightarrow \infty$.
\item\label{pt:2} For $r=r_n+\eps$ the solution to the right  of $P_s$  intersects the line $D_W = 0$.
\item\label{pt:3} For $r=r_{n+1}-\eps$ the solution to the right  of $P_s$  intersects the line $D_Z = 0$.
\end{enumerate}
More specifically, we show the above holds for $n=3$ and $\gamma \in (1, +\infty)$, as well as the case $\gamma=\frac75$ and $n$ sufficiently large.\footnote{We are however not aware of any counterexamples for $\gamma>1$ and $n\geq3$ odd. The requirement that $n$ is odd is used to ensure \ref{pt:1} holds.}  If we can prove \ref{pt:2} and \ref{pt:3}, by a simple shooting argument we obtain that there exists an $r\in(r_n,r_{n+1})$ such that the solution curve connects $P_s$ to $P_0$. Moreover, \ref{pt:1} implies that the solution connects $P_s$ to $P_\infty$. We note that for the Einstein-Euler and Euler-Poisson systems, Guo,  Had\v zi\'c and Jang in \cite{Guo_Relativity} and \cite{Guo_LarsonPenston} apply similar arguments in the context of non-autonomous ODEs.

For the special case $\gamma=\frac75$, we aim at constructing a sequence of self-similar scalings $r^{(j)}$ satisfying $r_j<r^{(j)}<r_{j+1}$, for $j$ odd and sufficiently large, as well as the corresponding smooth global solutions. A key ingredient to proving this is to determine a sign and lower bound on the Taylor coefficients of order $j$. Contrarily to the $\gamma\geq \frac53$ case, for $\gamma<\frac53$, and $r=r^\ast$, one may obtain a \emph{non-trivial} Taylor expansion of the corresponding curve passing through $P_s$. By continuity, for $r<r^\ast$, the corresponding Taylor series converges to that of $r=r^\ast$. Then, fixing $r<r^\ast$, $r$ sufficiently close to $r^\ast$, one can deduce the sign and magnitudes of lower order Taylor coefficients from those of $r^\ast$. With this information, one can use an inductive argument to deduce information about the higher order coefficients. Furthermore, we employ a computer-assisted proof to compute the first 10000 coefficient pairs $(W_j,Z_j)$ at $r=r^\ast$ with rigorous error bounds. While there is certainly room for improvement in terms of the amount of coefficients that we had to calculate using a computer-assisted approach, we decided to keep the asymptotic part of the analysis that treats the higher order coefficients as simple as possible, at the expense of a slightly larger computation time. This part of the calculation takes about 14 hours on a single CPU.

In order to perform rigorous, error-free calculations, interval 
arithmetic will be used as part of the proof whenever 
needed. The main idea underlying this technique is to work with 
intervals which have representable numbers by the computer as 
endpoints in order to guarantee that the true result at any point 
belongs to the interval by which is represented. By doing so, we 
control all the errors (rounding, floating point arithmetic, etc.) 
incurred by the computer program while calculating the necessary 
quantities. Over the intervals, we define an arithmetic in such a way 
that we are guaranteed that for every $x \in X, y \in Y$
\begin{align*}
x \star y \in X \star Y,
\end{align*}
for any operation $\star$. For example,
\begin{align*}
[\underline{x},\overline{x}] + [\underline{y},\overline{y}] & = 
[\underline{x} + \underline{y}, \overline{x} + \overline{y}] \,,\\
[\underline{x},\overline{x}] \times [\underline{y},\overline{y}] & = 
[\min\{\underline{x}\underline{y},\underline{x}\overline{y},\overline{x}\underline{y},\overline{x}\overline{y}\},\max\{\underline{x}\underline{y},\underline{x}\overline{y},\overline{x}\underline{y},\overline{x}\overline{y}\}]
\,.
\end{align*}
We can also define the interval version of a function $f(X)$ as an
interval $I$ that satisfies that for every $x \in X$, $f(x) \in
I$. Even though in this paper we will only make use of basic functions 
more complicated ones (such as special functions) over intervals can 
be defined as well.

Very early computer-assisted proofs were mostly constrained to finite dimensional problems 
\cite{Fefferman-DeLaLLave:relativistic-stability-of-matter-I, 
Tucker:lorenz-focm}. Slowly, more and more computational power has enabled to tackle harder problems, including partial differential equations. We mention the pioneering papers of Plum
\cite{Plum:H2-estimates-elliptic-bvp,Plum:numerical-existence-nonlinear-elliptic-bvps} and Nakao
\cite{Nakao:numerical-approach-elliptic-problems,Nakao:nonlinear-parabolic-problems} in this context, and more recent advances done by Fazekas--Pacella--Plum 
\cite{Fazekas-Pacella-Plum:nonradial-lane-emden-ball} for the Lane Emden equation, van den Berg--H\'enot--Lessard \cite{vandenBerg-Henot-Lessard:radial-solutions-semilinear-elliptic} for semilinear elliptic equations, Dahne--G\'omez-Serrano--Hou and G\'omez-Serrano--Orriols 
\cite{Dahne-GomezSerrano-Hou:counterexample-payne,GomezSerrano-Orriols:negative-hearing-shape-triangle} for inverse spectral problems, Jaquette--Lessard--Takayasu 
\cite{Jaquette-Lessard-Takayasu:global-dynamics-nonconservative-nls} for the non-conservative NLS equations, Dahne--G\'omez-Serrano 
\cite{Dahne-GomezSerrano:highest-wave-burgershilbert} for the Burgers-Hilbert equation, 
Takayasu--Lessard--Jaquette--Okamoto \cite{Takayasu-Lessard-Jaquette-Okamoto:nonlinear-heat-complex-time} for the complex in time nonlinear heat equation and Breden--Engel 
\cite{Breden-Engel:cap-chaos-hopf-systems} for chaos in stochastically perturbed Hopf systems. We make no claim that this list is exhaustive, but we would like to emphasize the broad directions of the problems that the field (computer-assisted proofs in PDE) has been able to undertake over the last few years.

 In the context of fluid mechanics we highlight the following authors and equations: Kobayashi \cite{Kobayashi:global-uniqueness-stokes} and Stokes' extreme waves; Chen--Hou--Huang \cite{Chen-Hou-Huang:blowup-degregorio} and De Gregorio; Castro--C\'ordoba--G\'omez-Serrano \cite{Castro-Cordoba-GomezSerrano:global-smooth-solutions-sqg} and SQG; Enciso--G\'omez-Serrano--Vergara \cite{Enciso-GomezSerrano-Vergara:convexity-cusped-whitham} and Whitham; Arioli--Koch, Figueras--De la Llave, Gameiro--Lessard, Figueras--Gameiro--Lessard--De la Llave, Zgliczynski, Zgliczynski--Mischaikow \cite{Arioli-Koch:cap-stationary-ks,Figueras-DeLaLLave:cap-periodic-orbits-kuramoto,Gameiro-Lessard:periodic-orbits-ks,Figueras-Gameiro-Lessard-DeLaLLave:framework-cap-invariant-objects,Zgl,ZM1} and Kuramoto--Shivasinsky; van den Berg--Breden--Lessard--van Veen, Arioli--Gazzola--Koch, Bedrossian--Punshon-Smith \cite{vandenBerg-Breden-Lessard-vanVeen:periodic-orbits-ns,Arioli-Gazzola-Koch:uniqueness-bifurcation-ns,Bedrossian-PunshonSmith:chaos-stochastic-2d-galerkin-ns} and Navier-Stokes.
%

We also refer the reader to the books \cite{Moore-Bierbaum:methods-applications-interval-analysis,Tucker:validated-numerics-book} and to the survey \cite{GoSe19} and the book \cite{Nakao-Plum-Watanabe:cap-for-pde-book} for a more specific treatment of computer-assisted proofs in 
PDE.

%
%
In our concrete case, we will use the computer in two different parts of our strategy:
\begin{enumerate}
\item Computing (with rigorous bounds) a high amount of Taylor coefficients of a solution of an ODE at a singular point (Lemma \ref{lemma:aux_WioverZi_7o5}, Lemma \ref{lemma:tenthousand_7o5}).
\item Validating the sign of polynomials of degree 7-11 with coefficients depending on 2 parameters using a branch and bound method (Proposition \ref{prop:left_global}, Proposition \ref{prop:left_local_3}, Proposition \ref{prop:right_f}).
\end{enumerate}

There is a rich history of papers that have used any of these two strategies in other contexts. 

Regarding the first one, in the context of the parameterization method of stable and unstable manifolds for ODE, Cabr\'e--Fontich--de la Llave 
\cite{Cabre-Fontich-delaLlave:parameterization-method-I,Cabre-Fontich-delaLlave:parameterization-method-II,Cabre-Fontich-delaLlave:parameterization-method-III}; for parabolic PDE, van den Berg--Jaquette--Mireles-James \cite{vandenBerg-Jaquette-MirelesJames:validated-stable-manifolds-parabolic-pde} and Barker--Mireles-James--Morgan
\cite{Barker-MirelesJames-Morgan:parameterization-method-unstable-standing-waves}; and for DDEs, H\'enot--Lessard--Mireles-James \cite{Henot-Lessard-MirelesJames:parameterization-unstable-manifolds-ddes} are examples. See also the book by Haro--Canadell--Figueras--Luque--Mondelo \cite{Haro-Canadell-Figueras-Luque-Mondelo:parameterization-method-book} for a more comprehensive list of references. A similar strategy has been employed to solve ODE by means of a Taylor expansion and automatic differentiation by Berz--Makino 
\cite{Berz-Makino:verified-integration-ode-taylor} and is also implemented in the CAPD library by Kapela--Mrozek--Wilczak--Zgliczynski 
\cite{Kapela-Mrozek-Wilczak-Zgliczynski:capd} and the COSY INFINITY library by Makino--Berz \cite{Makino-Berz:cosy-infinity-9}.

Regarding the second one, we highlight examples of computer-assisted proofs involving branch and bound methods, for example the work of Day--Kalies--Wanner \cite{Day-Kalies-Wanner:verified-homology-nodal-domains} in homology, Tanaka \cite{Tanaka:aposteriori-sign-change-elliptic-pde} in elliptic PDE, G\'omez-Serrano--Granero-Belinch\'on \cite{GomezSerrano-GraneroBelinchon:turning-muskat-computer-assisted} in the Muskat problem, B\'anhelyi--Csendes--Krisztin--Neumaier \cite{Banheli-Csendes-Krisztin-Neumaier:global-attractivity-wright} in Wright's conjecture, Hales \cite{Hales:kepler-annals} in the Kepler conjecture, or Kearfott \cite{Kearfott:interval-branch-and-bound} in constrained optimization problems . See also the book by Hansen--Walster \cite{Hansen-Walster:global-optimization} for more references.

We remark that in the recent paper \cite{GuHaJaSc21},  the authors Guo,  Had\v zi\'c, Jang and Schrecker apply arguments of a very similar flavor (Taylor expansions, dynamical systems arguments and computer-assisted proofs) to construct smooth self-similar solutions the gravitational Euler-Poisson system.

\subsection{Stability of the Euler solutions and the existence of asymptotically self-similar Navier-Stokes solutions}

Let us begin by rewriting \eqref{eq:NS} under spherical symmetry, and in terms of the rescaled sound speed $\sigma$:
\begin{align} \begin{split} \label{eq:ice}
 \p_t u + u \p_R u + \alpha\sigma \p_R \sigma-\frac{1}{R^2\rho}\partial_R(R^2 \partial_Ru)+ \frac{2u}{R^2\rho}&=0\,, \\
 \p_t \sigma + u \p_R \sigma+\frac{\alpha}{R^2}\sigma\p_R(R^2u)&=0\,.
 \end{split} \end{align}
 We recall that for simplicity, we fixed $\mu_1=1$ and $\mu_2=-1$.
 We again define our Riemann invariants as in \eqref{eq:Riemann:invariants}. However, in place of the ansatz \eqref{eq:ansatz2}, we instead consider the following time dependent ansatz
   \begin{align*}
  w(R,t)&=r^{-1}(T-t)^{r^{-1}-1} \mathcal W(\tfrac{R}{(T-t)^{\frac1r}},-\tfrac{\log(T-t)}{r})\,,\\
 z(R,t)&=r^{-1} (T-t)^{r^{-1}-1} \mathcal Z(\tfrac{R}{(T-t)^{\frac1r}}, -\tfrac{\log(T-t)}{r})\,.
 \end{align*}
 We then define the self-similar variables
   \[s=-\frac{\log(T-t)}{r},\quad \zeta=\frac{R}{(T-t)^{\frac{1}{r}}}=e^s R=\exp(\xi)\,.\]
 The equation \eqref{eq:ice} then becomes
 \begin{align} \begin{split} \label{eq:main}&
(\partial_s +r-1)\mc W+(\zeta+\frac12(\mc W+\mc Z+\alpha(\mc W-\mc Z)))\p_{\zeta}  \mc W  +\frac{\alpha}{2\zeta}(\mc W^2-\mc Z^2)\\
&\qquad\qquad=\frac{r^{1+\frac{1}{\alpha}} 2^{1/\alpha}}{\alpha^{1/\alpha} \zeta^2((\mc W-\mc Z))^{\frac1\alpha}}e^{(2-r+\frac{1}{\alpha}(1-r))s}\left( \partial_{\zeta}(\zeta^2\partial_{\zeta}(\mc W+\mc Z))-2(\mc W+\mc Z)\right)\,,\\
& (\partial_s +r-1)\mc Z+(\zeta+\frac12(\mc W+\mc Z-\alpha(\mc W-\mc Z)))\p_{\zeta}  \mc Z  - \frac{\alpha}{2\zeta}(\mc W^2-\mc Z^2) \\
&\qquad\qquad=\frac{r^{1+\frac{1}{\alpha}} 2^{1/\alpha-1}}{\alpha^{1/\alpha} \zeta^2((\mc W-\mc Z))^{\frac1\alpha}}e^{(2-r+\frac{1}{\alpha}(1-r))s}\left( \partial_{\zeta}(\zeta^2\partial_{\zeta}(\mc W+\mc Z))- 2(\mc W+\mc Z)\right)\,.
 \end{split} \end{align} 
 The last term can be treated as an error so long as \begin{equation}\label{eq:delta:dis}
 -\delta_{\rm dis}=
2-r+\frac{1}{\alpha}(1-r)<0\,,\end{equation}
or equivalently
\begin{equation}\label{eq:r:restriction}
 r> \frac{2\gamma}{\gamma+1}.
\end{equation}
Given that we intend to restrict $r<r^\ast$, by the definition of $r^\ast$ given in \eqref{eq:rstar}, we conclude that we require 
\[\ga<1+\frac{2}{\sqrt 3}\,,\]
which is clearly satisfied for $\gamma=\frac75$.

Given a smooth globally defined self-similar solution to the Euler equation corresponding to a self-similar variable $r$ satisfying \eqref{eq:r:restriction}, to prove the existence of an asymptotically self-similar solution to the Navier-Stokes equations, it will be sufficient to show the nonlinear stability of such a solution modulo finite modes of instability. The main ingredient is to first show linear stability of the self-similar Euler solutions. This was achieved in \cite{MeRaRoSz19b} by writing the equation as a nonlinear wave equation and proving stability in terms of carefully weighted spaces. In the present work, a simpler approach is taken exploiting locality and the transport structure of the equation written in Riemann variables. In place of weighted spaces, we modify the equation outside a neighborhood of the backwards acoustic cone of the singularity in order to restrict the region of interest. Differing from the work \cite{MeRaRoSz19b}, we exploit the transport structure of the Riemann invariants in order to simplify the stability analysis. The nonlinear stability will rely on a topological argument in a similar vain to \cite{MeRaRoSz19b,OhPa21} (see \cite{BuIy2020} for an alternate approach based on a Newton scheme).
 
\subsection{Organization of paper}

Section \ref{sec:expansions} describes Frobenius-like series expansions of the solution at $P_s$ and $P_0$. In Section \ref{sec:left} and Section \ref{sec:right}, we apply barrier arguments to describe the solution in the region outside (respectively inside) the backwards acoustic cone of the singularity. In Section \ref{sec:left7on5}, we complete the analysis of Section \ref{sec:left} for the case of $\gamma=7/5$ and $r\rightarrow r^\ast$. The section collects analysis related to the $r\rightarrow r^\ast$ asymptotic limit. In Section \ref{sec:mainproof}, Theorem \ref{th:mainlarge} and Theorem \ref{th:mainr3} are proved by combining the result of Proposition \ref{prop:right_main} with a shooting argument in order to connect $P_s$ to $P_0$ by a smooth solution. Section \ref{sec:linear} is dedicated to showing that the linearized operator of the Euler equations around the self-similar profile generates a contraction semigroup modulo finitely many instabilities. Finally in Section \ref{sec:nonlinarstability}, we use the linear stability analysis of Section \ref{sec:linear} in combination with a bootstrap and a topological argument in order to prove nonlinear stability for the Navier-Stokes equations for a manifold of initial data of finite codimension. In particular, Section \ref{sec:nonlinarstability} contains the proof of Theorem \ref{th:stability}. Appendix \ref{sec:aux} contains technical lemmas used throughout the proofs and properties of the phase portrait in the case $\gamma = \frac75$. Appendix \ref{sec:computer} summarizes the details of the computer-assisted proofs.

\section{Expansion around $P_s$ and $P_0$}\label{sec:expansions}

In this section, we describe the Frobenius-like series expansions of the smooth solutions passing through $P_s$ and starting at $P_0$. The general approach will be to obtain a recurrence for the coefficients of the expansion from imposing the ODE on the expansion. In the case of $P_s$, we will obtain that the recurrence can be solved for $k \notin \mathbb N$ (due to a factor of the type $m-k$ on the equation for the $m$-th coefficient). In the case of $P_0$, we will need to ensure that the appropriate conditions are met so that the profiles, in cartesian variables, are smooth at $R=0$. The most elegant way of doing so will be reexpressing our recurrence in terms of a new function $\mathcal W$, that encodes $W$ for positive arguments and a reflected version of $Z$ for negative arguments. The smoothness of $\mathcal W$ at the origin will yield smoothness of both our profiles at the origin.

\subsection{First order expansion}

We label the two solutions to $D_Z = N_Z = 0$:
\begin{align} \label{eq:Ps} \begin{split}
P_s = (W_0, Z_0) &= \Big( \frac{\ga^2 r+(\ga+1) \mathcal{R}_1 -3 \ga^2-2 \ga r+10 \ga-3 r-3}{4 (\ga-1)^2} \\
& \qquad \frac{\ga^2 r+(\ga -3) \mathcal{R}_1 -3 \ga^2-6 \ga r+6 \ga+9 r-7}{4 (\ga-1)^2} \Big)\,,
\end{split} \\
\label{eq:Psbar} \begin{split}
\bar P_s = (\bar W_0, \bar Z_0) &= \Big( \frac{\ga^2 r-(\ga+1) \mathcal{R}_1-3 \ga^2-2 \ga r+10 \ga-3 r-3}{4 (\ga-1)^2}\ \\
  & \qquad \frac{\ga^2 r+(3-\ga) \mathcal{R}_1 -3 \ga^2-6 \ga r+6 \ga+9 r-7}{4 (\ga-1)^2} \Big)\,,
\end{split}
\end{align}
where
\begin{equation} \label{eq:R1}
\mathcal{R}_1 = \sqrt{\ga^2 (r-3)^2-2 \ga (3r^2-6r+7)+(9r^2-14r+9)}\,.
\end{equation}
The points $P_s$ and $\bar P_s$ are the only intersections of $D_Z = N_Z = 0$, so any possible smooth profile going from $P_0$ to $P_\infty$ will need to pass through one of them in order to cross $D_Z = 0$.

Let us consider the derivative of the smooth solutions to the ODE \eqref{eq:wz:ODE} with respect to $\xi$ at $P_s$. The derivative of  $W$  at $P_s$ to both solutions is given by 
\begin{equation}\label{eq:princeton:uni:parking1}
W_1 = \frac{N_W (P_s)}{  D_W (P_s)}\,.
\end{equation}
 Applying L'H\^{o}pital, the derivative of $Z$ at $P_s$
satisfies the second-degree equation 
\begin{equation}\label{eq:princeton:uni:parking2}
\nabla D_Z (P_s)\cdot (W_1, Z_1) Z_1 = \nabla N_Z (P_s)\cdot  (W_1, Z_1)\,,
\end{equation}
 which leads to  two possible values of $Z_1$ corresponding to the two smooth solutions passing through $P_s$. The solution that we will work with corresponds to the vector $\nu_-=(W_1,Z_1)$ where
\begin{align} \label{eq:W1Z1} \begin{split}
W_1 &= \frac{\ga \left(-3 \left(\mathcal{R}_1+6\right)-3 \ga (r-3)+2 r\right)+\mathcal{R}_1+5 r+5}{4 (\ga-1)^2} \,,\\
Z_1 &=  \frac{-\left(3 \gamma ^3-7 \gamma ^2+\gamma +11\right) r+\gamma  (\gamma  (9 \gamma -3 \mc R_1-25)+10 \mc R_1-4(\gamma-1) \mc R_2+27)-3 \mc R_1+4 (\gamma-1) \mc R_2-3}{4 (\gamma -1)^2 (\gamma +1)}\,,
\end{split} \end{align}
and 
\begin{align} \begin{split} \label{eq:def_R2}
\mathcal{R}_2 &=\frac{1}{\gamma-1}\bigg(\gamma  ((76-27 \gamma ) \gamma -71)-\left((3 \gamma -5) ((\gamma -5) \gamma +2) r^2\right)+(\gamma  (\gamma  (18 \gamma -52)+50)-8) 
   r\\
   &\qquad+\mc R_1 (9 (\gamma -2) \gamma +((2-3 \gamma ) \gamma +5) r+5)+18\bigg)^{\frac12}\,.
\end{split} \end{align}
This value of $Z_1$ corresponds to the smooth solution that agrees up to order $\lfloor k \rfloor$ with all the non-smooth solutions around $P_s$ (the simple example given in  Section \ref{ss:imploding:solution} is illustrative of this behavior).

Define the vector $\nu_+$ by $\nu_+=(W_1,\check Z_1)$, where
\[ \label{eq:checkZ1}
\check Z_1 =  -\frac{\left(3 \gamma ^3-7 \gamma ^2+\gamma +11\right) r-\gamma  (\gamma  (9 \gamma -3 \mathcal{R}_1-25)+10 \mathcal{R}_1+4 (\gamma-1)\mathcal{R}_2+27)+3 \mathcal{R}_1+4(\gamma-1)
   \mathcal{R}_2+3}{4 (\gamma -1)^2 (\gamma +1)}
\,.\]
Then the two smooth solutions at $P_s$ have $\xi$ derivatives $\nu_+=(W_1, Z_1)$ and $\nu_-=(W_1, \check Z_1)$.

The vectors $\nu_-$, $\nu_+$ are eigenvectors of the Jacobian $J$ at $P_s$ of the reparameterized system \eqref{eq:wz:ODE2}. In particular, the Jacobian $J$ at $P_s$ is  given by
\begin{equation} \label{eq:JacobianP2} J = \left( \begin{matrix}
N_W (P_s) \p_W D_Z (P_s) & N_W (P_s) \p_Z D_Z (P_s) \\
D_W (P_s) \p_W N_Z (P_s) & D_W (P_s) \p_Z N_Z (P_s) \\ 
\end{matrix} \right). \end{equation}
Note if $\nu = (\nu_W, \nu_Z)$ is an eigenvector of $J$ then it must satisfy the equation 
\begin{equation}\label{eq:princeton:uni:parking3}
\left( N_W (P_s) \nabla D_Z (P_s)\cdot \nu, D_W(P_s) \nabla N_Z (P_s)\cdot \nu \right) \wedge \nu = 0\,.
\end{equation}
Then, applying \eqref{eq:princeton:uni:parking1} and \eqref{eq:princeton:uni:parking2} we see that \eqref{eq:princeton:uni:parking3} is satisfied for $\nu=\nu_-,\nu_+$ and hence $\nu_-,\nu_+$ are eigenvectors of $J$. We let $\lambda_+$ and $\lambda_-$ be the eigenvalues corresponding to $\nu_-$ and $\nu_+$ respectively. We obtain the equations
\begin{align} 
\lambda_- &= \frac{N_W (P_s)}{W_1} \left( \p_W D_Z (P_s) W_1 + \p_Z D_Z (P_s) Z_1 \right), 
\label{eq:lambda-}\\
\lambda_- + \lambda_+ &= N_W(P_s)  \p_W D_Z (P_s) + D_W (P_s)  \p_Z N_Z (P_s)\,.
\label{eq:trace}
\end{align}

 \begin{lemma} \label{lemma:k} Let $D_{Z, 1} = \nabla D_Z (P_s) (W_1, Z_1)$. Then, 
\begin{equation} \label{eq:defk}
k(r) =   - \frac{ Z_1 \p_Z D_Z (P_s) - \p_Z N_Z (P_s) }{D_{Z,1}}\,,
\end{equation}
where we recall in \eqref{eq:k:def} we defined $k(r) = \frac{\lambda_+ }{\lambda_-}$. Moreover, we have that $k(r)$ is a smooth monotonically increasing function for $r\in [1, r^\ast (\gamma ))$ such that $r(1) = 1$, $\lim_{r\to r^\ast } k(r) = +\infty$ and $k'(r) > 0$ for all $r \in (1, r^\ast (\gamma ))$. Thus, $k(r)$ is a bijection between $[1, r^\ast (\gamma ))$ and $[1, +\infty)$.
\end{lemma}
\begin{proof} 
Note  that the parenthesis in \eqref{eq:lambda-} is $D_{Z, 1}$. Then from \eqref{eq:lambda-}  and \eqref{eq:trace} we obtain
\begin{align*}
k(r)& = \frac{(\lambda_+ + \lambda_-) - \lambda_-}{\lambda_-}\\
&= \frac{ W_1 \p_W D_Z (P_s) + \frac{W_1 D_W (P_s) }{N_W(P_s)} \p_Z N_Z(P_s) - \left( \p_W D_Z (P_s) W_1 + \p_Z D_Z (P_s) Z_1 \right)}{D_{Z, 1}} \\
&=  \frac{  \frac{W_1 D_W (P_s) }{N_W(P_s)}  \p_Z N_Z (P_s) -  \p_Z D_Z (P_s) Z_1 }{D_{Z, 1}} \,.
\end{align*}
Noting also that $W_1 = \frac{N_W (P_s)}{D_W (P_s)}$, we get the expression \eqref{eq:defk}. 

In terms of $\check Z_1$, we have that the fact that $(W_1, \check Z_1)$ is an eigenvector of $J$ means
\begin{equation} \label{eq:lambda+}
\lambda_+ = \frac{ N_W(P_s) \nabla D_Z (P_s ) \cdot (W_1, \check Z_1)}{W_1}\,.
\end{equation}
Dividing this by \eqref{eq:lambda-}, we get that 
\begin{equation} \label{eq:k_asquotientDZ1}
k(r) = \frac{\check D_{Z, 1}}{D_{Z, 1}} = \frac{-4+(1+\ga)\frac{r-1}{\ga-1} - \mathcal R_2}{-4+(1+\ga)\frac{r-1}{\ga-1} + \mathcal R_2},
\end{equation} where $\check D_{Z, 1} = \nabla D_Z \cdot (W_1, \check Z_1)$ and we have substituted $D_{Z, 1}$, $ \check D_{Z, 1}$ by their expressions in terms of $\gamma, r, \mc R_i$. 

Now, we claim that $k(1) = 1$. From \eqref{eq:k_asquotientDZ1}, it suffices to show that $\mc R_2 = 0$ at $r=1$. From \eqref{eq:R1} we have $\mc R_1 = 2(\gamma - 1)$ for $r=1$. Plugging that into \eqref{eq:def_R2}, we deduce that $\mc R_2 = 0$ for $r=1$.

We also claim that $k(r) \rightarrow +\infty$ as $r \rightarrow r^\ast$. From Lemma \ref{lemma:aux_DZ1}, we have that $D_{Z, 1} > 0$ for $r \in [1, r^\ast)$ and from Lemma \ref{lemma:aux_DZ1check}, we have $\check D_{Z, 1} > 0$ for $r \in [1, r^\ast ]$. Using that $D_{Z, 1} = 0$ for $r = r^\ast$ (Lemma \ref{lemma:aux_DZ1_cancellation}), we conclude the desired limit. 

Lastly, we show that $k'(r) > 0$ from equation \eqref{eq:k_asquotientDZ1} via a computer-assisted proof. The code can be found in the supplementary material and details about the implementation can be found in Appendix \ref{sec:computer}.
\end{proof}

\subsection{Taylor expansion around $P_s$  ($\xi=0$)} \label{sec:taylor}

Let $(W^{(r)}(\xi ) ,Z^{(r)}(\xi ) )$ denote the smooth solution corresponding to the direction $\nu_-$ defined in the previous section. Now consider its Taylor expansion around $P_s$, i.e.\  $\xi=0$:
\begin{align} \label{eq:series}\begin{split}
W^{(r)} (\xi ) = \sum_{n=0}^\infty \frac{1}{n!} W_n \xi ^n, \quad\mbox{and}\quad
Z^{(r)} (\xi ) = \sum_{n=0}^\infty \frac{1}{n!} Z_n \xi ^n.
\end{split} \end{align}
Let us also define the Taylor coefficients of $D_W, D_Z, N_W, N_Z$ as follows
\begin{align} \begin{split}
D_W(W^{(r)} (\xi ) , Z^{(r)} (\xi )) &= \sum_{n=0}^\infty  \frac{1}{n!} D_{W, n}  \xi^n, \qquad
D_Z(W^{(r)} (\xi ) , Z^{(r)} (\xi )) = \sum_{n=0}^\infty  \frac{1}{n!} D_{Z, n}  \xi^n, \\
N_W(W^{(r)} (\xi ) , Z^{(r)} (\xi )) &= \sum_{n=0}^\infty  \frac{1}{n!} N_{W, n}  \xi^n, \qquad
N_Z(W^{(r)} (\xi ) , Z^{(r)} (\xi )) = \sum_{n=0}^\infty  \frac{1}{n!} N_{Z, n}  \xi^n.
\end{split}\notag \end{align}
For $ \circ \in \{ W, Z \}$, $D_\circ (W, Z)$ are first-degree polynomials in $W$ and $Z$, hence
\begin{equation} \label{eq:Dn}
    D_{\circ , n}  = \nabla D_\circ \cdot (W_n, Z_n), \qquad \circ \in \{ W, Z \}, \;\; n\geq 1,
\end{equation} 
with the special case $D_{\circ , 0} = D_\circ (P_s)$. In the case of $N_\circ $, we have second-degree polynomials, so the gradient is not constant in $(W, Z)$. However, the Hessian matrix is, and we get the expression
\begin{equation} \label{eq:Nn}
    N_{\circ , n} = \nabla N_\circ \Big|_{P_s} \cdot (W_n, Z_n) + \sum_{j=1}^{n-1} \binom{n-1}{j-1} (W_{n-j}, Z_{n-j})  (HN_\circ )(W_j, Z_j)^\top, \qquad \circ \in \{ W, Z \},\;\; n\geq 1,
\end{equation}
with the special case $N_{\circ , 0} = N_\circ (P_s)$. 

\begin{proposition} \label{prop:taylor} For $n\in\mathbb N$, let $r\in(r_n,r_{n+1})$. If there is a solution to ODE \eqref{eq:wz:ODE} passing through $P_s$ at $\xi = 0$ with gradient $(W_1, Z_1)$, their Taylor coefficients $W_m$, $Z_m$  for $m\geq 2$ satisfy the recursion relation
\begin{align}
 \label{eq:Wn}
D_{W, 0} W_m &=  N_{W, m-1} - \sum_{j=0}^{m-2} \binom{m-1}{j}  D_{W, m-1-j} W_{j+1}\,,\\
 Z_m D_{Z, 1} (m-k) &= - \sum_{j=1}^{m-2} \binom{m}{j} D_{Z, m-j} Z_{j+1}  \notag \\
 &\qquad+\left(  N_{Z, m} - (\p_Z N_Z (P_s) ) Z_m\right) - Z_1  W_m \p_W D_Z (P_s) \, .\label{eq:Zn}
\end{align}
Moreover, the coefficients are iteratively solvable as both the coefficients $D_{W, 0}$ and $D_{Z, 1} (m-k)$ are non-zero and the expansion of the second line in \eqref{eq:Zn} contains no terms involving $Z_m$.
\end{proposition}
 \begin{proof} 
Taking $m-1$ derivatives in equation \eqref{eq:wz:ODE}, we obtain
\begin{equation}
\p_\xi^{m-1} (D_W (W, Z)\partial_\xi W ) = \p_\xi^{m-1} (N_W (W, Z))\,.\notag
\end{equation}
 Expanding the derivatives, we immediately obtain \eqref{eq:Wn}. Moreover, $D_{W, 0} \neq 0$ as a consequence of Lemma~\ref{lemma:aux_DW0}. 
 Taking $m$ derivatives in the equation  \eqref{eq:wz:ODE}, we analogously obtain
\begin{align*}
mD_{Z, 1} Z_m &=  N_{Z, m} - \sum_{j=0}^{m-2} \binom{m}{j}  D_{Z, m-j} Z_{j+1}. 
\end{align*}
Now, we subtract the terms with $Z_m$ in the quantities $N_{W, m}$ and $D_{Z, m}$, obtaining
\begin{align*}
Z_m \left( m D_{Z, 1} - \p_Z N_Z (P_s) + Z_1 \p_Z D_Z \right) &= -\sum_{j=1}^{m-2} \binom{m}{j}  D_{Z, m-j} Z_{j+1} \\
&\qquad+ \left( N_{Z, m} - \p_Z N_Z (P_s) Z_m \right) + Z_1 \left( - D_{Z, m} + Z_m \p_Z D_Z \right) . 
\end{align*}
Note that the terms in the second line do not depend on $Z_m$, as we have subtracted the dependence of $D_{Z, m}$ and $N_{Z, m}$ on $Z_m$ (see equations \eqref{eq:Dn} and \eqref{eq:Nn}). Then, applying \eqref{eq:defk} we obtain \eqref{eq:Zn}. We have  $D_{Z, 1} \neq 0$ as a result of Lemma \ref{lemma:aux_DZ1}.
 \end{proof} 
 
 \begin{proposition} \label{prop:smooth}
  Let $n \in \N$ and $I \subset (r_n, r_{n+1})$ a closed interval. There exists an absolute constant $C$ (depending on $I$ and $\ga$) such that we have the bounds
  \begin{align} \begin{split} \label{eq:ind_IH_WZ}
  |W_i| +  |Z_i|&\leq C^{i+1} i!\,.
  \end{split} \end{align}
   For $\xi < 1/C$, the series $W^{(r)} (\xi )  = \sum W_i \xi^i /i!$ and $Z^{(r)}(\xi ) = \sum Z_i \xi^i / i!$ solve the ODE \eqref{eq:wz:ODE}. Moreover, the functions $W^{(r)} (\xi ) , Z^{(r)} (\xi )$ are continuous with respect to $r \in I$.
  \end{proposition}
  \begin{proof}
 In this proof, we will use the following notation 
 \[w_i=W_i/i!,\quad z_i=Z_i/i!,\quad d_{\circ,i}=D_{\circ,i}/i!,\quad\mbox{and}\quad n_{\circ,i}=N_{\circ,i}/i!\,.\]
 Then, from equations \eqref{eq:Dn}--\eqref{eq:Zn} we have
 \begin{align}\label{eq:rock:wallaby1}
D_{W, 0} w_m m &=  n_{W, m-1} - \sum_{j=0}^{m-2}  d_{W, m-1-j} w_{j+1}(j+1)\,,\\
 z_m  D_{Z, 1} (m-k) &= - \sum_{j=1}^{m-2}  d_{Z, m-j} (j+1)z_{j+1}  \notag \\
 &\qquad+\underbrace{\left(   n_{Z, m} - (\p_Z N_Z (P_s) ) z_m\right)}_{g_m} - Z_1  w_m \p_W D_Z (P_s)  \,,\label{eq:rock:wallaby2}\\
 d_{\circ , m}  &= \nabla D_\circ \cdot (w_m, z_m)\label{eq:rock:wallaby3}\,,\\
    n_{\circ , m} &=  \nabla N_\circ \Big|_{P_s} \cdot (w_m, z_m) + \frac{1}{m}\sum_{j=1}^{m-1} j (w_{m-j}, z_{m-j})  (HN_\circ )(w_j, z_j)^\top\,.\label{eq:rock:wallaby4}
\end{align}
We recall $D_{W, 0}, D_{Z, 1} \neq 0$ from Lemmas \ref{lemma:aux_DW0} and \ref{lemma:aux_DZ1}, so they are lower bounded in $I$. Let $\mathfrak C_i$ denote the Catalan numbers, then for some constant $M$, we inductively assume the bounds
\begin{equation}\label{eq:induction:Dingo}
	\abs{w_{i}}+\abs{z_i}+\abs{d_{\circ,i}}+\abs{n_{\circ,i}}\leq \mathfrak C_i M^{i-2}\,,
\end{equation}
for $3\leq i\leq m$. Since the constant $C$ in \eqref{eq:ind_IH_WZ} is allowed to depend on $n$, choosing $C$ sufficiently large, \eqref{eq:ind_IH_WZ} trivially holds for all $i\leq 2n+2$. Let us now assume $m\geq  2n+2$. By \eqref{eq:rock:wallaby1} we have
\begin{align}
	\abs{w_{m+1}} &\les\frac{1}{m} \left(\abs{ n_{W, m}} + \sum_{j=2}^{m-3}  (j+1)\abs{d_{W, m-j} }\abs{w_{j+1}} + \sum_{ \substack{ j \in \{0, 1, \\ m-2, m-1 \} }}^{m-1}  (j+1)\abs{d_{W, m-j} }\abs{w_{j+1}}	\right)\notag\\
	&\les \mathfrak C_m M^{m-2}+M^{m-3}\sum_{j=0}^{m-2} \mathfrak C_{m-j} \mathfrak C_{j+1} +  \mathfrak C_m M^{m-2}\notag\\
	&\les M^{m-2}\mathfrak C_{m+2}\,.\label{eq:Barton}
\end{align}
Now we bound $g_{m+1}$ using \eqref{eq:induction:Dingo} and \eqref{eq:Barton}
\begin{align}
	\abs{g_{m+1}}&\les \abs{w_{m+1}} + \frac{1}{m+1}\sum_{j=1}^{m} j (\abs{w_{m-j}}+\abs{ z_{m-j}}) (\abs{w_j}+ \abs{z_j})\notag\\
	&\les M^{m-2}\mathfrak C_{m+2}\,.
	\label{eq:Deakin}
\end{align}
We bound $z_{m+1}$ using \eqref{eq:induction:Dingo}, \eqref{eq:Barton}, \eqref{eq:Deakin} and $k \leq n+1 \leq \frac{m}{2}$:
\begin{align}
	\abs{z_{m+1}}  & \les\frac{1}{m} \left(\sum_{j=1}^{m-1}  \abs{d_{Z, m+1-j} (j+1)z_{j+1} } 
+\abs{g_{m+1}} + \abs{  w_{m+1}} \right)\notag\\
&\les M^{m-2}\sum_{j=1}^{m-1} \mathfrak C_{m+1-j}\mathfrak C_{j+1}+M^{m-2}\mathfrak C_{m+2}\notag\\
&\les M^{m-2}\mathfrak C_{m+3}\,.\label{eq:Watson}
\end{align}
 Finally, from \eqref{eq:rock:wallaby3}--\eqref{eq:Barton} and \eqref{eq:Watson}, we obtain
 \begin{equation}\label{eq:Reid}
\abs{d_{\circ,m+1}}+\abs{n_{\circ,m+1}}\les \mathfrak C_{m+3} M^{m-2}\,.
 \end{equation}
 Then, \eqref{eq:induction:Dingo} follows for $i=m+1$ by \eqref{eq:Barton}, \eqref{eq:Watson}, \eqref{eq:Reid}, the assumption that $M$ is chosen to be sufficiently large and $\mathfrak C_{m+3} \leq 4\mathfrak C_{m+2} \leq 16 \mathfrak C_{m+1}$.

 Choosing $C$ sufficiently large, from \eqref{eq:induction:Dingo}, we obtain \eqref{eq:ind_IH_WZ}. In particular, the series has a radius of convergence of at least $1/C$, independently of $r\in I$ (although depending on $I$).
  
   Lastly, we need to prove the continuity of the series with respect to the parameter $r$. We introduce the dependence of $W_j, Z_j$ with $r$, denoting the coefficients by $W_j(r), Z_j(r)$. Let $W^{(r)} (\xi ) = \sum_j \frac{W_j(r)}{j!} \xi^j$, and denote similarly by $Z^{(r)}$ the series formed by $Z_j(r)$. We show continuity with respect to $r$ for $W^{(r)}$, an analogous proof applies for $Z^{(r)}$. Let $\eps > 0$, $\delta > 0$. For $r, \bar r \in I$ with $|r-\bar r| < \delta$, we can bound
  \begin{align*}
  |W^{(r)} (\xi ) - W^{(\bar r)}(\xi )| &\leq \sum_{j=0}^{N-1} |W_j(r) - W_j(\bar r ) | \frac{|\xi |^j }{j!} + \sum_{j=N}^\infty |W_j(r)| \frac{|\xi |^j}{j!} +   \sum_{j=N}^\infty |W_j(\bar r)| \frac{|\xi |^j}{j!} \\
  &\leq \sum_{j=0}^{N-1} |W_j(r) - W_j(\bar r ) | \frac{|\xi |^j }{j!} + 2C\sum_{j=N}^\infty C^j |\xi|^j   \,.
  \end{align*}
  As $\xi < 1/C$, we can take $N$ large enough so that the last sum is smaller than $\eps /2$. As the coefficients $W_j(r)$ are continuous, there also exist $\delta_j $ such that $|W_j(r) - W_j (\bar r)| < \eps/(2N)$ as long as $|r-\overline{r}| < \delta_j$. Therefore, taking $|\delta |< \min_{j=0,\ldots N-1} \delta_i$ we have that $|W^{(r)} - W^{(r+\delta )}| < \eps$, and this shows continuity with respect to $r$.
 \end{proof}
 
As one can see from \eqref{eq:Zn}, the $n$-th coefficient of the Taylor series is of order $O(|k-n|^{-1})$, as $k \to n$ (equivalently, $r \to r_n$). All the previous coefficients of the Taylor series are not singular as $k\to n$. However, the higher order coefficients will not be $O(1)$, since they depend on $Z_n$ via the Taylor series recursion. The following Corollary studies the order in $\frac{1}{|k-n|}$ of the higher order Taylor coefficients.

 \begin{corollary} \label{cor:asymptotics} We have the following asymptotics for $r$ in a neighborhood of $r_n$, with $n \in \N$. 
 \begin{equation}  \label{eq:asymptotics} 
 \begin{split}
 \abs{W_m} \les_m 1+|k-n|^{-\left\lfloor \frac{m-2}{n-1} \right\rfloor } \quad\mbox{and}\quad
 \abs{Z_m} \les_m 1+|k-n|^{-\left\lfloor \frac{m-1}{n-1} \right\rfloor  } \,.
 \end{split} \end{equation}
 In particular, for
   $m\leq n$, we have $ \abs{W_m}=O(1)$ and for $m<n$, we have $ \abs{Z_m}=O(1)$. We also have that $\abs{Z_n} = O(1 + |k-n|^{-1})$.
 \end{corollary}
 \begin{proof}
  From Proposition \ref{prop:taylor}, we can iteratively calculate $W_m$ and $Z_m$ with equations \eqref{eq:Dn}, \eqref{eq:Nn}, \eqref{eq:Wn} and \eqref{eq:Zn}. Each coefficient is a rational function of the previous coefficients. From Lemma~\ref{lemma:aux_DW0} and Lemma \ref{lemma:aux_DZ1}, we have $D_{W, 0}$ and $D_{Z, 1}$ remain bounded away from $0$ for $r$ in a neighborhood of $r_n$. We trivially have that for $m<n$, the factor $m-k$ in  \eqref{eq:Zn} also remains  bounded away from $0$. Then the result holds trivially for $W_m$ in the case $m\leq n$ and for $Z_m$ in the case $m<n$. The case $m=n$ for $Z_m$ similarly holds using in addition that $k(r)$ has non-zero derivative at $r=r_m$ (Lemma \ref{lemma:k}).

For $m > n$, each coefficient is a rational function of the previous one, with denominators only involving $D_{Z, 1}$, $D_{W, 0}$ and $m-k$, all of them bounded away from zero and infinity in a neighborhood of $r_m$ (as $m > n$). We will prove  \eqref{eq:asymptotics} via induction in $m$, supposing it holds for all coefficients of order $<m$.

We start proving the induction step for $W_m$. As a consequence of the induction hypothesis, we know that $D_{\circ, i}$ and $N_{\circ , i}$ are of the order $O \left( 1 + |k-n|^{-\left\lfloor \frac{i-1}{n-1} \right\rfloor} \right)$ for every $i < m$. From expression \eqref{eq:Wn} and the induction hypothesis, we have
\begin{align} \label{eq:Zn_asym_aux1}
\abs{W_m}\les  1+|k-n|^{-\left\lfloor \frac{m-2}{n-1} \right\rfloor} + \sum_{j=0}^{m-2}  (1+ |k-n|^{-\left\lfloor \frac{m-j-2}{n-1} \right\rfloor})(1+ |k-n|^{-\left\lfloor \frac{j-1}{n-1} \right\rfloor})\les 1+ |k-n|^{-\left\lfloor \frac{m-2}{n-1} \right\rfloor}\,,
\end{align}
where we have used the floor concavity property $\left\lfloor x \right\rfloor + \left\lfloor y \right\rfloor \geq \left\lfloor x+y \right\rfloor$, for all $ x, y \in \R^+$. Thus we obtain the desired estimate on $\abs{W_m}$.

Now consider  $Z_m$. For the first line of \eqref{eq:Zn}, we have
\begin{align}
\abs{\sum_{j=1}^{m-2} \binom{m}{j} D_{Z, m-j} Z_{j+1}}
&\les \sum_{j=1}^{m-2}  (1+ |k-n|^{-\left\lfloor \frac{m-j-1}{n-1}  \right\rfloor} )(1+ |k-n|^{-\left\lfloor \frac{j}{n-1} \right\rfloor} )\les 1+ |k-n|^{-\left\lfloor \frac{m-1}{n-1} \right\rfloor}\,. \label{eq:Tuatara1}
\end{align}

Now consider the second line of \eqref{eq:Zn}. Since the expansion of $N_{Z, m} - \p_Z N_Z (P_s) Z_m$ does not involve $Z_m$, we have from the induction hypothesis and \eqref{eq:Zn_asym_aux1} that
\begin{align}&
\abs{\left(  N_{Z, m} - (\p_Z N_Z (P_s) ) Z_m\right) - Z_1 W_m \p_W D_Z (P_s) }\notag\\
&\qquad
\les  \sum_{i=0}^{m-1} (\abs{W_i}+\abs{Z_i})\abs{W_{m-i}}+\sum_{i=1}^{m-1}\abs{Z_i}\abs{Z_{m-i}}+\abs{Z_1}\abs{W_m} \notag\\&\qquad
\les  \sum_{i=0}^{i-1}  (1+ |k-n|^{-\left\lfloor \frac{i-1}{n-1}  \right\rfloor} )(1+ |k-n|^{-\left\lfloor \frac{m-i-2}{n-1} \right\rfloor} )+\sum_{i=1}^{m-1}(1+|k-n|^{-\left\lfloor \frac{i-1}{n-1}  \right\rfloor})(1+  |k-n|^{-\left\lfloor \frac{m-i-1}{n-1} \right\rfloor})\notag\\&
\qquad\qquad+1 +|k-n|^{-\left\lfloor \frac{m-2}{n-1}  \right\rfloor} \notag\\
&\qquad\les 1+|k-n|^{-\left\lfloor \frac{m-2}{n-1} \right\rfloor} \,.\label{eq:Tuatara2}
\end{align}
Combining \eqref{eq:Tuatara1} and \eqref{eq:Tuatara2}, we obtain the desired estimate on $\abs{Z_m}$ and hence conclude  the proof by induction.
 \end{proof}
  
\subsection{Taylor expansion around $P_0$}

We now aim to construct smooth solutions emanating from $P_0$ and reaching $P_s$ at  $\xi=0$. Let us recall that $P_0$ is a point in the compactification of the phase portrait, that corresponds to $\xi = 0$, and where $S = +\infty$ and $U$ has finite value (in $W, Z$ coordinates, $P_0$ is at infinity along a line parallel to $W-Z = 0$). Due to the singular nature of the coordinate change $R\mapsto \xi$ near $R=0$, and the singular nature of $P_0$, it is useful to instead work in terms of the self-similar coordinate $\zeta=\exp(\xi)$. Moreover, we will extend the values of $\zeta$ to negative $\zeta$ as well considering a function $\mathcal W (\zeta)$ on the whole real line that is associated with $W$ for $\zeta > 0$ and with $Z$ for negative $\zeta$ (see a precise definition below). In particular, we will search  for a solution $\mc W$ to \eqref{eq:Euler:SS:alt2} for $\zeta\in[-1,1]$ satisfying $(\mc W(1),-\mc W(-1))=(W_0,Z_0)$. Such a solution would correspond to a profile 
\begin{equation} \label{eq:drwho}
W(\xi ) = \exp (-\xi ) \mc W (\exp (\xi )), \qquad Z(\xi ) = -\exp (-\xi ) \mc W(-\exp (\xi ) ),
\end{equation}
solving equation \eqref{eq:wz:ODE}. Thus, one can understand $Z$ as the natural continuation along $P_0$ ($\zeta = 0$) of the $W$ solution. Moreover, we will see that the singular nature of the point $P_0$ is captured in the factors $\exp(-\xi)$ of \eqref{eq:drwho}, so that the function $\mathcal W (\zeta)$ is smooth at $\zeta = 0$.
 
 \begin{proposition} \label{prop:solnearxi0}
 For any $A > 0$, there exists a solution $\mc W$ to \eqref{eq:Euler:SS:alt2} in a neighborhood of $\zeta=0$ which can be written in terms of a convergent power series
  \begin{align} \label{eq:P1expansion}
  \mc W(\zeta)=\sum_{i=0}^{\infty}  w_i \zeta^i \,,
  \end{align}
  such that $w_0=A$. Moreover, letting $\ga = 7/5$ and $r$ sufficiently close to $r^\ast (\gamma )$, or $\gamma > 1$ with $r \in (r_3, r_4)$, there exists a value of $A$ such that the solution can be continued to $\zeta\in[-1,1]$ and $(\mc W(1),-\mc W(-1))=(W_0,Z_0)$. The solution corresponding to that value of $A$ is continuous with respect to $r$.
 \end{proposition}
\begin{proof} 
Let us start by writing
\begin{align*}
\mathcal V=\zeta+\frac12(\mc W(\zeta)-\mc W(-\zeta)+\alpha(\mc W(\zeta)+\mc W(-\zeta)))  \quad\mbox{and}\quad \mathcal G=-\left((r-1)\mc W+\frac{\alpha}{2\zeta}(\mc W^2(\zeta)-\mc W^2(-\zeta))\right),
\end{align*}
so that \eqref{eq:Euler:SS:alt2} can be rewritten as
 \begin{equation}\label{eq:bilby}
\mathcal V \p_{\zeta}  \mc W=\mathcal G\,.
 \end{equation}
 The coefficients $w_i$ will be determined by substituting the series \eqref{eq:P1expansion} into the equation \eqref{eq:bilby} in order to obtain a recursion formula. Writing $\mc V(\zeta)=\sum_{i=0}^{\infty}  v_i \zeta^i$ and $\mc G(\zeta)=\sum_{i=0}^{\infty}  g_i \zeta^i$, yields the expression
 \begin{equation}\label{eq:bilby2}
  (n+1)v_0 w_{n+1}=g_n-\sum_{i=0}^{n-1}(i+1)v_{n-i}w_{i+1}\,,
  \end{equation}
  where
  \begin{equation}\label{eq:bilby3}
  \begin{split}
  v_i&=\mathbbm{1}_{i=1}+\frac{w_i}2\left(1+\alpha+(1-\alpha)(-1)^{i+1}\right)\,,\\
    g_i&=\underbrace{(1-r)w_i-\mathbbm{1}_{2\mid i}\alpha \sum_{j=1}^{i}w_jw_{i-j+1}}_{\bar g_i}-2\mathbbm{1}_{2\mid i}\alpha w_0 w_{i+1}\,.
  \end{split}
 \end{equation}
Rewriting \eqref{eq:bilby2} and using $v_0 = \alpha w_0$, we obtain 
 \begin{equation}\label{eq:bilby4}
\alpha w_0 (n+1+2\mathbbm{1}_{2\mid n })w_{n+1}=\bar g_n-\sum_{i=0}^{n-1}(i+1)v_{n-i}w_{i+1}\,,
  \end{equation}
  which gives an inductive definition of $w_i$ given $w_0=A$. We now prove that the corresponding series is analytic in a small neighborhood of the origin.
We inductively assume that 
 \[\abs{w_i}+\abs{\bar g_i}+\abs{v_i}\leq \mf C_i M^{i-1}\quad\mbox{for all }2\leq i\leq n\,,\]
 where we recall  $\mf C_i$ denotes the Catalan numbers.
 We trivially get for $i=0,1$ that $\abs{w_i}+\abs{\bar g_i}+\abs{v_i}\les 1$.
Then from \eqref{eq:bilby4} we have
 \begin{align}
 \abs{w_{n+1}}&\leq \frac{1}{\alpha(n+1+2\mathbbm{1}_{2\mid n})}\left(\abs{\bar g_n}+\sum_{i=1}^{n-2}(i+1) \abs{v_{n-i}}\abs{w_{i+1}} + | v_{n} | |w_1| + n |v_1| |w_n|\right)\notag\\
 &\les \mf C_n M^{n-1}+ M^{n-1}\sum_{i=1}^{n-2} \mf C_{n-i}\mf C_{i+1}\notag\\
 &\les M^{n-1} ( \mf C_{n+1} + \mf C_{n+2} ) \les M^{n-1} \mf C_{n+1}\,, \label{eq:bilby5}
 \end{align}
here, we used that $\mf C_{n+1}=\sum_{i=0}^n\mf C_i\mf C_{n-i}$ and $ \mf C_{i+1}\les \mf C_i\les \mf C_{i+1}$.
We can then use this bound together with the inductive hypothesis to bound $\bar g_{n+1}$ and $v_{n+1}$. For $\bar g_{n+1}$ we have
 \begin{align}
 \abs{\bar g_{n+1}}&\les   \abs{w_{n+1}}+\left(\sum_{i=2}^{n}\abs{w_iw_{n-i+2}}+ |w_1||w_{n+1}|\right)\les M^{n-1}\mf C_{n+1}\,.\notag
 \end{align}
 Finally, \eqref{eq:bilby5}  implies $\abs{v_{n+1}}\les  M^{n-1} \mf C_{n+1}$, closing the induction. Since $\mf C_i\leq 4^i$, then we obtain that the power series  that the series $\mc W(\zeta)=\sum_{i=0}^{\infty}  w_i \zeta^i$ is analytic in a small neighborhood $(-\delta,\delta)$ of the origin. 
 
 Under the change of variables $W(\log \zeta) = \frac{1}{\zeta} \mc W(\zeta)$ (and analogously with $Z$), we have that \eqref{eq:bilby} reads like ODE \eqref{eq:wz:ODE}. Thus, we set $(W(\log \frac\delta 2),Z(\log \frac\delta 2))= \frac 2 \delta (\mc W(\frac\delta 2),-\mc W(-\frac\delta 2))$ and solve \eqref{eq:wz:ODE}. As each coefficient $w_i$ is continuous with $r$, we have that $\mc W $, for a fixed $\zeta \in (-\delta/2, \delta_2)$, is continuous with respect to $r$. This is done in exactly the same way as we did for $(W^{(r)} (\xi), Z^{(r)} (\xi))$ in Proposition \ref{prop:smooth}. Thus, the continuations $W$ and $Z$ are also continuous with respect to $r$ because of the stability of the ODE with respect to $r$.


Finally, we need to prove that the solution $(W, Z)$ reaches $P_s$. If we consider the field $(N_W D_Z, N_Z D_W)$ (which reverses time because $D_Z < 0$), our solution  corresponds to a trajectory arriving at $P_0$. Let $H$ be sufficiently large, $P^{(H)}$ be the point $P_s + (H, -H)$ and $\check P^{(H)}$ be the point in the same vertical as $P^{(H)}$ and lying in $D_Z = 0$. We call $\mc {T}^{(H)}$ the triangle formed by $P_s$, $P^{(H)}$ and $\check P^{(H)}$, which is drawn in Figure \ref{fig:plotprop}. We have that $$ W \left( \log \zeta \right) + Z \left( \log  \zeta  \right) = 2w_1 + O( \zeta ) = \frac{-4(r-1)}{3(\ga - 1)} + O(\delta ).$$ Let us argue that for $\zeta >0$ sufficiently small and $H$ sufficiently large (depending on $\zeta$), we will have that $(W (\log \zeta), Z (\log \zeta ))$ lies in $\mathcal{T}^{(H)}$. We just need to show that $\frac{-4(r-1)}{3(\ga - 1)}$ is bigger than $W_0 + Z_0$ (which is the value of $W+Z$ along the line $\overline{P_s P^{(H)}}$). We have that
\begin{equation*}
\frac{-4(r-1)}{3(\gamma - 1)} - (W_0 + Z_0) = \frac{1}{6 (\gamma - 1)} \left(-7 + 9 \gamma + r - 3 \gamma r - 3 \mathcal R_1 \right),
\end{equation*}
so we just need to show the right parenthesis is positive. Now,
\begin{equation*}
\left( -7 + 9 \gamma + r - 3 \gamma r \right)^2 - 9 \mathcal R_1^2 = 16 (r-1) (2 + (-5 + 3 \gamma) r) > 16 (r-1) \frac{(\gamma - 1) (6\gamma - 5)}{\gamma }> 0,
\end{equation*}
where in the last inequality we used $r < 2 - \frac{1}{\gamma}$ from Lemma \ref{lemma:connecticut}. Therefore, it suffices to show $-7 + 9 \gamma + r - 3 \gamma r > 0$. Using again that $r < 2 - \frac{1}{\gamma}$, we have that
\begin{equation*}
-7 + 9 \gamma + r - 3 \gamma r > \frac{(\ga - 1) ( 1 + 3\gamma )}{\ga} > 0.
\end{equation*}
We conclude that $(W (\log \zeta), Z (\log \zeta ))$ lies in $\mathcal{T}^{(H)}$. 

\begin{figure}
\centering
  \includegraphics[width=.5\linewidth]{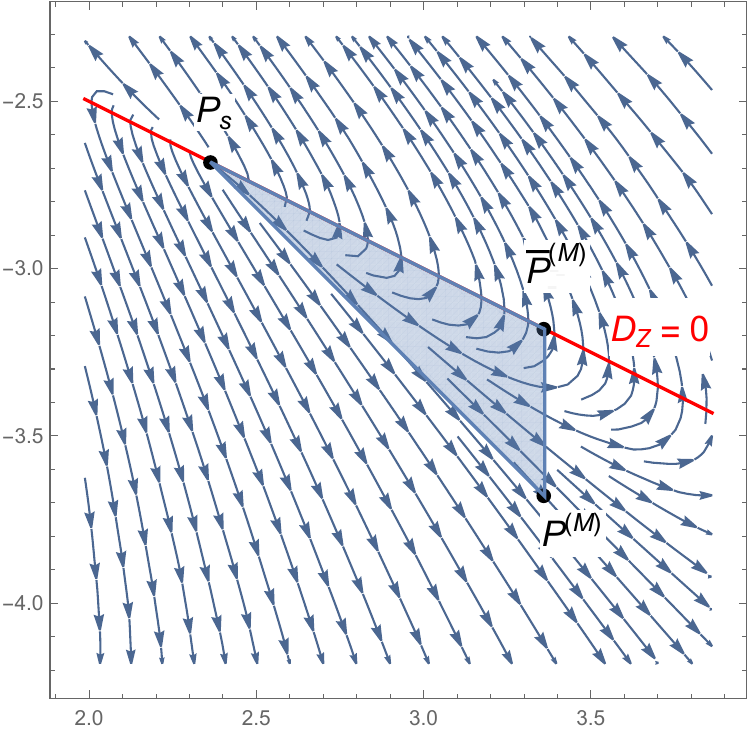}
  \captionof{figure}{\footnotesize Field $(N_W D_Z, N_Z D_W)$ in $(W, Z)$ coordinates for $\gamma = \frac53$ and $r = \frac{11}{10}$. The shaded area corresponds to the triangle $\mc T^{(1)}$.}
  \label{fig:plotprop}
\end{figure}
 
 We now show that the solution cannot come from the boundary of $\mathcal{T}^{(H)}$ except from point $P_s$. It suffices checking that the normal component of the field $(N_Z D_W, N_W D_Z)$ at each side of $\mathcal{T}^{(H)}$ points outwards (except at the extremum $P_s$). For the vertical segment $\overline{P_s P^{(H)} }$ this is guaranteed by Lemma \ref{lemma:aux_diagonal_from_P2}, for the side $\overline{P_s \check P^{(H)}}$ this is guaranteed by Lemma \ref{lemma:aux_DZ=0_repels}. For the side $\overline{P^{(H)} \check P^{(H)} }$ this follows from the fact that $N_W D_Z > 0$ (Lemma \ref{lemma:aux_NWtriangle}). Using Proposition \ref{prop:existence}, the fact that $\mc T^{(H)}$ is bounded and that there are no equilibrium points in $\mc T^{(H)}$ the trajectory has to come from $P_s$. The fact that there are no equilibrium points on $\mc T^{(H)}$ follows from $\mc T^{(H)} \subset \Omega$ and Lemma \ref{lemma:equilibrium_points} except for $P_\eye$ in the case $\ga = 7/5$, $r$ sufficiently close to $r^\ast$ and the equilibrium point $P_\eye$. In that case, note that as $\ga = 7/5$, $P_s$ is the point with highest $Z$ of $\mc T^{(H)}$, and by Lemma \ref{lemma:equilibrium_points}, $P_\eye$ has larger $Z$, so we deduce $P_\eye \notin \mc T^{(H)}$.
 \end{proof}
 
From now on, we will always choose $A$ such that the solution reaches $P_s$ at $\xi = 0$ (that is, $\zeta = 1$). We will let $(W_r^o, Z_r^o)$ to be that solution in $(W, Z)$ variables. Therefore, in $(W, Z)$ variables we get the Taylor expansion
\begin{equation} \label{eq:newtaylor}
W_r^o (\xi ) = Ae^{-\xi} + \sum_{j=0}^{\infty} W^o_{ r,j} \frac{e^{j \xi }}{j!}, \qquad Z_r^o (\xi) = - Ae^{-\xi} + \sum_{j=0}^{\infty} Z^o_{r,j} \frac{e^{j \xi} }{j!}\,,
\end{equation}
for some $W^o_{r,j} = (-1)^{j} Z^o_{r,j}$. Moreover, the series (without the $e^{-\xi}$ term) converges uniformly for $\xi < M < 0$ for some $M$ sufficiently negative.

\begin{remark} \label{rem:horizontal}
Note from the proof that the solution $(W_r^o, Z_r^o)$ will stay in the region $\Omega$ for $\xi < 0$. That is because the curve $(W_r^o(\xi), Z_r^o (\xi))$ for $\xi \in [-M', 0]$ will stay in some $\mc T^{(M)}$ for sufficiently large $M$ (as seen in the proof) and therefore, as $\mc T^{(M)} \subset \Omega$, we get that $(W_r^o, Z_r^o)$ stays in $\Omega$ for all $\xi < 0$.

Note also that $W_r^o (\xi)$ is decreasing for all $\xi \in (-\infty, 0]$. This follows from the fact that $D_W > 0$ in $\Omega$ and $N_W < 0$ in $\mc T^{(M)}$ for every $M$ from Lemma \ref{lemma:aux_NWtriangle}. \end{remark}

\section{Left of $P_s$} \label{sec:left}

This section is dedicated to showing properties of the solution left of $P_s$ in the phase portrait, which in the self-similar radial variable $\xi$ corresponds to the region $\xi>0$, or equivalently the region outside the backwards acoustic cone of the singularity.  In particular, the main goal of this section is to prove the following proposition.
\begin{proposition} \label{prop:left_main} Suppose either $n=3$ with  $\gamma \in (1, +\infty)$ or that $n\in \N$ is odd and sufficiently large with $\gamma = 7/5$. Let $r\in (r_n, r_{n+1})$. The smooth solution $W^{(r)} (\xi ), Z^{(r)} (\xi )$ defined in Proposition \ref{prop:smooth} can be continued up to  $\xi = +\infty$ and it satisfies $\lim_{\xi \to +\infty } (W^{(r)} (\xi ), Z^{(r)} (\xi )) = (0,0) = P_\infty$. Moreover, the solution stays in the region where $D_W > 0, D_Z > 0$ for all $\xi > 0$.
\end{proposition}

We will prove Proposition \ref{prop:left_main} using a double barrier argument. We will consider a barrier for the near-left region ($b^{\rm nl}(s)$) and another one for the far-left region ($b^{\rm fl}(t)$). The field will point upwards along the barrier $b^{\rm fl}(t)$. The smooth solution starts below it and so the barrier will be insufficient to bound the behavior of the smooth solution. The barrier $b^{\rm nl}(s)$ will have the field pointing upwards, start above the smooth solution and will be valid over an interval sufficiently long to intersect $b^{\rm fl}(t)$. Then, concatenating $b^{\rm nl}(s)$ (up to its intersection) with $b^{\rm fl}(t)$, one obtains a barrier bounding the trajectory of the smooth solution.

Let us define the far-left barrier as $b^{\rm fl}(t) = (b^{\rm fl}_W (t), b^{\rm fl}_Z (t))$, where
\begin{align} \begin{split} \label{eq:bfl}
b^{\rm fl}_W (t) = W_0 + B_1 W_1 t + \frac{1}{2} B_2 t^2\quad\mbox{and}\quad
b^{\rm fl}_Z (t) = Z_0 + B_1 Z_1 t + \frac{1}{2} B_3 t^2\,,
\end{split} \end{align}
where $B_1$, $B_2$ and $B_3$ will be chosen to enforce that $(b^{\rm fl}_W (1), b^{\rm fl}_Z (1)) = \Peye$ and a first order cancellation at this point. The point $\Peye$ is defined as the only solution to $N_W(W, Z) = N_Z(W, Z) = 0$ in the region $\{ W > Z \}$ (there are two solutions in the symmetry axis $W=Z$ and another two solutions outside the axis, one in each halfplane). $P_\eye$ is given explicitly by
\begin{equation} \label{eq:Peye}
P_\eye=(X_0, Y_0)= \left( \frac{2 \left(\sqrt{3}-1\right) r}{3 \ga-1},-\frac{2 \left(1+\sqrt{3}\right) r}{3 \ga-1}\right).
\end{equation}
Moreover, if we do a first order expansion around $\Peye$, we observe that one eigenvector of the matrix $((\nabla N_W)/D_W, (\nabla N_Z )/D_Z)$ is given by $(X_1, Y_1)$ with
\begin{align*} \begin{split}
\Theta &= -\sqrt{2 \left(53-3 \ga \left(45 \ga^3-110 \ga+88\right)\right) r+\left(27 \ga^2-30 \ga+7\right)^2+(3 \ga (\ga (3 \ga
   (\ga+20)-94)+28)+25) r^2 } \\
   &\qquad+3 \sqrt{3} \ga^2 (r-5)+2 \sqrt{3} \ga (r+7)-\sqrt{3} (r+3)\,, \\
X_1 &= -2 \left(3 \ga+4 \sqrt{3}-9\right) \left(3 \ga^2 \left(r^2-3\right)-6 \ga (r-2) r+6 \ga-r (r+4)-1\right)\,,\\
Y_1 &= \Theta \left( 3 \left(\sqrt{3}-1\right) \ga r-\left(\sqrt{3}-3\right) (3 \ga-1)-\left(3+\sqrt{3}\right) r \right)\,.
\end{split} \end{align*}

In order to  achieve the desired cancellations, we then choose
\begin{align} \label{eq:defBi} \begin{split}
B_1 &= 2\frac{ (Y_0 - Z_0) X_1 - (X_0 - W_0 )Y_1 }{ X_1 Z_1 - W_1 Y_1 }, \\
B_2 &=  2\left( X_0 - W_0 - B_1 W_1 \right) , \\
B_3 &= 2 \left( Y_0 - Z_0 - B_1 Z_1 \right).
\end{split} \end{align}
It is clear that the definitions of $B_2$ and $B_3$ ensure that $b^{\rm fl}_W (1) = X_0$ and $b^{\rm fl}_Z (1) = Y_0$ respectively. $B_1$ is defined so that $(X_1, Y_1)$ is proportional to $b^{\mathrm{fl}\, \prime}(0)$. In particular, we will require
\begin{align*}
0&=(X_1, Y_1)\wedge \frac{d}{dt}b^{\rm fl}\vert_{t=1} \\
&=(X_1, Y_1)\wedge (B_1 W_1 + B_2 , B_1 Z_1 + B_3)\\
&=(X_1, Y_1)\wedge \left( B_1 (-W_1 , -Z_1 ) + 2( X_0- W_0,Y_0- Z_0 ) \right)\,.
\end{align*}
Solving for $B_1$, one obtains the first equation of \eqref{eq:defBi}. 

In order to check the validity of the barrier, we need to show the positivity of  the seventh degree polynomial
\begin{equation} \label{eq:Pfl}
P^{\rm fl} (t) = b_Z^{\mathrm{fl} \, \prime} (t) N_W (b^{\rm fl}(t)) D_Z (b^{\rm fl}(t)) -  b_W^{\mathrm{fl} \, \prime} (t) N_Z (b^{\rm fl}(t)) D_W (b^{\rm fl}(t))\,.
\end{equation}
Note that the vector $(b_Z^{\mathrm{fl} \, \prime}(t), -b_W^{\mathrm{fl} \, \prime}(t))$ is normal to the curve $b^{\rm fl}(t)$ and points in the upwards direction. 

\begin{proposition} \label{prop:left_global} Let $\gamma \in (1, +\infty)$ and $r\in (r_3, r_{4})$. We have that $P^{\rm fl} (t) > 0$ for every $t\in (0, 1)$. Moreover, we have that $D_W (b^{\rm fl}(t)) > 0$ and $D_Z (b^{\rm fl}(t)) > 0$ for any $t \in (0, 1]$. 
\end{proposition}
\begin{proof} 
The statement $P^{\rm fl}(t) > 0$ for every $t \in (0, 1)$ is proven via a computer-assisted proof. The code can be found in the supplementary material and we refer to Appendix \ref{sec:computer} for details about the implementation.


With respect to $D_Z (b^{\rm fl}(t)) > 0$, note that this is a second-degree polynomial vanishing at $t=0$. For $t > 0$ small enough, $D_Z (b^{\rm fl}(t)) > 0$ since the slope of $b^{\rm fl}(t)$ coincides with the slope of the smooth solution (and $D_{Z, 1} > 0$ by Lemma \ref{lemma:aux_DZ1}). Now, as $D_Z(\Peye) > 0$ (Lemma \ref{lemma:aux_34bounds}), the second-degree polynomial $D_Z(b^{\rm fl}(t))$ cannot be non-positive at any $0 < t_0 < 1$ because otherwise $D_Z(b^{\rm fl}(t))$ would have three roots by continuity (one at $t=0$, and two in $(0, 1)$).

Lastly, $D_W(b^{\rm fl}(t))$ is also a second-degree polynomial which is positive at $t=0$ and $t=1$, it would need to have two roots in $[0, 1]$ in order to be negative for some $t \in [0, 1]$. That is impossible since its derivative at $0$ is positive: $b^{\rm fl}(t)$ agrees up to first order with the smooth solution and $D_{W, 1} > 0$ by Lemma \ref{lemma:aux_34bounds}. 
\end{proof}

For the case $\gamma = 7/5$ and $r\in (r_n, r_{n+1})$ for $n$ odd large enough, we instead consider the barrier
\begin{equation} \label{eq:bfl7o5}
 b^{\rm fl}_{7/5}(t) = \left( W_0 + W_1t + \frac{W_2}{2} t^2 - \left( W_0 + W_1 + \frac{W_2}{2} \right) t^3, Z_0 + Z_1t + \frac{Z_2}{2}t^2 - \left( Z_0 + Z_1 + \frac{Z_2}{2} \right) t^3 \right) .
 \end{equation}
  It is clear that the barrier matches up to second order at zero and that $b^{\rm fl}_{7/5}(1) = P_\infty = (0, 0)$. In the very same way as before, we define the polynomial
\begin{equation} \label{eq:global_left_P_7o5}
P_{7/5}^{\rm fl} (t) = b^{\mathrm{fl} \; \prime}_{7/5, Z}(t) N_W (b^{\rm fl}_{7/5}(t) )D_Z (b^{\rm fl}_{7/5}(t) ) - b^{\mathrm{fl} \; \prime}_{7/5, W}(t) N_Z (b^{\rm fl}_{7/5}(t) )D_W (b^{\rm fl}_{7/5}(t) ).
\end{equation}
We have the same type of result.

\begin{proposition}
 \label{prop:left_global_7o5} Let $\gamma  = 7/5$, $n \in \N$ large enough and $r\in (r_n, r_{n+1})$. We have that $P^{\rm fl}_{7/5} (t) > 0$ for every $t\in (0,1)$. Moreover, we also have that $D_W (b^{\rm fl}_{7/5} (t)) > 0$ for $t \in (0, 1)$ and $D_Z (b^{\rm fl}_{7/5} (t)) > 0$ for $t \in (0, 1) \setminus [t^{\rm in}_{7/5}, t^{\rm out}_{7/5}]$, for some $t^{\rm in}_{7/5}, t^{\rm out}_{7/5}$ such that $D_Z(b^{\rm fl}_{7/5}(t^{\rm in}_{7/5})) = D_Z(b^{\rm fl}_{7/5}(t^{\rm out}_{7/5})) = 0$.
 
Moreover the points $P^{\rm in}_{7/5} = b^{\rm fl}_{7/5}(t^{\rm in}_{7/5})$ and $P^{\rm out}_{7/5} = b^{\rm fl}_{7/5}(t^{\rm out}_{7/5})$ are located to the left of $\bar P_s$.
\end{proposition}
\begin{proof} We have that $D_W (b^{\rm fl}_{7/5}(t))$ is a third-degree polynomial. Calculating this polynomial at $r = r^\ast (7/5)$, we obtain
\begin{equation*}
D_W (b^{\rm fl}_{7/5}(t))\Big|_{r=r^\ast} = \frac{1}{396} \left(-52 \left(3 \sqrt{5}-10\right) t^3+\left(57 \sqrt{5}-91\right) t^2-33 \left(3 \sqrt{5}-5\right) t+198 \left(\sqrt{5}-1\right)\right).
\end{equation*}
As all the coefficients are positive, we obtain that the coefficients of $D_W (b^{\rm fl}_{7/5}(t))$ are still positive for $r$ sufficiently close to $r^\ast (7/5)$, so that $D_W (b^{\rm fl}_{7/5}(t)) > 0$ for $t \in (0, 1)$.

With respect to $P^{\rm fl}_{7/5}(t)$, we observe that this polynomial is a multiple of $t^3$ since the first three coefficients of $b^{\rm fl}_{7/5} (t)$ agree with those of the smooth solution passing through $P_s$. We also have that $P_{7/5}^{\rm fl}(1) = 0$ because $N_W(P_\infty) = N_Z (P_\infty ) = 0$ and $D_W(P_\infty) = D_Z(P_\infty) = 1$. Moreover, we have that
\begin{align*}
P_{7/5}^{\mathrm{fl} \; \prime}(1) &= b^{\mathrm{fl} \; \prime}_{7/5, Z}(1) D_Z (P_\infty )  \nabla N_W (P_\infty) \cdot  b^{\mathrm{fl} \; \prime}_{7/5} (1)
 - b^{\mathrm{fl} \; \prime}_{7/5, W} (1) D_W (P_\infty ) \nabla N_Z (P_\infty ) \cdot b^{\mathrm{fl} \; \prime}_{7/5} (1) \\
 &= b^{\mathrm{fl} \; \prime}_{7/5, Z}(1) (-r, 0) \cdot  b^{\mathrm{fl} \; \prime}_{7/5} (1)
 - b^{\mathrm{fl} \; \prime}_{7/5, W} (1) (0, -r) \cdot b^{\mathrm{fl} \; \prime}_{7/5} (1) 
 = 0\,.
\end{align*}
Therefore, we see that $P_{7/5}^{\rm fl}(t) = (1-t)^2 t^3 Q_{7/5}^{\rm fl}(t)$ for some sixth-degree polynomial $Q_{7/5}^{\rm fl}(t)$. We just need to show that $Q_{7/5}^{\rm fl}(t) > 0$ for $t \in (0, 1)$. Calculating $Q_{7/5}^{\rm fl}(t)$ at $r=r^\ast$, we obtain
\begin{align*}
Q_{7/5}^{\rm fl}(t)\Big|_{r=r^\ast} &= 375 \left(17845372 \sqrt{5}-22507109\right) t^6+25 \left(242253290 \sqrt{5}-88777981\right) t^5 \\
&\qquad+10 \left(1945028708 \sqrt{5}-2421950855\right) t^4+660 \left(12759056
   \sqrt{5}-8257439\right) t^3 \\
   &\qquad +7260 \left(1604200 \sqrt{5}-2790277\right) t^2-21780 \left(137 \sqrt{5}-38799\right) t+313632 \left(6133 \sqrt{5}-7995\right).
\end{align*}
As all the coefficients are positive, they will be positive for $r$ sufficiently close to $r^\ast$, and therefore, for $r$ sufficiently close to $r^\ast$, we have that $Q_{7/5}^{\rm fl}(t) > 0$ for $t \in (0, 1)$.

With respect to $D_Z$, we have that $D_Z (b^{\rm fl}_{7/5}(t))$ is a multiple of $t$ (as it vanishes at zero), and moreover
\begin{equation} \label{eq:patatillas}
\frac{1}{t} D_Z (b^{\rm fl}_{7/5}(t)) = D_{Z, 1} + \frac{D_{Z, 2}}{2} t + \frac16 D_{Z, 3}^{\rm fl}t^2,
\end{equation}
for some $D_{Z, 3}^{\rm fl}$. We have that $D_{Z, 1} > 0$ from Lemma \ref{lemma:aux_DZ1}, so the polynomial is initially positive. On the other hand, at $r = r^\ast$, we have that $D_{Z, 1} = 0$, $\frac{D_{Z, 2}}{2} = \frac{19-9\sqrt{5}}{264} < 0$ and $\frac{D_{Z, 3}^{\rm fl}}{6} = \frac{245 + 9\sqrt{5}}{264} > 0$. Therefore, it is clear that for $r$ sufficiently close to $r^\ast$, we have two real roots of the second-degree polynomial \eqref{eq:patatillas}, which we define to be $t^{\rm in}_{7/5}$ and $t^{\rm out}_{7/5}$. Moreover, as $D_{Z, 3}^{\rm fl} > 0$, the sign of \eqref{eq:patatillas} is positive except for $t \in [t^{\rm in}_{7/5}, t^{\rm out}_{7/5}]$.

Lastly, we just need to show that $b^{\rm fl}_{7/5}(t^{\rm in}_{7/5})$ is located to the left of $\bar P_s$, that is, we need to check
\begin{equation} \label{eq:patatillas2}
b^{\rm fl}_{7/5, W} \left( \frac{- D_{Z, 2} - \sqrt{D_{Z, 2}^2 - 8 D_{Z, 1}D_{Z, 3}^{\rm fl}/3 }}{2 D_{Z, 3}^{\rm fl} / 3} \right) < \mw W_0\,.
\end{equation}
This is checked in Lemma \ref{lemma:aux_patatillas3}.
\end{proof}

We define the near-left barrier to be
\begin{align} \begin{split} \label{eq:bnl}
b^{\rm nl}_{n, W} (s) = \sum_{i=0}^n \frac{1}{i!} W_i s^i, \\
b^{\rm nl}_{n, Z} (s) = \sum_{i=0}^n \frac{1}{i!} Z_i s^i,
\end{split} \end{align}
and define $b^{\rm nl}_n (s) = (b^{\rm nl}_{n, W}(s), b^{\rm nl}_{n, Z}(s))$. We have that
\begin{lemma} \label{lemma:left_initial} Let $n=3$ with any $\gamma \in (1, +\infty)$ or either $n\in \N$ odd and sufficiently large with $\gamma = 7/5$. Let $r\in (r_n, r_{n+1})$. We have that $(W^{(r)} (\xi), Z^{(r)} (\xi ))$ is initially above $b^{\rm nl}_n (s)$ for $s$ and $\xi$ sufficiently small. That is, for the same value of $W$, $Z$ is higher for the smooth solution. 
\end{lemma}
\begin{proof}
Both curves $(W^{(r)}, Z^{(r)})(\xi )$ and $b^{\rm nl}_n(s)$ agree in their Taylor expansions around $P_s$ up to order $n$. If the $(n+1)$-th coefficient  is given by a term $(W_{n+1}, Z_{n+1}) \frac{t^{n+1}}{(n+1)!}$, this will have a normal component over the tangent line of size $- (W_1, Z_1) \wedge (W_{n+1}, Z_{n+1}) \frac{t^{n+1}}{(n+1)!}$ (where the sign is positive for a deviation above the tangent line at $P_s$ and negative for a deviation below the tangent line at $P_s$). 

Therefore, we just need to check
\begin{align*}
- (W_1, Z_1) \wedge (W_{n+1}, Z_{n+1}) > 0 \Leftrightarrow W_1 Z_{n+1} - Z_1 W_{n+1} < 0.
\end{align*}

This is done in Lemma \ref{lemma:aux_34bounds} for the case of $n = 3$. For the case of $\gamma = 7/5$ and $n$ odd sufficiently large, this will follow from Section \ref{sec:left7on5}, concretely from Corollary \ref{cor:Znplus1mWnplus1overW1}.
\end{proof}

We now consider the $(4n-1)$-th degree polynomial
\begin{equation} \label{eq:Pc}
P^{\rm nl}_n (s) = b^{\rm nl \; \prime}_{n, Z}(s) N_W (b^{\rm nl}_{n}(s) ) D_Z (b^{\rm nl}_{n}(s)) -  b^{\rm nl \; \prime}_{n, W}(s)N_Z (b^{\rm nl}_{n}(s)) D_W (b^{\rm nl}_{n}(s)),
\end{equation}
whose sign determines the direction of the normal component of the barrier along $b^{\rm nl}_n(s)$. In particular, we want $P^{\rm nl}_n(t)$ to be positive, as this corresponds to the field pointing upwards.

\begin{proposition} \label{prop:left_local_3}  Let $\gamma > 1$, $n=3$ and $r\in (r_{3}, r_{4})$. There exist $s_\star $ and $t_\star $ depending on $r$ such that $b^{\rm fl}_n (t_\star ) = b^{\rm nl}_{n}(s_\star )$. Moreover, $P^{\rm nl}_{n}(s) > 0$, $D_Z( b^{\rm nl}_{n}(s) ) > 0$ and $D_W (b^{\rm nl}_{n}(s)) > 0$  for every $s\in (0, s_\star ]$.
\end{proposition}
\begin{proof} We first formulate the barrier $b^{\rm fl}(t)$ in implicit form using the resultant
\begin{equation} \label{eq:Bfl}
B^{\rm fl} (W, Z) = \left| \begin{matrix}
\frac{B_2}{2} & B_1 W_1 & W_0 - W & 0 \\
0 & \frac{B_2}{2} & B_1 W_1 & W_0 - W  \\
\frac{B_3}{2} & B_1 Z_1 & Z_0 - Z & 0 \\
0 & \frac{B_3}{2} & B_1 Z_1 & Z_0 - Z  \\
\end{matrix} \right|,
\end{equation}
so that the equation $b^{\rm fl}(t_\star ) = b^{\rm nl}(s_\star )$ can be reformulated as $B^{\rm fl}(b^{\rm nl}_n (t_\star ) ) = 0$. Let us fix $s_- = 35/100$. We will divide the proof in five steps: \begin{enumerate}
\item \label{item:initial} For every $\gamma > 1$, the polynomial $B^{\rm fl}(b^{\rm nl}_3 (s))$ is negative for $s>0$ sufficiently small  for all $r\in (r_3, r_4)$.
\item \label{item:intersec} For every $\gamma > 1$, the polynomial $B^{\rm fl}(b^{\rm nl}_3 (s_-(k-3 ))$ is positive for all $r\in (r_3, r_4)$.
\item \label{item:positivity} For every $\gamma > 1$, we have $P^{\rm nl}_3(s) > 0$ for all $s\in (0, s_-(k-3))$ and $r\in (r_3, r_4)$.
\item \label{item:DWbnl} For every $\gamma > 1$, we have $D_W (b^{\rm nl}_3(s)) > 0$ for all $s\in (0, s_-(k-3))$ and $r\in (r_3, r_4)$.
\item \label{item:DZbnl} For every $\gamma > 1$, we have $D_Z (b^{\rm nl}_3(s)) > 0$ for all $s\in (0, s_- (k-3))$ and $r\in (r_3, r_4)$.
\end{enumerate}
From items \ref{item:initial} and \ref{item:intersec}, by continuity, there exists a value $s_\star \in (0, s_-(r-r_3)]$ such that $B^{\rm fl}(b^{\rm nl}_3(s_\star )) = 0$, and therefore, there exists $t_\star$ such that $b^{\rm fl}(t_\star ) = b^{\rm nl}_3(s_\star )$. Then, items \ref{item:positivity}, \ref{item:DWbnl} and \ref{item:DZbnl} give us the desired result.

Finally, we prove each of those steps with a computer-assisted proof. The code can be found in the supplementary material. We refer to Appendix \ref{sec:computer} for details about the implementation.
\end{proof}

Lastly, we require an analogous Proposition for the case where $\gamma = 7/5$ and $k$ sufficiently large. 

\begin{proposition} \label{prop:mainleft7o5}
 Let $\gamma = 7/5$, and $n$ odd sufficiently large. There exist $s_{7/5,\rm int} $ and $t_{7/5, \rm int} $ such that $b^{\rm fl}_{7/5, n} (t_{7/5, \rm int} ) = b^{\rm nl}_{n}(s_{7/5, \rm int})$. Moreover, $P^{\rm nl}_{n}(s) > 0$ and $D_W (b^{\rm nl}_{n}(s)) > 0$ for every $s \in (0, s_{7/5,\rm int}  ]$. Lastly, either $D_Z( b^{\rm nl}_{n}(s) ) > 0$ for $s \in (0, s_{7/5, \rm int}  ]$, or there exists some $s_{7/5, \rm int}' < s_{7/5, \rm int}$ such that $D_Z (b^{\rm nl}_{n}(s_{7/5, \rm int}')) = 0$ and the point $b^{\rm nl}_{n}(s_{7/5, \rm int}')$ is located to the left of $\bar P_s$.
\end{proposition}
The proof will require an asymptotic analysis of the Taylor series done in Section \ref{sec:left7on5} and it can be found at the end of that section.

\begin{proof}[Proof of Proposition \ref{prop:left_main}] 
\begin{figure}
\centering
  \includegraphics[width=.5\linewidth]{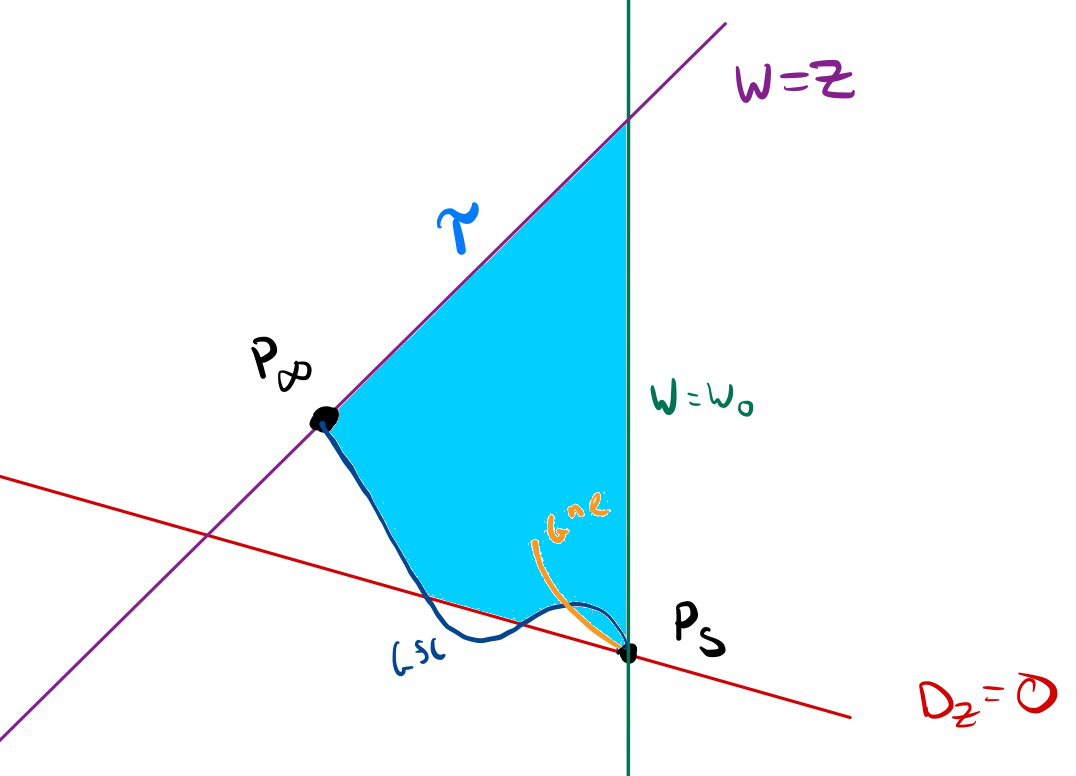}
  \captionof{figure}{\footnotesize Region $\mathcal T$ for the case $\ga = 7/5$ and $r$ sufficiently close to $r^\ast$. }
\label{fig:figurilla}
\end{figure}
\begin{figure}
\centering
  \includegraphics[width=.5\linewidth]{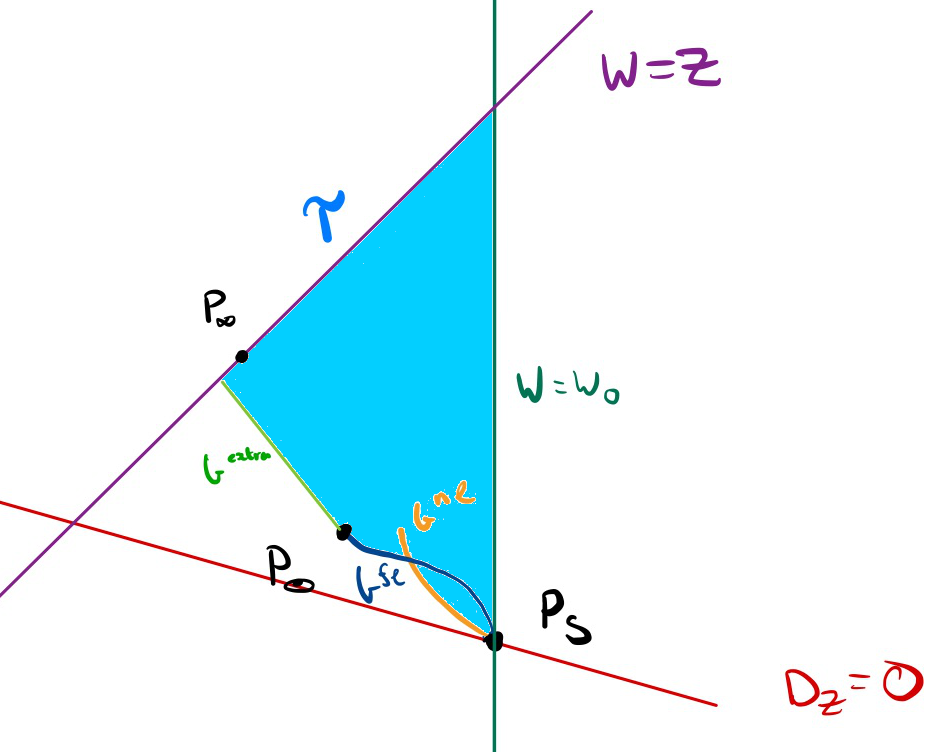}
  \captionof{figure}{\footnotesize Region $\mathcal T$ for the case $\ga > 1$ and $r \in (r_3, r_4)$.}
\label{fig:figurilla2}
\end{figure}
	For $\gamma = 7/5$ and $n$ sufficiently large, we can consider the closed region $\mc T$ of the plane which has a corner at $P_s$, and is enclosed by $b^{\rm nl}_{7/5,n}(t)$, $b^{\rm fl}_{7/5}(t)$, $D_Z = 0$, the diagonal $W = Z$ and the vertical line $W = W_0$ (starting at $(W_0, W_0)$, ending at $P_s$). The intersection between $b^{\rm nl}_{7/5,n}(t)$ and $b^{\rm fl}_{7/5}(t)$ is proven in Proposition \ref{prop:mainleft7o5}. There are two intersection points between $b^{\rm fl}_{7/5}(t)$ and $D_Z = 0$ as indicated in Proposition \ref{prop:left_global_7o5}. Note that Proposition \ref{prop:left_global_7o5} gives us two cases, in the first case we go from $P_s$ to $P_\infty$ by following $b^{\rm nl}_{7/5,n}(s)$ for $0\leq s \leq s_{7/5, \rm int}$, then $b^{fl}_{7/5}$ up to $P^{\rm in}_{7/5}$, then $D_Z = 0$ up to $P^{\rm out}_{7/5}$ and finally $b^{\rm fl}_{7/5}(t)$ up to $P_\infty$. In the second case, the path is the same except that we connect directly $b^{\rm nl}_{3}$ with $D_Z = 0$ at $b^{\rm nl}_3(s_{7/5, \rm int}')$. We should notice that in any case the region of $D_Z = 0$ which forms part of $\mc T$ is always located to the left of $\bar P_s$ (by Proposition \ref{prop:mainleft7o5} or Proposition \ref{prop:left_global_7o5}). For general $\gamma$ and $r \in (r_3, r_4)$, the endpoint of $b^{\rm fl}(t)$ at $t=1$ is $P_{\eye}$, so we consider the same region with the addition of the barrier $b^{\rm extra} (t)= (X_0 - t, Y_0 + t)$ for $t \in [0, \frac12 (X_0 - Y_0)]$. That is, we take $\mc T$ to be enclosed by $b^{\rm nl}_3(t)$, $b^{\rm fl}(t)$, $b^{\rm extra}(t)$, $W+Z = 0$ and $W=W_0$. The intersection between $b^{\rm nl}_3(t)$ and $b^{\rm fl}(t)$ is guaranteed by Proposition \ref{prop:left_local_3}. We show a sketch of region $\mathcal T$ in Figure \ref{fig:figurilla} and Figure \ref{fig:figurilla2}.

We will now show that the  region $\mathcal{T}$ does not intersect the line $D_W = 0$. We first show that the line $D_W = 0$ does not intersect $b^{\rm nl}(t)$ and $b^{\rm fl}(t)$. For the case   $r \in (r_3, r_4)$, this follows as a consequence of  Proposition \ref{prop:left_local_3} and Proposition \ref{prop:left_global}, and for the case $\gamma = 7/5$, $n$ odd sufficiently large, this is follows from Proposition \ref{prop:left_global_7o5} and Proposition \ref{prop:mainleft7o5}. In the case $r \in (r_3, r_4)$, the fact that the line $D_W = 0$ does not intersect  $b^{\rm extra}(t)$ follows from Lemma \ref{lemma:aux_bextra}. Moreover, as $D_W$ is increasing with $Z$ along $W = W_0$ and is also increasing with $t$ along $(t, t)$, we conclude that $D_W > 0$ in all $\p \mc T$. As $D_W = 0$ is a straight line and $\mc T$ is connected,  we obtain that $D_W > 0$ in all $\mc T$. The same reasoning allows us to say that $\p \mc T$ only intersects $D_Z = 0$ at $P_s$ and (in the case where $\ga = \frac75$ and $r$ sufficiently close to $r^\ast$) the points in the segment between $P^{\rm in}$ and $P^{\rm out}$ defined in Proposition \ref{prop:left_global_7o5}. 

The smooth solution $W^{(r)}(\xi ), Z^{(r)} (\xi )$ is in $\mc T$ for small enough $\xi$ due to Lemma \ref{lemma:left_initial} and the fact that $W_1 < 0$ (Lemmas \ref{lemma:aux_34bounds} and \ref{lemma:aux_limits}). It cannot exit $\mc T$ through $b^{\rm nl}(t)$, $b^{\rm fl}(t)$ or $b^{\rm extra}(t)$ due to Proposition \ref{prop:left_local_3} and Proposition \ref{prop:mainleft7o5} (for $b^{\rm nl}(t)$), Proposition \ref{prop:left_global} and  Proposition \ref{prop:left_global_7o5} (for $b^{\rm fl}(t)$) and Lemma \ref{lemma:aux_bextra} (for $b^{\rm extra}(t)$, $r \in (r_3, r_4)$). In the case of $\ga = \frac75$ and $r$ close enough to $r^\ast$, it cannot exit through $D_Z = 0$ because that region is always located to the left of $\bar P_s$, and the field points inwards there (Lemma \ref{lemma:aux_DZ=0_repels}). It can not exit through the line $W = Z$ because that line is an invariant of the field ($N_W / D_W, N_Z / D_Z$ is proportional to $(1, 1)$ in that diagonal). Also, it cannot exit through $W = W_0$ since $N_W / D_W < 0$ ($D_W > 0$ from last paragraph and $N_W < 0$ from Lemma \ref{lemma:aux_extralemma}). 

Therefore, as $\mc T$ is bounded, Proposition \ref{prop:existence} yields that either $(W^{(r)}, Z^{(r)})$ converges to some point of $\mc T$ with $D_W = 0$ or $D_Z=0$ or it converges to some equilibrium point inside $\mc T$ as $\xi \rightarrow + \infty$. There are no points with $D_W = 0$ due to the second paragraph. The solution can not converge to the points of $D_Z = 0$ in the segment $[P^{\rm in}_{7/5}, P^{\rm out}_{7/5}]$ because the field points inwards to $\mc T$ in that segment (due to Proposition \ref{prop:left_global_7o5} and Lemma \ref{lemma:aux_DZ=0_repels}). Now, we show that the situation where the solution converges to $P_s$ is also not possible. As $P_s$ is the point with minimum $Z$ in $B_\eps (P_s) \cap \mc T$ for $\eps$ small enough,\footnote{This is a consequence of $\mc T\cap B_{\eps}(P_s)$ lying the the sector to the left of $P_s$ bounded by $W=W_0$ and $D_Z=0$, for sufficiently small $\eps$. The fact that $\mc T \cap B_{\eps}(P_s)$ lies on $D_Z \geq 0$ follows from the previous considerations, as we showed that $\mc T$ is inside $D_Z \geq 0$. The fact that $\mc T \cap B_{\eps} (P_s)$ lies in $W \leq W_0$ for $\eps$ small enough follows from noting that $P_s$ is the corner of $\p \mc T$ formed between $W=W_0$ and $b^{\rm nl}(t)$ and $b^{\rm nl\; \prime}_W(0) = W_1 < 0$ by Lemmas \ref{lemma:aux_34bounds} and \ref{lemma:aux_limits}.} 
we would need to have points in $\mc T$ arbitrarily close to $P_s$ with $\frac{N_Z}{D_Z} > 0$. In order to show that this does not happen, we just need to check that $D_Z > 0$ and $N_Z > 0$ in $(\mc T \setminus \{ P_s \}) \cap B_\eps (P_s)$ for $\eps$ small enough. As the corner of $\p \mc T$ at $P_s$ has less than $\pi$ radians (because $W_1 < 0$), we just need to check $\nabla D_Z (P_s) \cdot (W_1, Z_1) > 0$, $\nabla D_Z (P_s) \cdot (0, 1)$, $\nabla N_Z (P_s) \cdot (W_1, Z_1) > 0$ and $\nabla N_Z (P_s)\cdot(0,1) > 0$. Clearly $\p_Z D_Z = (\gamma + 1)/4 > 0$ and $D_{Z, 1}, N_{Z, 1}, \p_Z N_Z (P_s) > 0$ due to Lemmas \ref{lemma:aux_DZ1}, \ref{lemma:aux_34bounds} and \ref{lemma:aux_limits}. 

Thus, we conclude that $W^{(r)}, Z^{(r)}$ converge to some equilibrium point of $\mc T$, that is, some point with $N_W = N_Z = 0$. There are four solutions to that system, which are $(0, 0), (-r, -r), P_\eye, -P_\eye$. The point $ -P_\eye$ clearly lies in the half-plane $Z > W$, so it is not in $\mc T$. The point $(-r, -r)$ is also never in $\mc T$, for the case $\gamma = 7/5$, $n$ sufficiently large this is trivial and for the $r \in (r_3, r_4)$ case this is because $b^{\rm extra}(t)$ intersects $W=Z$ at some point above $(-r, -r)$ (Lemma \ref{lemma:aux_34bounds}). 

Finally, we show that the solution does not converge to $P_{\eye}$.  For the case $\ga = 7/5$ with $n$ sufficiently large this point is discarded arguing that $P_\eye$ is not in $\mc T$. We can parametrize $b^{\rm fl}_{7/5}(t)$ by $W$, because $b^{\rm fl}_{7/5}(t)$ is decreasing by Lemma \ref{lemma:aux_Peyeout}, so all the points $(b^{\rm fl}_{7/5, W}(t), Z)$ with $Z < b^{\rm fl}_{7/5, Z}(t)$ are outside $\mc T$. That is the case of $P_\eye$ due to Lemma \ref{lemma:aux_Peyeout}. Now we show that the solution does not converge to $P_\eye$ for the general $\gamma$ case with $r \in (r_3, r_4)$. Note that $P_\eye$ is a saddle point of the field $(N_W D_Z, N_Z D_W)$ due to Lemma \ref{lemma:equilibrium_points}. Therefore, there is only one direction along for trajectories converging to $P_\eye$, and will be given by the eigenvector of negative eigenvalue, $v_-$. We will see that there are not points in $\mc T$ approaching $P_\eye$ in direction $v_-$ (or $-v_-$) and that will conclude the proof. 

Let us fix the angles $-\pi \leq \theta_- < \theta_+ < \theta_- +\pi < \pi$, so that $\theta_-, \theta_- +\pi$ indicate the angles of $v_-, -v_-$ and $\theta_+, \theta_+ + \pi$ indicate the angles of $v_+, -v_+$. Locally around $P_\eye$, the angular component of the field around a saddle point points counterclockwise in the region $\Theta_\circlearrowleft =  (\theta_-, \theta_+)\cup (\theta_- + \pi, \theta_+ + \pi)$, while it points clockwise in the region $\Theta_\circlearrowright = (\theta_+, \theta_- +\pi)\cup (\theta_- + \pi, \theta_+ + 2\pi)$. Let $\theta_{\rm fl}, \theta_{\rm extra} \in [-\pi, \pi)$ be the angles at which those barriers arrive to $P_\eye$. Note $\theta_{\rm extra} = 3\pi/4$ and note also $\theta_{\rm fl} < 3\pi/4$ by Lemma \ref{lemma:paella}. As the field points inwards to $\mc T$ on the barriers $b^{\rm fl}$ and $b^{\rm extra}$, we get that $\theta_{\rm fl} \in \Theta_\circlearrowleft$ and $\theta_{\rm extra} \in \Theta_\circlearrowright$. Therefore, the set $(\theta_{\rm fl}, \theta_{\rm extra}) \cap \{ \theta_+ - \pi, \theta_-, \theta_+, \theta_-+\pi, \theta_++\pi \}$ has an odd number of elements. If there are three elements, we get $\theta_{\rm extra} - \theta_{\rm fl} > \pi$, which contradicts Lemma \ref{lemma:paella}, so there is only one element. If that element is $\theta_-$ or $\theta_- + \pi$ we would get that $\theta_{\rm fl} \in \Theta_\circlearrowright$ and $\theta_{\rm extra} \in \Theta_\circlearrowleft$, which is also a contradiction. Thus, $\theta_-, \theta_- + \pi \notin (\theta_{\rm fl}, \theta_{\rm extra})$.
\end{proof}

\section{Right of $P_s$} \label{sec:right}
This section is dedicated to showing properties of the solution right of $P_s$ in the phase portrait, which in the self-similar radial variable $\xi$ corresponds to the region $\xi<0$, or equivalently the region within acoustic cone of the singularity.  In particular, the main goal of this section is to prove the following proposition.
\begin{proposition} \label{prop:right_main} Let us consider the smooth solution of Proposition \ref{prop:smooth} for $\xi < 0$. Let either $n =3$ for $\ga >1$ or $n \in \N$  odd, and sufficiently large for $\ga = 7/5$. Then, there exist $r_u \in (r_{n}, r_{n+1})$ such that the smooth solution $(W^{(r_u)}(\xi ), Z^{(r_u) }(\xi ))$ lies in $\Omega_1^{(r_u)}$ and $r_d \in (r_{n}, r_{n+1})$ such that the smooth solution lies in $\Omega_2^{(r_d)}$ (where here $\Omega_1^{(r)}$ and $\Omega_2^{(r)}$ are defined in Remark \ref{rem:Omega12}).
\end{proposition} 

The strategy for the proof of this proposition will be similar to the proof of Proposition \ref{prop:left_main}. We will consider a near-right barrier that matches up to the $n$-th (or $(n+1)$-th) coefficients with the smooth solution and we will also consider a far-right barrier  that intersects the near-right barrier within the interval of its validity. This approach is similar to the one employed on the left, since in both cases we need to use a local barrier that matches up to $n$-th order with the smooth solution in order to capture the singular behavior of $Z_n$. As in Section \ref{sec:left}, we also concatenate this barrier with a global barrier (the far-right barrier) that matches better the behavior of the solution far from $P_s$. The main difference with respect to Section \ref{sec:left} is that here we will work asymptotically as $r \rightarrow r_n^+$ or $r \rightarrow r_{n+1}^-$. 

We consider the near-right barrier
\begin{align*} \begin{split}
b^{\rm nr}_{n, W} (t) &= \sum_{i=0}^n \frac{1}{i!} W_i (-t)^i, \\
b^{\rm nr}_{n, Z} (t) &= \sum_{i=0}^n \frac{1}{i!} Z_i (-t)^i + \frac{\beta_n}{(n+1)!}Z_{n} (-t)^{n+1},
\end{split} \end{align*}
where $\beta_n$ is sufficiently large. We will always assume that $r-r_n$ (or $r_{n+1}-r$) is sufficiently small depending on $\beta_n$ (or $\beta_{n+1}$). We will use the standard big-$O$ notation whenever the implicit constant does not depend on $\beta_n$ and whenever we use $O_{\beta_n}$, the implicit constant is allowed to depend on $\beta_n$.

We also define the curve $b^{\rm nr}_n (t) = (b^{\rm nr}_{n, W}(t), b^{\rm nr}_{n, Z}(t))$ and consider the $(4n+3)$-th degree polynomial
\begin{equation} \label{eq:Pd}
P^{\rm nr}_n (t) = b^{\mathrm{nr} \; \prime }_{n, Z} (t) N_W (  b^{\rm nr }_{n}(t)) D_Z ( b^{\rm nr }_{n}(t) ) -  b^{\mathrm{nr} \; \prime }_{n, W} (t) N_Z (b^{\rm nr }_{n}(t) ) D_W (b^{\rm nr }_{n}(t) )\,.
\end{equation}
As usual, the sign of $P^{\rm nr}_n(t)$ indicates the direction of the normal component of the field in this barrier. We have the following result for this sign.

\begin{proposition} \label{prop:right_local} Let either $n \in \{ 3, 4 \}$ for $\ga >1$ or $n \in \N$ sufficiently large for $\ga = 7/5$. There exist constants $\eps$, $c$ such that: \begin{itemize}
\item For $n$ even and $r\in (r_{n} - \eps, r_{n})$, we have $P^{\rm nr}_{n} (t) < 0$ for all $t\in (0, c(n-k)^{1/n} )$.
\item For $n$ odd, $\beta_{n}$ sufficiently large and $r\in (r_{n}, r_{n}+\eps)$, we have $P^{\rm nr}_{n} (t) > 0$ for all $t\in (0, c(k-n)^{1/(n-2)} \beta_{n})$. 
\end{itemize}
\end{proposition}
\begin{proof} First, note that $P^{\rm nr}_n (t)$ is a multiple of $t^{n+1}$ since $b^{\rm nr}_n(t)$ matches the smooth solution up to $n$-th order. By Corollary \ref{cor:asymptotics}, the only terms in $b^{\rm nr}_n (t)$ which are not $O(1)$ in a neighborhood of $r_n$ are $\frac{1}{n!} Z_n (-t)^n$ and $\frac{1}{(n+1)!} Z_n \beta_n (-t)^{n+1}$, which are $O_{\beta_n}\left( \frac{1}{k-n}  \right)$. Summing the terms, the asymptotics for $P^{\rm nr}_n (t)$ are given by
\begin{align} \label{eq:right_local_one} \begin{split}
P^{\rm nr}_n (t) &=  C_1 Z_n t^{n+1} + C_2 Z_n^2 t^{2n-1} \\
&\qquad + O_{\beta_n}(t^{n+1} ) + O_{\beta_n}\left( \frac{t^{n+2}}{k-n} \right) + O_{\beta_n}\left(  \frac{t^{2n}}{(k-n)^2 } \right) + O_{\beta_n}\left(  \frac{t^{3n-1}}{(k-n)^3 } \right) + O_{\beta_n}\left(  \frac{t^{4n-1}}{(k-n)^4 } \right),
\end{split} \end{align}
for small $t$ and $|k-n|$. In order to calculate $C_1$, we take $n+1$ derivatives in $P^{\rm nr}_n (t)$ and look for the terms with a factor $Z_n$ (as the rest of terms will be $O(1)$, thus contributing to $O(t^{n+1})$ in \eqref{eq:right_local_one}). Note that $P^{\rm nl}_n(t)$ already involves one derivative, so we will have terms with $n+2$ derivatives in total. At $t=0$, we have
\begin{align*}
(-1)^{n+2} \p_t^{n+1} P^{\rm nr}_n &= (n+1)\beta_n Z_n  (-\p_t) D_Z (b^{\rm nr}_{n}) N_W (b^{\rm nr}_n)  +  Z_1 (-\p_t)^{n+1} D_Z (b^{\rm nr}_n ) N_W (b^{\rm nr}_n )\\
& \qquad - W_1 (-\p_t)^{n+1} N_Z (b^{\rm nr}_n ) D_W (b^{\rm nr}_n ) + O(|Z_n|) \,.\end{align*}
Collecting the terms involving $Z_n\beta_n$, we define
 \begin{align} \begin{split} \label{eq:C3_1}
 C_3 &= -(n+1)D_{Z, 1} N_W (P_s) - Z_1 \p_Z D_Z (P_s) N_W (P_s) + W_1 \p_Z N_Z (P_s) D_W (P_s) \\
 &= N_W(P_s) \left( -(n+1)D_{Z, 1} - Z_1 \p_Z D_Z (P_s) + \p_Z N_Z (P_s)\right)\,,
 \end{split}\end{align}
 so that,
we obtain
\[
(-\p_t)^{n+1} P^{\rm nr}_n (0) =  C_3 Z_n \beta_n + O(|Z_n|)\,.
\]

 As $\abs{r_n-r}\leq \eps$ and we are taking $\eps$ small, we may evaluate the sign of $C_3$ by looking at its sign at $r=r_n$, which by continuity will remain the same sign in a  neighborhood of $k=n$. Let us recall from Lemma \ref{lemma:k} that $k = \frac{\check D_{Z, 1}}{D_{Z, 1}}$, where $\check D_{Z, 1} = \nabla D_Z (P_s) \cdot (W_1, \check Z_1)$ and $\check Z_1$ was defined in \eqref{eq:checkZ1}. Thus, using \eqref{eq:C3_1} at $k = n = \frac{\check D_{Z, 1}}{D_{Z, 1}}$, yields
  \begin{align}  
   \frac{C_3}{N_{W, 0}} &= - \check D_{Z, 1} - D_{Z, 1} - Z_1 \p_Z D_Z(P_s)  + \p_Z N_Z (P_s)\notag  \\
 &= -D_{Z, 1} - \p_W D_Z (P_s) W_1 - \p_Z D_Z (P_s) \left( \check Z_1  + Z_1 \right) + \p_Z N_Z(P_s)\,. \label{eq:C3_22}
 \end{align}
 Note $Z_1, \check Z_1$ are the two solutions of the second degree equation 
 \begin{equation*}
 \p_Z D_Z z_1^2 + \left( \p_W D_Z (P_s) W_1 - \p_W N_Z (P_s) \right)z_1 - \p_W N_Z (P_s) W_1 = 0,
 \end{equation*}
 and hence
 \[Z_1+ \check Z_1=-\frac{\p_W D_Z (P_s) W_1 - \p_W N_Z (P_s)}{\p_Z D_Z (P_s)}\,. \]
 Substituting this expression into \eqref{eq:C3_22} and cancelling terms, we deduce 
\[
C_3 = - N_{W, 0} D_{Z, 1}\,.
\]
 Using Lemmas \ref{lemma:aux_34bounds} and \ref{lemma:aux_DZ1}, we conclude that $C_3 > 0$ for $k=n$ and thus it is positive for $r$ in a sufficiently small  neighborhood of $r_n$. 
 
We have that 
\begin{equation*}
C_1 = \frac{1}{(n+1)!Z_n}(-\p_t)^{n+1} P^{\rm nr}_n (0) \geq \frac{(-1)^{n+1}}{(n+1)!}   \left( \beta_n C_3 + O(1)\right).
\end{equation*}
Therefore, choosing $\beta_n$ sufficiently large, $r$ close enough to $r_n$ and using that $C_3 > 0$
\begin{equation} \label{eq:C1final}
\text{sign}(C_1) = (-1)^{n+1} \,, \quad \quad \left| C_1 \right| > \frac{\beta_n |C_3|}{2(n+1)!}\,.
\end{equation}
 
 
 We calculate the term $C_2$ in \eqref{eq:right_local_one}. Taking $2n-1$ derivatives in $P^{\rm nr}_n (t)$ and looking for terms with $Z_n^2$, we obtain
 \begin{align*}
 (-1)^{2n} \p_t ^{2n-1} P^{\rm nr}_n (t) = \binom{2n-1}{n} Z_n \p_t^n D_Z (b^{\rm nr}_n ) N_W (P_s) + O_{\beta_n}\left( Z_n \right) \,.
 \end{align*}
 Therefore, 
 \begin{equation} \label{eq:C2final}
 C_2 = \frac{1}{(2n-1)!} \binom{2n-1}{n} \p_Z D_Z (P_s) N_W (P_s) = \frac{1+\alpha}{2} \frac{1}{n!(n-1)!} N_W(P_s) < 0\,,
 \end{equation}
 where the sign is due to Lemma \ref{lemma:aux_34bounds}.

Now that we have analyzed $C_1$ and $C_2$ in \eqref{eq:C1final} and \eqref{eq:C2final}, let us go back to \eqref{eq:right_local_one} and consider the cases of $n$ odd and even separately. 

\textbf{Odd $n$ case.} We have that 
\[C_1 > \frac{\beta_n C_3}{2(n+1)!} > 0\quad\mbox{and}\quad C_2 < 0\,,\]
by \eqref{eq:C1final} and \eqref{eq:C2final}. Moreover, we have that for odd $n$ and $r$ sufficiently close  to $r_n$ from above, $Z_n > 0$ (Corollary \ref{cor:Znsign} and Lemma \ref{lemma:aux_signsZ3Z4}); moreover, assuming in addition that $\beta_n$ is sufficiently large, and $t^{n-2} \leq c^{n-2} (k-n) \beta_n$, the equation \eqref{eq:right_local_one} yields
 \begin{equation*} 
 P^{\rm nr}_n(t) > \frac{1}{2} C_1 Z_n t^{n+1} + C_2 Z_n^2 t^{2n-1} > t^{n+1} Z_n \beta_n \left( \frac{C_3}{4(n+1)!} - c^{n-2} |C_2| Z_n (k-n) \right)\,.
 \end{equation*} 
 Hence, since $Z_n (k-n) = O(1)$, and we may choose $c$ sufficiently small, we obtain that $ P^{\rm nr}_n(t) >0$
  
\textbf{Even $n$ case.} For this case, we obtain that for $\beta_n$ sufficiently large
\[C_1 < -\frac{\beta_n C_3}{2(n+1)!} < 0\quad\mbox{and}\quad C_2 < 0\,,\] 
by \eqref{eq:C1final} and \eqref{eq:C2final}. Moreover, for $n$ even and $r$ close to $r_n$ from below, we have that $Z_n > 0$ (Lemma \ref{lemma:aux_signsZ3Z4} and Corollary \ref{cor:Znsign}). Therefore, taking $\beta_n$ sufficiently large, $r$ close enough from below to $r_n$ and $t \leq c(n-k)^{1/n}$, \eqref{eq:right_local_one} yields:
 \begin{align*}
P^{\rm nr}_n(t) &<  \frac{1}{2}C_2  Z_n^2 t^{2n-1} + O_{\beta_n}\left( \frac{t^{2n}}{(n-k)^2} + \frac{t^{3n-1}}{(n-k)^3} + \frac{t^{4n-1}}{(n-k)^4} \right) \\
&= Z_n^2 t^{2n-1} \left( \frac{C_2}{2} + O_{\beta_n} \left( c(n-k)^{1/n} +c^n+c^{2n}\right) \right),
 \end{align*}
 which is negative as long as $c$ is chosen small enough since $C_2 < 0$.
\end{proof}

\begin{lemma} \label{lemma:right_up_intersec} Let $n=3$ with $\ga > 1$ or $n \in \N$ odd and sufficiently large for $\ga = 7/5$. There exists $\eps > 0$ sufficiently small such that the following holds. For every $r\in (r_{n+1}-\eps, r_{n})$, we have a value $ t_D>0$ with $D_Z (b^{\rm nr}_{n+1} (t )) < 0$ for $t \in (0, t_D)$ and $D_Z (b^{\rm nr}_{n+1} (t_D )) = 0$. Moreover, we have $t_D \lesssim (n+1-k)^{1/n}$ as $r \rightarrow r_{n+1}^-$.
\end{lemma}
\begin{proof} We have that $D_Z (b^{\rm nr}_{n+1} (t))$ is a $(n+2)$-th degree polynomial, which is multiple of $t$ since \[D_Z (b^{\rm nr}_{n+1} (0)) = D_Z (P_s) = 0\,.\] Moreover, we have that at $t=0$
\begin{equation} \label{eq:right_up_intersec_1} \begin{cases}
\p_t^i D_Z (b^{\rm nr}_{n+1} (t)) = (-1)^i \nabla D_Z (P_s) (W_{n+1}, Z_{n+1})\,, \qquad \text{ for } 1\leq i \leq n+1, \\
\p_t^{n+2}D_Z (b^{\rm nr}_{n+1} (t)) = (-1) \nabla D_Z (P_s) (0, Z_{n+1} \beta_{n+1} )\,.
\end{cases} \end{equation}
Corollary \ref{cor:asymptotics} tells us that all the coefficients of $D_Z(b^{\rm nr}_{n+1}(t))$ are $O(1)$ for $r$ in a small neighborhood of $r_{n+1}$, except for the $(n+1)$-th and $(n+2)$-th coefficients, which are $O_{\beta_{n+1}} \left( \frac{1}{n+1-k} \right)$. We can thus write 
\begin{equation} \label{eq:right_up_intersec_2}
D_Z (b^{\rm nr}_{n+1} (t)) = -D_{Z, 1}t + \frac{\p_Z D_Z (P_s) Z_{n+1} t^{n+1}}{(n+1)!} + O(t^2) + O \left( \frac{t^{n+2} \beta_{n+1}}{n+1-k} \right)\,.
\end{equation}
We have that $-D_{Z, 1} < 0$ as a consequence of  Lemma \ref{lemma:aux_DZ1} and Lemma \ref{lemma:aux_limits} and 
\[\p_Z D_Z (P_s) Z_{n+1} = \frac{1+\alpha}{2} Z_{n+1} > 0\,,\]
 by Corollary \ref{cor:Znsign} and Lemma \ref{lemma:aux_signsZ3Z4}. Therefore, the polynomial $D_Z (b^{\rm nr}_{n+1} (t))$ is initially negative. 

Taking $t=C (n+1-k)^{1/n}$ and using Corollary \ref{cor:asymptotics}, we can choose $C$ sufficiently large so that $|D_{Z, 1} t| \leq \frac{1}{3} \left| \frac{\p_Z D_Z (P_s) Z_{n+1}  t^{n+1} }{(n+1)!} \right|$ for $r \in (r_{n+1}-\eps, r_{n+1})$. This constant $C$ is allowed to depend on $n$. We can also choose $\eps$ sufficiently small so that the error $O(t^2) + O(t^{n+2}\beta_{n+1} /(n+1-k))$ in \eqref{eq:right_up_intersec_2} is smaller than $\frac{1}{3} \left| \frac{\p_Z D_Z (P_s) Z_{n+1}  t^{n+1} }{(n+1)!} \right|$. With those choices, using \eqref{eq:right_up_intersec_2} yields
\begin{equation*}
D_Z (b^{\rm nr}_{n+1} (t)) \geq \frac{\p_Z D_Z (P_s) Z_{n+1} t^{n+1}}{(n+1)!} - \frac{2}{3} \left| \frac{\p_Z D_Z (P_s) Z_{n+1} t^{n+1}}{(n+1)!} \right| > 0 ,
\end{equation*}
because $Z_{n+1} > 0$. The proof of $Z_{n+1} > 0$ is in Lemma \ref{lemma:aux_signsZ3Z4} for $n = 3$ and will be done in Corollary \ref{cor:Znsign} for the case $\gamma = 7/5$.

As $D_Z(b^{\rm nr}_{n+1}(t))$ is initially negative and is positive at $t=C (n+1-k)^{1/n}$, by continuity there exists a first time $t_D \in (0, C(n+1-k)^{1/n} )$ at which $D_Z (b^{\rm nr}_{n+1} (t_D)) = 0$ and $D_Z(b^{\rm nr}_{n+1} (t))$ is negative up to $t_D$.
\end{proof}

\begin{lemma} \label{lemma:right_up_initial} Let either $\gamma \in (1, +\infty)$ with $n=3$ or $\gamma = 7/5$ with $n$ a sufficiently large odd number. There exist $\beta_{n+1}, \eps > 0$ such that for every $r \in (r_{n+1} - \eps , r_{n+1} )$, the smooth solution $(W^{(r)} (\xi ), Z^{(r)} (\xi ))$ is above the near-right barrier $b^{\rm nr}_{n+1} (t)$ for $\xi < 0$ small enough in absolute value.
\end{lemma}
\begin{proof} 
First note that by definition, the Taylor expansion of $b^{\rm nr}_{n+1} (-t)$ agrees with the smooth solution up to order $n+1$ at $P_s$. In particular we have
\[b^{\rm nr}_{n+1} (t)-(W^{(r)}(-t),Z^{(r)}(-t))=\tfrac{ t^{n+2}}{(n+2)!} \left(\allowbreak (0, -\beta_{n+1} Z_{n+1} )- \allowbreak (-W_{n+2}, -Z_{n+2})\right)+O_{\beta_{n+1}}(t^{n+3})\,.\]
Both $b^{\rm nr}_{n+1}(t)$ and $(W^{(r)}(-t), Z^{(r)}(-t))$ start with slope $(-W_1, -Z_1)$. Then, in order for $b^{\rm nr}_{n+1} (t)$ to begin below the smooth solution it suffices to check the geometric condition
\[
 (-W_1, -Z_1) \wedge \left(\allowbreak (0, -\beta_{n+1} Z_{n+1} )- \allowbreak (-W_{n+2}, -Z_{n+2})\right) <  0\,, \]
 or equivalently
 \[
\beta_{n+1} Z_{n+1} W_1 < W_1 Z_{n+2} - W_{n+2} Z_1\,.
\]
Note that both sides are of order $\frac{1}{|k-n-1|}$, by Corollary \ref{cor:asymptotics} and $n \geq 3$. The case $n = 3$ follows from the fact that $Z_{4} > 0$ (Lemma \ref{lemma:aux_signsZ3Z4}) and $W_1 < 0$ (Lemma \ref{lemma:aux_34bounds}), taking $\beta_4$ sufficiently large. The case $\gamma = 7/5$ with $n$ odd and sufficiently large would follow as long as we have $Z_{n+1} > 0$, since $W_1<0$ is guaranteed by Lemma \ref{lemma:aux_limits}. The proof that $Z_{n+1} > 0$ will be delayed to  Section \ref{sec:left7on5} -- specifically this will be a consequence of Corollary \ref{cor:Znsign}).
\end{proof}

We define a far-right barrier by
\begin{equation} \label{eq:f}
B^{\rm fr}(W, Z) = (W-W_0-F_2Z + F_2Z_0)(W+Z-F_0) - F_1 (W+Z-W_0-Z_0) \,.
\end{equation}
We define the coefficient $F_0 = \frac{-4(r-1)}{3(\ga - 1)} = 2w_1$ where $w_1$ is given in Proposition \ref{prop:solnearxi0} and we set $F_2 = 1/2$. The coefficient $F_1$ will be fixed later (given by \eqref{eq:F1}). 

It is clear that $B^{\rm fr}(W_0, Z_0) = 0$, and thus the curve passes through $P_s$. As the second degree summand has a factor $W+Z-F_0$, it is clear that $B^{\rm fr}(W, Z) = 0$ has an asymptotic line parallel to the direction of $P_0$. The value of $F_0$ is the asymptotic value of $W+Z$ for the trajectory matching at $P_0$ (this corresponds to matching another order at $P_0$). The value of $F_1$ is chosen so that the slope of $B^{\rm fr}(W,Z) = 0$ at $P_s$ matches  that of the smooth solution, that is $\nabla B^{\rm fr} (P_s) (W_1, Z_1) = 0$. Therefore,
\begin{equation*} 
\left( W_0 + Z_0 - F_0  - F_1, -F_2(W_0 + Z_0 - F_0) - F_1\right) \cdot (W_1, Z_1) = 0 \,,
\end{equation*}
which yields
\begin{equation} \label{eq:F1}
  F_1 = (W_0 + Z_0 - F_0) \frac{W_1 - F_2 Z_1}{W_1+Z_1} = s_\infty^{\rm fr} \frac{F_2Z_1 - W_1}{W_1+Z_1}\,.
\end{equation}

In order to parametrize the curve $B^{\rm fr} = 0$, we solve the system
\begin{align*}
&\left( b^{\rm fr}_W(s) - W_0 \right) + \left( b^{\rm fr}_Z(s) - Z_0 \right) = s\,,\\
&B^{\rm fr} (b^{\rm fr}_W (s), b^{\rm fr}_Z(s) ) = \left( (b^{\rm fr}_W(s) - W_0) - F_2 (b^{\rm fr}_Z(s) - Z_0) \right) (s+W_0+Z_0-F_0) - F_1 s = 0\,. 
\end{align*}
for $s$ ranging  from $0$ to $s^{\rm fr}_\infty = F_0 - W_0 - Z_0$.
Solving the system, we obtain
\begin{equation} \label{eq:bfr} \begin{cases}
b^{\rm fr}_W (s) &= W_0 + \frac{1}{F_2 + 1} \left(F_2  s  + \frac{F_1 s}{s+W_0+Z_0 - F_0} \right) = W_0 + \frac{1}{F_2+1} \left( F_2 s + \frac{F_1 s}{s-s^{\rm fr}_\I} \right)  , \\
b^{\rm fr}_Z (s) &= Z_0 + \frac{1}{F_2 + 1} \left( s  - \frac{F_1 s}{s+W_0+Z_0 - F_0} \right) = Z_0 + \frac{1}{F_2+1} \left( s - \frac{F_1 s}{s-s_\I^{\rm fr}} \right),
\end{cases} \end{equation}
and we define $b^{\rm fr}(s) = (b^{\rm fr}_W(s), b^{\rm fr}_Z (s))$ which goes from $P_s$ at time $s=0$ to $P_0$ at $s = s^{\rm fr}_\infty$. We also define the barrier condition as
\begin{equation} \label{eq:Pfr}
P^{\rm fr} (s) =  b^{\rm fr\; \prime}_Z (t) D_Z (b^{\rm fr} (t)) N_W (b^{\rm fr}(t)) - b^{\rm fr \; \prime}_W (t) D_W (b^{\rm fr} (t)) N_Z (b^{\rm fr}(t))\,,
\end{equation}
which is a rational function with powers of $(s-s_\infty^{\rm fr})$ in the denominator.

Let us remark that the value of $F_2 = 1/2$ is chosen ad-hoc so that $P^{\rm fr}(s)$ is positive. In contrast, $F_0$ and $F_1$ are carefully chosen to get cancellations of $B^{\rm fr}$ at $P_s$ or $P_0$. In particular, $F_2$ does not enforce any cancellation and any slight modification of it would still yield a valid barrier.

\begin{proposition} \label{prop:right_f} Let either $n=3$ with $\ga > 1$ or $n \in \N$ odd and sufficiently large with $\ga = 7/5$. For $r$ close enough from above to $r_n$, we have that $P^{\rm fr} (s) > 0$ for all $s\in (0, s^{\rm fr}_\infty)$.
\end{proposition}
\begin{proof} By continuity of $P^{\rm fr}$ with respect to $r$, we just need to prove this for $r = r_3$ in the case of $n = 3$ and for $r$ close enough to $r^\ast(7/5)$ for $\ga = 7/5$. We prove the statement via a computer-assisted proof. The code can be found in the supplementary material and we refer to Appendix \ref{sec:computer} for details about the implementation. 
\end{proof}

\begin{lemma} \label{lemma:right_down_intersec} Let $n = 3$ with $\ga > 1$ or $n \in \N$ odd sufficiently large with $\ga = 7/5$. There exists $\eps > 0$ such that for every $r\in (r_{n}, r_{n}+\eps)$, there exists a value $t_F$ with $B^{\rm fr} (b^{\rm nr}_{n} (t_F )) = 0$ and $B^{\rm fr} (b^{\rm nr}_{n} (t )) < 0$ for $t \in (0, t_F)$. Moreover, we have that $t_F \lesssim_{n} (k-n )^{1/(n-2)}$.
\end{lemma}
\begin{proof} We consider the $(2n+2)$-th degree polynomial $B^{\rm fr} (b^{\rm nr}_n (t))$. We have that $B^{\rm fr} (b^{\rm nr}_n (0)) = B^{\rm fr} (P_s) = 0$ and $\nabla B^{\rm fr} (P_s) (W_1, Z_1) = 0$, due to our choice of $F_1$. Therefore, the terms in the polynomial have  $t^2$ as a common factor. Taking into account Corollary \ref{cor:asymptotics}, all the coefficients in $b^{\rm nr}_n (t)$ are $O(1)$ around $r\approx r_n$ except the $n$-th and $(n+1)$-th coefficients of $b_Z^{\rm nr}(t)$. We get
\begin{align} \label{eq:Bfrbnr} \begin{split}
B^{\rm fr} (b^{\rm nr}_n (t)) &= \frac{a^{\rm nr}_2}{2} t^2 + \frac{a^{\rm nr}_n}{n!(k-n)} t^n + O(t^3)  + O_{\beta_n}\left( \frac{t^{n+1}}{k-n} \right) + O\left( \frac{ t^{2n}}{(k-n)^2} \right) + O_{\beta_n} \left( \frac{ t^{2n+1} }{(k-n)^2} \right) \\
 &=   \frac{a^{\rm nr}_2}{2} t^2 (1+o(1)) + \frac{a^{\rm nr}_n}{n!(k - n)} t^n (1+o(1))\, ,
 \end{split} \end{align}
where in the last equality we are assuming $t \leq C (k-n)^{1/(n-2)}$ for some sufficiently large $C$ and using that $k-n$ is sufficiently small depending on $\beta_n$.

We proceed to calculate $a^{\rm nr}_2$ and $a^{\rm nr}_n$. We have that
\begin{align*}
B^{\rm fr}(W, Z) &= \nabla B^{\rm fr}(P_s) (W-W_0, Z-Z_0)^\top + (W-W_0, Z-Z_0) \frac{HB^{\rm fr}}{2} (W-W_0, Z-Z_0)^\top \\
&= (-s_\infty^{\rm fr} - F_1, F_2s_\infty^{\rm fr} - F_1) (W-W_0, Z-Z_0)^\top \\
&\qquad + (W-W_0, Z-Z_0) \begin{pmatrix} 1 & (1-F_2)/2 \\ (1-F_2)/2 & -F_2 \end{pmatrix} (W-W_0, Z-Z_0)^\top\,.
\end{align*}
Therefore, using $F_2 = 1/2$, we obtain that
\begin{align*}
\frac{a_2^{\rm nr} }{2}&= \frac{1}{2} \left( -s_\infty^{\rm fr} W_2 - F_1W_2 + s_\infty^{\rm fr} Z_2/2 - F_1Z_2 \right) + W_1^2 + W_1Z_1/2 - Z_1^2/2 \\
&= \frac{s_\infty^{\rm fr}}{2} \left( -W_2 + Z_2/2 - (W_2+Z_2)\frac{Z_1/2-W_1}{W_1+Z_1}\right) + W_1^2 + W_1Z_1/2 - Z_1^2/2\,,
\end{align*}
and
\begin{equation*}
\frac{a_n^{\rm nr}}{k-n} = \left( s_\infty^{\rm fr}/2 - F_1 \right) Z_n(-1)^n = -s^{\rm fr}_\infty \left( \frac{1}{2} - \frac{Z_1/2 - W_1}{W_1+Z_1} \right) Z_n\,,
\end{equation*}
where have used $F_1 = s_\infty^{\rm fr} \frac{Z_1/2 - W_1}{W_1+Z_1}$. First, note that $s_\infty^{\rm fr} > 0$ due to Lemma \ref{lemma:aux_sinf}. Second, we have that $\frac{Z_1/2 - W_1}{W_1 + Z_1} \leq -1 $ for $r = r_3$ with $\ga > 1$ and $r=r^\ast(7/5)$ for $\ga = 7/5$ due to Lemma \ref{lemma:aux_F1}. These two facts allow us to conclude 
\begin{equation}\label{eq:honeyeater}
\frac{a_n^{\rm nr}}{k-n} \leq \frac{-3}{2} s_\infty^{\rm fr}Z_n < 0\,,	
\end{equation}
where $Z_n > 0$ (Corollary \ref{cor:Znsign} for $\ga = 7/5$, $n$ sufficiently large and Lemma \ref{lemma:aux_signsZ3Z4} for $r \in (r_3, r_4)$).

On the other hand, we have that Lemma \ref{lemma:aux_a2nr} guarantees $a_2^{\rm nr} > 0$. 

Going back to \eqref{eq:Bfrbnr}, we have that $B^{\rm fr}(b^{\rm nr}_n(t))$ is initially positive for sufficiently small $t > 0$. Define
\begin{equation} \label{eq:deadpoetsociety}
t_G = \left| \frac{n! a_2^{\rm nr}}{ 2 s_\infty^{\rm fr} Z_n} \right|^{1/(n-2)}.
\end{equation}

Using \eqref{eq:Bfrbnr} and \eqref{eq:honeyeater} we have that
\begin{align*}
B^{\rm fr} (b^{\rm nr}_n (t_G)) &= \frac{a^{\rm nr}_2}{2} t^2 (1+o(1)) + \frac{a^{\rm nr}_n}{n!(k - n)} t^2 \cdot  \left| \frac{n! a_2^{\rm nr}}{ 2 s_\infty^{\rm fr} Z_n}  \right| (1+o(1)) \\
&= \frac{a^{\rm nr}_2}{2} t^2 \left(1- \left| \frac{a^{\rm nr}_n}{(k-n)s_\infty^{\rm fr} Z_n } \right|(1+o(1)) +o(1)\right) \\
&\leq -\frac{a^{\rm nr}_2}{4} t^2 (1+o(1))\,.
\end{align*}
By continuity, there exists some $t_F \in (0, t_G)$ such that $B^{\rm fr}(b^{\rm nr}_n(t_F)) = 0$. Moreover, as $B^{\rm fr}(b^{\rm nr}(t))$ is a polynomial, we can take $t_F$ to be the first such zero, so that $B^{\rm fr}(b^{\rm nr}_n(t)) > 0$ for $0 < t < t_F$. Lastly, by \eqref{eq:deadpoetsociety} we conclude that $t_F \leq t_G \lesssim_n (k-n)^{1/(n-2)}$.
\end{proof}

\begin{lemma} \label{lemma:right_down_initial} Let either $\gamma > 1$ with $n=3$ or $\gamma = 7/5$ with $n\in \N$ odd, large enough. There exists $\eps, \delta, \beta_{n, 0 } > 0$ such that for every $r\in (r_{n}, r_{n} + \eps )$, $\beta_{n} > \beta_{n, 0 }$ and $\xi\in(0,\delta)$, the smooth solution $(W^{(r)} (\xi ), Z^{(r)} (\xi ))$ 
is below the near-right barrier $b^{\rm nr}_{n} (t)$. \end{lemma}
\begin{proof} We have to compare the first Taylor coefficients on the barrier which differ from those of the smooth solution (with $\xi$ replaced with $-\xi$), that is the $n+1$-th. As $n+1$ is even, they are given by $(W_{n+1}, Z_{n+1})$ for the smooth solution and $(0, \beta_n Z_n )$ for $b^{\rm nr} (t)$. Therefore, $b^{\rm nr}_n (t)$ is above the smooth one close to $P_s$ if
\begin{align*}
(-W_1, -Z_1) \wedge (0, \beta_n Z_n )& > (-W_1, -Z_1) \wedge (W_{n+1}, Z_{n+1} )\,, \end{align*}  
or equivalently
\[
\beta_n Z_n W_1 < W_1 Z_{n+1} - W_{n+1} Z_1\,.\]
The case $n = 3$ follows from the fact that $W_1 < 0$ (Lemma \ref{lemma:aux_34bounds}) and $Z_3 > 0$ (Lemma \ref{lemma:aux_signsZ3Z4}) provided that we take $\beta_n$ sufficiently large. For the case $\gamma = 7/5$, $n$ large enough, Lemma \ref{lemma:aux_limits} guarantees $W_1 < 0$, so we just need $Z_n > 0$ in order to conclude the statement. The proof that $Z_n > 0$ will be delayed to Section \ref{sec:left7on5}: specifically this will follow from Corollary \ref{cor:Znsign}.
\end{proof}

\begin{proof}[Proof of Proposition \ref{prop:right_main}] Let us start with the existence of $r_u$. From Proposition \ref{prop:right_local} and Lemmas \ref{lemma:right_up_intersec} and \ref{lemma:right_up_initial}, we have an $\eps > 0$ small enough and a $\beta_{n+1}$ sufficiently large such that for every $r\in (r_{n+1}-\eps , r_{n+1})$ the following holds:
\begin{enumerate} [label=\roman*.]
\item The smooth solution $(W^{(r)} (\xi ), Z^{(r)} (\xi ))$ is above the barrier $b^{\rm nr}_{n+1}(t)$ for $0 > \xi > -\epsilon$, for some $\epsilon > 0$ sufficiently small. 
\item There is a constant $C_1$ such that $P^{\rm nr}_{n+1} (t) < 0$ for every $t\in (0, C_1 (n+1-k)^{1/n})$.
\item There exists $t_D>0$ such that $D_Z ( b^{\rm nr}_{n+1} (t_D)) = 0$, $D_Z(b^{\rm nr}_{n+1}(t)) < 0$ for $t \leq t_D$; with $t_D \leq C_2 (n+1-k)^{1/(n-1)}$.
\end{enumerate}
Therefore, taking $r$ close enough to $r_{n+1}$ from below, we can ensure that $P^{\rm nr}_{n+1} (t) < 0$ up to the intersection of $b^{\rm nr}_{n+1}(t)$ with $D_Z = 0$. Consider the region delimited by $b^{\rm nr}_{n+1}(t)$ and $D_Z = 0$ inside $\Omega$. The region is bounded and the solution cannot exit that region. Moreover, the solution can not converge to an equilibrium point by Lemma \ref{lemma:equilibrium_points} (and Lemma \ref{lemma:aux_snes} for the case of $P_\eye$, $\ga = 7/5$, $r$ sufficiently close to $r^\ast$). Therefore, by Proposition \ref{prop:existence}, the solution must end in the right half-line of $D_Z = 0$ starting at $P_s$, thus lying in $\Omega_1^{(r_u)}$. This concludes the first part of Proposition \ref{prop:right_main}.

From Proposition \ref{prop:right_local}, Lemma \ref{lemma:right_down_intersec} and Lemma \ref{lemma:right_down_initial}, we choose $\beta_{n}$ sufficiently large and $\eps > 0$ sufficiently small, depending on $\beta_{n}$ (in particular, we assume $\beta_n\ll (k-n)^{-1/(2n)}$ ), then the following holds:
\begin{enumerate}[label=\roman*.]
\item\label{ptt:1} The smooth solution $(W^{(r)} (\xi ), Z^{(r)} (\xi ))$ stays initially below the barrier $b^{\rm nr}_{n}(t)$ for $-\epsilon < \xi < 0$ for some $\epsilon > 0$ sufficiently small.
\item There is a constant $C_1$ such that $P^{\rm nr}_{n} (t) > 0$ for every $t\in (0, \beta_{n} C_1 (k-n)^{1/(n-2)}$.
\item There exists a $0<t_F \leq C_2 (k-n)^{1/(n-2)}$ such that $B^{\rm fr} ( b^{\rm nr}_{n}(t_F)) = 0$ and $B^{\rm fr} ( b^{\rm nr}_{n}(t)) < 0$ for $0 < t < t_F$.
\end{enumerate}
We conclude that $P^{\rm nr}_n(t) > 0$ up to $t_F$, for $r$ close enough to $r_n$ from above. 

We define the barrier $b^{\rm nr}_{n} \ast B^{\rm fr}$ to be $b^{\rm nr}_{n} (t)$ up to $t_F$ and then $B^{\rm fr} = 0$ (with some parametrization starting at $t_F$) up to $P_0$. Using Proposition \ref{prop:right_f} and the fact that $P^{\rm nr}_n(t) > 0$ up to $t_F$, we know that the component of the field $\left( N_W D_Z ,  N_Z D_W \right)$ points downwards. Let us consider $\Omega_d$, the part of $\Omega$ below $b^{\rm nr}_{n}  \ast B^{\rm fr} $. 

The smooth solution $(W^{(r)} (\xi ), Z^{(r)} (\xi ))$ is in $\Omega_d$ for $-\xi >0$ sufficiently small by point \ref{ptt:1} above, and cannot exit through $b^{\rm nr}_{n} \ast B^{\rm fr}$ since the normal component points downwards. In particular, it cannot hit the halfline of $D_Z = 0$ to the right of $P_s$, so by Remark \ref{rem:Omega12}, it cannot lie in $\Omega_1^{(r_d)}$. In other words $\Omega_d \subset \bar \Omega_2^{(r_d)}$. As $\Omega_d$ is open, it cannot have points of $\p \Omega_2^{(r_d)}$. Therefore, $\Omega_d \subset \Omega_2^{(r_d )}$, and we are done. \end{proof}

\section{Complete Section \ref{sec:left} for the case $r\rightarrow r^\ast$} \label{sec:left7on5}

In this section, we fix $\gamma = 7/5$. Our objective here is to complete the analysis we did in Section \ref{sec:left} for the case of $\gamma = 7/5$ and $n $ sufficiently large. We will prove Proposition \ref{prop:mainleft7o5}, and in Corollary \ref{cor:Znplus1mWnplus1overW1}, we will also conclude the proof of Lemma \ref{lemma:left_initial} for this case of $\gamma = 7/5$ and $k$ sufficiently large. 

In order to do so, first we need to control the growth of the Taylor coefficients of the solution. In particular, we need to obtain the sign of $Z_n$ (with a lower bound of its magnitude)  in order to guarantee that the behavior of the smooth solution when $r \approx r_n^+$ or $r\approx r_{n+1}^-$ is the expected one. 

Let us define $C_\ast = \lim_{r\to r^\ast}  \frac{D_{Z, 2}}{2 k D_{Z, 1}}$. With this definition, we will have that for $1 \ll i \leq n$, the coefficient $Z_i$ is approximately equal to $C_\ast \frac{i^2 k}{i-k} Z_{i-1}$. This idea is formalized in the following lemma. Let us define also the quantity $\bar C_\ast=0.95 C_\ast$.

\begin{lemma}\label{lemma:kennedy} For  $k$ sufficiently large  and  $i \leq k$ we have that: 
\begin{align}
\left| \frac{W_{i+1}}{Z_i}\right| &\leq 4\,, \label{eq:kennedy1}\\
\left| \bar C_\ast \frac{(i+1)^2k}{i+1-k} \right| \leq\left| \frac{Z_{i+1}}{Z_i} \right| &\leq
\left| 500\bar C_\ast \frac{(i+1) k}{i+1-k} \right|&\mbox{for }i< 160\,,  \label{eq:kennedy2}\\
\left| \bar C_\ast \frac{(i+1)^2 k}{i+1-k} \right|
\leq\left| \frac{Z_{i+1}}{Z_i} \right|&\leq \left| 3\bar C_\ast \frac{(i+1)^2k}{i+1-k} \right|&\mbox{for }i\geq 160\,,
\label{eq:kennedy3}\\
\abs{Z_{i+1} - \bar C_\ast \frac{(i+1)^2 k}{i+1-k} Z_i}&\leq0.05|Z_{i+1}|
&\mbox{for }i >10000\,.\label{eq:kennedy4}
\end{align}
\end{lemma}
Note that the inequality \eqref{eq:kennedy4} is strictly stronger than \eqref{eq:kennedy3}.

Inequality \eqref{eq:kennedy2} and inequality \eqref{eq:kennedy3} for $i \leq 10000$, are proven via a computer-assisted argument (see Lemma \ref{lemma:tenthousand_7o5}). Moreover, \eqref{eq:kennedy1} for $i\leq 10000$ is also proven via a computer-assisted proof (see Lemma \ref{lemma:aux_WioverZi_7o5}). In both cases,  strict inequalities at $r = r^\ast (7/5)$ are shown and then one obtains the result by invoking continuity.  

The proof of Lemma \ref{lemma:kennedy} will follow by induction. In Lemma \ref{lemma:W:ind} we show that given that the estimates in Lemma  \ref{lemma:kennedy} hold for all $i< m$ for $10000\leq m\leq k$, then  \eqref{eq:kennedy1} holds for $i=m$. In Lemma \ref{lemma:Z:ind} we show that given that the estimates in Lemma  \ref{lemma:kennedy} hold for all $i< m$ and  \eqref{eq:kennedy1} holds for $i\leq m$, where $10000\leq m\leq k$, then  \eqref{eq:kennedy4} holds for $i=m$. 

The proof of \eqref{eq:kennedy4} will be strongly based on the $log$-convexity of the Taylor terms. We will have that $|Z_a| \cdot |Z_{i-a}| \gg |Z_b| \cdot |Z_{i-b}|$ when $a \leq b \leq i-b \leq i-a$. This will allow us to justify that all the terms in the Taylor recursion are dominated by the most extreme ones (those terms $Z_a Z_{i-a}$ on which either $a \in \{0, 1 \}$ or $a \in \{ i-1, i \}$). Note that this $log$-convexity does not hold when $i > k$ since, for example, taking $k$ approximately equal to $n$, and $n$ even, we have that $|Z_{3n/2}| \cdot | Z_{n/2} | = O(|k-n|^{-1})$ as $k\to n$, while $|Z_n|^2 = O(|k-n|^{-2})$ as $k \to n$. Therefore, for $k$ sufficiently close to $n$, we have $|Z_n|^2 \gg  |Z_{3n/2}| \cdot | Z_{n/2} |$. However, since $Z_n$ is the first coefficient that blows up as $k \to n$, it will be enough to have a lower bound on that coefficient to control the Taylor series, since the next ones will go with higher powers of $\xi$.

\begin{remark} \label{rem:C4}
From Lemma \ref{lemma:aux_limits}, we have that that
\begin{equation*}
C_\ast = \lim_{r\to r^\ast}  \frac{D_{Z, 2}}{2 k D_{Z, 1}}=1/726 (-29 + 12 \sqrt 5 ) = -0.0029851\ldots,
\end{equation*}
so that $\bar C_\ast = 0.95 \bar C_\ast $ is negative and $0.00283 < |\bar C_\ast| < 0.00284$.
\end{remark}

\begin{corollary}
Assume \eqref{eq:kennedy2}--\eqref{eq:kennedy4} hold for all $i\leq m$. If $k$ is sufficiently large, then for all $i\leq m$, we have
\begin{equation}\label{eq:WZ:quotient}
    \abs{\frac{W_i}{Z_i}}\leq 2\,.
\end{equation}
\end{corollary}
\begin{proof}
The inequality is shown for $i\leq 160$ in Lemma \ref{lemma:aux_WioverZi_7o5}.

Recall that the inequality \eqref{eq:kennedy4} is stronger than \eqref{eq:kennedy3}. Then, \eqref{eq:kennedy1}, \eqref{eq:kennedy3} and \eqref{eq:kennedy4} imply
\[
 \abs{\frac{W_i}{Z_i}}\leq 4  \abs{\frac{Z_{i-1}}{Z_i}}\leq 4 |\bar C_\ast| \abs{\frac{i-k}{i^2k}}\leq 
8 |\bar C_\ast|\,.
\]
Then, we conclude by using Remark \ref{rem:C4}.
\end{proof}

Whenever we use the sign $\lesssim$ in this section, the implicit constant will not depend on $n$. We also use the descending Pochhammer notation $a_{(b)} = a(a-1)(a-2)\ldots (a-b+1)$ for any real $a$ and positive integer $b$. For simplicity, assume that $\sum_{j=a}^b$ is the sum starting at $\lceil a \rceil$ and ending at $\lfloor b \rfloor$ whenever $a, b$ are not integers.

\subsection{Convexity lemmas from our assumptions}
\begin{lemma} \label{lemma:washington} Assume \eqref{eq:kennedy2}--\eqref{eq:kennedy4} hold for all $i\leq m$. Let $j \geq 2$ and $m < k$. 
Assume also $j \leq (m+1)/2$, and $m > 10000$. Then we have the following:
\begin{itemize}
\item For $j < 160$, we have:
\begin{align} \label{eq:Sarkozy} \begin{split}
\binom{m+1}{j} |Z_jZ_{m+1-j}| &\leq \binom{m+1}{1} |Z_1Z_m| \left( \frac{600}{m}\right)^{j-1} \,,\\
\binom{m+2}{j} |Z_j Z_{m+2-j}| &\leq \binom{m+2}{2} |Z_2 Z_m| \left( \frac{600}{m} \right)^{j-2}\,.
\end{split} \end{align}
\item For $160 \leq j \leq m/10$, we have that: 
\begin{align}\label{eq:Hollande}  \begin{split}
\binom{m+1}{j} |Z_jZ_{m+1-j}| &\leq \binom{m+1}{1} |Z_1Z_m| \frac{1}{m^9} (1/3)^j \,,\\
\binom{m+2}{j} |Z_j Z_{m+2-j}| &\leq \binom{m+2}{2} |Z_2 Z_m| \frac{1}{m^9} (1/3)^j\,.
\end{split} \end{align}
\item For $j \geq m/10$, we have that:
\begin{align}\label{eq:Macron} \begin{split}
\binom{m+1}{j} |Z_jZ_{m+1-j}| &\leq \binom{m+1}{1} |Z_1Z_m|(3/4)^{j/11}\,, \\
\binom{m+2}{j} |Z_j Z_{m+2-j}| &\leq \binom{m+2}{2} |Z_2 Z_m| (3/4)^{j/11}\,.
\end{split} \end{align}
We note that the first inequality of \eqref{eq:Sarkozy} is also trivially true for $j=1$.
\end{itemize}
\end{lemma}
Before we prove this lemma, let us prove an elementary bound on binomial coefficients. We first  recall the classical bound
\begin{equation}\label{eq:factorial:bnd}
    \sqrt{2\pi n} \left(\frac{n}{e}\right)^{n} \leq n! \leq\sqrt{2\pi n}\left(\frac{n}{e}\right)^ne^{\frac1{12n}}\,.
\end{equation}
From this bound we obtain the following bound on binomial coefficients. 
\begin{lemma}  \label{lemma:eisenhower} The following holds for any $n \geq 1$. \begin{itemize}
\item If $1 \leq j \leq n/2$, it holds that $\binom{n}{j}^{-1} \leq 3 \frac{\sqrt{j} }{4^j}$.
\item If $j \leq n/10$, it holds that $\binom{n}{j}^{-1} \leq \frac{j!}{(9n/10)^j}$.
\end{itemize}
\end{lemma}
\begin{proof} 
For the first one, use Stirling's bound and note that $\frac{j^j (n-j)^{n-j} }{n^n} = f(\alpha)^n$ where $\alpha = j/n$ and $f(\alpha) = \alpha^\alpha (1-\alpha)^{1-\alpha}$. Checking the bound $f(\alpha) \leq 4^{-\alpha}$ for any $\alpha \in [0, 1/2]$ concludes the proof. The second claim is clear from the definition of binomial number.
\end{proof} 

\begin{proof}[Proof of Lemma \ref{lemma:washington}] For the sake of brevity, let us just prove the first bound on each item (the proof for the second is exactly the same changing the indices accordingly). Let us start with \eqref{eq:Macron}. We bound
\begin{align}\label{eq:Chirac1}
\abs{Z_j}&\leq \abs{Z_1} \prod_{i=1}^{159}\left| 500\bar C_\ast \frac{(i+1) k}{i+1-k} \right| \prod_{i=160}^{j-1} \left| 3\bar C_\ast \frac{(i+1)^2k}{i+1-k} \right|  \leq \abs{Z_1} \frac{\left(500 / 3\right)^{159}} {160!} \frac{(j!)^2\abs{3\bar C_\ast k} ^{j-1}}{(k-2)_{(j-1)}}\,,
\end{align}
and similarly
\begin{align}\label{eq:Chirac2}
\abs{Z_{m+1-j}}
& \leq  |Z_m|\frac{\abs{(m-k)_{(j-1)}}((m+1-j)!)^2}{\abs{\bar C_\ast k}^{j-1} (m!)^2 } \,.
\end{align}
Substituting in these bounds and using the bound  $\frac{\abs{(m-k)_{(j-1)}}}{(k-2)_{(j-1)}} \leq 1$, along with Lemma \ref{lemma:eisenhower}, we obtain
\begin{align*}
\binom{m+1}{j} |Z_j Z_{m+1-j}| &\leq \binom{m+1}{j}\abs{Z_1}\abs{Z_m} \frac{3 ^{j-160} \cdot500 ^{159} (m+1)^2 } {160!} {\binom{m+1}{j}}^{-2}\ \\
&\leq \binom{m+1}{1} |Z_1||Z_m| (m+1) \frac{3 ^{j-159} \cdot500 ^{159}  } {160!} \frac{\sqrt j}{4^j}\\
&\leq \binom{m+1}{1} |Z_1||Z_m|  (3/4)^{j/11}\left((3/4)^{m/11}\sqrt{m}(m+1) \frac{(500/3)^{159}}{160!}\right) \\
 &\leq \binom{m+1}{1} |Z_1||Z_m| (3/4)^{j/11},
\end{align*}
where we used $j\leq m/10$ in the third inequality and in the last inequality we using that the last factor is bounded by 1 for $m\geq 10000$. To obtain this last observation, we note
\[\frac{d}{dm} \left( (3/4)^{m/11}\sqrt{m}(m+1) \right) \leq 0\quad \mbox{for }m\geq\frac{33+\sqrt{1089+4 \left(\log \left(\frac{4}{3}\right)-11\right) \log
   \left(\frac{4}{3}\right)}-4 \log (2)+\log (9)}{\log \left(\frac{256}{81}\right)}\,.\]
 In particular, the inequality holds for $m\geq 100$, thus the desired inequality holds as a consequence of
 \[1000100\cdot (3/4)^{10000/11} \frac{(500/3)^{159}}{160!}\leq 1\,.\]
Let us now show \eqref{eq:Hollande}. We apply \eqref{eq:Chirac1}, \eqref{eq:Chirac2}, Lemma \ref{lemma:eisenhower} using $160\leq j \leq m/10$,
\begin{align*} \begin{split} 
\binom{m+1}{j} |Z_j Z_{m+1-j}| &\leq \binom{m+1}{1} |Z_1||Z_m| (m+1) \frac{3 ^{j-160} \cdot500 ^{159}  } {160!} \frac{j!}{(9(m+1)/10)^j} \\
&\leq \binom{m+1}{1} |Z_1||Z_m| \frac1{3^j m^9}\frac{10^j\cdot j!}{m^{j-9}} \frac{500^{159} }{ 3^{160}\cdot 160!} \\
&\leq \binom{m+1}{1} |Z_1||Z_m| \frac1{3^j m^9} \sqrt{2\pi j}m^9 \left(\frac{10 j}{e m}\right)^j  \frac{500^{159} e^{\frac1{12j}}}{ 3^{160}\cdot 160!} \\
&\leq \binom{m+1}{1} |Z_1||Z_m| \frac1{3^j m^9} m^{10} \left(\frac{10j}{ e m}\right)^j  \frac{2\cdot 500^{159} }{ 3^{160}\cdot 160!}\,,
\end{split} \end{align*}
where we used \eqref{eq:factorial:bnd}.
Observe that $\frac{d}{dj}\left(\frac{10j}{ e m}\right)^j \leq 0$ for $j\leq \frac{m}{10}$. Then \eqref{eq:Hollande} follows once we note that for $160\leq j\leq \frac m{10}$ and $m\geq 10000$ we have
\[m^{10} \left(\frac{10j}{ e m}\right)^j  \frac{2\cdot 500^{159} }{ 3^{160}\cdot 160!}\leq m^{-150}  \frac{2\cdot 500^{159}\cdot 1600^{160} }{ (3e)^{160}\cdot 160!}\leq 1\,.\]

Finally, let us show \eqref{eq:Sarkozy} for $j < 160$. Then by \eqref{eq:kennedy2}, we have
\begin{align}\label{eq:Chirac3}
\abs{Z_j}&\leq \abs{Z_1} \prod_{i=1}^{j-1} \left| 500\bar C_\ast \frac{(i+1) k}{i+1-k}  \right| \leq \abs{Z_1}  \frac{j!\abs{500 \bar C_\ast k} ^{j-1}}{(k-2)_{(j-1)}}\,,
\end{align}
and \eqref{eq:Chirac2} holds. Thus
\begin{align*} \begin{split}
\binom{m+1}{1} |Z_j Z_{m+1-j}| &\leq  \binom{m+1}{j} |Z_1||Z_m|   \frac{(m-j+1)!500^{j-1}}{m!} \\
&\leq  \binom{m+1}{j} |Z_1||Z_m|   \frac{101}{100}\frac{(m-j+1)^{m-j+\frac32}(500e)^{j-1}}{m^{m+\frac12}}\\
&\leq \binom{m+1}{1} |Z_1||Z_m| \left( \frac{600}{m} \right)^{j-1}  \frac{101}{100}\left(\frac{m-j+1}{m}\right )^{m-j+1}\left(\frac{5 e}{6}\right)^{j-1}\,.
\end{split} \end{align*}
Then to conclude, we observe that
\[\frac{d}{dj} \left( \left(\frac{m-j+1}{m}\right )^{m-j+1}\left(\frac{5 e}{6}\right)^{j-1} \right)\leq 0\,,\]
which is negative so long as $j\leq \frac{6+m}{6}$, which is clearly satisfied. Then, substituting $j=2$ we obtain \eqref{eq:Sarkozy} as a consequence of
\[\frac{101}{100}\left(\frac{m-j+1}{m}\right )^{m-j+1}\left(\frac{5 e}{6}\right)^{j-1}\leq \frac{101}{100} \left(\frac{m-1}{m}\right )^{m-1}\left(\frac{5 e}{6}\right)\leq 1\,,\]
for $m\geq 10000 $.
\end{proof}

\subsection{Closing the induction for $W_i$}

Let us recall the following bounds from Lemma \ref{lemma:aux_bounds7o5}. If  $r \in (r_n, r_{n+1})$, for $n$ sufficiently large, then for any $i, j \in \left\{ W, Z \right\}$
\begin{equation}\label{eq:useful}
\begin{gathered}
D_{W, 0} \leq 2|Z_0|,\quad |Z_1| \leq 3/10,
\quad |W_1| \leq 1/2,\\
|\p_i N_\circ (P_s)| \leq 2,\quad |\p_i D_\circ | \leq 3/5,
\quad |\p_i  \p_j N_\circ (P_s) | \leq 7/5\,.
\end{gathered}
\end{equation}

\begin{lemma}\label{lemma:W:ind}
Assuming the estimates in Lemma \ref{lemma:kennedy} hold for all $i< m$ where we take $10000\leq m\leq k$, then  \eqref{eq:kennedy1} holds for $i=m$.
\end{lemma}
\begin{proof}
Let us rewrite equation \eqref{eq:Wn} as
\begin{align} \begin{split} \label{eq:Wanalysis}
D_{W, 0} W_{m+1} &=  (\p_Z N_W (P_s) - W_1 \p_Z D_W ) Z_m+  (N_{W, m} - \p_ZN_W(P_s)Z_m)  \\
&\qquad - W_1\p_W D_Z(P_s) W_m   - \sum_{j=1}^{m-1} \binom{m}{j}D_{W, m-j} W_{j+1} =\rm{I}+\rm{II}+\rm{III}+\rm{IV}\,.
\end{split} \end{align}
Taking the limit $r\rightarrow r^\ast$ yields
\[ \lim_{r\rightarrow r^\ast}(\p_Z N_W (P_s) - W_1 \p_Z D_W )=\frac{2}{5+3 \sqrt{5}}\,,\]
(see Lemma \ref{lemma:aux_limits}). Thus, if we take $k$ sufficiently large, we obtain
\begin{equation}
\abs{\rm{I}}\leq \left( \frac{2}{5+3 \sqrt{5}}+o(1) \right) \abs{Z_m}\leq \frac{2}{5}\abs{Z_m}\,.\label{eq:Lorikeet1}
\end{equation}

Now consider $\rm{II}$. Applying \eqref{eq:kennedy1}--\eqref{eq:kennedy4} ,\eqref{eq:WZ:quotient}, Lemma  \ref{lemma:washington} and \eqref{eq:useful}
\begin{align}
\abs {\rm{II}} &= \left| \p_W N_W (P_s) W_{m} + \frac{1}{2}\sum_{j=1}^{m-1} \binom{m}{j} (W_j, Z_j)^\top HN_W \cdot (W_{m-j}, Z_{m-j}) \right|\notag \\
&\leq 2|W_{m}| + 5 \sum_{j=1}^{m/2} \binom{m}{j} |Z_{j}| \abs{HN_W}|Z_{m-j}| \notag\\
&\leq 8|Z_{m-1}| + 14\sum_{j=1}^{m/2} \binom{m}{j} |Z_{j}| |Z_{m-j}|\notag \\
&\leq 8|Z_m| + 14\binom{m}{1} |Z_1 Z_{m-1}| \left( \sum_{j=1}^{160} \left( \frac{600}{m} \right)^{j-1} + \frac{1}{m^9}\sum_{j=161}^{m/10} (1/3)^{j} + \sum_{j=m/10}^{m/2} (3/4)^{j/11} \right)\notag \\
&\leq 8|Z_m| + \frac{21}{5}m |Z_{m-1}|  \left( \frac{10}{9} + \frac{1}{1000} + \frac{1}{1000} \right) \notag\\
&  \leq 5m|Z_{m-1}|\leq \frac{5\abs{k-m}}{|\bar C_\ast|m k} |Z_{m}|\leq \frac{1700}{m } |Z_{m}|, \label{eq:Lorikeet2}
\end{align}
where in the last inequality we used Remark \ref{rem:C4}. Here, we used the estimate \[\abs{HN_W}_{\infty}\leq \max_{i=1,2}\sum_{j=1,2}\abs{(HN_W)_{ij}}\leq \frac{14}{5}\,.\]

To estimate $\rm{III}$, we first note that by \eqref{eq:kennedy1}--\eqref{eq:kennedy4}
\[|W_m| \leq 4|Z_{m-1}| \leq \frac{4\abs{k-m}}{|\bar C_\ast|m^2k} |Z_m| \leq \frac{1500}{m^2}|Z_m|, \]
using again Remark \ref{rem:C4}. Therefore, using \eqref{eq:useful} we obtain
 \begin{equation}
 |{\rm{III}}| \leq \frac{1}{2} \cdot \frac{3}{5} \cdot \frac{1500}{m^2} |Z_m| = \frac{450}{m^2}|Z_m|\,.\label{eq:Lorikeet3}
 \end{equation}

Lastly, let us consider $\rm{IV}$. We remark that $D_{W, i}= \frac{3W_i+2Z_i}5$ for $i\geq 1$. Then applying \eqref{eq:kennedy1}--\eqref{eq:kennedy4}, \eqref{eq:WZ:quotient}, Lemma~\ref{lemma:washington} and \eqref{eq:useful}
\begin{align}
\abs{\rm{IV}} &\leq \frac{1}{5}\sum_{j=1}^{m-1} \binom{m}{j} |3W_{m-j}+2Z_{m-j}| |W_{j+1}| \notag\\
 &\leq \frac{32}{5}\sum_{j=1}^{m-1} \binom{m}{j} |Z_{m-j}| |Z_{j}| \notag
\\&\leq 13 \sum_{j=1}^{m/2} \binom{m}{j} |Z_j||Z_{m-j}|\notag \\
&\leq 13 m |Z_1||Z_{m-1}| \left( \sum_{j=1}^{160} \left( \frac{600}{m} \right)^{j-1} + \frac{1}{m^9}\sum_{j=160}^{m/10} (1/3)^j + \sum_{j=m/10}^{m/2} (3/4)^{j/11}  \right)\notag \\
&\leq 14 m |Z_1||Z_{m-1}| \leq   \frac{5}{m|\bar C_\ast|} |Z_m| \leq \frac{2000}{m} |Z_m|,\label{eq:Lorikeet4}
\end{align}
where we used Remark \ref{rem:C4} in the last inequality.

Then combining \eqref{eq:Lorikeet1}--\eqref{eq:Lorikeet4} we obtain
\[\abs{D_{W, 0} W_{m+1}} \leq \left(\frac{2}{5}+\frac{1700}{m }+\frac{450}{m^2}+\frac{2000}{m}\right)\abs{Z_m} \leq \abs{Z_m}\,. \]

Noting
\[\lim_{r\rightarrow r^\ast} D_{W,0}=\frac{2}{1+\sqrt{5}}\,,\]
from Lemma \ref{lemma:aux_limits}, then for $k$ sufficiently large, we obtain our claim.
\end{proof}

\subsection{Closing the induction for $Z_i$}

\begin{lemma}\label{lemma:Z:ind}
Assuming the estimates in Lemma \ref{lemma:kennedy} hold for all $i< m$ where we take $10000\leq m\leq k$, and \eqref{eq:kennedy1} holds for $i=m$, then 
\eqref{eq:kennedy4} holds for $i=m$.
\end{lemma}
\begin{proof}
Let us recall \eqref{eq:Zn}
\begin{align*}
 Z_{m+1} D_{Z, 1} (m+1-k) &= -\binom{m+1}{m-1} D_{Z, 2} Z_{m}  - \sum_{j=1}^{m-2} \binom{m+1}{j} D_{Z, m+1-j} Z_{j+1}  \notag \\
 &\qquad+\left(  N_{Z, m+1} - (\p_Z N_Z (P_s) ) Z_{m+1}\right) - Z_1  W_{m+1} \p_W D_Z (P_s) \\
 &= -\frac m 2 (1 + m)D_{Z, 2} Z_{m} +\rm{I}+\rm{II}+\rm{III}\,.
\end{align*}

Applying \eqref{eq:WZ:quotient}, Lemma \ref{lemma:washington} and \eqref{eq:useful}
\begin{align}
\abs{\rm{I}}&\leq  \frac{1}{5}\sum_{j=1}^{m-2} \binom{m+1}{j} (3\abs{W_{ m+1-j }}+2\abs{Z_{ m+1-j }}) \abs{Z_{j+1}} \nonumber \\
&\leq  \frac{8}{5}\sum_{j=1}^{m-2} \binom{m+1}{j} \abs{Z_{m+1-j}} \abs{Z_{j+1}} \nonumber \\
&\leq  \frac{16}{5}\sum_{j=1}^{(m+1)/2} \binom{m+1}{j} \abs{Z_{m+1-j}} \abs{Z_{j+1}} \nonumber \\
&=  \frac{16}{5}\sum_{j=2}^{(m+2)/2} \frac{j}{m+2} \binom{m+2}{j} \abs{Z_{m+2-j}} \abs{Z_{j}} \nonumber \\
&\leq  \frac{16}{5}(m+1) |Z_1 Z_m| \left( \sum_{j=2}^{160} j\left( \frac{600}{m} \right)^{j-1} + \frac{1}{m^9}\sum_{j=161}^{m/10} j(1/3)^{j} + \sum_{j=m/10}^{(m-2)/2} j(3/4)^{j/11} \right) \nonumber \\
& \leq \frac {m}{10}|Z_m|\,.\label{eq:currawong1}
\end{align}

Applying \eqref{eq:kennedy1}, \eqref{eq:WZ:quotient}, Lemma \ref{lemma:washington} and \eqref{eq:useful}
\begin{align}
\abs{\rm{II}} &= \abs{\p_W N_Z (P_s) W_{m+1} }+ \sum_{j=1}^{m/2} \binom{m+1}{j} \abs{(W_j, Z_j)^\top HN_Z \cdot (W_{m+1-j}, Z_{m+1-j}) }\notag\\
&\leq 2|W_{m+1}| +(m+1) \abs{(W_1, Z_1)^\top HN_Z \cdot (W_{m}, Z_{m}) }+ \sum_{j=2}^{m/2} \binom{m+1}{j} \abs{(W_j, Z_j)^\top HN_Z \cdot (W_{m+1-j}, Z_{m+1-j}) }\notag\\
&\leq 8|Z_{m}| +(m+1) \abs{(W_1, Z_1)^\top HN_Z \cdot (W_{m}, Z_{m}) }+ 5\sum_{j=2}^{m/2} \binom{m+1}{j}  |Z_{j}| |Z_{m+1-j}|\notag\\
&\leq 8|Z_{m}| +(m+1) \abs{(W_1, Z_1)^\top HN_Z \cdot (W_{m}, Z_{m}) }\notag\\
&\qquad+ 5(m+1) |Z_1 Z_m| \left( \sum_{j=2}^{160} \left( \frac{600}{m} \right)^{j-1} + \frac{1}{m^9}\sum_{j=161}^{m/10} (1/3)^{j} + \sum_{j=m/10}^{(m+1)/2} (3/4)^{j/11} \right) \notag\\
&\leq (m+1) \abs{(W_1, Z_1)^\top HN_Z \cdot (W_{m}, Z_{m}) }+ \frac{m}{10} | Z_m|\,.\label{eq:Kea}
\end{align}
Note
\[HN_Z (P_s) =
\begin{pmatrix}
 \frac{\gamma-1}{2} & \frac{\gamma-3}{4} \\
 \frac{\gamma-3}{4} & -\gamma \\
\end{pmatrix}
= \frac{1}{5} \begin{pmatrix} 1 & -2 \\ -2 & -7 \end{pmatrix}\,,\]
and $\lim_{r\rightarrow r^\ast}(W_1,Z_1)=(5-3 \sqrt{5}) \left(\frac14,-\frac16\right)$ from Lemma \ref{lemma:aux_limits}. Hence
\[\lim_{r\rightarrow r^\ast} (W_1, Z_1)^\top HN_Z \cdot (W_{m}, Z_{m}) =-\frac{1}{60} \left(3 \sqrt{5}-5\right) (7 W_m+8 Z_m)\,.\]
 Thus, for $k$ sufficiently large, we have
\[\abs{ (W_1, Z_1)^\top HN_Z \cdot (W_{m}, Z_{m})}\leq \frac 15 (\abs{W_m}+\abs{Z_m})\leq \frac{3}{5}\abs{Z_m}\,.\]
Combining this estimate with \eqref{eq:Kea}, we obtain
\begin{equation*}
\abs{\rm{II}} \leq m| Z_m|\,.
\end{equation*}

Finally, applying \eqref{eq:kennedy1}, \eqref{eq:useful} to $\rm{III}$ yields
\begin{equation}\label{eq:currawong3}
\abs{\rm{III}}= |- Z_1 W_{m+1} \p_W D_Z| \leq \frac{3}{10}\cdot 4 |Z_m| \cdot \frac35 \leq  |Z_m|\,.
\end{equation}

Hence combining \eqref{eq:currawong1}--\eqref{eq:currawong3} we obtain
\begin{align*}
\abs{ Z_{m+1} D_{Z, 1} (m+1-k) +\frac m 2 (1 + m)D_{Z, 2} Z_{m}}  \leq \abs{\rm{I}}+ \abs{\rm{II}}+ \abs{\rm{III}}
\leq 
\left( \frac{1}{10}  + 1 + \frac 1 m\right) m|Z_m|\leq \frac32 m|Z_m|\,.
\end{align*}
Thus
\begin{align*}
\left| Z_{m+1} (m+1-k) \frac{2 D_{Z, 1}}{(m+1)m D_{Z, 2}} + Z_m \right| \leq \frac{3}{(m+1)|D_{Z, 2}|}  |Z_m| \,.
\end{align*}
Using
\[\lim_{r\rightarrow r^\ast }D_{Z, 2}=\frac{1}{132} (19 - 9 \sqrt 5)\,,\]
from Lemma \ref{lemma:aux_limits} we obtain
\[\left| Z_{m+1} (m+1-k) \frac{2 D_{Z, 1}}{(m+1)m D_{Z, 2}} + Z_m \right| \leq \frac{2}{50} |Z_m|\,,\]
from which we obtain \eqref{eq:kennedy4}.
\end{proof}

\subsection{Further corollaries of Lemma \ref{lemma:kennedy}}

\begin{corollary} \label{cor:Znsign} Let $\gamma = 7/5$, $n$ odd sufficiently large and $r \in (r_n, r_{n+1})$. We have that $Z_n > 0$ and $Z_{n+1} < 0$. Moreover, we have
\begin{equation}\label{eq:kookaburra}
\frac{ k^n}{k-n} |\bar C_\ast|^n \lesssim\frac{Z_n}{n!} \lesssim \frac{n k^n}{k-n } |\bar C_\ast|^n \left( \frac{1.05}{0.95} \right)^{n} \,.
\end{equation}
In particular, $\left( \frac{|Z_n|}{n!} (k-n) \right)^{1/n} \approx n$.
\end{corollary}
\begin{proof} 
By Lemma \ref{lemma:tenthousand_7o5} we have
\[ 
\abs{Z_{10000}+6\cdot 10^{46770}}\leq 10^{46770}, \]
which by a liberal choice of implicit constant may be rewritten as
\[
|\bar C_\ast|^{10000} \frac{{10000}! k^{10000}}{|({10000}-k)_{(10000)}|} \lesssim \frac{-Z_{10000}}{10000!} \lesssim   |\bar C_\ast|^{10000}  \frac{n! k^{10000}}{|({10000}-k)_{(10000)}|}\left( \frac{1.05}{0.95} \right)^{10000}\,.\]
Then, applying \eqref{eq:kennedy4} successively yields
\[
|\bar C_\ast|^n \frac{n! k^n}{|(n-k)_{(n)}|} \lesssim \frac{Z_n}{n!} \lesssim |\bar C_\ast|^n  \frac{n! k^n}{|(n-k)_{(n)}|}\left( \frac{1.05}{0.95} \right)^{n}.
\]
We notice that $(k-n)(n-1)! \leq |(n-k)_{(n)}| \leq (k-n) n!$ which follows by cancelling the factor $(k-n)$ and bounding $1 \leq (k-n+1) \leq 2$, $2 \leq (k-n+2) \leq 3$, $\dots$. Thus, we obtain \eqref{eq:kookaburra}.

Finally, we note the statement $Z_{n+1}<0$ follows as a consequence of \eqref{eq:kennedy4} and \eqref{eq:kookaburra}.
\end{proof}
Note that this Corollary concludes the proof of Lemmas \ref{lemma:right_up_initial} and \ref{lemma:right_down_initial}.

\begin{corollary} \label{cor:Znplus1mWnplus1overW1} We have that for $\gamma = 7/5$ and $n$ odd and sufficiently large ($r \in (r_n, r_{n+1})$),
\begin{equation*}
Z_{n+1} - \frac{Z_1 W_{n+1}}{W_1} < 0\,.
\end{equation*}
\end{corollary}
\begin{proof} From Lemma \ref{lemma:kennedy}, we have that $|W_{n+1}| \lesssim |Z_n| \lesssim \frac{n+1-k}{(n+1)^2k }|Z_{n+1}| \lesssim \frac{|Z_{n+1}|}{n^2}$. Using Corollary \ref{cor:Znsign}, we have $Z_{n+1} < 0$.\end{proof}

Let us note that this Corollary ends the proof of Lemma \ref{lemma:left_initial} for the case of $\gamma = 7/5$ and sufficiently large odd $n$. We will now show how Lemma \ref{lemma:kennedy} implies $\log$-convexity of $\frac{Z_i}{i!}$.

From Lemma \ref{lemma:kennedy} we can compare the size of $|Z_i|$ with $|Z_{i+1}|$. The following Lemma just applies Lemma \ref{lemma:kennedy} for all $i = a, a+1, \ldots n-1$ and concatenates those bounds in order to obtain a direct comparison between $|Z_a|$ and $|Z_n|$.
\begin{lemma} \label{lemma:grant} For $k$ sufficiently large, we have that for any $a \leq  n = \lfloor k \rfloor$, 
\begin{equation*}
\frac{|Z_{a}| }{a!} \lesssim \left( \frac{Z_n}{n!} \right)^{a/n} \binom{n}{a}^{-1} \left( \frac{1.05}{0.95} \right)^{\frac{a(n-a)}{n}} (k-n)^{a/n-\lfloor a/n \rfloor} \,.
\end{equation*}
\end{lemma} 
\begin{proof}
Let us start supposing $a < n$, so that $a/n - \lfloor a/n \rfloor = a/n$. Writing $\frac{Z_a}{a!} = \prod_{j=1}^a  \frac{1}{j} \frac{Z_j}{Z_{j-1}}$ and $\frac{Z_n}{n!} = \prod_{j=1}^n  \frac{1}{j} \frac{Z_j}{Z_{j-1}}$, the statement is equivalent to
\begin{equation*}
 \left( \prod_{j=1}^{a} \frac{|Z_j|}{|Z_{j-1}| jk} \right)^{1-a/n} \lesssim \left( \prod_{j=a+1}^{n} \frac{|Z_j|}{|Z_{j-1}| jk} \right)^{a/n} \binom{n}{a}^{-1}  \left( \frac{1.05}{0.95} \right)^{ \frac{a(n-a)}{n}} (k-n)^{a/n}\,.
\end{equation*}
Using Lemma \ref{lemma:kennedy}, the left-hand-side is $\lesssim \prod_{j=1}^a \frac{j}{k-j}\bar C_\ast\left( \frac{1.05}{0.95} \right)^a$ and the parenthesis in the right-hand-side is $\gtrsim  \prod_{j=a+1}^n \frac{j}{k-j}\bar C_\ast$. Thus, it just suffices to show
\begin{align*}
&\left( \prod_{j=1}^{a} \frac{j}{k-j} \bar C_\ast \right)^{1-a/n} \left( \frac{1.05}{0.95} \right)^{a(1-a/n)} \lesssim \left( \prod_{j=a+1}^{n} \frac{j\bar C_\ast}{k-j} \right)^{a/n} \binom{n}{a}^{-1}  \left( \frac{1.05}{0.95} \right)^{ \frac{a(n-a)}{n}} (k-n)^{a/n} \\
\Leftrightarrow &\left( \prod_{j=1}^{a} \frac{j}{k-j} \right)^{1-a/n} \lesssim \left( \prod_{j=a+1}^{n} \frac{j}{k-j} \right)^{\frac{a}{n}} \binom{n}{a}^{-1} (k-n)^{a/n} \\
\Leftarrow &\left( \frac{a!}{n! /(n-a)! }\right)^{1-a/n} \lesssim \left( \frac{n!/a!}{(n-a)! (k-n)} \right)^{\frac{a}{n}} \binom{n}{a}^{-1} (k-n)^{a/n}\,.
\end{align*}
As the last inequality is in fact an equality, we are done with the case $a < n$. For the case $a=n$ all the implications work the same except the last one, as we should erase the factor $(k-n)$ (this is because $\prod_{j=a+1}^n (k-j)$ has no factors for $a=n$ so we cannot extract a factor $(k-n)$). Thus erasing the factors $(k-n)$ (and their powers) in all the equations, the proof works the same for $a=n$ . As $a/n - \lfloor a/n \rfloor  = 0$ for $a=n$, the statement is also correct in that case.
\end{proof}

\subsection{Estimates on $P$ - Validity of the near-left barrier}
Let us recall  $b^{\rm nl}_{n, W}(s) = \sum_{i=0}^n \frac{W_i}{i!} s^i$, $b^{\rm nl}_{n, Z}(s) = \sum_{i=0}^n \frac{Z_i}{i!} s^i$ and 
\begin{equation*}
P^{\rm nl}_n (s) = b^{\rm nl \, \prime}_{n, Z}(s) N_W (b^{\rm nl}_n (s)) D_Z (b^{\rm nl}_n (s)) - b^{\rm nl \, \prime}_{n, W}(s) N_Z (b^{\rm nl}_n (s)) D_W (b^{\rm nl}_n (s))\,.
\end{equation*}

In order to obtain the sign of $P^{\rm nl}_n(s)$, we prove the following lemma:
\begin{proposition} \label{prop:fdr} For any $s$ with $0 \leq s^{n-2} \leq 3 \frac{|D_{Z, 2}^\ast| }{2 \p_Z D_Z}\frac{n!}{Z_n} $, we have
\begin{align} \label{eq:Pnlfdr}
P^{\rm nl}_n(s) =    n \frac{Z_n}{n!} \frac{N_{W, 0} D_{Z, 2}}{2} s^{n+1} \left( 1 +O \left( \frac1n \right) \right) + n\left(\frac{Z_n}{n!} \right)^2 \p_Z D_ZN_{W, 0} s^{2n-1}  \left( 1 +O \left( \frac1n \right) \right).
\end{align}
Moreover, letting
\begin{equation} \label{eq:s75val}
s_{7/5, \rm val} = \left( \frac{99}{100}\frac{|D_{Z, 2}|}{2 \p_Z D_Z} \frac{n!}{Z_n} \right)^{\frac{1}{n-2}},
\end{equation}
we have that $P^{\rm nl}_n(s) > 0$ for $s \in (0, s_{7/5, \rm val} )$.
\end{proposition}

Before we prove Proposition \ref{prop:fdr}, we need a few auxiliary lemmas. 

In what follows, it will be useful to introduce the following notation: Let us define the discrepancy $\ell$ of a number $a$ as 
\[\ell(a)=\min \{ a, n-a\}\quad\mbox{and}\quad \ell (a, b, c, d)=\ell(a)+\ell(b)+\ell(c)+\ell(d)\,.\]
We also define
\[\stirling{n}{a} = \binom{n}{a} \left( \frac{0.95}{1.05} \right)^{(a(n-a))/n}\,.\]
The brackets behave asymptotically like the binomial coefficients in the sense that they are symmetric, we have $\stirling{n}{a}^{-1} \lesssim \frac{1}{n^a}$ if $a$ is fixed and we also have $\stirling{n}{a}^{-1}  \lesssim \frac{1}{n^{10}} 3^{-\ell (a)}$ for any $1\leq a\leq n$ with $\ell (a) \geq 10$. 

The strategy towards Proposition \ref{prop:fdr} is to develop a Taylor series for $P^{\rm{nl}}_n(s)$ and bound all the terms except the dominant ones. In Lemma \ref{lemma:nixon}, we will bound the $i$-th coefficient for $i \leq n$, and in Lemma \ref{lemma:carter}, we will extract the main contributions of the $(n+1)$-th and $(2n-1)$-th terms. Those correspond to the main terms in \eqref{eq:Pnlfdr}. Lastly, in the proof of Proposition \ref{prop:fdr}, we will deal with all the $i$-th Taylor coefficients, for $i \geq n+2$, $i \neq 2n-1$. In all those proofs, we will develop the expressions in terms of $|Z_i|$ (or $|W_i| \les |Z_{i-1}|$) and we will observe that the term arising from $|Z_a| |Z_b| |Z_c| | Z_d |$ is proportional to $\stirling{n}{a}^{-1} \stirling{n}{b}^{-1} \stirling{n}{c}^{-1} \stirling{n}{d}^{-1}$. Since that is smaller than $n^{-\ell (a, b, c, d) }$, we will be able to show that the terms with small discrepancy dominate, and they will correspond exactly to the main contributions from the $(n+1)$-th and $(2n-1)$-th terms of the Taylor series.

\begin{lemma} \label{lemma:nixon} Let $n$ be an odd number sufficiently large and $ i \leq n$. We have that \begin{equation*}
\frac{\p_s^{i-1} P^{\rm nl}_{n} (0)}{(i-1)!} \lesssim i \left( \frac{|Z_n|}{n!}\right)^{i/n} \frac{(k-n)^{i/n - \lfloor i/n \rfloor}}{n^{\min \{7, \ell (i) \} }}.
\end{equation*}
\end{lemma}
\begin{proof} Let us recall \eqref{eq:Pc}:
\begin{equation*}
P^{\rm nl}_n (s) = b^{\rm nl \; \prime}_{n, Z}(s) N_W (b^{\rm nl}_{n}(s) ) D_Z (b^{\rm nl}_{n}(s)) -  b^{\rm nl \; \prime}_{n, W}(s)N_Z (b^{\rm nl}_{n}(s)) D_W (b^{\rm nl}_{n}(s)).
\end{equation*}
 We have that $|D_{\circ, i}| \lesssim |\p^i_s b_{n, Z}^{\rm nl}(s)|$ and $|\p^i_s (N_Z(b_n^{\rm nl}(s)))| \lesssim |\p^i_s (b_{n,Z}^{\rm nl}(s)^2)| + |\p_s^i b_Z^{\rm nl}(s)| \lesssim |\p^i_s (b_Z^{\rm nl}(s)^2)|$. Using Lemma \ref{lemma:grant}:
\begin{align}
\p^{i-1}_s P^{\rm nl}_{n} (0) &\lesssim |\p^{i-1}_s (b_{n, Z}^{\rm nl\, \prime}(s) b_{n, Z}^{\rm nl}(s)^3)(0)| = \sum_{\mc J_i} \frac{(i-1)!}{(a-1)!b!c!d!} |Z_a||Z_b||Z_c||Z_d| \notag \\
& \lesssim \sum_{\mc J_i} \frac{i!}{a!b!c!d!} |Z_a||Z_b||Z_c||Z_d|  \notag \\
&\lesssim i! \left( \frac{|Z_n|}{n!} \right)^{i/n} (k-n)^{i/n - \lfloor i/n \rfloor} \sum_{\mc J_i} \stirling{n}{a}^{-1} \stirling{n}{b}^{-1} \stirling{n}{c}^{-1} \stirling{n}{d}^{-1} \,,\label{eq:mercurio}
\end{align}
where
\begin{equation*}
\mc J_i = \left\{ (a, b, c, d) \in \mathbb{Z}^4: \quad 0 \leq a, b, c, d \leq n \quad \mbox{and} \quad a+b+c+d = i \right\}.
\end{equation*}
In the last inequality we also used that the function $x-\lfloor x \rfloor$ is subadditive.
Let us decompose $\mc J_i = \mc J_i' \sqcup \mc J_i''$ with $\mc J_i', \mc J_i''$ given by
\begin{align*}
\mc J_i'  = \mc J_i \cap \left\{ (a, b, c, d) \in \mathbb{Z}^4: \ell (a, b, c, d) \geq 10 \right\}, \\
\mc J_i'' = \mc J_i \cap \left\{ (a, b, c, d) \in \mathbb{Z}^4: \ell (a, b, c, d) < 10    \right\}.
\end{align*}

As $|\mc J_i'| \les n^3$, we have that
\begin{equation} \label{eq:venus}
\sum_{\mc J_i'} \stirling{n}{a}^{-1} \stirling{n}{b}^{-1} \stirling{n}{c}^{-1} \stirling{n}{d}^{-1} \les |\mc J_i'| \frac{1}{n^{10}} \les \frac{1}{n^7}.
\end{equation}
On the other hand, $| \mc J_i'' | \les 1$, so we deduce
\begin{equation} \label{eq:marte}
\sum_{\mc J_i''} \stirling{n}{a}^{-1} \stirling{n}{b}^{-1} \stirling{n}{c}^{-1} \stirling{n}{d}^{-1} \les | \mc J_i''| \max_{\mc J_i''} \frac{1}{n^{\ell (a, b, c, d)}} \les \frac{1}{n^{\ell (i)}},
\end{equation}
where we are using superadditivity of the discrepancy: $\ell(a)+\ell(b)+\ell(c)+\ell(d) \geq \ell(i)$.

Plugging in \eqref{eq:venus} and \eqref{eq:marte} into \eqref{eq:mercurio} we conclude the proof.\end{proof}


Now, recall that $\p^{i-1}_s P^{\rm nl}_n (0) = 0$ for any $i \leq n+1$. The previous lemma will guarantee that $\p^{i-1}_s P^{\rm nl}_n (0)$ are not dominant for $i \neq n+2, 2n$. Let us thus analyze the precise order of the $(n+2)$-th and $(2n)$-th derivatives of $P$.

\begin{lemma} \label{lemma:carter} We have that for odd $n$, sufficiently large
\begin{align}
\frac{\p^{n+1}_s P^{\rm nl}_n (0)}{(n+1)!} &=  n \frac{Z_n}{n!} \frac{N_{W, 0} D_{Z, 2}}{2}\left(1+O\left( \frac1n \right) \right) \,,\label{eq:bilby1}\\
\frac{\p^{2n-1}_s P^{\rm nl}_n (0)}{(2n-1)!} &=  n\left(\frac{Z_n}{n!} \right)^2 N_{W, 0} \p_Z D_Z\left(1+O\left( \frac1n \right) \right)\,.\label{eq:bilbyyy2}
\end{align}
\end{lemma}
\begin{proof} The strategy is similar to the one in Lemma \ref{lemma:nixon}. The main difference is that in Lemma \ref{lemma:nixon} we bounded every term while here we will identify the dominant terms (which will be the ones with the least discrepancy) and bound the rest.

Let us start showing \eqref{eq:bilby1}. Let us define
\begin{align} \begin{split} \label{eq:jupiter}
\mathcal{B} &= \binom{n+1}{2} Z_n N_{W, 0} D_{Z, 2} + (n+1)n Z_n N_{W, 1} D_{Z, 1} + (n+1) Z_1 N_{W, n} D_{Z, 1} + (n+1) Z_2 N_{W, 0} D_{Z, n} \\
&\qquad + (n+1) Z_1 N_{W, 1} Z_n  - \binom{n+1}{2} W_n N_{Z, 2} D_{W, 0} - (n+1)n W_n N_{Z, 1} D_{W, 1}\\
&\qquad  - (n+1) W_1 N_{Z, n} D_{W, 1} - (n+1) W_2 N_{Z, n} D_{W, 0} - (n+1) W_1 N_{Z, 1} D_{W, n} \\
&\qquad  -W_1N_{Z,n+1}D_{W,0}\,,
\end{split}\end{align}
which corresponds to all the monomials of $\p_s^{n+1} P^{\rm nl}_n (0)$ where there are factors with $n$ derivatives. In particular, following the same reasoning as in \eqref{eq:mercurio}, we have that
\begin{equation} \label{eq:saturno}
\left| \p_s^{n+1} P^{\rm nl}_n (0) - \mc B \right| \lesssim  (n+2)! \left( \frac{|Z_n|}{n!} \right)^{1+2/n} (k-n)^{2/n} \sum_{\mt{\mc J}} \stirling{n}{a}^{-1} \stirling{n}{b}^{-1} \stirling{n}{c}^{-1} \stirling{n}{d}^{-1}\,,
\end{equation}
where 
\begin{equation*}
\mathcal{\mt{\mc J}} = \left\{ (a, b, c, d) \in \mathbb{Z}^4: \quad  0 \leq a, b, c, d \leq n, \quad a+b+c+d = n+2 \quad \mbox{and} \quad \ell(a)+\ell(b)+\ell(c)+\ell(d) \geq 4 \right\}. 
\end{equation*}
Note that for any tuple $(a, b, c, d)$ with $a+b+c+d = n+2$ we have that $\ell(a,b,c,d) \geq \ell (n+2) = 2$ and that the discrepancy is even. As $\mc B$ from equation \eqref{eq:jupiter} contains precisely the monomials of discrepancy $2$ of $\p_s^{n+1} P^{\rm nl}_n(0)$, every addend in the right hand side of \eqref{eq:saturno} has discrepancy greater or equal than $4$.

Now, using the same reasoning as in the proof of Lemma \ref{lemma:nixon} (decomposing $\mt{\mc J}$ according to the discrepancy being smaller or greater than $10$), we have that
\begin{equation} \label{eq:neptuno}
\sum_{\mt{\mc J}} \stirling{n}{a}^{-1} \stirling{n}{b}^{-1} \stirling{n}{c}^{-1} \stirling{n}{d}^{-1} \lesssim \frac{1}{n^4}\,.
\end{equation}

On the other hand, from \eqref{eq:jupiter}, we see that
\begin{align}
\mc B &= \binom{n+1}{2} Z_n N_{W, 0} D_{Z, 2} + O\left( n^2 |Z_n| D_{Z, 1} \right) + O \left( n|Z_n| \right) + O \left( n^2 |W_n| \right)\notag  \\
&= \binom{n+1}{2} Z_n N_{W, 0} D_{Z, 2} + O(n|Z_n|)\,,  \label{eq:urano}
\end{align}
since $|D_{Z, 1}| \les \frac{1}{k} \les \frac{1}{n}$ (Lemma \ref{lemma:aux_limits}) and $|W_n| \les |Z_{n-1}| \les n|Z_n|$ (Lemma \ref{lemma:kennedy}).

From equations \eqref{eq:saturno}--\eqref{eq:urano}, we conclude that
\begin{align*}
\frac{\p_s^{n+1} P_n^{\rm nl}(0)}{(n+1)!} &= \frac{1}{(n+1)!} \binom{n+1}{2} Z_n N_{W, 0} D_{Z, 2} + O \left( \frac{n|Z_n|}{(n+1)!} \right) + O \left( \frac{1}{n^3}  \left( \frac{|Z_n|}{n!} \right)^{1+2/n} (k-n)^{2/n} \right) \\
&= n \frac{N_{W, 0} D_{Z, 2}}{2} \frac{Z_n}{n!} + O\left( \frac{|Z_n|}{n!} \right) +  O \left( \frac{|Z_n|}{n!} \cdot \frac{1}{n^3} \left( \frac{|Z_n|(n-k) }{n!}\right)^{2/n} \right) \\
&=  n \frac{Z_n}{n!} \frac{N_{W, 0} D_{Z, 2}}{2}\left(1+O\left( \frac1n \right) \right),
\end{align*}
where we have used Corollary \ref{cor:Znsign} in the last equality and the fact that $|N_{W, 0} D_{Z, 2}| \gtrsim 1$ (Lemma \ref{lemma:aux_limits}).

Now, let us show \eqref{eq:bilbyyy2}. Let us define
\begin{equation} \label{eq:io}
\mw{\mc B} = \binom{2n-1}{n-1} \left( Z_n N_{W, 0} D_{Z, n} - W_n N_{Z, n} D_{W, 0} - W_n N_{Z, 0} D_{W, n} \right),
\end{equation}
which corresponds to all the monomials of $\p_s^{n+1} P^{\rm nl}_n (0)$ where there are two factors with $n$ derivatives. In particular, following the same reasoning as in \eqref{eq:mercurio}, we have that
\begin{equation} \label{eq:ganimedes}
\left| \p_s^{2n-1} P^{\rm nl}_n (0) - \mc B \right| \lesssim  (2n)! \left( \frac{|Z_n|}{n!} \right)^{2} \sum_{\mw{\mc J}} \stirling{n}{a}^{-1} \stirling{n}{b}^{-1} \stirling{n}{c}^{-1} \stirling{n}{d}^{-1}\,,
\end{equation}
where 
\begin{equation*}
\mathcal{\mw{\mc J}} = \left\{ (a, b, c, d) \in \mathbb{Z}^4: \quad  0 \leq a, b, c, d \leq n, \quad a+b+c+d = 2n \quad \mbox{and} \quad \ell(a)+\ell(b)+\ell(c)+\ell(d) \geq 2 \right\}. 
\end{equation*}
Note that for any tuple $(a, b, c, d)$ with $a+b+c+d = 2n$ we have that the discrepancy is even. As $\mw{\mc B}$ from equation \eqref{eq:io} contains precisely the monomials of discrepancy $0$ of $\p_s^{n+1} P^{\rm nl}_n(0)$, every addend in the right hand side of \eqref{eq:ganimedes} has discrepancy greater or equal than $2$.

Now, using the same reasoning as in the proof of Lemma \ref{lemma:nixon}, we have that
\begin{equation} \label{eq:europa}
\sum_{\mw{\mc J}} \stirling{n}{a}^{-1} \stirling{n}{b}^{-1} \stirling{n}{c}^{-1} \stirling{n}{d}^{-1} \lesssim \frac{1}{n^2}\,.
\end{equation}
On the other hand, from \eqref{eq:io} and using $|W_n| \les |Z_{n-1}| \les \frac{1}{n^3} |Z_n|$ (Lemma \ref{lemma:kennedy}), we have that
\begin{align} 
\mw{\mc B} &= \binom{2n-1}{n-1} N_{W, 0} \p_Z D_Z  Z_n^2 + O \left( \binom{2n-1}{n-1} |Z_n|^2 \frac{1}{n^3}  \right) \notag \\
&= (2n-1)! \cdot n  N_{W, 0} \p_Z D_Z  \left( \frac{Z_n}{n!} \right)^2  + O \left( (2n-1)! \left( \frac{Z_n}{n!} \right)^2 \frac{1}{n^2} \right). \label{eq:calisto}
\end{align}
Finally, using \eqref{eq:ganimedes}--\eqref{eq:calisto} we conclude that
\begin{align*}
\frac{\p_s^{n+1} P_n^{\rm nl}(0)}{(n+1)!} &=  n  N_{W, 0} \p_Z D_Z  \left( \frac{Z_n}{n!} \right)^2 + O \left( \left( \frac{Z_n}{n!} \right)^2 \frac{1}{n^2} \right) + O \left( \left( \frac{|Z_n|}{n!} \right)^{2} \frac{1}{n} \right) \\
&=  n\left(\frac{Z_n}{n!} \right)^2 N_{W, 0} \p_Z D_Z\left(1+O\left( \frac1n \right) \right).
\end{align*}
In the second equality, we used $|N_{W, 0} \p_Z D_Z| = \frac{1+\gamma}{4} |N_{W, 0}| \gtrsim 1$ from Lemma \ref{lemma:aux_limits}. \end{proof}

\begin{proof}[Proof of Proposition \ref{prop:fdr}]

Let us start showing equation \eqref{eq:Pnlfdr}. First of all, let us notice that the assumption $s^{n-2} \leq 3 |\frac{D_{Z, 2}^\ast n!}{2\p_Z D_Z Z_n}|$ yields 
\[s \lesssim \left( \frac{n!}{|Z_n|} \right)^{1/(n-2)} \lesssim \frac{1}{n}\,,\]
as a consequence of Corollary \ref{cor:Znsign}. Applying Lemma \ref{lemma:nixon}, we obtain
\begin{align*}
\frac{\p_s^{i-1} P^{\rm nl}_n(0)}{(i-1)!}s^{i-1} &\lesssim \frac{n}{s} \frac{1}{n^{\min \{7, \ell (i) \}}} \left( \frac{|Z_n| s}{n!}\right)^{i/n} \lesssim n s \frac{|Z_n| s^n}{n!} \frac{1}{s^2 n^{\min \{ 7, \ell (i) \}}} \left( s^2 \right)^{i/n - 1}\\
& = n s \frac{|Z_n| s^n}{n!} \frac{ s^{2i/n - 4}}{ n^{\min \{ 7, \ell (i) \}}},
\end{align*}
where we have used $\frac{|Z_n| s^n}{n!} \lesssim s^2$ from our hypothesis on $s$. Let us consider the case $ i \geq 2n+1$. Then, since $s \lesssim 1/n$ and as $2i/n-4 > 0$ we have
\[\frac{ s^{2i/n - 4}}{ n^{\min \{ 7, \ell (i) \}}}\les \frac{1}{n^{2i/n-4}  n^{\min \{ 7, \ell (i) \}}} \,.\]
This can be further bounded, up to a constant multiple by $ \frac{1}{n^2}$ for all $i > 2n+1$ and $i\neq 3n$. For the cases $i=2n+1,3n$, this quantity can be bounded by a constant multiple of  $\frac{1}{n}$. Thus,
\begin{equation*}
\sum_{i=2n+1}^{4n} \frac{\p_s^{i-1} P^{\rm nl}_n(0)}{(i-1)!}s^{i-1} = O \left( n s \frac{|Z_n| s^n}{n!} \cdot \frac{1}{n} \right)\,.
\end{equation*}
Now, for the case $n+3 \leq i \leq 2n-1$, we have $i/n-\lfloor i/n\rfloor = i/n-1$ and we argue with Lemma \ref{lemma:nixon} as follows
\begin{align} \begin{split} \label{eq:cheney}
\frac{\p_s^{i-1} P^{\rm nl}_n(0)}{(i-1)!}s^{i-1} &\lesssim \frac{1}{n^{\min \{7, \ell (i) \}}} \frac{n}{s}  \left( \frac{|Z_n| s^n}{n!}\right)^{i/n} (k-n)^{i/n-1} \\
&\leq   n s \left( \frac{|Z_n| s^n}{n!} \right) s^{i-n-2} \left( \frac{|Z_n| (k-n)}{n!}\right)^{(i-n)/n}  \frac{1}{n^{\min \{7, \ell (i) \}}}\,.
\end{split} \end{align}
Labelling $\mathcal{I} = s^{i-n-2} \left( \frac{|Z_n|(k-n)}{n!} \right)^{(i-n)/n}$, we have by Corollary \ref{cor:Znsign}
\begin{align*}
\mathcal{I}^{(n-2)/(i-n-2)} &= \frac{|Z_n| s^{n-2}}{n!} \left( \frac{|Z_n|(k-n)}{n!} \right)^{\frac{-2i+4n}{in-n^2-2n}}(k-n) \\&\lesssim \left( nk^n |\bar C_\ast|^n \left( \frac{1.05}{0.95} \right)^n \right)^{\frac{-2i+4n}{in-n^2-2n}} \\
&\lesssim \left( n^{1/n} (n+1) \frac{1.05 |\bar C_\ast|}{0.95} \right)^{\frac{n(-2i+4n)}{in-n^2-2n}}\,,
\end{align*}
where we used the hypothesis $s^{n-2} \lesssim \frac{n!}{|Z_n|}$ in the second line. Now, raising the inequality to the $\frac{i-n-2}{n-2}$-th power, we obtain $\mathcal{I} \lesssim   n^{(4n-2i)/(n-2)}$. Going back to \eqref{eq:cheney}, we have that
\begin{equation*}
\frac{\p_s^{i-1} P^{\rm nl}_n(0)}{(i-1)!}s^{i-1} \lesssim  n s \left( \frac{|Z_n| s^n}{n!} \right)  \frac{ n^{(4n-2i)/(n-2)}}{n^{\min \{7, \ell (i) \}}}.
\end{equation*}
Note
\[
\frac{ n^{(4n-2i)/(n-2)}}{n^{\min \{7, \ell (i) \}}}\les
\begin{cases}
\frac1n & \mbox{for }i=n+3,2n-1\\
\frac1{n^2}	& \mbox{for }n+4\leq i \leq 2n-2
\end{cases}\,.
\]
Therefore, we obtain
\begin{equation*}
\sum_{i=n+3}^{2n-1} \frac{\p_s^{i-1} P^{\rm nl}_n(0)}{(i-1)!}s^{i-1} = O \left( n s \frac{|Z_n| s^n}{n!} \cdot \frac{1}{n} \right),
\end{equation*}
which concludes the proof of equation \eqref{eq:Pnlfdr}.

Now, from equation \eqref{eq:Pnlfdr}, we have that 
\begin{equation}\label{eq:bilby:bilby}
P^{\rm nl}_n(s) =  \frac{Z_n}{n!}ns^{n+1}N_{W, 0} \left( \frac{D_{Z, 2}}{2}\left( 1+O \left( \frac1n \right) \right)
 + \frac{Z_n}{n!} \p_Z D_Z s^{n-2}\left( 1+O \left( \frac1n \right) \right)
 \right) \,, \end{equation}
up to $s^{n-2} = \frac{3|D_{Z, 2}|}{2 \p_Z D_Z} \frac{n!}{Z_n}$. In particular, taking $s_{7/5, \rm val}$ such that $s_{7/5, \rm val}^{n-2} = \frac{99}{100}\frac{|D_{Z, 2}|}{2 \p_Z D_Z} \frac{n!}{Z_n}$, the  approximation \eqref{eq:bilby:bilby} is valid and moreover 
\[\abs{ \frac{D_{Z, 2}}{2}}\geq \frac{100}{99}\abs{\frac{Z_n}{n!} \p_Z D_Z s^{n-2}}\,.\]
for $s \leq s_{7/5, \rm val}$. Therefore, for $n$ sufficiently large, the sign of $P^{\rm nl}_n(s )$ for every $0 \leq s \leq s_{7/5, \rm val}$ is given by the sign of $Z_n N_{W, 0} D_{Z, 2}$. Using Lemma \ref{lemma:aux_limits}, and the fact that $Z_n > 0$ from Corollary \ref{cor:Znsign}, we have that $P^{\rm nl}_n(s) > 0$ up to $s_{7/5, \rm val}$.

\end{proof}

\subsection{Intersection with the far-left barrier and proof of Proposition \ref{prop:mainleft7o5}}

Let us recall that for $\ga = 7/5$ and $n$ sufficiently large, we take the far-left barrier to be 
\begin{equation*}
b_{7/5}^{\rm fl} (t) = \left( W_0 + W_1 t + \frac12 W_2 t^2 - \left(W_0 + W_1 + \frac{W_2}{2}\right) t^3, Z_0 + Z_1 t + \frac12 Z_2 t^2 - \left(Z_0 + Z_1 + \frac{Z_2}{2}\right)t^3 \right)\,.
\end{equation*}
We formulate the barrier in implicit form looking at
\begin{align*}
B^{\rm fl}_{7/5} (W, Z) &= \mathrm{Res}_t \Big( -\left( W_0 + W_1 + \frac12 W_2 \right) t^3 + \frac12 W_2t^2 + W_1 t + W_0 - W, \\
& \qquad -\left( Z_0 + Z_1 + \frac12 Z_2 \right) t^3 + \frac12 Z_2t^2 + Z_1 t + Z_0 - Z \Big) \\
 &= \begin{vmatrix}
 - \left( W_0 + W_1 + \frac{W_2}{2} \right) & \frac{W_2}{2} & W_1 & W_0 - W &0  & 0 \\
 0 & - \left( W_0 + W_1 +  \frac{W_2}{2} \right) &  \frac{W_2}{2} & W_1 & W_0 - W & 0 \\
 0 & 0 &- \left( W_0 + W_1 +  \frac{W_2}{2} \right) &  \frac{W_2}{2} & W_1 & W_0 - W\\
 - \left( Z_0 + Z_1 + \frac{Z_2}{2} \right) & \frac{Z_2}{2} & Z_1 & Z_0 - Z &0  & 0 \\
 0 & - \left( Z_0 + Z_1 + \frac{Z_2}{2} \right) & \frac{Z_2}{2} & Z_1 & Z_0 - Z & 0 \\
 0 & 0 &- \left( Z_0 + Z_1 +\frac{Z_2}{2}\right) & \frac{Z_2}{2} & Z_1 & Z_0 - Z\\
 \end{vmatrix}, 
\end{align*}
which is positive if $(W, Z)$ is above the the barrier and negative if it is below. Thus, we study the function $f(s)= B^{\rm fl}(b^{\rm nl}_n (s))$. Let us define
\begin{equation} \label{eq:defa3an}
a_3 = \frac16 \p_s^3 \Big|_{s=0} f(s), \qquad \mbox{ and } \qquad a_n = \frac{Z_n}{n!} \p_Z B^{\rm fl}_{7/5} (W_0, Z_0).
\end{equation}

\begin{lemma} \label{lemma:biglebowski} Let $n$ be an odd number sufficiently large, $\gamma = \frac75$ and $r \in (r_n, r_{n+1})$. We have that $a_3 < 0$ and $a_n > 0$. 
\end{lemma}
\begin{proof}
The statement $a_n > 0$ follows from $\p_Z B^{\rm fl}_{7/5} (W_0, Z_0)>0$ (Lemma \ref{lemma:aux_limits}) and $Z_n > 0$ (Corollary \ref{cor:Znsign}). The proof of $a_3 < 0$ is computer-assisted and can be found in the supplementary material and we refer to Appendix \ref{sec:computer} for the implementation. 
\end{proof}

\begin{lemma} \label{lemma:lbj} For any given constant $C>0$, and $0\leq s^{n-3} \leq C\frac{n!}{Z_n}$,
\begin{equation}  \label{eq:lbj}
f(s) = a_3 s^3+ a_n s^n + O\left( \frac{|a_3| s^3 +| a_n | s^n }{n}\right)\,,
 \end{equation}
 where the implicit constant in the big-O notation is permitted to depend on $C$.
\end{lemma}
\begin{proof} This will follow in a very similar way as Lemma \ref{lemma:carter}, using Lemma \ref{lemma:grant} to bound intermediate terms, so we omit most of the details. First of all, let us note that $s \lesssim 1/n$ from Corollary \ref{cor:Znsign}.

Now, we note that $f(s)$ is a $3n$-th degree polynomial multiple of $s^3$, because both $B^{\rm fl}$ and $b^{\rm nl}$ agree up to second order. Therefore
\begin{equation} \label{eq:phenylalamine}
f(s) = a_3 s^3 + a_n s^n + \sum_{i=4, i\neq n}^{3n} \frac{1}{i!} (\p_s^i f)(0) + \left( \frac{1}{n!} |(\p_s^n f)(0)| - a_n \right) s^n\,.
\end{equation}

As $B^{\rm fl}_{7/5}(W, Z)$ is a third-degree polynomial, similarly to Lemma \ref{lemma:carter}, we have that the $i$-th coefficient is bounded as
\begin{align}\notag
\frac{1}{i!} |(\p_s^i f)(0)| s^i &\lesssim s^i \sum_{\substack{ a+b+c = i, \\ 0\leq a, b, c \leq n}} \frac{|Z_a||Z_b||Z_c|}{a!b!c!}  \\
&\lesssim \left(\frac{Z_n s^n}{n!} \right)^{i/n} \sum_{\substack{ a+b+c = i, \\ 0\leq a, b, c \leq n}} \left( \stirling{n}{a}^{-1} \stirling{n}{b}^{-1} \stirling{n}{c}^{-1} (k-n)^{i/n-\mathbbm{1}_{a=n} -\mathbbm{1}_{b=n} -\mathbbm{1}_{c=n} } \right)\,, \label{eq:reloj}
\end{align}
and
\begin{align}\notag
\left( \frac{1}{n!} |(\p_s^n f)(0)| - a_n \right) s^n &\lesssim s^n \left( \sum_{\substack{ a+b+c = n, \\ 0\leq a, b, c < n}} \frac{|Z_a||Z_b||Z_c|}{a!b!c!} + \frac{|W_n|}{n!} \right)
\lesssim s^n \left( \sum_{\substack{ a+b+c = n, \\ 0\leq a, b, c < n}} \frac{|Z_a||Z_b||Z_c|}{a!b!c!} + \frac{|Z_{n-1}|}{n!} \right) \\
&\lesssim \left(\frac{Z_n s^n}{n!} \right) \left( \sum_{\substack{ a+b+c = n, \\ 0\leq a, b, c < n}} \left( \stirling{n}{a}^{-1} \stirling{n}{b}^{-1} \stirling{n}{c}^{-1} (k-n) \right) + \frac{k-n}{n^3} \right)\,.
\label{eq:reloj2}
\end{align}

From \eqref{eq:reloj2}, we directly see that the last term of \eqref{eq:phenylalamine} satisfies the stated bound in \eqref{eq:lbj}. Thus, we just need to bound the sum in \eqref{eq:phenylalamine}.

From \eqref{eq:reloj}, we see that the term with $i = n + 1$ also satisfies the bound in \eqref{eq:lbj}, since $\abs{\frac{Z_n s^n}{n!}}=O( 1/n )$ and $Z_n s^n/n! \leq 1$. The terms with $n+2 \leq i \leq 2n-2$ added all together also satisfy the bound, as the parenthesis is $O( 1/n^2 )$ (so that the sum of all those terms is $O(\frac1n |a_n| s^n)$). For the terms with $2n-1 \leq i \leq 3n$, note that
\begin{equation*}
\left( \frac{Z_n s^n}{n!} \right)^{i/n} \lesssim \left( |a_n| s^n  \right)\left( \frac{Z_n s^n}{n!} \right)^{i/n-1} \lesssim  \left( |a_n| s^n  \right) s^{3 (n-2)/n},
\end{equation*}
so they are all $O(1/n^2)|a_n| s^n$ and their sum also satisfies the stated bound in \eqref{eq:lbj}. 

Lastly, we need to show that the contribution of terms with $4 \leq i \leq n-1$ in \eqref{eq:phenylalamine} is $O \left( \frac{|a_3|s^3 + |a_n|s^n}{n}\right)$. We have
\begin{align}
\sum_{i=4}^{n-1} \frac{|a_i| s^i}{i!} &\lesssim s^3 \sum_{i=4}^{n-1} s^{i-3} \left( \frac{Z_n (k-n) }{n!} \right)^{i/n}  \sum_{\substack{ a+b+c = i, \\ 0\leq a, b, c \leq n}} \stirling{n}{a}^{-1} \stirling{n}{b}^{-1} \stirling{n}{c}^{-1} \nonumber \\
& \lesssim s^3 \sum_{i=4}^{n-1} \left( \frac{s^{n-3} |Z_n| }{n!} (k-n)\right)^{(i-3)/(n-3)} \left( \frac{Z_n (n-k)}{n!}\right)^{(3n-3i)/(n(n-3))}  \sum_{\substack{ a+b+c = i, \\ 0\leq a, b, c \leq n}} \stirling{n}{a}^{-1} \stirling{n}{b}^{-1} \stirling{n}{c}^{-1} \nonumber \\
& \lesssim s^3  \sum_{i=4}^{n-1} n^{\frac{3n}{n^2/2}}  \sum_{\substack{ a+b+c = i, \\ 0\leq a, b, c \leq n}} \stirling{n}{a}^{-1} \stirling{n}{b}^{-1} \stirling{n}{c}^{-1}  \nonumber \\
&\lesssim s^3 \sum_{i=4}^{n-1} \sum_{\substack{ a+b+c = i, \\ 0\leq a, b, c \leq n}} \stirling{n}{a}^{-1} \stirling{n}{b}^{-1} \stirling{n}{c}^{-1}, \label{eq:botella}
\end{align}
where we have used Corollary \ref{cor:Znsign} for bounding $\frac{Z_n (n-k)}{n!}$ and our hypothesis for bounding $\frac{Z_ns^{n-3}}{n!} \lesssim 1$. Lastly, one can check that the sum in \eqref{eq:botella} is $O(1/n)$ and this concludes our statement because $|a_3| \gtrsim 1$.
\end{proof}

\begin{proof}[Proof of Proposition \ref{prop:mainleft7o5}]
Combining Lemma \ref{lemma:biglebowski} and Lemma \ref{lemma:lbj}, we have that $f(s) < 0$ for $s$ small enough and $f(s) > 0$ for 
\begin{equation*}
s = \left( \frac{2|a_3|}{|Z_n \p_Z B^{\rm fl}_{7/5} (W_0, Z_0)| / n!} \right)^{1/(n-3)}.
\end{equation*}
 In particular, there exists a value of $s_{7/5, int}$ with $s_{7/5, int}^{n-3} \lesssim \frac{n!}{Z_n}$ such that $B_{7/5}^{\rm fl} (b^{\rm nl}_{7/5}(s_{7/5, int})) = 0$, that is, the far-left and near-left barriers intersect.

As we know that the near-left barrier is valid up to 
\begin{equation*}
s_{7/5,\rm  val} = \left( \frac{99}{100}\frac{|D_{Z, 2}|}{2 \p_Z D_Z} \frac{n!}{Z_n} \right)^{1/(n-2)},
\end{equation*}
 for $n$ sufficiently large, it is clear that $s_{7/5, \rm val} > s_{7/5, \rm int}$.
 
  From Proposition \ref{prop:fdr}, we have that $P^{\rm nl}_{n}(s) > 0$ up to $s_{7/5, \rm val} > s_{7/5, \rm int}$, and that $b^{\rm nl}_n(s)$ intersects $b^{\rm fl}_{7/5}(t)$ at $s_{7/5, \rm int}$. Let us check that $D_Z (b^{\rm nl}_n (s)) > 0$ and $D_W (b^{\rm nl}_n(s)) > 0$ for $s \in (0, s_{7/5}^{\rm int})$. Now, notice that
 \begin{align*}
 D_W (b^{\rm nl}_n(s)) &= D_{W, 0} + \sum_{i=1}^{n} \frac{s^i}{i!} \nabla D_W (P_s) \cdot (W_i, Z_i) 
 = D_{W, 0} + \frac{\p_Z D_W Z_n}{n!}s^n + O \left( \frac{1 + |Z_n|s^n/n!}{n}\right) \\
 \end{align*}
for $s^{n-1} \lesssim \frac{n!}{Z_n}$, using the same reasoning as the one used in Proposition \ref{prop:fdr} or Lemma \ref{lemma:lbj}. Noting that $\p_Z D_W = \frac25$ (as $\ga = \frac75$) and $D_{W, 0} > 0$ (Lemma \ref{lemma:aux_DW0}), we get the result for $D_W (b^{\rm nl}_n(s))$.

In order to treat the case of $D_Z (b^{\rm nl}_n(s))$, note that $D_{Z, 1} > 0$, so $D_Z (b^{\rm nl}_n(s))$ is initially positive. On the other hand, if $b^{\rm nl}_n(s)$ crosses $D_Z = 0$ between $P_s$ and $\bar P_s$ before $s_{7/5}^{\rm int}$, by Lemma \ref{lemma:aux_DZ=0_repels}, we would have that the field $(N_W D_Z, N_Z D_W)$ points upwards at that point contradicting $P^{\rm nl}_n(s) > 0$, from Proposition \ref{prop:mainleft7o5}. If $b^{\rm nl}_n(s)$ crosses $D_Z = 0$ to the left of $\bar P_s$, at some other time $s < s_{7/5}^{\rm int}$, we fall under the second case considered in Proposition \ref{prop:mainleft7o5}.
 \end{proof}
 
\section{Proof of Theorems \ref{th:mainlarge} and \ref{th:mainr3}} \label{sec:mainproof}

We finally give the proof of Theorem \ref{th:mainlarge} and Theorem \ref{th:mainr3}. By Proposition \ref{prop:left_main}, we have that in either the case  $\ga \in (1, +\I )$ and $r \in (r_3, r_{4})$ or the case $\gamma=\frac75$ and  $r\in (r_{n}, r_{n+1})$ for $n$ odd and sufficiently large, the smooth solution $(W^{(r)}, Z^{(r)})$,  given in Proposition \ref{prop:smooth}, connects $P_s$ to $P_\infty$. It remains to apply a shooting argument, in conjunction with Proposition \ref{prop:right_main}, to show that the smooth solution connects $P_0$ to $P_s$.

Let us fix $n \in \N$ odd and $r\in (r_{n}, r_{n+1})$.  Due to Proposition \ref{prop:right_main}, we know that there exist $r_d, r_u \in (r_{n}, r_{n+1})$ such that $(W^{(r)}, Z^{(r)})$ lies in $\Omega_2^{(r)}$ and $\Omega_1^{(r)}$ respectively. Set $\delta>0$ sufficiently small such that for all $r\in[r_d,r_u]$, the Taylor series \eqref{eq:series} for  $(W^{(r)},Z^{(r)})$ converges for all $\xi\in[-\delta,0]$. By Lemma \ref{lemma:aux_34bounds} and Lemma \ref{lemma:aux_limits}, $W_1<0$. Thus, by continuity and compactness, we may take $\delta$ smaller if need be to guarantee that 
\[\frac{d}{d\xi}W^{(r)}(\xi)<0\,,\]
 for all $\xi\in[-\delta,0]$. In particular, the curve $(W^{(r)},Z^{(r)})$ for $\xi\in[-\delta,0]$ is a graph with respect to its $W$ coordinate. 
 
 Let $(W_r^o (\xi ), Z_r^o (\xi ))$ be the curve defined in \eqref{eq:newtaylor}. By Remark \ref{rem:horizontal},  $W_r^o(\xi)$ is increasing with $| \xi |$ for $\xi\in(-\infty,0]$. Thus, the curve $(W_r^o (\xi ), Z_r^o (\xi ))$ is also a graph  with respect to its $W$ coordinate. 
 
 Fix $(W_\ast,Z_\ast)=(W^{(r)}(-\delta),Z^{(r)}(-\delta))$, and define $Z^o_\ast$ to be such that $(W_\ast,Z^o_\ast)$ is a point on the curve $(W_r^o (\xi ), Z_r^o (\xi ))$. We then define $e:[r_d,r_u]\mapsto \mathbb R$ by
 \[e(r)=Z_\ast-Z^o_\ast\,.
 \]
 By definition, $e$ is a continuous function in $r$. Moreover, as a consequence of Proposition  \ref{prop:right_main} we have $e(r_d)<0$ and $e(r_u)>0$. Hence by continuity, there exists a $r^{(n)}\in (r_d,r_u)$ such that $e(r)=0$.  Therefore,  by the uniqueness in Proposition \ref{prop:existence}, for $r= r^{(n)}$, we have $(W^{(r)},Z^{(r)})=(W_r^o (\xi ), Z_r^o (\xi ))$. Thus, we conclude that the smooth curve corresponding to $r = r^{(n)}$ connects $P_0$ to $P_\infty$ through the point $P_s$, concluding the proofs of  Theorem \ref{th:mainlarge} and Theorem \ref{th:mainr3}.

\section{Linear Stability of the Profile} \label{sec:linear}
In this section, we will study the linearized operator of the Euler equations around the self-similar profiles we have found. The stability for the Euler equation will follow in general, while in the Navier-Stokes case we need to restrict the parameter $r$ to a regime where the self-similar profile dominates the dissipation. The strategy will be to cut-off the equation and study the linearized operator in a compact region $| \zeta | < 2$. Following the strategy of \cite{MeRaRoSz19b}, we show that the linearized operator is maximal and accretive in the appropriate spaces. Maximality corresponds to the existence of solutions of the ODE determined by this operator, while accretivity corresponds to the fact that the operator has damping. Both properties give us, via a functional analysis argument, that the compactified linearized operator generates a contraction semigroup modulo finitely many instabilities. That is the main result of this section. The nonlinear stability and the treatment of the equation outside our compact region $| \zeta | < 2$ will be delayed to Section \ref{sec:nonlinarstability}.

\begin{remark} \label{remark:torino} Let us note that from the proof that for $r = r^{(n)}$, $W^{(r)}(\zeta )$ is increasing in $\xi \in (-\infty, 0]$, so in particular, $W^{(r)}(\xi) > 0$ for all $\xi \in (-\infty, 0]$. Moreover, we have that $D_Z ( W^{(r)}, Z^{(r)}) < 0$ for $\xi < 0$, and for the $\gamma = \frac75$ case, this implies that $Z^{(r)}(\xi) < 0$ for $\xi < 0$. Thus, in the $\gamma = \frac75$, we have that $(W^{(r)} (\xi), Z^{(r)} (\xi))$ lies in the region $W>0, Z<0$ for $\xi \in (-\infty, 0]$. 
\end{remark}

\subsection{Linearization and localization}

Let  $(\mw W, \mw Z)$ represent an exact self-similar solution to the Euler equations solving \eqref{eq:Euler:SS:alt}. We now consider a solution $(\mc W,\mc Z)$ to the time dependent Navier-Stokes equation \eqref{eq:main} and the difference
\[(\mt W, \mt Z)=(\mc W-\mw W,\mc Z- \mw Z)\,.\]
Then, $(\mt W, \mt Z)$ satisfy the equations
 \begin{equation}\begin{aligned} \label{eq:tildeWZ:def}&
(\partial_s +r-1+\frac{\alpha}{\zeta}\widebar W+\frac{1+\alpha}2\partial_{\zeta}\widebar W)\widetilde W+(\zeta+\frac12(\widebar W+\widebar Z+\alpha(\widebar W-\widebar Z)))\partial_{\zeta}  \widetilde W  +
\left( \frac{1-\alpha}{2} \p_{\zeta}  \mw W -\frac{\alpha \mw Z}{\zeta}  \right)\mt Z
\\
&\qquad\qquad = 
\frac{r^{1+\frac{1}{\alpha}} 2^{1/\alpha-1}}{\alpha^{1/\alpha} \zeta^2((\mathcal W-\mathcal Z))^{\frac1\alpha}  }
e^{(2-r+\frac{1}{\alpha}(1-r))s}\left(\partial_{\zeta}(\zeta^2\partial_{\zeta}(\mathcal W+\mathcal Z))-2(\mathcal W+\mathcal Z)\right)\\
&\qquad\qquad\qquad
-\frac12(\widetilde W+\widetilde Z+\alpha(\widetilde W-\widetilde Z))\partial_{\zeta} \widetilde W
-\frac{\alpha}{2\zeta}(\widetilde W^2-\widetilde Z^2)\,,\\
&
(\partial_s +r-1+\frac{\alpha}{\zeta}\widebar Z+\frac{1+\alpha}2\partial_{\zeta}\widebar Z)\mt Z+(\zeta+\frac12(\mw W+\mw Z-\alpha(\mw W-\mw Z)))\partial_{\zeta}  \mt Z  + \left( \frac{1-\alpha}{2} \p_{\zeta}  \mw Z -\frac{\alpha \mw W}{\zeta}  \right)\mt W\\
&\qquad\qquad= 
\frac{r^{1+\frac{1}{\alpha}} 2^{1/\alpha-1}}{\alpha^{1/\alpha} \zeta^2((\mathcal W-\mathcal Z))^{\frac1\alpha}  }
e^{(2-r+\frac{1}{\alpha}(1-r))s}\left(\partial_{\zeta}(\zeta^2\partial_{\zeta}(\mathcal W+\mathcal Z))-2(\mathcal W+\mathcal Z)\right)\,\\
&\qquad\qquad\qquad
 -\frac12(\widetilde W+\widetilde Z-\alpha(\widetilde W-\widetilde Z))\partial_{\zeta}  \widetilde Z
 +\frac{\alpha}{2\zeta}(\widetilde W^2-\widetilde Z^2)\,.
 \end{aligned}
 \end{equation}
 Defining
 \begin{equation} \label{eq:mattmurdock}
 \begin{gathered}
 \mc D_{\mw W}=r-1+\frac{\alpha}{\zeta}\widebar W+\frac{1+\alpha}2\partial_{\zeta}\widebar W,
 \quad  \mc D_{\mw Z} =r-1+\frac{\alpha}{\zeta}\widebar Z+\frac{1+\alpha}2\partial_{\zeta}\widebar Z,\\
 \mc V_{\mw W}=\zeta+\frac12(\widebar W+\widebar Z+\alpha(\widebar W-\widebar Z)),
 \quad
  \mc V_{\mw Z}=\zeta+\frac12(\mw W+\mw Z-\alpha(\mw W-\mw Z)),\\
 \mc H_{\mw W}=\frac{1-\alpha}{2} \p_{\zeta}  \mw W -\frac{\alpha \mw Z}{\zeta},\quad
  \mc H_{\mw Z}=\frac{1-\alpha}{2} \p_{\zeta}  \mw Z -\frac{\alpha \mw W}{\zeta} , \\
  \mc F_{\mathrm nl, \mt W} =-\frac12(\mt W+\mt Z+\alpha(\mt W-\mt Z))\partial_{\zeta} \mt W
-\frac{\alpha}{2\zeta}(\mt W^2-\mt Z^2) ,\\
  \mc F_{\mathrm nl, \mt Z} = -\frac12(\widetilde W+\widetilde Z-\alpha(\widetilde W-\widetilde Z))\partial_{\zeta}  \widetilde Z
 +\frac{\alpha}{2\zeta}(\widetilde W^2-\widetilde Z^2), \\
  \mc F_{\rm dis} = 
  \frac{r^{1+\frac{1}{\alpha}} 2^{1/\alpha-1}}{\alpha^{1/\alpha} \zeta^2((\mathcal W-\mathcal Z))^{\frac1\alpha}  }
  e^{(2-r+\frac{1}{\alpha}(1-r))s}\left(\partial_{\zeta}(\zeta^2\partial_{\zeta}(\mathcal W+\mathcal Z))-2(\mathcal W+\mathcal Z)\right),
 \end{gathered}
 \end{equation}
 then \eqref{eq:tildeWZ:def} becomes
 \begin{equation}\label{eq:tildeWZ:def2}
  \begin{aligned}
 (\p_s+\mc D_{\mw W})\mt W+\mc V_{\mw W}\p_{\zeta}\mt W+\mc H_{\mw W}\mt Z&=\mc F_{\rm dis}+\mc F_{\mathrm{nl}, \mt W}=:\mc F_{ \mt W}\,,\\
  (\p_s+\mc D_{\mw Z})\mt Z+\mc V_{\mw Z}\p_{\zeta}\mt Z+\mc H_{\mw Z}\mt W&=\mc F_{\rm dis}+\mc F_{\mathrm{nl}, \mt Z}=:\mc F_{ \mt Z}\,.
 \end{aligned}
 \end{equation}
 where $\mc F_{\rm dis}$, $\mc F_{\mathrm{nl}, \mt W}$, $\mc F_{\mathrm{nl}, \mt Z}$ are respectively, the dissipative forcing, the nonlinear forcing term in the equation for $\mt W$ and the  the nonlinear forcing term in the equation for $\mt Z$. Since we look for  solutions $(\mc W,\mc Z)$ which are smooth when transformed into Cartesian coordinates, we may extend the solutions to $\zeta\in\mathbb R$ by imposing the restriction
 \begin{equation*}
 \mc Z(\zeta)=-\mc W(-\zeta)\quad\mbox{or equivalently}\quad  \mt Z(\zeta)=-\mt W(-\zeta)\,.
 \end{equation*}
Then, \eqref{eq:tildeWZ:def} becomes 
\begin{equation}\label{eq:tildeWZ:def3}
 (\p_s+\mc D_{\mw W}(\zeta))\mt W(\zeta)+\mc V_{\mw W}(\zeta)\p_{\zeta}\mt W(\zeta)-\mc H_{\mw W}(\zeta)\mt W(-\zeta)=\mc F_{\rm dis}(\zeta)+\mc F_{\mathrm{nl}, \mt W}(\zeta)\,.
\end{equation}
We also let $\mathcal U=\frac{\mc W+\mc Z}{2}$ and $\mathcal S=\frac{\mc W-\mc Z}{2}$ be the self-similar velocity and sound speed respectively, which satisfy the equations
\begin{equation} \label{eq:US:system} \begin{aligned}
(\partial_s +r-1)\mc U+(\zeta+\mc U)\p_{\zeta}  \mc U  +\alpha \mc S\p_{\zeta}S&=\mc F_{\rm dis}\,,\\
 (\partial_s +r-1)\mc S+(\zeta+\mc U)\p_{\zeta}  \mc S  +\frac{\alpha \mc S}{\zeta^2}\p_{\zeta}(\zeta^2 \mc U)&=0\,.
\end{aligned}
\end{equation}
We let $\mw U=\frac{\mw W+\mw Z}{2}$ and $\mw S=\frac{\mw W-\mw Z}{2}$ denote the self-similar velocity and sound speed of the exact self-similar Euler profile. Taking the difference
\[(\mt U, \mt S)=(\mc U-\mw S,\mc U- \mw S)\,,\]
leads to the equation
 \begin{align} \label{eq:US:tilde}
 \begin{split}
(\partial_s +r-1)\mt U+(\zeta+\mw U)\p_{\zeta}  \mt U  +\alpha \mw S\p_{\zeta}\mt S +\mt U\p_{\zeta}  \mw U  +\alpha \mt S\p_{\zeta}\mw S  &= \mc F_{\rm dis} + \frac{\mc F_{\mathrm{nl}, \mt W} + \mc F_{\mathrm{nl}, \mt Z}}{2}, \\
 (\partial_s +r-1)\mt S+(\zeta+\mw U)\p_{\zeta}  \mt S  +\frac{\alpha \mw S}{\zeta^2}\p_{\zeta}(\zeta^2 \mt U)+\mt U\p_{\zeta}  \mw S+\frac{\alpha \mt S}{\zeta^2}\p_{\zeta}(\zeta^2 \mw U)
&= \frac{\mc F_{\mathrm{nl}, \mt W} - \mc F_{\mathrm{nl}, \mt Z}}{2}.
\end{split} \end{align}

In order to simply our analysis, we will now introduce cut-offs and additional damping to  \eqref{eq:tildeWZ:def2} and \eqref{eq:US:tilde} which will have the effect of localizing our analysis around a neighborhood of the acoustic light-cone of the singularity. Let $\chi_1$ be a cut-off function which is $1$ for $| \zeta | \leq \frac{6}{5}$ and it is supported on $ |\zeta | \leq \frac{7}{5}$. Define also $\chi_2$ to be a cut-off function such that $\chi_2(\zeta)=1$ for $| \zeta | \leq \frac{8}{5}$ and it is supported on $| \zeta | \leq \frac{9}{5}$.  For a large constant $J>0$, define
\begin{equation*}
 \begin{gathered}
 \mc D_{t,\mw W}=J(1-\chi_1)+\chi_2 \mc D_{\mw W},
 \quad  \mc D_{t,\mw Z} =J(1-\chi_1)+\chi_2 \mc D_{\mw Z},\\
 \mc V_{t,\mw W}=\chi_2 \mc V_{\mw W},
 \quad
  \mc V_{t,\mw Z}=\chi_2 \mc V_{\mw Z},\quad
 \mc H_{t,\mw W}= \chi_2 \mc H_{\mw W},\quad
  \mc H_{t,\mw Z}=\chi_2  \mc H_{\mw Z},\quad
   \mc F_{t,\tilde W}=\chi_2\mc F_{\tilde W}\,\quad
     \mc F_{t,\tilde Z}=\chi_2\mc F_{\tilde Z}\,.
 \end{gathered}
 \end{equation*}
We then consider the \emph{truncated} equations 
\begin{align} \begin{split} \label{eq:CE}
 (\p_s+\mc D_{t,\mw W})\mt W_{t}+\mc V_{t,\mw W}\p_{\zeta}\mt W_{t}+\mc H_{t,\mw W}\mt Z_{t}&=\mc F_{t,\tilde W}\,,\\
  (\p_s+\mc D_{t,\mw Z})\mt Z_{t}+\mc V_{t,\mw Z}\p_{\zeta}\mt Z_{t}+\mc H_{t,\mw Z}\mt W_{t}&=\mc F_{t,\tilde W}\,.
\end{split} \end{align}
Note that the truncated equations are not themselves closed since $(\mt W,\mt Z)$ appear in the forcing terms. Adding the equations \eqref{eq:tildeWZ:def2}, closes \eqref{eq:CE}.  The truncated analogue of $(\mt U,\mt S)$, given by $(\mt U_t,\mt S_t)=\frac12(\mt W_t+\mt Z_t,\mt W_t-\mt Z_t)$ satisfy the equations
\begin{align} \begin{split} \label{eq:CE2}
(\partial_s +J(1-\chi_1)+\chi_2(r-1))\mt U_t+\chi_2\left((\zeta+\mw U)\p_{\zeta}  \mt U  +\alpha \mw S\p_{\zeta}\mt S\right) +\chi_2\left(\mt U\p_{\zeta}  \mw U  +\alpha \mt S\p_{\zeta}\mw S \right) &= \tfrac{\mc F_{t, \mt W} + \mc F_{t, \mt Z}}{2}, \\
 (\partial_s +J(1-\chi_1)+\chi_2(r-1))\mt S_t+\chi_2\left((\zeta+\mw U)\p_{\zeta}  \mt S_t  +\tfrac{\alpha \mw S}{\zeta^2}\p_{\zeta}(\zeta^2 \mt U)\right)+\chi_2\left(\mt U\p_{\zeta}  \mw S+\tfrac{\alpha \mt S}{\zeta^2}\p_{\zeta}(\zeta^2 \mw U)\right)
&= \tfrac{\mc F_{t, \mt W} - \mc F_{t, \mt Z}}{2}\,.
\end{split} \end{align}
To further distinguish the original equation, we will adopt the notation
\begin{gather*}
\mt U_e = \mt U, \quad \mt S_e  = \mt S,
 \quad
\mt W_e=\mt W,\quad \mt Z_e=\mt Z,\quad \mc D_{e,\mw W}= \mc D_{\mw W},\quad
\mc D_{e,\mw Z} =\mc D_{\mw Z},\quad
 \\
\mc V_{e,\mw W}=\mc V_{\mw W},
 \quad
  \mc V_{e,\mw Z}=  \mc V_{\mw Z},\quad 
 \mc H_{e,\mw W}=  \mc H_{\mw W},\quad
  \mc H_{e,\mw Z}= \mc H_{\mw Z},\quad
   \mc F_{e,\tilde W}=\mc F_{\tilde W}\,,\quad
     \mc F_{e,\tilde Z}=\mc F_{\tilde Z}\,.
\end{gather*}
where here the subscript `e' stands for \emph{extended}. Then by an abuse of notation, we free up the notation $\mt U$, $\mt S$, $\mt W$, $\mt Z$, $\mc D_{\mw W}$, $
\mc D_{\mw Z}$, $
 \mc V_{\mw W}$, $ \mc V_{\mw Z}$, $\mc H_{\mw W}$, $ \mc H_{\mw Z}$, $
   \mc F_{\tilde W}$, $
     \mc F_{\tilde Z}$ to refer to either the corresponding notation with the `t' or `e' subscript. For the remainder of the section we will restrict our attention to the truncated equation and so we will drop the `t' subscript. In particular, we will consider the  linear operator $\mathcal L=(\mc L_W,\mc L_Z)$ associated with \eqref{eq:CE} where
\begin{equation*}
-\mc L_W(W,Z)= \mc D_{\mw W}  W+\mc V_{\mw W}\p_{\zeta} W + \mc H_{\mw W} Z \quad\mbox{and}\quad
-\mc L_Z(W,Z)= \mc D_{\mw Z}  Z+\mc V_{\mw Z}\p_{\zeta} Z + \mc H_{\mw Z} W\,,
\end{equation*}
or in $(U,S)$ variables, $\mathcal L=(\mathcal L_U,\mc L_S)$ where
\begin{equation} \label{eq:lviv}
\begin{split}
-\mc L_U(U,S)&=(J(1-\chi_1)+\chi_2(r-1))U+\chi_2\left((\zeta+\mw U)\p_{\zeta}  U  +\alpha \mw S\p_{\zeta}S\right) +\chi_2\left(U\p_{\zeta}  \mw U  +\alpha  S\p_{\zeta}\mw S \right)\,, \\
-\mc L_S(U,S)&= (J(1-\chi_1)+\chi_2(r-1)) S+\chi_2\left((\zeta+\mw U)\p_{\zeta}   S  +\tfrac{\alpha \mw S}{\zeta^2}\p_{\zeta}(\zeta^2 U)\right)+\chi_2\left(U\p_{\zeta}  \mw S+\tfrac{\alpha  S}{\zeta^2}\p_{\zeta}(\zeta^2 \mw U)\right)\,.
\end{split}
\end{equation}
The parameter $J$ will be chosen sufficiently large in order that the operator $\mathcal L$ is well behaved in the region $[\frac{6}{5},2]$.

Using the definition of $\mc L = (\mc L_U, \mc L_S)$ in \eqref{eq:lviv}, we may rewrite \eqref{eq:CE} as
\begin{equation} \label{eq:ohare}
\p_s (\mt U, \mt S ) = \mc L (\mt U, \mt S ) + \mc F_{t, \rm nl} + \mc F_{t, \rm dis}\,,
\end{equation}
where
\begin{equation*}
\mc F_{t, \rm nl} = (\mc F_{t, \mt U}, \mc F_{t, \mt S}) = \left( \chi_2 \frac{\mc F_{\mathrm{nl}, \mt W} + \mc F_{\mathrm{nl}, \mt Z} }{2}, \chi_2 \frac{\mc F_{\mathrm{nl}, \mt W} - \mc F_{\mathrm{nl}, \mt Z} }{2}\right) \qquad \mbox{ and } \qquad \mc F_{t, \rm dis} = (\chi_2 \mc F_{\rm dis}, 0 )\,.
\end{equation*}
 
\subsection*{Dissipativity of the operator}

\begin{remark} \label{rem:espartero} For some $m$ (which will be chosen to be sufficiently large), we consider the space $X$ to be the subspace of tuples $(U, S)$, where $U$ is a radially symmetric vector field and $S$ a radially symmetric smooth function, and where $U, S \in H_0^{2m} (B(0, 2))$. We equip $X$ with the usual  $H^{2m}$ norm
\begin{equation*}
 \| (U, S) \|_{H^{2m}}^2 = \int_{B(0, 2)}\left(  \left| \Delta^m U \right|^2 + \left( \Delta^m S\right)^2 +| U |^2 + S^2 \right)\,.
\end{equation*}
Similarly, we let $\dot H^{2m}$ denote the corresponding homogeneous norm. 

 Moreover, we will sometimes consider the function $W (\zeta) = S (| \zeta |) + \text{sign}(\zeta ) U(| \zeta |)$ defined for $\zeta \in [-2, 2]$. We will say that $W \in X$ if the corresponding pair $(U, S)$ (which can be uniquely determined from $W$), is in $X$. By abuse of notation, we define the $H^{2m}$ norm on $W$ as $\| W \|_{H^{2m}} = \|(U, S) \|_{H^{2m}}$ in that case.
 
 We then define the domain of our linear operator $\mathcal L=(\mathcal L_U,\mathcal L_S)$ to be the space
 \[\mathcal D(\mathcal L)=\{(U,S)\in X\vert (\mathcal L_U U,\mathcal L_S S)\in X \}\,.\]
\end{remark} 

\begin{lemma} \label{lemma:spacefreq} For any $N$ there exists a finite codimension subspace of $X$ where for any $0 \leq i \leq 2m$, the following holds
\begin{equation*}
\|U \|_{H^i} \leq \frac{1}{N^{2m-i}} \| \Delta^m U \|_{L^2}, \qquad \| S \|_{H^i} \leq \frac{1}{N^{2m-i}}  \| S \|_{L^2}.
\end{equation*}
We let $Y_N$ denote that subspace. 
\end{lemma}
\begin{proof} First of all, note that by interpolation, it suffices to show the claim for $i=0$. Now, for a pair $(U, S)$, we consider the torus $\mathbb{T}^3 = [-\pi, \pi]^3$ and extend $U_i(y), S(y)$ to be zero for $y \in \mathbb{T}^3 \setminus B(0, 2)$. Consider the Fourier series of $U_i, S$ as functions over $\mathbb{T}^3$.

Now, we let the space $Y_N$ to be the finite codimension subspace of $X$ defined by the finite set of linear equations
\begin{equation*}
\hat{U}_i (k) = 0, \qquad \hat{S}(k) = 0, \qquad \forall k \in \mathbb{Z}^3, \; |k| < N\,,
\end{equation*}
for $i=1,2,3$.
Then, for such $U_i$, we have that
\begin{equation*}
\| U_i \|_{H^{2m}}^2 \geq \sum_{k \in \mathbb{Z}^3} |k|^{2m}\abs{ \hat{U_i}(k)}^2 \geq N^{2m} \sum_{k \in \mathbb{Z}^3}\abs{ \hat{U_i}(k)}^2 = N^{2m} \| U_i \|_{L^2}.
\end{equation*}
The same reasoning applies to $S$ and we conclude our result.

\end{proof}

\begin{lemma} \label{lemma:accretivity} There exists sufficiently large $N$ depending on $J$, which is chosen sufficiently large depending on $m$ such that if
 $P_N$ is the orthogonal projection $P_N: X \rightarrow Y_N$, then
 $P_N \circ \mc L $ is dissipative on $Y_N$, satisfying the bound
\[\Re \langle P_N \circ \mc L (U,S),(U,S)\rangle_{H^{2m}} \leq - \norm{(U,S)}_{H^{2m}}^2\,, \]
for all $(U,S)\in Y_N$.
\end{lemma}
\begin{proof}
We will use the notation $O_m$ to indicate cases where the constant may depend on $m$, while we use $O$ as usual our usual big-$O$ notation (the constant is universal).

First note that since $P_N \circ \mc L $ is a real operator, mapping real valued function to real valued functions, it suffices to prove the bound
\[ \langle P_N \circ \mc L (U,S),(U,S)\rangle_{H^{2m}} \leq - \norm{(U,S)}_{H^{2m}} \,,\]
for $(U,S)$ real valued.

Let us study the inner product $\langle \mc L (U, S), (U, S) \rangle_{\dot H^{2m}}$, and we will treat the projection at the end. Let us recall that 
\begin{align*}
-\mc L_U(U,S)&=
J(1-\chi_1)U
+\chi_2\left( (r-1))U  + U\p_{\zeta}  \mw U  +\alpha  S\p_{\zeta}\mw S\right) 
+\left(\chi_2(\zeta+\mw U)\p_{\zeta}  U  +\alpha \chi_2 \mw S\p_{\zeta}S\right)  \\
&= J(1-\chi_1)U  + \chi_2 \mc K_U + \mc V_U\,, \\ 
-\mc L_S(U,S)&= J(1-\chi_1) S
+\chi_2\left((r-1)S + U\p_{\zeta}  \mw S+\alpha  S \div ( \mw U )\right)
+\left(\chi_2 (\zeta+\mw U)\p_{\zeta}   S  +\alpha \chi_2 \mw S \div (U)\right) \\
&= J(1-\chi_1)S  + \chi_2 \mc K_S + \mc V_S\,.
\end{align*}
Now, we proceed to study the terms in 
\begin{align} \begin{split} \label{eq:sagasta}
-\langle \mc L(U, S), (U, S) \rangle_{\dot H^{2m}} &= 
\int_{B(0, 2)}\Delta^m  \left( J(1-\chi_1) U + \chi_2 \mc K_U + \mc V_U \right)\cdot  \Delta^m U \\
& \qquad + \int_{B(0, 2)}\Delta^m  \left( J(1-\chi_1) S + \chi_2 \mc K_S + \mc V_S \right)  \Delta^m S\,.
\end{split} \end{align}
 First of all, let us note that
\begin{align*} 
J \int_{B(0, 2)} \Delta^m \left( (1-\chi_1) U \right) \cdot  \Delta^m U &= J \int_{B(0, 2)} (1-\chi_1) | \Delta^m U|^2 + O_m \left( J \| U \|_{H^{2m}} \left( \| U \|_{H^{2m-1}} + \| S \|_{H^{2m-1}} \right) \right)\,,\nonumber \\
J \int_{B(0, 2)} \Delta^m \left( (1-\chi_1) S \right) \cdot  \Delta^m S &= J \int_{B(0, 2)} (1-\chi_1) | \Delta^m S|^2 + O_m \left( J \| S \|_{H^{2m}} \left( \| U \|_{H^{2m-1}} + \| S \|_{H^{2m-1}} \right) \right)\,.
\end{align*} 
Choosing $N$ sufficiently large, dependent on $J$, which in turn is chosen sufficiently large, dependent on $m$, we can ensure as a consequence of Lemma \ref{lemma:spacefreq} that the error is $O( \| (U, S) \|_{ H^{2m}}^2 )$. Therefore, we get that
\begin{align} \begin{split} \label{eq:aristoteles}
J \int_{B(0, 2)} \Delta^m \left( (1-\chi_1) U \right) \cdot  \Delta^m U &= J \int_{B(0, 2)} (1-\chi_1) | \Delta^m U|^2 + O \left( \| (U, S) \|_{\dot  H^{2m}}^2 \right) \,,\\
J \int_{B(0, 2)} \Delta^m \left( (1-\chi_1) S \right) \cdot  \Delta^m S &= J \int_{B(0, 2)} (1-\chi_1) | \Delta^m S|^2 + O \left( \| (U, S) \|_{\dot  H^{2m}}^2 \right)\,.
\end{split} \end{align} 
 
 For the terms $\Delta^m (\chi_2 \mc K)$ in \eqref{eq:sagasta}, we note the following: The terms where all the derivatives fall on $U$ or $S$ are bounded in $L^2$ as $O( \|(U, S) \|_{\dot H^{2m}})$, where the implicit constant is independent of $m$. The rest of the terms have at most $2m-1$ derivatives on $U$ or $S$, so are simply bounded in $L^2$ as $O_m (\| U \|_{H^{2m-1}} + \| S \|_{H^{2m-1}} )$. Putting this altogether yields
\begin{align}
\int_{B(0, 2)} \left( \Delta^m \mc (\chi_2 K_U) \cdot \Delta^m U + \Delta^m (\chi_2 \mc K_S) \Delta^m S \right) &=  O_m \left( \left( \| U \|_{H^{2m-1}} + \| S \|_{H^{2m-1}} \right) \| (U, S) \|_{\dot H^{2m}} \right) \notag\\&\qquad+O \left( \| (U, S) \|_{\dot H^{2m}}^2 \right) \nonumber  \\
& = O \left( \| (U, S) \|_{\dot H^{2m}}^2 \right)\,.\label{eq:platon}
\end{align}
The last equality is due to the fact that we take $N$ large enough in terms of $m$ and apply  Lemma \ref{lemma:spacefreq}.

Lastly, let us treat the terms coming from $\mc V_U$ and $\mc V_S$ in \eqref{eq:sagasta}. From Lemma \ref{lemma:leibnitzlaplace}:
\begin{align}  \label{eq:epicuro1}
\int_{B(0, 2)} \Delta^m \mc V_U \cdot \Delta^m U &= 
\int_{B(0, 2)} \chi_2 (y + \mw U) \cdot \nabla \Delta^m U \cdot \Delta^m U  + 2m \int_{B(0, 2)} \nabla \left( \chi_2 (y + \mw U) \right) ( \Delta^m U )^2 \nonumber \\
&\qquad + \alpha \int_{B(0, 2)} \chi_2 \mw S \nabla \Delta^{m} S \cdot \Delta^m U + 2m \alpha \int_{B(0, 2)}  \Delta^m S \nabla \left( \chi_2 \mw S \right) \Delta^m U \nonumber \\
&\qquad + O_m \left(  \| U \|_{H^{2m}} (\| U \|_{H^{2m-1}} + \| S \|_{H^{2m-1}} ) \right)\,, \nonumber \\
\int_{B(0, 2)} \Delta^m \mc V_S \cdot \Delta^m S &= 
\int_{B(0, 2)} \chi_2 (y + \mw U) \cdot \nabla \Delta^m S \Delta^m S + 2m \int_{B(0, 2)} \p_\zeta \left( \chi_2 (\zeta + \mw U) \right) (\Delta^m S )^2 \nonumber \\
&\qquad + \alpha \int_{B(0, 2)} \chi_2 \mw S \div (\Delta^m U ) \Delta^m S + 2m \alpha \int_{B(0, 2)} \nabla (\chi_2 \mw S) \Delta^m U \Delta^m S \nonumber \\
&\qquad + O_m \left( \| S \|_{H^{2m}} (\| U \|_{H^{2m-1}} + \| S \|_{H^{2m-1}} ) \right)\,.
 \end{align}
If we take $N$ sufficiently large in terms of $m$ and use Lemma \ref{lemma:spacefreq}, these errors are $O( \| (U, S) \|_{\dot  H^{2m}}^2 )$. Therefore, we see from \eqref{eq:epicuro1} that
\begin{align} \begin{split} \label{eq:epicuro2}
\int_{B(0, 2)} \Delta^m \mc V_U \cdot \Delta^m U &= 2m \int_{B(0, 2)}  \nabla \left( \chi_2 (y + \mw U) \right) | \Delta^m U |^2 + \mc I_1 + \mc I_2 + 2m\mc I_3 + O \left( \| (U, S) \|_{\dot  H^{2m}}^2 \right), \\
\int_{B(0, 2)} \Delta^m \mc V_S \Delta^m S &= 2m \int_{B(0, 2)}\nabla \left( \chi_2 (y + \mw U) \right)( \Delta^m S )^2 + \mc I_4 + \mc I_5 + 2m\mc I_3 + O \left( \| (U, S) \|_{\dot  H^{2m}}^2 \right).
\end{split} \end{align}
where we have defined
\begin{align}  \label{eq:Ii}
&\mc I_{1} = \int_{B(0, 2)} \chi_2 (y + \mw U ) \cdot \nabla \Delta^m U \cdot \Delta^m U,  \qquad
&&\mc I_{2} = \alpha \int_{B(0, 2)}  \chi_2 \mw S \nabla \Delta^m S \cdot \Delta^m U , \nonumber \\
&\mc I_{3} = \alpha \int_{B(0, 2)} \Delta^m S \nabla (\chi_2 \mw S) \cdot  \Delta^m U, \qquad
&&\mc I_4 = \int_{B(0, 2)} \chi_2 ( y + \mw U ) \cdot \nabla \Delta^m S \Delta^m S, \nonumber \\
&\mc I_5 = \alpha \int_{B(0, 2)} \chi_2 \mw S \Delta^m S \div (\Delta^m U).
 \end{align}
 
 Integrating by parts, we see that
 \begin{equation} \label{eq:hipatia1}
 \mc I_1 + \mc I_4 = -\frac{1}{2} \int_{B(0, 2)} \div (\chi_2 (y + \mw U )) \left( | \Delta^m U |^2 + ( \Delta^m S)^2 \right) = O\left( \| (U, S) \|_{\dot  H^{2m}}^2 \right).
 \end{equation}
 Integration by parts also shows that
 \begin{equation} \label{eq:hipatia2}
 \mc I_2 + \mc I_5 = -\alpha \int_{B(0, 2)} \nabla (\chi_2 \mw S) \cdot \Delta^m U \Delta^m S = O \left( \| (U, S) \|_{\dot  H^{2m}}^2 \right).
 \end{equation}
 Finally, we also have
 \begin{equation} \label{eq:hipatia3}
 4m \mc I_3 \geq -2m \alpha \int_{B(0, 2)} \left( | \Delta^m U |^2 + (\Delta^m S)^2 \right) | \p_\zeta (\chi_2 \mw S) |\,.
 \end{equation}
 Plugging \eqref{eq:hipatia1}--\eqref{eq:hipatia3} into \eqref{eq:epicuro2}, we obtain
 \begin{align} \label{eq:thales}
\int_{B(0, 2)} \left( \Delta^m \mc V_U \cdot \Delta^m U + \Delta^m \mc V_S \Delta^m S \right) &\geq 2m \int_{B(0, 2)} \left( \nabla (\chi_2 (y + \mw U)) - \alpha |\p_\zeta (\chi_2 \mw S) | \right) \left( | \Delta^m U |^2 + ( \Delta^m S )^2 \right)  \nonumber \\
& \qquad + O \left( \| (U, S) \|_{\dot H^{2m}}^2 \right).
 \end{align}
 
 Plugging \eqref{eq:aristoteles}, \eqref{eq:platon} and \eqref{eq:thales} in \eqref{eq:sagasta}, we obtain that
 \begin{align} \label{eq:socrates}
 -\langle \mc L(U, S), (U, S) \rangle_{\dot  H^{2m}} &\geq \int_{B(0, 2)} \left( J(1-\chi_1) + 2m \left( \nabla (\chi_2 (y + \mw U)) - \alpha |\p_\zeta (\chi_2 \mw S) | \right) \right) \left( | \Delta^m U |^2 + ( \Delta^m S )^2 \right) \nonumber \\
 &\qquad - C \| (U, S) \|_{\dot H^{2m}}^,,
 \end{align}
 for some absolute constant $C$. 
 Now, we claim that we can choose $J \gg m$ such that
 \begin{equation} \label{eq:demostenes}
 \left( J(1-\chi_1) + 2m \left( \nabla (\chi_2 (y + \mw U)) - \alpha |\p_\zeta (\chi_2 \mw S) | \right) \right) \geq C + 2\,.
 \end{equation}

 In order to show \eqref{eq:demostenes}, let us divide $B(0, 2)$ in two regions. We define $R_1$ as the region where $\chi_2 = 1$ and define $R_2$ as the region of $B(0, 2)$ where $\chi_2 < 1$. In particular, we have that $\chi_1 = 0$ on $R_2$.
 
 \textbf{Region $R_1$.}
In this region, as $\chi_2 = 1$ and $\chi_1 \leq 1$, it suffices to show
\begin{equation*}
2m \left( \nabla ((y + \mw U)) - \alpha |\p_\zeta (\mw S) | \right)  \geq C + 2\,.
\end{equation*}
As we can choose $m$ sufficiently large, we just need to show that
\begin{equation*}
 1 + \p_\zeta \mw U - \alpha | \p_\zeta \mw S | \geq \eps\,,
\end{equation*}
for some $\eps > 0$ and every $\zeta \in [0, 7/5]$. This is implied by Lemma \ref{lemma:Vdamping} taking $\eps = \eta_{\rm damp}$.
 
\textbf{Region $R_2$.}
In this region, as $\chi_1 = 0$, it suffices to satisfy
\begin{equation*}
  J >  - 2m \left( \nabla (\chi_2 (y + \mw U)) - \alpha |\p_\zeta (\chi_2 \mw S) |  \right) + C + 2\,.
\end{equation*}
It is trivial that we can satisfy this inequality because the right-hand-side is a bounded function for $\zeta \in [0, 2]$ and we can take $J$ sufficiently large, depending on $m$.

Therefore, we conclude that \eqref{eq:demostenes} holds, and inserting this in \eqref{eq:socrates}, we conclude
\begin{equation}\label{eq:diss1}
\langle \mc L (U, S), (U, S) \rangle_{\dot  H^{2m}} \leq -2 \| (U, S) \|_{\dot H^{2m}}^2\,.
\end{equation}
Assuming $N$ is sufficiently large, applying Lemma \ref{lemma:spacefreq}, the bound \eqref{eq:diss1} yields
\[\langle \mc L (U, S), (U, S) \rangle_{  H^{2m}} \leq - \| (U, S) \|_{H^{2m}}^2\,.\]
Note that  $(1-P_N) \circ \mc L (U, S)$ has image in a finite dimension space which is the orthogonal complement to the space $Y_N$ where $(U, S)$ lie. Therefore, $\langle (1-P_N)\circ \mc L (U, S), (U, S) \rangle_{ H^{2m}} = 0$ and we conclude
\begin{equation*}
\langle P_N \circ \mc L(U, S), (U, S) \rangle_{ H^{2m}} \leq -\| (U, S) \|_{H^{2m}}^2\,,
\end{equation*}
that is, that our operator $P_N \circ \mc L$ is dissipative on $Y_N$.
\end{proof}

\subsection*{Maximality}

Before we prove our main maximality result, let us prove the following auxiliary lemma that we will help us deal with the point $\zeta=\frac95$ where both $\mc V_{\mw W}$ and $\mc V_{\mw Z}$ vanish. 
\begin{lemma}\label{lem:Korimako}
	For $\lambda>0$ and $a<2$, consider the following ODE 
		\begin{equation}\label{eq:silly:ode}
			(\lambda+\mathcal D)u+\mathcal V u'=f, \quad u(a)=u_0\,,
		\end{equation}
		on the region $[a,2]$, for
		 smooth $\mc D$, $\mc V$ and $f$. 
		For some $a<b<2$, let us further assume that $\mathcal V(x)=0$ for $x\in[b,2]$ and $\mc V(x)>0$ for $x\in[a,b)$. Then, assuming $\lambda>0$ is sufficiently large, \eqref{eq:silly:ode} has a unique smooth solution. 
	Moreover, $(u,f)$ may be taken to be vector valued, in which case $u_0$ is a vector, and $\mc D$ is taken to be matrix valued and $\mc V$ remains scalar valued.
	\end{lemma}
\begin{proof}
	For concreteness, let us assume $a=0$ and $b=1$. We also assume $u$ to be scalar valued since the vector valued case will follow from an identical proof.
	
	By standard ODE theory, there exists a unique smooth $u$ to \eqref{eq:silly:ode} on the region $[0,1)$. Moreover assuming $ \lambda>0$ is sufficiently large, on the region $[1,2]$, \eqref{eq:silly:ode} has the unique smooth solution
	$u=\frac{f}{\lambda+\mathcal D}$. Thus, it suffices to verify that the resulting solution $u$ is smooth at $x=1$. In particular, we need to show
	\begin{equation*}
		\lim_{x\rightarrow 1^-}u^{(n)}\quad\mbox{exists for all }n\,.
	\end{equation*}
		
	By the Leibniz rule
\[
\mc V u^{(n+1)}=-(\lambda+\mc D+n{\mc V}') u^{(n)}\underbrace{-\sum_{k=2}^{n} \binom{n}{k} u^{(n-k+1)}{\mc V}^{(k)}-\sum_{k=1}^{n} \binom{n}{k} u^{(n-k)}{\mc D}^{(k)}+f^{(n)}}_{\mc F_n}\,.
\]
By Gr\"onwall's inequality, for $0<x_0<x<1$, $x_0$ sufficiently close to $1$ we have
\begin{align*}
	\abs{u^{(n)}(x)}&\leq e^{-\int_{x_0}^x  \frac{\lambda}{\mc V}}\left(\abs{u^{(n)}({x_0})}+\int_{x_0}^x \frac{e^{\int_{x_0}^{x'}  \frac{\lambda}{\mc V}}}{\mc V(x')} \left(\abs{\mathcal F_n(x')}+\abs{u^{(n)}({x'})}\left(\abs{\mc D(x')}+n\abs{\mc V'(x')}\right)\right)\,dx'\ \right)\\
&\leq e^{-\int_{x_0}^x \frac{\lambda}{\mc V}}\left(\abs{u^{(n)}({x_0})}+\int_{x_0}^x \frac{1}{\lambda}\left(\frac{d}{dx'}e^{\int_{x_0}^{x'}  \frac{\lambda}{\mc V}} \right)\left(\abs{\mathcal F_n(x')}+C\abs{u^{(n)}({x'})}\right)\,dx' \right)\\
&\leq e^{-\int_{x_0}^x \frac{\lambda}{\mc V}}\abs{u^{(n)}({x_0})}+\frac{1}{\lambda}\left(\norm{\mathcal F_n}_{L^{\infty}[x_0,x]}+C\norm{u^{(n)}}_{L^{\infty}[x_0,x]}\right)\,,
\end{align*}
for some constant $C$ independent of $n$, where we used that $n\norm{\mc V'}_{L^\infty[x_0,1]}$ can be made arbitrarily small by assuming $x_0$ to be sufficiently close to $1$. Assuming $\frac{C}{ \lambda}\le\frac 12$, we obtain
\[\norm{u^{(n)}}_{L^{\infty}[x_0,x]}\les \abs{u^{(n)}({x_0})}+\norm{\mathcal F_n}_{L^{\infty}[x_0,x]}\,.\]
By induction on $n$ (and an appropriate choice of $x_0$ for each $n$), we conclude that $u$ is smooth at $x=1$.
\end{proof}

\begin{lemma} \label{lemma:maximality} Consider $J, m, N$ chosen as in Lemma \ref{lemma:accretivity}. For sufficiently large $\lambda>0$ we have that for every $F=(F_U,F_S)\in X$,  then there exists $(U,S)\in \mathcal D(\mathcal L)$ such that
  \begin{align} \label{eq:max:prob}
(  -\mathcal L_U+\lambda )U=F_U\quad\mbox{and}\quad
(  -\mathcal L_S+\lambda )S=F_S\,.
   \end{align}
\end{lemma}
\begin{proof} 
  Let us rewrite the equation \eqref{eq:max:prob} in terms of
   \[W(\zeta)=\begin{cases}U(\zeta)+S(\zeta)&\mbox{for }\zeta\geq 0\\
   -U(-\zeta)+S(-\zeta)&\mbox{for }\zeta\leq 0\end{cases}\,,\]
   which leads to the equation
    \begin{align}\label{eq:max:prob2}
(\lambda + \mc D_{\mw W}) W+\mc V_{\mw W}\p_{\zeta}  W + \mc H_{\mw W} Z    = F_W\,.
 \end{align}
 where
  \[
 Z ( \zeta) = -W (-\zeta)\quad\mbox{and}\quad F_W(\zeta)=\begin{cases}F_U(\zeta)+F_S(\zeta)&\mbox{for }\zeta\geq 0\\
   -F_U(-\zeta)+F_S(-\zeta)&\mbox{for }\zeta\leq 0\end{cases} \,.\] 
 We consider first the problem  \eqref{eq:max:prob2} for the case $F_W=\chi\mc F$, where $\mathcal F$ is analytic and $\chi$ is a cut-off function that is $1$ on $[-\frac32,\frac32]$ and has compact support in $[-2,2]$. Clearly such $F_W$ are dense in $X$.
 Let us rewrite \eqref{eq:max:prob2} as
 \begin{equation}\label{eq:fairywren1}
 \mc V_{\mw W} \partial_{\zeta} W= \underbrace{ F_{ W}-(\lambda + \mc D_{\mw W}) W-\mc H_{\mw W}  Z   }_{\mc G_{ W}}\,.
 \end{equation}
 For the analysis around $P_s$, it will also be useful to write a separate equation for $ Z$:
 \[
(\lambda + \mc D_{\mw Z}) Z+\mc V_{\mw Z}\p_{\zeta}   Z + \mc H_{\mw Z}  W  =  F_{ Z}\,.\]
 and
 \begin{equation}\label{eq:fairywren2}
 \mc V_{\mw Z}  \partial_{\zeta}Z= \underbrace{ F_{ Z}-(\lambda + \mc D_{\mw Z}) Z-\mc H_{\mw Z}  W}_{\mc G_{ Z}(W,Z)}\,.
 \end{equation}

Consider the formal power series expansions of $(W, Z)$ at $\zeta=1$, i.e.\ $  W=\sum_{i \geq 0} w_i (\zeta-1)^i$ and $ Z = \sum_{i \geq 0} z_i(\zeta-1)^i$. Writing in addition $\mc V_{\mw Z}=\sum_{i \geq 0} v_{\mw Z,i} (\zeta-1)^i$ and $\mc G_{ Z}(W,Z)=\sum_{i \geq 0} g_{ Z,i} (\zeta-1)^i$, then substituting these expansions into \eqref{eq:fairywren1} and  \eqref{eq:fairywren2} yields
\begin{align}\label{eq:DingoW}
(n+1)v_{\mw W,0} w_{n+1}&=g_{ W,n}-\sum_{i=0}^{n-1} (i+1)v_{\mw W,n-i}w_{i+1}\,,\\
(n+1)v_{\mw Z,1}z_{n+1}&=g_{ Z,n+1}-\sum_{i=0}^{n-1}(i+1)v_{\mw Z,n+1-i}z_{i+1}\,.\label{eq:DingoZ}
\end{align}
Let us rewrite $g_{ Z,n+1}$ as
\begin{align*}
g_{ Z,n+1}=\check g_{ Z,n+1}-(\lambda+\mathcal D_{\mw Z}(1))z_{n+1}\,.
\end{align*}
Then, \eqref{eq:DingoZ} can be rewritten as
\begin{equation}\label{eq:DingoZ2}
(\lambda+\mathcal D_{\mw Z}(1)+(n+1)v_{\mw Z,1})z_{n+1}=\check g_{ Z,n+1}-\sum_{i=0}^{n-1}(i+1)v_{\mw Z,n+1-i}z_{i+1}\,.
\end{equation}
Thus, assuming $\lambda$ is sufficiently large and using $v_{\mw Z,1}=D_{Z,1}> 0$ (Lemma \eqref{lemma:aux_DZ1}), one may solve the recurrence relations \eqref{eq:DingoW} and \eqref{eq:DingoZ} uniquely by setting $w_0 = A$ and $z_0$ is determined from \eqref{eq:fairywren2}. Furthermore, from the analyticity of $(\mw W, \mw Z)$ at $\zeta=1$ (which is a consequence of Proposition Proposition \ref{prop:smooth}) and $F$, we obtain that the series converges absolutely to obtain a solution $(W,Z)$ in a neighborhood of $\zeta=1$. The solution can be extended to a $C^\infty$ solution on $\zeta\in(0,\frac95)$ by standard ODE arguments, using that the only zero of $\mathcal V_{\bar Z} (\zeta)$ or $\mathcal V_{\bar W} (\zeta)$ with $\zeta \in (0, \frac95 )$ is $\mathcal V_{\bar Z} (1) = 0$. This just follows from the observation that $\mathcal V_{\bar Z} = \zeta D_Z^E$ and $\mathcal V_{\bar W} = \zeta D_W^E$ (where we use the superindex $E$ to indicate we refer to the Euler scaling and self-similar profiles from Sections \ref{sec:expansions}--\ref{sec:mainproof}). Applying Lemma \ref{lem:Korimako}, we can further extend the solution to a  $C^\infty$ solution on $\zeta\in(0,2]$. Note in order to apply Lemma \ref{lem:Korimako}, we let $(a,b)=(\frac65,\frac95)$ and $u=(W,\frac{\mc V_{\mw Z}}{\mc V_{\mw W}}V)$.

We will apply a shooting argument in order to choose $A$ such that $W$ is smooth at $\zeta=0$. First we show that there exists $A^\ast$ such that if $(W_{A^\ast},Z_{A^\ast})$ and $(W_{-A^\ast},Z_{-A^\ast})$ correspond to the smooth solutions to \eqref{eq:fairywren1} and \eqref{eq:fairywren2} for $\zeta\in(0,1]$ satisfying $W(1)=A^\ast$ and $W(1)=-A^\ast$ respectively, then we have
\begin{equation}\label{eq:Thylacine}
W_{A^\ast}+Z_{A^\ast}\geq 1,\quad W_{-A^\ast}+Z_{-A^\ast}\leq -1\,,
\end{equation}
for all $\zeta\in (0,1]$.

Note that by
 \eqref{eq:fairywren2} we have that
\begin{equation}\label{eq:Z(1)}
 Z_{A^\ast}(1) =  \frac{1}{\lambda + \mc D_{\mw Z}(1)}\left(-\mc H_{\mw Z}(1)  A^\ast +  F_Z(1)\right)\,.
\end{equation}
Note for $\gamma=\frac75$ we have
\[\lim_{r\rightarrow r^\ast} \mc H_{\mw Z}(1)   =\frac{1}{30} \left(5-3 \sqrt{5}\right)<0\,.\]
Hence for $k$ sufficiently large, by continuity, $H_{\mw Z}(1)<0$. Thus  choosing  $\lambda\gg \mc D_{\mw Z}(1)$ and $A^\ast\gg\frac{\mc F_Z(1)}{-\mc H_{\mw Z}(1) }$, from \eqref{eq:Z(1)}, we obtain 
$ Z_{A^\ast}(1)\geq \frac{A^\ast}{C\lambda}$ for $C$ some depending on $ \mc D_{\mw Z}(1)$. 

We claim that $W_{A^\ast}\geq \frac{A^\ast}{2}$ and $Z_{A^\ast}\geq \frac{A^\ast}{C'\lambda}$ for all $\zeta\in(0,1]$ and some large constant $C' > C$. Suppose the statement is false, then there must  exist a largest $\zeta'\in(0,1)$ such that either 
\begin{enumerate}
\item $W_{A^\ast}(\zeta')=\frac{A^\ast}{2}$ and $W_{A^\ast}'(\zeta')\geq 0$.
\item $Z_{A^\ast}(\zeta')=\frac{A^\ast}{\lambda C'}$  and $Z_{A^\ast}'(\zeta')\geq 0$.
\end{enumerate}

Consider the first the case, if  $W_{A^\ast}(\zeta')=\frac{A^\ast}{2}$ then
\begin{align}
W_{A^\ast}'(\zeta') &= \frac{1}{\mathcal V_{\mw W}}\left(F_{ W}-\lambda W_{A^\ast}- \mc D_{\mw W} W_{A^\ast}-\mc H_{\mw W}  Z_{A^\ast} \right)\bigg\vert_{\zeta=\zeta'} \notag\\
&\leq  \frac{1}{ \mathcal V_{\mw W}(\zeta')}\left(-\frac{\lambda W_{A^\ast} (\zeta')}{2} - \frac{\alpha  (\mw W(\zeta')W_{A^\ast}(\zeta')-\mw Z (\zeta')Z_{A^\ast}(\zeta'))}{\zeta' } - \frac{1-\alpha}{2} \p_\zeta \mw W(\zeta') Z_{A^\ast}(\zeta') \right) \notag
 \\
&\leq \frac{-\lambda A^\ast}{4 \mathcal V_{\mw W}(\zeta')} + \left( \frac{\alpha \mw Z(\zeta')}{\zeta'} - \frac{1-\alpha}{2} \p_\zeta \mw W(\zeta') \right) \frac{Z_{A^\ast} (\zeta')}{\mc V_{\mw W} (\zeta ') }\,. \label{eq:milan}
\end{align}
In the second line, we absorbed many terms by $-\frac{ \lambda W_{A^\ast}}{2} $. In the third line, we used that $W_{A^\ast}(\zeta')=\frac{A^\ast}{2}$, $\mathcal V_{\mw W}$ is positive and $\mw W$ is positive by Remark \ref{remark:torino}. Thus, we arrive at a contradiction, using that
\begin{equation} \label{eq:venecia}
\frac{\alpha \mw Z}{\zeta} - \frac{1-\alpha}{2} \p_\zeta \mw W < 0\,,
\end{equation}
from Lemma \ref{lemma:venecia} and $Z_{A^\ast} \geq \frac{A^\ast}{C' \lambda} > 0$.

Now consider the second case, $Z_{A^\ast(\zeta')}= \frac{A^\ast}{C'\lambda}$. Thus, we get
\begin{align}
Z_{A^\ast}'(\zeta') &= \frac{1}{|\mathcal V_{\mw Z}|}\left( -F_{ Z}+\lambda Z_{A^\ast}+ \mc D_{\mw Z} Z_{A^\ast}+\mc H_{\mw Z}  W_{A^\ast}\right)\bigg\vert_{\zeta=\zeta'} \notag \\
&\leq  \frac{1}{|\mathcal V_{\mw Z}|}\left( \|F_{ Z}\|_\infty + (\lambda + C'') Z_{A^\ast} + \frac{1-\alpha}{2}\p_\zeta \mw Z  W_{A^\ast} - \frac{\alpha  (\mw WW_{A^\ast}-\mw Z Z_{A^\ast})}{\zeta } \right)\bigg\vert_{\zeta=\zeta'} \notag \\
&\leq \frac{1}{| \mc V_{\mw Z} |} \left( \| F_{Z} \|_\infty + A^\ast \left( \frac{1}{C'}+\frac{C''}{C'\lambda } \right)  - \frac{1}{100} W_{A^\ast}  \right)\bigg\vert_{\zeta=\zeta'} \notag  \\
&\leq \frac{1}{| \mc V_{\mw Z}(\zeta') |} \left( C''' + \frac{A^\ast}{C'} + \frac{A^\ast C''}{C'\lambda} - \frac{A^\ast}{200} \right) \,.\label{eq:napoles}
\end{align}
In the first inequality, we have used that $\mc V_{\mw Z} (\zeta')< 0$. In the second one, we have bounded most of the terms from $\mc D_{\mw Z}A_{A^\ast}$ simply by $C''Z_{A^\ast}$, being $C''$ a constant sufficiently large. In the third one, we used our value for $Z_{A^\ast} = \frac{A^\ast}{C'\lambda}$, the fact that $\mw Z$ is negative for $\zeta \in (0, 1]$ (Remark \ref{remark:torino}) and we also used
\begin{equation} \label{eq:florencia}
\left( \frac{1-\alpha}{2} \p_\zeta \mw Z - \frac{\alpha \mw W}{\zeta}  \right) < \frac{-1}{100}\,.
\end{equation}
from Lemma \ref{lemma:florencia}. Choosing $C' = 400$ and $A^\ast, \lambda$ to be sufficiently large, we get a contradiction from \eqref{eq:napoles}.

The second inequality of \eqref{eq:Thylacine} follows analogously, enlarging $A^\ast$ if needed.

 As a consequence of \eqref{eq:Thylacine}, for any $0<\delta<1$, there exists a map $\mathcal F_\delta:[-1,1]\rightarrow [-A^\ast,A^\ast]$ such that if $(W,Z)$ is the smooth solution on $(0,2]$ corresponding to $W(1)=\mathcal F_\delta(z)$, then
 \[ \label{eq:liguria} W(\delta)+Z(\delta)=z\,.\] 
 We now want to show such solutions $(W,Z)$ can be bounded on the region $[\delta,2]$, independent of the choice of $\delta$ and $z\in[-1,1]$.  We introduce a parameter $M$ (that will be taken sufficiently large) and we note that on the region $\left[ \frac{1}{M \lambda}, 2 \right]$ we have the bound
 \[|W|+ |Z|\leq C_1\,,\]
 for some constant $C_1$, depending on $\lambda$ and $M$, independent of $z$ and $\delta$.
 
 Since $\lambda$ may be chosen sufficiently large, dependent on $F_W, F_Z$, we can rewrite \eqref{eq:fairywren1} and \eqref{eq:fairywren2} as
 \begin{align} \begin{split} \label{eq:piemonte}
 \mc V_{\mw W} \p_\zeta W &= - \frac{\alpha}{\zeta} \mw W W + \frac{\alpha}{\zeta} \mw Z Z + O \left( \lambda \sqrt{ W^2 + Z^2 } \right) \,,\\
 \mc V_{\mw Z} \p_\zeta Z &= - \frac{\alpha}{\zeta} \mw Z Z + \frac{\alpha}{\zeta} \mw W W + O \left( \lambda \sqrt{ W^2 + Z^2 } \right)\,.
 \end{split} \end{align}
 Setting $U = \frac{W+Z}{2}, S = \frac{W-Z}{2}$, and using
 \begin{equation*}
 -\mw W W + \mw Z Z = - (\mw U + \mw S) (U + S) + (\mw U - \mw S) (U-S) = -2 \mw U S - 2 \mw S U,
 \end{equation*}
 we obtain 
 \begin{align} \begin{split} \label{eq:lombardia}
 \p_\zeta U &= \left( \frac{1}{2\mc V_{\mw W}} - \frac{1}{2\mc V_{\mw Z}} \right) \frac{\alpha}{\zeta} \left( -2 \mw U S - 2 \mw S U \right) + O \left( \lambda \sqrt{ U^2 + S^2 } \right) \\
 &=  \frac{\alpha}{\mc V_{\mw W} (0)} \frac{ - 2 \mw S(0) U }{\zeta} + O \left( \lambda \sqrt{ U^2 + S^2 } \right) \,,\\
 \p_\zeta S &= \left( \frac{1}{2\mc V_{\mw W}} + \frac{1}{2\mc V_{\mw Z}} \right) \frac{\alpha}{\zeta} \left( -2 \mw U S - 2 \mw S U \right) + O \left( \lambda \sqrt{ U^2 + S^2 } \right) \\
 &= O \left( \lambda \sqrt{ U^2 + S^2 } \right).
 \end{split} \end{align}
We used that $\mc V_{\mw W} (0) = - \mc V_{\mw Z} (0)$ and  that $\frac{\mw U}{\zeta}$ is uniformly bounded close to $\zeta = 0$. 

Defining $C_2 = \frac{2 \mw S(0)\alpha}{ \mc V_{\mw W} (0)  }$, From \eqref{eq:lombardia}, we obtain 
\begin{align} \begin{split} \label{eq:lombardia2}
\left| \p_\zeta U + \frac{C_2 U}{\zeta} \right| \leq C_3 \lambda \sqrt{U^2 + S^2} \qquad \mbox{ and } \qquad 
\left| \p_\zeta S \right| \leq C_3 \lambda \sqrt{U^2 + S^2}\,.
\end{split} \end{align}
for some new constant $C_3>0$. Let us define
\begin{equation} \label{eq:umbria}
A = \left| \frac{C_2 U(\zeta)}{\zeta} \right| ,\qquad B = C_3 \lambda \sqrt{U^2+S^2}\qquad \mbox{ and } \qquad G = U^2 + S^2\,.
\end{equation}
Let us work under the hypothesis that $A > 10B$ for some $\zeta$. Without loss of generality, let us also assume that $U(\zeta) > 0$. Then, 
\begin{align*}
-\p_{\zeta} A = -\p_{\zeta} \frac{C_2 U(\zeta)}{\zeta} \geq \frac{C_2 U}{\zeta^2} - \frac{C_2 \left(-C_2U/\zeta \right)}{\zeta} - \left| \frac{C_2 \left( C_3 \lambda \sqrt{U^2 + S^2} \right)}{\zeta} \right| = \frac{1+C_2}{\zeta} A - \frac{C_2}{\zeta} B \geq  M\lambda A \,, \\
-\p_{\zeta} B = -\frac{BG'}{2G } \leq B\frac{\zeta A^2 /C_2}{B^2 /(C_3^2 \lambda^2) } + 2C_3 \lambda B 
\leq \zeta\frac{\lambda^2 C_3^2}{10 C_2} A + 2 C_3 \lambda B = \frac{\lambda C_3^2}{10 M C_2} A + 2 C_3 \lambda B \leq C_4 \lambda A\,,
\end{align*}
where $C_4$ is independent of $M, \lambda$. 
Then, taking $M$ to be sufficiently large so that $M > 10 C_4$, we see that the hypothesis $A > 10B$ for some $\zeta'$, implies that the same hypothesis holds for all smaller $\zeta \in [\delta, \zeta']$. Taking $\zeta'$ to be the largest $\zeta \in [\delta, \frac{1}{M \lambda}]$ such that $A \geq 10B$ (or $\zeta' = \delta$ if there is no such $\zeta$ exists), we obtain the following: \begin{itemize}
\item In the region $(\delta, \zeta')$, we have $A \geq 10B$ and moreover $U(\zeta)$ has constant sign in this interval.
\item In the region $\left( \zeta', \frac{1}{M \lambda} \right)$ we have $A < 10B$.
\end{itemize}

Let us first treat the region $\zeta' < \zeta < \frac{1}{M \lambda}$. As $A < 10B$, we obtain that $|U'| \leq 11 C_3 \lambda \sqrt{U^2 + S^2}$. Thus
\begin{equation*}
\left| G'(\zeta) \right| \leq 24 C_3 \lambda G.
\end{equation*} 
Integrating, we get that for all $\zeta \in \left( \zeta', \frac{1}{\lambda M} \right)$
\begin{equation} \label{eq:sicilia}
G(\zeta) \leq G \left( \frac{1}{\lambda M} \right) e^{\frac{24 C_3}{M}} \leq C_{M, \lambda} ,
\end{equation}
where $C_{M, \lambda}$ is sufficiently large depending on $M$ and $\lambda$ (independent of $z, \delta$).

In the region $\delta < \zeta < \zeta'$, using $A > 10B$ in equation \eqref{eq:lombardia2}, we see that $U'$ has the same sign as $\frac{-C_2 U}{\zeta}$, so $U^2$ is decreasing in $\delta < \zeta < \zeta'$. Therefore, for  $\zeta \in [\delta, \zeta']$ we have
\begin{equation} \label{eq:calabria}
|U(\zeta)| \leq |U(\delta )| \leq 2|z| \leq 2\,.
\end{equation}
Using \eqref{eq:calabria} in \eqref{eq:lombardia2}, we obtain that
\begin{equation*}
|S'| \leq 2 C_3 \lambda + C_3 \lambda |S|.
\end{equation*}
Using the initial condition $|S(\zeta')| \leq C_{\lambda, M}$ implied by \eqref{eq:sicilia}, we deduce that for $\zeta \in [\delta, \zeta']$.
\begin{equation} \label{eq:sardegna}
| S(\zeta) | \leq \bar C_{\lambda, M}\,,
\end{equation} 
for some $\bar C_{M, \lambda}$ sufficiently large, independent of $z$ and $\delta$. Finally, recalling $W = U+S$, $Z = U-S$, \eqref{eq:sicilia}--\eqref{eq:sardegna} give us a bound on $|W|$ and $|Z|$ independent of $\delta, z$ for all $\zeta \in \left( \delta, \frac{1}{M \lambda} \right)$.

To summarize, there exists a map $\mathcal F_\delta:[-1,1]\rightarrow [-A^\ast,A^\ast]$ such that if $(W,Z)$ is the smooth solution on $(0,2]$ corresponding to $W(1)=\mathcal F_\delta(z)$, then $W(\delta)+Z(\delta)=z$
 and we have a bound on $|W|$ and $|Z|$ independent of $\delta, z$ for all $\zeta \in \left( \delta, \frac{1}{M \lambda} \right)$. Now we wish to show that exists a continuous solution map $\mathcal G$, mapping any $w_0$ to an analytic solution  $(W,Z)$ to equation \eqref{eq:max:prob2} on the region $\zeta\in(-\frac12,\frac12)$ such that $W(0)=w_0$. With these two maps, $\mathcal F$ and $\mathcal G$, we will be able use a fixed point argument to construct a smooth solution to \eqref{eq:max:prob2} on $\zeta\in[-2,2]$.
 
 We will now repeat an expansion argument in line with the expansion in the proof of Proposition \ref{prop:solnearxi0}. Writing $ W=\sum_{i \geq 0} w_i \zeta^i$, $\mathcal V_{\mw W}=\sum_{i \geq 0} v_i \zeta^i$, $\mathcal G_{ W}=\sum_{i \geq 0} g_i \zeta^i$ and $\mw W=\sum_{i \geq 0} \bar w_i \zeta^i$,  then substituting these formal expansions into \eqref{eq:fairywren1}, we obtain
\[(n+1)v_{\mw W,0} w_{n+1}=g_{n}-\sum_{i=0}^{n-1}(i+1)v_{\mw W,n-i}w_{i+1}\,.
\]
Let us rewrite $g_{ n}$ as
\[g_{n}=\check g_{n}-2\mathbbm{1}_{ 2\mid n}\alpha \mw w_0 w_{n+1}\,,\]
then, using $v_{\mw W,0} =\alpha \mw w_0$ we have
\begin{equation}\label{eq:lace:monitor}
\alpha \mw w_0\left(n+1 +2\mathbbm{1}_{2\mid n} \right)w_{n+1}=\check g_{n}-\sum_{i=0}^{n-1}(i+1)v_{\mw W,n-i}w_{i+1}\,.
\end{equation}
Since $\mw w_0\neq 0$, then \eqref{eq:lace:monitor} can be used to define $w_i$ given $w_0$. Since $(\bar W,\bar Z,F_W)$ are analytic, it is easy to see that the formal series converges producing an analytic solution in a small neighborhood of $\zeta=0$.
 
 Let us denote the solution map $w_0\mapsto W$ restricted to $w_0\in[-2C, 2C]$ by $\mathcal G$. Then by continuity, there exists some $\delta'>0$ such that
 the solutions in the range of $\mathcal G$ are all analytic on the region $\zeta\in[-\delta',\delta']$.

  Now, we consider the map $w_0 \mapsto \mathcal G (w_0)(\delta )$ for some $0 < \delta < \delta'$ sufficiently small. We can take $\delta$ sufficiently small so that: \begin{itemize}
  \item In the range of $\mathcal G$ we have $\abs{W+Z} = \abs{W(\zeta) - W(-\zeta)} \leq 1$ for all $\zeta\in(-\delta,\delta)$.
  \item The map $w_0 \mapsto \mathcal G (w_0)(\delta )$ has a range that covers $[-C, C]$ (recall $w_0 \in [-2C, 2C]$).
  \item The map $w_0 \mapsto \mathcal G (w_0)(\delta )$ is injective.
  \end{itemize} 
  The first two items follow from $W(\zeta) = w_0 + O(\delta)$ when $\zeta \in (-\delta, \delta)$, so let us argue the third one. Considering $w_0, w_0' \in [-2C, 2C]$, we use \eqref{eq:lace:monitor} to define $w_i, w_i'$ via the Taylor recurrence. We have
  \begin{equation} \label{eq:montecristo}
  \mathcal G (w_0)(\delta) - \mathcal G (w_0')(\delta) = w_0 - w_0' + \sum_{i \geq 1} (w_i - w_i') \zeta^i.
  \end{equation}
  Now, from \eqref{eq:lace:monitor}, and using the formula for $\bar g_n$ from \eqref{eq:bilby3}, we have 
  \begin{align*}
  \alpha \mw w_0\left(n+1 +2\mathbbm{1}_{2 | n} \right) (w_{n+1} - w_{n+1}' )&=
  (1-r) (w_n - w_n') - \mathbbm{1}_{2 | i} \alpha \sum_{j=1}^{n} (w_j w_{n+1-j} - w_j' w_{n+1-j}')\\
  &\qquad -\sum_{i=0}^{n-1}(i+1)v_{\mw W,n-i} (w_{i+1} - w_{i+1}').
  \end{align*}
  Letting $d_i = w_{i} - w_i'$, we see that
  \begin{equation*}
  n |d_{n+1}| \les  n \sum_{i=0}^n |d_i| + \sum_{j=1}^n |d_j| \cdot |w_{n+1-j}|.
  \end{equation*}
  Choosing $\delta$ sufficiently small, we can assume $|w_i| \les 1/(3\delta)^i$, and in that case the equation above allows to close an induction argument for the bound $|d_i| \leq C |d_0| / (2\delta)^i$, for some constant $C$ independent of $\delta, w_0, w_0'$ and $\delta$ sufficiently small. Therefore, $|w_i - w_i'| \les |w_0-w_0'|/(2\delta)^i$, and from \eqref{eq:montecristo} we conclude the injectivity.

  Now, since $w_0 \mapsto \mathcal G (w_0)(\delta )$ is injective and covers $[-C, C]$, we define $\mathcal H:[-C,C]\rightarrow [-2C,2C]$ to be its inverse map restricted to $[-C, C]$. Therefore, $\mathcal H$ takes the value of a smooth solution $W$ at $\zeta = \delta$ and outputs the value that $W$ has at $\zeta = 0$. Now, we consider the following map $z \mapsto z'$, where $z'$ is defined as follows. First, we construct $(W, Z) = \mathcal F_\delta (z)$, which we recall that is the smooth solution on $(0, 2]$ with $W(\delta) + Z(\delta) = z$. Then, we apply $\mathcal H (W(\delta))$ to obtain the value of $w_0$ that generates a smooth solution around $\zeta = 0$ with that corresponding value of $W(\delta)$. Finally, we let $\check W = \mathcal G (\mathcal H (W(\delta))$ to be the smooth solution generated by that $w_0$, and define $z' = \check W(\delta) - \check W(-\delta )$ to be its corresponding value of $\check W+\check Z$ at $\zeta = \delta$. Since $z \mapsto z'$ maps the interval $[-1, 1]$ to $[-1, 1]$ and it is continuous, Brouwer's fixed point theorem ensures the existence of a fixed point $z$. For such $z$, we have that $W(\delta) + Z(\delta) = z = z' = \check W(\delta) -\check W(-\delta)$. Note that by construction of the map $z \mapsto z'$, we always have $W(\delta) = \check W(\delta )$, since the definition of $\check W(\delta) = \mathcal G (\mathcal H ( W(\delta ))$ is the solution constructed via a Taylor series at zero that passes through $(\delta, W(\delta ))$. Therefore, we have that $(W(\delta), Z(\delta) ) = (\check W (\delta) , -\check W(-\delta ))$, and by standard uniqueness of ODE solutions, we conclude they are the same solution. Since $\check W (\zeta )$ is smooth for $\zeta \in (-\delta, \delta)$ and $(W(\zeta), Z(\zeta ))$ for $\zeta \in (0, 2]$, we see that $\check W(\zeta)$ is smooth for $\zeta \in [-2, 2]$.

 Let $\aleph$ be the space of functions that can be written as $F=\chi(\mc F_U,\mc F_S)$, for $(\mc F_U,\mc F_S)$  analytic and $\chi$  a smooth cut-off function, $1$ on $[0,\frac32]$ and compactly supported on $[0,2)$.\footnote{It is important to note that we are not fixing $\chi$, each element of $\aleph$ may be defined in terms of a different $\chi$.}  Then for every $F\in \aleph$, we have shown there exists $(U,S)\in \mathcal D(\mathcal L)$ satisfying \eqref{eq:max:prob}. Now given a $F\in X$ and a sequence  $F_{j}\in \aleph$ converging to $F$ in $X$, it remains to show that the corresponding sequence $(U_j,S_j)\in \mathcal D(\mathcal L)$ solving \eqref{eq:max:prob} converges in $\mathcal D(\mathcal L)$.

Observe that
\begin{align*}
\langle F_i -F_j,~(U_i-U_j,~S_i-S_j) \rangle_{H^{2m}}&=
\langle -\mathcal L(U_i-U_j,~S_i-S_j) ,~(U_i-U_j,~S_i-S_j) \rangle_{H^{2m}}\\&\qquad+\lambda \norm{(U_i-U_j,~S_i-S_j)}_{H^{2m}}^2
\\
&\gtrsim \norm {P_N\mathcal L(U_i-U_j,~S_i-S_j) }_{H^{2m}}^2+ \norm {(I-P_N)\mathcal L(U_i-U_j,~S_i-S_j) }_{H^{2m}}^2\\&\qquad
+\lambda \norm{(U_i-U_j,~S_i-S_j)}_{H^{2m}}^2\\
&\gtrsim \lambda \norm{(U_i-U_j,~S_i-S_j)}_{H^{2m}}^2\,,
\end{align*}
where we used Lemma  \ref{lemma:accretivity}, that $\mathcal L$ is bounded on the finite dimensional orthogonal complement of $Y_N$ and that we are free to take $\lambda $ arbitrarily large. By Cauchy–Schwarz, we conclude
\[ \norm{(U_i-U_j,~S_i-S_j)}_{H^{2m}}\les \norm{F_i -F_j}_{H^{2m}}\,. \]
Thus since $F_j$ forms a Cauchy sequence, so is $(U_j,S_j)$, which concludes the proof.
\end{proof}

\begin{corollary}\label{cor:A0} For all $\delta_g > 0$ sufficiently small, we can write our linearized operator $\mc L$  as 
\[\mc L= A_0 - \delta_g + K\,,\]
 for some $A_0$ maximally dissipative on $X$ and  $K$  is some compact operator on $X$
\end{corollary}
\begin{proof} First of all, recall that $P_N : X \rightarrow Y_N$ is the projection onto the finite dimensional subspace $Y_N$. From Lemma \ref{lemma:accretivity} we have that 
\begin{equation*}
\langle \mathcal L(U, S), (U, S) \rangle_{ H^{2m}} \leq - \| (U, S) \|_{H^{2m}}^2,
\end{equation*}
for every $(U, S) \in Y_N$, so in particular $P_N \mc L + \delta_g$ is dissipative on $Y_N$.

 Lemma \ref{lemma:maximality} implies that  $P_N \mc L + \delta_g -  \lambda :X \rightarrow X$ is surjective for sufficiently large $\lambda$.
Since $ Y_N$ has finite codimension $\text{cod}_X(Y_N)$, the surjectivity of $\mc L+\delta_g - \lambda$ on $X$ implies that there exists some finite codimension space $\mt Y_N \subset Y_N$ (of finite codimension $\text{cod}_X(\mt Y_N) \leq 2 \text{cod}_X(Y_N)$) such that $\mt P_N \mc L+\delta_g - \lambda : \mt Y_N \rightarrow \mt Y_N$ is surjective on $\mt Y_N$. Here $\mt P_N$ denotes the orthogonal projection to $\mt Y_N$.

Thus, we get that $\mt P_N \mc L + \delta_g$ is a maximally dissipative operator on the finite codimension space $\mt Y_N$. Let  $A_0$ be a maximally dissipative operator on $X$ that agrees with $\mc L+\delta_g$ on $\mt Y_N$. For instance, one may define $A_0$ to be $-\text{Id}$ on $\mt Y_N^\perp$, and this clearly makes $A_0: X \rightarrow X$ maximally dissipative. In particular, we have the equality $\mt P_N \mc L = A_0 - \delta_g$ over the finite codimension space $\mt Y_N$. Letting $K_1 = (1-\mt P_N)\mc L$, which is compact because it has finite dimensional rank, we get $\mc L = A_0 - \delta_g + K_1$ over $\mt Y_N$.

Then, let $K_2$  be an operator which is zero over $\mt Y_N$ and it is defined as $\mc L - A_0 + \delta_g - K_1$ on $\mt Y_N^\perp$. Again, $K_2$ is compact, as it has finite dimensional rank. Moreover, we get that the equality
\begin{equation*}
\mc L = A_0 - \delta_g + K_1 + K_2\,,
\end{equation*}
holds both over $\mt Y_N$ and over $\mt Y_N^\perp$. Therefore, it holds over all $X$ and letting $K = K_1 + K_2$, we conclude that $\mc L  = A_0 - \delta_g + K$ for some $A_0$ maximally dissipative, $\delta_g > 0$ and $K$ compact. \end{proof}

\subsection{Abstract results on maximally dissipative operators}

We first recall some basic properties of maximally dissipative operators (see \cite{MR2953553,MR132403}).

\begin{lemma}[Properties of maximally dissipative operators] \label{lemma:propertiesmaximalaccretive} Let $A_0$ be a maximally dissipative operator on a Hilbert space $H$. Then, we have that: \begin{itemize}
\item $A_0$ is closed.
\item $\sigma (A_0) \subset \{ \lambda \in \mathbb{C} : \Re (\lambda) \leq 0 \}$.
\item For every $\lambda \in \mathbb{C}$ with $\Re \lambda > 0$, we have that $(-A_0 + \lambda  ) : D(A_0) \rightarrow H$ is a bijection and moreover 
$\| (-A_0 +  \lambda )^{-1} \|_{L(H \rightarrow D(A_0))} \leq \Re (\lambda )^{-1}$.
\item $A_0^\ast$ is also maximally dissipative.
\item (Lumer-Phillips theorem): $A_0$ generates a strongly continuous semigroup on $H$.
\end{itemize}
\end{lemma}

\begin{definition} We define the growth bound of a semigroup $T$ on $H$ as
\begin{equation*}
w_0(T) = \inf \left\{ w \in \mathbb{R} , \; \exists M_w \text{ such that } \forall t \geq 0, \; \| T(t) \| \leq M_w e^{wt} \right\}.
\end{equation*}
We also define 
\begin{equation*}
w_{\rm ess}(T) = \inf \left\{ w \in \mathbb{R} , \; \exists M_w \text{ such that } \forall t \geq 0, \; \| T(t) \|_{\rm ess} \leq M_w e^{wt} \right\},
\end{equation*}
where the essential seminorm is defined as
\begin{equation*}
\| T(t) \|_{\rm ess} = \inf_{K \text{ compact }} \|T(t) - K \|\,.
\end{equation*}
\end{definition} 

\begin{lemma}[Growth bound \cite{EngelNagel}] \label{lemma:growthbound} Let us suppose that $A$ generates the continuous semigroup $T$. Let $\sigma (A)$ be the spectrum of $A$ and let us consider $s(A) = \sup_{\lambda \in \sigma(A)} \Re (\lambda )$. Then
\begin{equation*}
w_{0}(T) = \max \left\{ w_{\rm ess}(T), s(A)\right\}\,.
\end{equation*}
Moreover for any $w > w_{\rm ess}(T)$ have that $\sigma (A) \cap \{ \lambda \in \mathbb{C} : \Re (\lambda) > w \}$ is a finite set of eigenvalues with finite algebraic multiplicity. 
\end{lemma} 
 
\begin{lemma} \label{lemma:abstract_result} Let $A_0$ be a maximally dissipative operator and consider $A = A_0 - \delta_g + K$ where $\delta_g > 0$ and $K$ compact. Then we have: \begin{enumerate}
\item\label{item:spectrum} The set $\Lambda  = \sigma (A) \cap \{ \lambda \in \mathbb{C} : \Re (\lambda) > -\delta_g/2\}$ is finite and formed only by eigenvalues of $A$. Moreover each $\lambda \in \Lambda$ has finite algebraic multiplicity. That is, if we let $\mu_{\lambda}$ to be the first natural such that ${\rm ker} (A-\lambda \text{Id} )^{\mu_{\lambda}} = {\rm ker} (A- \lambda \text{Id} )^{\mu_\lambda + 1}$, we have that the vector space
\begin{equation} \label{eq:spaceV}
V = \bigoplus_{\lambda \in \Lambda} {\rm ker} (A - \lambda \text{Id})^{\mu_{\lambda}}\,,
\end{equation}
is finite dimensional. 
\item \label{item:invariance} Consider $A^\ast = A_0^\ast - \delta_g + K^\ast$ and let $\Lambda^\ast = \sigma (A^\ast) \cap \{ \lambda \in \mathbb{C} : \Re (\lambda) >- \delta_g/2\}$. As before, we define
\begin{equation} \label{eq:spaceVast}
V^\ast = \bigoplus_{\lambda \in \Lambda^\ast} {\rm ker} (A^\ast - \lambda \text{Id})^{\mu_{\lambda}^\ast}\,.
\end{equation}
We have that both $V$ and $V^{\ast \, \perp}$ are invariant under $A$. We also have that $\Lambda^\ast = \overline{\Lambda}$ and $\mu_\lambda = \mu_{\overline{\lambda}}^\ast$. Moreover, we have the decomposition $H = V \oplus V^{\ast \, \perp}$.

\item \label{item:stability_outwards} The linear transformation $A|_V: V\rightarrow V$ obtained by restricting $A$ to the finite dimensional space $V$ has all its eigenvalues with real part larger than $-\delta_g / 2$. In particular, there is some basis such that we can express
\begin{equation*}
A|_V  = 
\begin{bmatrix}
J_1 &  & &  \\
 & J_2 & &  \\
 & & \ddots & \\
 &  & & J_\ell
 \end{bmatrix},  
 \qquad \mbox{ where } \qquad J_i = 
 \begin{bmatrix}
\lambda_i & \delta_g/10 & & \\
 & \lambda_i & \ddots & \\
 & & \ddots &  \delta_g/10\\
 & & & \lambda_i
 \end{bmatrix},
\end{equation*}
where $\lambda_i$ are the eigenvalues of $A|_V$. In that basis, we have that
\begin{equation} \label{eq:chisinau}
w^T \cdot A|_V \cdot w \geq \frac{-6\delta_g}{10} \| w \|^2, \qquad \forall w \in \mathbb{R}^N.
\end{equation}

Moreover, letting $T(t)$ be the semigroup generated by $A$, for any  $v \in V^{\ast \, \perp}$ we have
\begin{equation} \label{eq:prim}
\| T(t) v \|_H \lesssim e^{-\delta_g t/2} \| v \|_H.
\end{equation}
\end{enumerate}
\end{lemma}
 Lemma \ref{lemma:abstract_result} is very similar to Lemma 3.3 in \cite{MeRaRoSz19b}. We however provide a proof for completeness.
\begin{proof}

\textbf{Item \ref{item:spectrum}.}  Lemma \ref{lemma:propertiesmaximalaccretive} tells us that $A_0$ generates a contraction semigroup $T_0(t)$. Therefore, we have that $A_0-\delta_g$ generates a contraction semigroup $T_1(t) = e^{-\delta_g t}T_0(t)$, such that $w_0(T_1), w_{\rm ess}(T_1) \leq -\delta_g$.

Thus as $K$ is compact, $A_0 -\delta_g-K$ generates a continuous semigroup $T(t)$ as well, and as $w_{\rm ess}$ is invariant under compact perturbations, we have that $w_{\rm ess}(T) \leq -\delta_g$. In particular, applying Lemma \ref{lemma:growthbound} for $w = \frac{-1}{2} \delta_g$, we see that $\Lambda$ consist of finitely many eigenvalues with finite algebraic multiplicity.

The fact that the spaces ${\rm ker}(A-\lambda_i \text{Id})^{\mu_{\lambda}}$ are linearly independent for a finite set of different $\lambda_i$ is well-known in linear algebra.

\textbf{Item \ref{item:invariance}.} It is clear that ${\rm ker}(A-\lambda \text{Id})^{\mu_\lambda}$ is invariant under $A$: if $v \in {\rm ker}(A-\lambda \text{Id})^{\mu_\lambda}$, we just have 
\begin{equation*}
0 = (A-\lambda \text{Id})^{\mu_{\lambda} + 1} v = (A - \lambda \text{Id})^{\mu_{\lambda}} (Av) - \lambda (A - \lambda \text{Id})^{\mu_{\lambda}} v = (A - \lambda \text{Id})^{\mu_{\lambda}} (Av)\,.
\end{equation*}
 As a consequence, $V$ remains invariant under $A$.

Applying the  argument from Item \ref{item:spectrum} to $A^\ast$, we find that $\Lambda^\ast$ is finite and consequently  $V^\ast$ is finite dimensional. In addition, as above, we find that $V^\ast$ is invariant under $A^\ast$. Therefore, for any $v \in V^{\ast \, \perp}, w \in V^\ast$ 
\begin{align*}
\langle A v, w \rangle = \langle v, A^\ast w \rangle = 0,
\end{align*}
since $A^\ast w \in V^\ast$. Thus, $Av \in V^{ \ast \, \perp}$ and we have shown that $V^{\ast \, \perp}$ is invariant under $A$.

Now suppose $\lambda \in \Lambda^\ast \setminus \Lambda$. As $\Re (\lambda ) > -\delta_g/2$ and $\lambda \notin \Lambda$, hence  the resolvent $(A-\lambda \text{Id} )^{-1}$ is a bounded operator. Therefore, $(A^\ast - \overline{\lambda} \text{Id} ) u = v$ is equivalent to
\begin{equation*}
\langle w, u \rangle = \langle (A - \lambda \text{Id} )^{-1} w, v \rangle, \qquad \forall w \in H.
\end{equation*}
By the Riesz representation theorem, we have a unique such $u\in H$ and moreover the map $v\mapsto u$ is bounded. This shows that $\overline{\Lambda^\ast} \subset \Lambda$; however, since $A^{\ast \, \ast} = A$, an analogous argument for $A^\ast$ shows that  $\Lambda\subset\overline{\Lambda^\ast} $ and hence $\overline{\Lambda^\ast} =\Lambda$. 

Before showing that the multiplicities are equal, let us show that $H = V \oplus V^{\ast \, \perp}$. Let us note that $V$ is the image of $H$ under the projector
\begin{equation*}
P(A) = \frac{1}{2\pi i} \int_{\Gamma} \left( \lambda \text{Id} - A \right)^{-1} d \lambda,
\end{equation*}
for some curve $\Gamma$ enclosing $\Lambda$. On the other hand, as $\overline{\Lambda} = \Lambda^\ast$, we get that $P(A^\ast) = P(A)^\ast$. Therefore, using $\text{Im}(P(A^\ast ))^\perp = {\rm ker}(P(A))$ and the decomposition $H = {\rm ker}(P(A)) \oplus \text{Im}(P(A))$ (since $P(A)$ is a projector),  yields the desired decomposition $H = V^{\ast \, \perp} \oplus V$.

Lastly, let us show that $\mu_\lambda = \mu_{\overline{\lambda}}^\ast$. Without loss of generality assume $\mu_{\lambda} > \mu_{\overline{\lambda}}^\ast$, then we have that
\begin{equation*}
 \left\langle (A - \lambda \text{Id})^{\mu_{\overline{\lambda}}^\ast} v, w \right\rangle
=  \left\langle  v, (A^\ast - \overline{\lambda} \text{Id})^{\mu_{\overline{\lambda}}^\ast} w \right\rangle = 0
\qquad \forall v \in {\rm ker} (A-\lambda \text{Id} )^{\mu_{\lambda}}, \; w \in {\rm ker} (A - \lambda \text{Id} )^{\mu_{\overline{\lambda}}^\ast}.
\end{equation*}
However, as $\mu_{\lambda} > \mu_{\overline{\lambda}}^\ast$, there exists some $v \in V$ such that the term $v' = \left( A - \lambda \text{Id} \right)^{\mu_{\overline{\lambda}}^\ast}v \neq 0$. It is clear that $v' \in V$, since $V$ is an invariant subspace of $A$. Therefore, we have that $v' \in V$, $v' \in V^{\ast \, \perp}$ and $v \neq 0$. However, this is impossible due to the decomposition $H = V\oplus V^{\ast \, \perp}$.

\textbf{Item \ref{item:stability_outwards}.} If $\lambda$ is an eigenvalue of $A|_V$, we have that there exists $v \in V$ with $A v = \lambda v$. As $v \in V$, we have that $\lambda \in \Lambda$ and therefore $\Re  (\lambda ) > -\delta_g/2$.

Now, we express $A|_V$ in its Jordan normal form, and obtain some blocks $\mt J_i$ with $\lambda_i$ on the diagonal and $1$ on the superdiagonal of each block. Consider $D_i$ to be the diagonal matrix with elements $1, \delta_g/10, (\delta_g/10)^2, \ldots$ on its diagonal. Then, we have that $D_i^{-1} \mt J_i D_i = J_i$, so we can obtain the desired form by applying the change of basis dictated by $D_i$ to the Jordan normal form.

In order to show \eqref{eq:chisinau}, we show that $A|_V + \frac{6\delta_g}{10} I$ is semipositive definite. It suffices to show that each of its blocks $J_i + \frac{6 \delta_g}{10} I$ is semipositive definite. Suppose that the block is of size $k$. Then, as $\Re \lambda \geq - \frac{\delta_g}{2}$, we have
\begin{equation*}
w^T \cdot (J_i + \frac{6 \delta_g}{10} I) \cdot w \geq \sum_{j=1}^k \frac{\delta_g}{10} w_i^2 -  \frac{\delta_g}{10} \sum_{j=1}^{k-1} |w_j | |w_{j+1}| \geq  \frac{\delta_g}{20} \sum_{j=1}^{k-1} (w_j^2 + w_{j+1}^2 - 2|w_j | |w_{j+1}|) \geq 0\,.
\end{equation*}

As $V^{\ast \, \perp}$ is invariant under $A$, we can consider $T_{\rm sta}(t) = T|_{V^{\ast \, \perp}}(t)$, the restriction to the semigroup to that space, which is clearly generated by $A_s = A|_{V^{\ast \, \perp}}$. 

On the one hand, we have that 
\begin{equation}\label{eq:Magnus}
w_{\rm ess}(T_{\rm sta}) = w_{\rm ess}(T) \leq -\delta_g\,,
\end{equation}
 since those contraction semigroups differ by a compact operator.

On the other hand, if $\sigma (A_s)$ had any element $\lambda $ with $\Re (\lambda ) > -\delta_g/2$, we can apply the same reasoning as in \ref{item:spectrum} to say that $\lambda$ has to be an eigenvalue. Thus, we would have an eigenvector $v \in V^{\ast \, \perp}$ with $Av = \lambda v$ and $\Re (\lambda) > -\delta_g/2$. This is a contradiction since ${\rm ker}(A - \lambda \text{Id}) \subset V$ for those $\lambda$. Thus, we get that 
\begin{equation}\label{eq:Nepo}
\sigma (A_s) \subset \{ \lambda \in \mathbb{C}: \Re (\lambda ) \leq -\delta_g/2 \}\,.
\end{equation}

Combining \eqref{eq:Magnus} and \eqref{eq:Nepo} via Lemma \ref{lemma:growthbound}, we get that $w_0(T_{\rm sta}) \leq -\delta_g/2$ and we conclude the proof.

\end{proof}

\subsection{Smoothness of eigenfunctions}\label{ss:smoothness:eigen}

Let us remark that due to Corollary \ref{cor:A0}, we have that our operator $\mc L$ can be written as $A_0 - \delta_g + K$, for some $K$ compact, $\delta_g > 0$ and $A_0$ maximally dissipative. Therefore, we are under the hypothesis of Lemma \ref{lemma:abstract_result} on the space $X$. From now on, let us denote
\begin{equation*}
P_{\rm sta} (\mt U, \mt S) = P|_V (\mt U, \mt S) \qquad \mbox{ and } \qquad P_{\rm uns} (\mt U, \mt S) = P|_{V^{\ast  \perp }} (\mt U, \mt S)\,.
\end{equation*}

\begin{lemma} \label{lemma:canada} Let $(\mt U_t, \mt S_t)$ be radially symmetric. Let $\chi_0$ be a cut-off function supported on $[0, \frac65]$ which takes value $1$ in $[0, 1]$. Then, there exists an absolute constant $C$ independent of $m, J, N$ such that 
\begin{equation*}
\int \chi_0 (\zeta) \mc L_U (\mt U_t, \mt S_t) \cdot \mt U_t + \int \chi_0 (\zeta) \mc L_S (\mt U_t, \mt S_t) \mt S_t \leq C \int \chi_0 (\zeta) \left(  | \mt U_t |^2 +   \mt S_t^2 \right)\,.
\end{equation*}
\end{lemma}
\begin{proof} Let us denote 
\begin{equation*}
\mc I = \int \chi_0 (\zeta) \mc L_U (\mt U_t, \mt S_t) \cdot \mt U_t + \int \chi_0 (\zeta) \mc L_S (\mt U_t, \mt S_t) \mt S_t\,.
\end{equation*}
 Let us note that on the support of $\chi_0$, we have that $\chi_1 = \chi_2 = 1$. Therefore, in this region we have the equalities
\begin{align*}
\mc L_U (\mt U_t, \mt S_t) &= \left( r-1+\p_\zeta \left( \mw U \cdot \frac{y}{\zeta}\right) \right) \mt U_t + \alpha \mt S_t \nabla \mw S + (y+\mw U) \nabla \mt U_t + \alpha \mw S \nabla \mt S_t\,, \\
\mc L_S (\mt U_t, \mt S_t) &= (r-1+\alpha \div(\mw U)) \mt S_t + \p_\zeta \mw S \mt U_t + (y + \mw U)\cdot \nabla \mt S_t + \alpha \mw S \div (\mt U_t)\,.
\end{align*}
Taking $C$ big enough such that $r-1+\| \mw S \|_{\dot{H}^1} + \| \mw U \|_{\dot{H^1}} \leq \frac{C}{4}$, we have
\begin{align*}
\mc I &\leq \frac{C}{2} \int \chi_0 (\zeta) \left(  | \mt U_t |^2 +   \mt S_t^2 \right) +\int \chi_0 (\zeta)  (y+\mw U)\cdot \nabla \mt U_t \cdot \mt U_t + \alpha \int \chi_0(\zeta) \mw S \nabla \mt S_t \cdot \mt U_t \\
&\qquad + \int \chi_0(\zeta) (y + \mw U)\cdot  \nabla \mt S_t \mt S_t + \alpha \int \chi_0 (\zeta) \mw S \div (\mt U_t) \mt S_t\,.
\end{align*}
Therefore, we get
\begin{align*}
\mc I &\leq \frac{C}{2} \int \chi_0 (\zeta) \left(  | \mt U_t |^2 +   \mt S_t^2 \right) 
+\frac{1}{2} \int \chi_0 (\zeta)  (y+\mw U)\cdot \nabla \left(|\mt U_t|^2 + \mt S_t^2 \right) + \alpha \int \chi_0(\zeta) \mw S \div \left( \mt S_t \cdot \mt U_t \right) \\
&\leq \left( \frac{C}{2} + \left\| \frac{\div (\chi_0 (y+\mw U))}{2} \right\|_{L^\infty} + \left\| \alpha \div (\chi_0 \mw S) \right\|_{L^\infty} \right)  \int \chi_0 (\zeta)  (y+\mw U)\cdot \nabla \left(|\mt U_t|^2 + \mt S_t^2 \right) \,.
\end{align*}
\end{proof}

\begin{corollary} If $\lambda$ is an eigenvalue of our operator $\mc L$ we necessarily have that $\Re (\lambda) \leq C$.
\end{corollary}
\begin{proof} Let $\nu = (\nu_U, \nu_S)$ be an eigenfunction of $\mc L$. As these may be complex, let us write $\nu_U = \nu_U^r + i\nu_U^i$ and $\nu_S = \nu_S^r + i\nu_S^i$ for their decompositions into real and imaginary parts. Let us also denote $\nu^r = (\nu_U^r, \nu_S^r)$ and $\nu^i = (\nu_U^i, \nu_S^i)$. As the operator $\mc L$ sends real functions into real functions, we have that 
\begin{align*}
\int \chi_0 \Re  \left( \nu_U \cdot \overline{\mc L_U \nu} + \nu_S \overline{\mc L_S \nu_S} \right) &=   
\int \chi_0\left( \nu_U^r \cdot \mc L_U \nu_U^r + \nu_S^r \mc L_S \nu_S^r \right) 
+ \int \chi_0\left( \nu_U^i \cdot \mc L_U \nu_U^i + \nu_S^i \mc L_S \nu_S^i \right) \\
&\leq C \int \chi_0 \left(| \nu_U^r |^2 + (\nu_S^r)^2 +| \nu_U^i |^2 + (\nu_S^i)^2   \right),
\end{align*}
where in the last inequality we used Lemma \ref{lemma:canada} for the pairs $(\nu_U^r, \nu_S^r)$ and $(\nu_U^i, \nu_S^i)$. Using that $\nu$ is an eigenfunction of eigenvalue $\lambda$ in the left-hand-side of the previous equation, we obtain that
\begin{align*}
\Re (\lambda) \int \chi_0 \left( | \nu_U |^2 + | \nu_S |^2 \right) \leq C \int \chi_0 \left( | \nu_U |^2 + | \nu_S |^2 \right).
\end{align*}
This concludes our claim.
\end{proof}

\begin{lemma}\label{lemma:smooth:eigen}
	If $\delta_g>0$ is chosen sufficiently small and $\{\nu_{i,U},\nu_{i,S}\}_{i=1,\dots,N}$ are the eigenfunctions corresponding to the eigenvalues $\Lambda $ defined in Lemma \ref{lemma:abstract_result} applied to the operator $A_0$ defined in Corollary \ref{cor:A0}, then then  $\{\nu_{i,U},\nu_{i,S}\}_{i=1,\dots,N}$ are smooth.
\end{lemma}
\begin{proof}
Fixing $\delta_g>0$ sufficiently small, let  $\{\nu_{i,U},\nu_{i,S}\}_{i=1,\dots,N}$ be the eigenfunctions corresponding to the eigenvalues $\Lambda $ defined in Lemma \ref{lemma:abstract_result} applied to the operator $A_0$ defined in Corollary \ref{cor:A0}. By Sobolev embedding the eigenfunctions are $C^{2m-1}$.

Fix $i$, let $\lambda+\delta_g$ be the eigenvalue associated with $(\nu_{i,U},\nu_{i,S})$ and define $(W,Z)=(\nu_{i,U}+\nu_{i,S},\nu_{i,U}-\nu_{i,S})$. Then
\begin{align}\notag \begin{split} 
 (\lambda+\mc D_{\mw W}) W+\mc V_{\mw W}\p_{\zeta}W+\mc H_{\mw W} Z&=0\\
  (\lambda+\mc D_{\mw Z}) Z+\mc V_{\mw Z}\p_{\zeta} Z+\mc H_{\mw Z} W&=0
\end{split} \end{align}
We extend $(W,Z)$ in the usual way by requiring $Z(\zeta)=-W(-\zeta)$. 
By simple ODE analysis we obtain that $(W,Z)$ are smooth away from $\zeta =0,1$. At $\zeta=0,1$, we compare  $(W,Z)$ to the power series. At $\zeta=1$, we may use \eqref{eq:DingoZ2} in order to construct a power series expansion around $\zeta=1$. In order to construct the series we are using that $\Re \lambda >-\frac{\delta_g}2$ and hence the prefactor
\[
(\lambda+\mathcal D_{\mw Z}(1)+(n+1)v_{\mw Z,1})\]
in \eqref{eq:DingoZ2} is positive assuming that $n\geq 2m-1$ and $m$ is chosen sufficiently large, dependent on $\delta_g$ -- here we are also using the lower bound Lemma \ref{lemma:aux_DZ1} on $v_{\mw Z,1}$. Let $(\check W, \check Z)$ denote the solution obtained via power series expansion in a small neighborhood $[1-\delta,1+\delta]$ of $\zeta=1$ such that $W(1)=\check W(1)$. We necessarily have that all derivatives of $(W,Z)$ and $(\check W,\check Z)$ agree up to order $2m-1$ at $\zeta=1$. Let $(\mt W,\mt Z)=( W-\check W, Z-\check W)$, then for $\zeta\in[1-\delta,1+\delta]$ we have $(\mt W,\mt Z)=O((\zeta-1)^{2m-1})$.

Suppose for $\zeta \in [1-\delta,1+\delta]$, $C$ is chosen larger enough such that
\[\abs{\mathcal D_{\mw W}}+\abs{\mathcal H_{\mw W}}+\abs{\mathcal D_{\mw Z}}+\abs{\mathcal H_{\mw Z}}\leq C\quad\mbox{and}\quad \frac{1}{\abs{\mathcal V_{\mw W}}}+\frac1{\abs{\mathcal V_{\mw Z}}}\leq \frac{C}{\zeta-1}\,.\]
Then, by Gr\"onwall, we have for $1+\bar \delta<\zeta<1+\delta$
\begin{align*}
\abs{\check W}+\abs{\check Z}\les \frac{1}{{\bar \delta}^{C(C+\abs{\lambda})}} \left(\abs{\check W (1+\bar \delta)}+\abs{\check Z(1+\bar \delta) }\right) \les {\bar \delta}^{2m-1-C(C+\abs{\lambda})}\,.
\end{align*}
By an energy estimate, we can bound $\lambda$ independent of $m$. Then since  $m$ can be chosen sufficiently large, we can take $\bar \delta$ to zero to conclude that $(\check W,\check Z)\equiv 0$ in the region $\zeta\in[1,1+\delta]$. An analogous argument holds in the region $\zeta\in[1-\delta,1]$. In particular we have shown that $(W,Z)$ is smooth in a neighborhood of $\zeta=1$. A similar Gr\"onwall argument using \eqref{eq:lace:monitor} to construct the local analytic solution can be used to prove $(W,Z)$ are smooth in a neighborhood of $\zeta=0$.
\end{proof}

\begin{corollary} \label{cor:findimbon}
	There exists a finite dimensional orthonormal basis of smooth functions $\{\psi_{i,U},\psi_{i,S}\}_{i=1,\dots,N}$  for the space $V$ defined in Lemma \ref{lemma:abstract_result} for the operator $A_0$ defined in Corollary \ref{cor:A0}.
\end{corollary}
\begin{proof}
	Let $\{\nu_{i,U},\nu_{i,S}\}_{i=1,\dots,M}$ be the sequence of smooth eigenvectors defined in Lemma \ref{lemma:smooth:eigen}. For each $i= 1,\dots,M $, define
	\[(\tilde \psi_{i,U},\tilde \psi_{i,S})=\Re (\tilde \nu_{i,U},\tilde \nu_{i,S})\quad\mbox{and}\quad
	(\tilde \psi_{i+M,U},\tilde \psi_{i+M,S})=\Im (\tilde \nu_{i,U},\tilde \nu_{i,S})\,.\]
	Each $(\tilde \psi_{i,U},\tilde \psi_{i,S}) $ is smooth by definition and span $V$. The sequence of functions $\{\psi_{i,U},\psi_{i,S}\}_{i=1,\dots,N}$ can then be constructed via a standard Gram-Schmidt argument.
\end{proof}

\begin{remark} \label{rem:eigen_compsup} Let us note that, moreover, with our definition of $\mc L$, the functions $\psi_{i, U}, \psi_{i, S}$ are compactly supported in $\zeta \leq \frac95$, which is the support of $\chi_2$. As the functions $\psi_{i, U}$ and $\psi_{i, S}$ are linear combinations of $\tilde \psi_{i, U}$ and $\tilde \psi_{i, S}$ respectively, it suffices to check that $\tilde \psi_{i, U}$ and $\tilde \psi_{i, S}$ are supported on $\zeta \leq \frac95$.

Indeed, note on the one hand that $\mc L(\tilde \psi_i) = \lambda_i (\tilde \psi_{i, U}, \tilde \psi_{i, S})$ with $\Re  (\lambda_i) > -\delta_g/2$. On the other hand, from \eqref{eq:lviv}, we get that  
\begin{equation} \label{eq:kherson}
\mc L (\psi_{i, U}, \psi_{i, S}) = -J (\psi_{i, U}, \psi_{i, S})\,,
\end{equation}
outside the support of $\chi_2$. 

As $J$ is taken to be sufficiently large, $-J \ll 1\ll\frac{-\delta_g}{2} < \lambda_i$. Equation \eqref{eq:kherson} contradicts that $\lambda_i$ is the eigenvalue of $\tilde \psi_i$ unless both $\tilde \psi_{i, U}$ and $\tilde \psi_{i, S}$ vanish identically outside the support of $\chi_2$. Thus, $\psi_{i, U}, \psi_{i, S}$ are compactly supported on $\zeta \leq \frac95$.
\end{remark}

\section{Nonlinear stability}\label{sec:nonlinarstability}

For brevity, we will use the notation $\norm{\cdot}_{X}$ in place of $\norm{\cdot}_{H^{2m}(B(0,2))}$ (this is consistent with the definition of the space $X$ given in Remark \ref{rem:espartero}).

Our aim is to show that there exists a finite codimensional manifold of initial data that lead to asymptotically self-similar implosion. To make this more precise, suppose we are given initial data $(\mc U_0',\mc S_0')$ such that the difference $ (\mt U_0',\mt S_0')=(\mc U_0',\mc S_0')-(\mw U,\mw S)$ satisfies the following assumptions
  \begin{equation}\label{eq:US:initial:assumption1}
\| \mt U_0' \|_{L^\infty}, \| \mt S_0' \|_{L^\infty} \leq \delta_1\,,\quad
 \| \chi_2 \mt U_{0, i}' \|_X , \| \chi_2 \mt  S_0' \|_X  \leq \delta_0\quad\mbox{and}\quad\mc  S_0' (\zeta)  \geq \frac12 \delta_1 \quad\mbox{for every }\zeta \in \mathbb{R}^+\,,
\end{equation}
where here $\delta_0$ and $\delta_1$ are constants satisfying the relation 
\[\delta_0^{3/2} \ll \delta_1 \ll \delta_0\ll1\,,\]
and $\chi_2$ is the cut-off function defined in Section \ref{sec:linear}.

We will moreover assume a high order weighted energy estimate on $ (\mt U_0',\mt S_0')$. For some $\zeta_0$ and $\eta_{w}$ yet to be determined, we let $\phi$ be  a smooth function that is $1$ on the region $[-\zeta_0,  \zeta_0]$ and behaves like $\zeta^{2(1-\eta_{w})}$ for $\zeta\geq \zeta_0$. We then assume $ (\mt U_0',\mt S_0')$ satisfies the bound
 \begin{equation}\label{eq:US:initial:assumption2}
 4\pi\int_{0}^{\infty} \left( (\Delta^K \mc U_0')^2 + (\Delta^K  \mc S_0')^2 \right) \phi^{2K}(\zeta) \zeta^2\,d\zeta \leq \bar{E}^2\,,\\
\end{equation}
for $K$ satisfying  $ \delta_0 \ll 1/K \ll \eta_{w}$.

For the convenience of the reader, we collect the following chain of inequalities
\begin{equation*}
\frac{1}{s_0} \ll \delta_0^{3/2} \ll \delta_1 \ll \delta_g\delta_0\ll \delta_0 \ll \frac{1}{\bar E} \ll \frac{1}{K} \ll \frac{1}{m} \ll  \eta_w \ll \delta_g \ll \delta_{\rm dis} = O(1)\,,
\end{equation*}
where we recall $\delta_g$ is defined at the beginning of Section \ref{ss:smoothness:eigen} and $\delta_{\rm dis}$ is defined in \eqref{eq:delta:dis}.

With the assumptions \eqref{eq:US:initial:assumption1} and \eqref{eq:US:initial:assumption2}, we will show there exists  $\{ a_i \}_{i=1}^N$ satisfying $|a| \leq \delta_1$, such that the initial data
\begin{equation}\label{eq:US:initial:data}
( \mc U_0', \mc S_0' ) = \left( \mw U + \mt U_0' + \sum_{i=1}^N a_i \psi_{i, U} ,~ \mw S + \mt S_0' + \sum_{i=1}^N a_i \psi_{i, S} \right)\,,
\end{equation}
leads to a global solution $(\mc U,\mc S)$ to \eqref{eq:US:system}, and moreover, if one sets $ (\mt U,\mt S)=(\mc U,\mc S)-(\mw U,\mw S)$ then
\[\lim_{s\rightarrow \infty } \mt U(\zeta,s)=0\,,\]
for any $\zeta$. The key ingredient to proving this statement is the linear stability of  truncated problem considered  in Section \ref{sec:linear}. To make this link precise, given $( \mc U_0, \mc S_0 ) $, define  its truncation $( \mc U_{0,t}, \mc S_{0,t})$ as
\begin{equation} \label{eq:initialdata_truncated}
 ( \mc U_{0,t},  \mc S_{0,t})=\chi_2 ( \mc  U_0, \mc  S_0)\,.
 \end{equation}
We then let  $( \mc U_{t}, \mc  S_{t})$ be the solution to truncated equation \eqref{eq:CE} corresponding to such initial data. Let us recall that all the cut-offs introduced in the truncated equation are constantly equal to $1$ for $|\zeta| \leq \frac{6}{5}$. We thus have the following. 

\begin{lemma} \label{lemma:agree} The solution to the truncated equation and the solution to the original equation agree on $[0, \frac65 ]$. 
\end{lemma}
\begin{proof}
Subtracting the two solutions written in terms of their $(U, S)$ variables, we obtain that their difference $(U, S)$ satisfies
\begin{align*}
(\p_s + r - 1) U + (\zeta + \mw U) \p_\zeta U + \alpha \mw S \p_\zeta S + \mt U \p_\zeta \mw U + \alpha S \p_\zeta \mw S &= 0\,, \\
(\p_s + r - 1) S + (\zeta + \mw U) \p_\zeta  S + \alpha \mw S \div ( U) + \mt U \p_\zeta \mw S + \alpha S \div (\mw U) &= 0\,.
\end{align*}
on a ball $B(0, 6/5)$ and with zero initial conditions on that ball. Applying energy estimates, we see that
\begin{align*}
\frac{\p_s}{2} \int_{B(0, 6/5)} \left( | U |^2 + S^2 \right) &\leq -\int_{B(0, 6/5)} (\zeta + \mw U) \frac{\p_\zeta}{2} \left( | U |^2 + S^2 \right) - \int_{B(0, 6/5)} \alpha \mw S \div(S U) + C_1  \int_{B(0, 6/5)} \left(  U^2 + S^2 \right)  \\
&\leq -\int_{\p B(0, 6/5)} \left( \frac{\zeta + \mw U}{2} (| U |^2 + S^2) + \alpha \mw S S U \right)  + C_2  \int_{B(0, 6/5)} \left(  U^2 + S^2 \right) \\
&\leq \int_{\p B(0, 6/5)} \frac{-\zeta - \mw U + \alpha \mw S}{2} (| U |^2 + S^2)  + C_3  \int_{B(0, 6/5)} \left(  U^2 + S^2 \right) \\
&\leq  C_3  \int_{B(0, 6/5)} \left(  |U|^2 + S^2 \right),
\end{align*}
where $C_1, C_2, C_3$ are some absolute constants and we have used Lemma \ref{lemma:DWDZproperties} in the last inequality. In particular, as $U$ and $S$ are zero at time $s_0$ in $B\left( 0, \frac65 \right)$, we conclude that they are zero for all times, and both solutions agree for all times and $\zeta \in \left[ 0, \frac65 \right]$.
\end{proof}

Given a solution  $( \mc U_{t}, \mc  S_{t})$ to \eqref{eq:CE}, we will consider $\| P_{\rm uns} (\mc U_t, \mc S_t ) \|_X$. Our first result is to show that as long as the unstable modes are controlled, we can control the extended solution $(\mathcal U,\mathcal S)$ in a high order weighted Sobolev norm. In particular, we will bound
\begin{equation} \label{eq:weightedhighenery}
E_{2K}(s)^2 = \int_{\mathbb{R}^3} \left( (\Delta^K \mc U (\zeta, s))^2 + (\Delta^K \mc S(\zeta, s))^2 \right) \phi^{2K}(\zeta) d\zeta\,.
\end{equation}
Specifically, in Section \ref{sec:bootstrap}, we will prove the following:
\begin{proposition} \label{prop:bootstrap} Let us take $\delta_0^{3/2} \ll \delta_1 \ll \delta_0$. Let us assume that our initial data satisfies
\begin{equation} \label{eq:initialdataassumptions}
\| (\mt U_{0,t}, \mt S_{0,t}) \|_X \leq \frac{\delta_0}{2}, \qquad \| \mt U_0 \|_{L^\infty}, \| \mt S_0 \|_{L^\infty} \leq \delta_1, \qquad E_{2K}(s_0) \leq \frac{\bar{E}}{2} , \qquad \mt S_0 + \mw S \geq \frac{\delta_1}{2}\,,
\end{equation}
and that letting $\mc U_0 = \mw U + \mt U_0$, $\mc S = \mw S + \mt S_0$, its derivatives satisfy the decay estimates
\begin{equation} \label{eq:extrahypothesis}
 | \nabla  \mc U_0 |+ |  \nabla  \mc S_0 | \les \frac{1}{\langle \zeta \rangle} \qquad\mbox{and}\qquad |  \nabla ^2 \mc U_0 |+ |  \nabla ^2 \mc S_0 | \les \frac{1}{\langle \zeta \rangle^2}\,.
\end{equation}
Moreover, let us assume that our solution is defined for $s \in [s_0, s_1]$ and for every $s \in [s_0, s_1]$ we have 
\begin{equation}
\| P_{\rm uns} (\mt U_t, \mt S_t) (s) \|_X \leq \delta_1\,.\label{eq:kappa:bnd}
\end{equation}

 Then, we have the bounds
\begin{equation} \label{eq:bootstrap_hyp}
E_{2K} < \bar{E} \quad \mbox{and}\quad \| \mt U \|_{L^\infty}, \| \mt S \|_{L^\infty} < \delta_0,
\end{equation}
for all $s \in [s_0, s_1]$. 
\end{proposition}

By local existence and a standard continuation argument, Proposition \ref{prop:bootstrap} implies that the solution $( \mc U, \mc  S)$ is well defined and satisfies the bound \eqref{eq:bootstrap_hyp} so long as the unstable modes satisfy the bound $\| P_{\rm uns} (\mt U_t, \mt S_t) \|_X \leq \delta_1$.

In Section \ref{sec:top}, we will prove, via a standard topological argument, the existence of a choice of $\{a_i\}$ leading to a global bounded converging asymptotically to $(\mw U, \mw S)$.

\begin{proposition}  \label{prop:topological} 
Let us consider $(\mt U_0', \mt S_0')$ smooth and satisfying the initial conditions  
\begin{align} \begin{split} \label{eq:initialconditions_prime}
\| (\mt U_{0, t}', \mt S_{0, t})' \|_X \leq \frac{\delta_0}{4}, \qquad E_{2K} (\mc U_0', \mc S_0'; s_0) \leq \frac{\bar E}{4}, \qquad \| \mt U_0' \|_{L^\infty}, \| \mt S_0' \|_{L^\infty} \leq \frac{9\delta_1}{10} ,\\
\mw S + \mt S_0' \geq \frac{\delta_1}{2}, \qquad | \nabla \mc S_0' | +  | \nabla \mc U_0' | \les \frac{1}{\langle \zeta \rangle} \qquad | \nabla^2 \mc U_0'  | + | \nabla^2 \mc S_0' | \les \frac{1}{ \langle \zeta \rangle^2},
\end{split} \end{align}
and moreover such that $P_{\rm uns} (\mt U_{0, t}', \mt S_{0, t}' ) = 0$, for $\mt U_{0,t}' = \chi_2 \mt U_0'$ and $\mt S_{0, t}' = \chi_2 \mt S_0'$
Then, we have that there exist specific values of $a_i$ such that we have the following: Let $(\mc U_0, \mc S_0)$ be defined by \eqref{eq:US:initial:data} and $(\mc U_{0, t}, \mc S_{0, t})$ defined by \eqref{eq:initialdata_truncated}. Let also $\mt U_0 = \mc U_0 - \mw U$, $\mt S_0 = \mc S_0 - \mw S$, $\mt U_{0, t} = \mc U_{0, t} - \chi_2 \mw U$, $\mt S_{0, t} = \mc S_{0, t} - \chi_2 \mw S$.

Then, the equations \eqref{eq:US:tilde} and \eqref{eq:CE2} can be solved globally for all $s \geq s_0$, and moreover, we obtain smooth solutions that satisfy the estimates:
\begin{equation}
\label{eq:roy}
 \|  (\mt U_{t}, \mt S_{t}) \|_X \lesssim \delta_1 e^{- \frac{7}{10} \delta_g (s-s_0)},
 \end{equation}
 and
\begin{equation}
 \label{eq:ike} \| \mt U_e \|_{L^\infty} + \| \mt S_e \|_{L^\infty} \lesssim \delta_1 e^{- \frac{7}{10} \delta_g (s-s_0)} \,.
 \end{equation}
 where we recall $\mt U_e$ and $\mt S_e$ refers to the whole perturbation (without any cut-off) solving the extended equation \eqref{eq:US:tilde}.
\end{proposition}
Given the inequalities in \eqref{eq:initialconditions_prime}, we may safely assume that the equations \eqref{eq:initialdataassumptions}--\eqref{eq:extrahypothesis} will be satisfied for $\mc U_0, \mc S_0$.

Now, let us see how to conclude Theorem \ref{th:stability} from Proposition \ref{prop:topological}. The proofs of Proposition \ref{prop:bootstrap} and Proposition \ref{prop:topological} will be delayed for now and will constitute the bulk of this section. 

Let us specify the initial data $(\mt U_0', \mt S_0')$. We consider $\zeta_u > 3$ such that $\mw U \leq \frac34 \delta_1$ for $\zeta > \zeta_u$ and let 
\begin{equation*}
\mathfrak B_U = \left\{\zeta \, : \zeta > \zeta_u \right\} \qquad \mbox{ and } \qquad 
\mathfrak B_S = \left\{ \zeta \, : \, \mw S(\zeta) \leq \frac34 \delta_1 \right \}\,.
\end{equation*}
We let $\chi_3 : \mathbb{R}_{\geq 0} \rightarrow \mathbb{R}$ be a smooth cut-off function supported on $[0, 1]$ and equal to $1$ on $[0, 1/2]$, and we let $\lambda > 1$ be a parameter to be fixed. We define the initial data $\mt U_0', \mt S_0'$ as follows:
\begin{equation*}
\mt U_0'(\zeta) = -\mw U (\zeta) \chi_3 \left( \lambda \cdot \mathrm{d}(\zeta, \mathfrak B_U )\right), \qquad \mbox{ and } \qquad \mt S_0' (\zeta) = \left( \frac34 \delta_1 - \mw S (\zeta) \right)\chi_3 \left( \lambda \cdot \mathrm{d}( \zeta, \mathfrak B_S ) \right),
\end{equation*}
where $\mathrm{d}$ is the distance function. That is, we fix $\mt U_0'$ and $\mt S_0'$ to be $-\mw U$ and $\frac34 \delta_1 - \mw S$ in the regions $\mathfrak B_U$ and $\mathfrak B_S$ respectively. Outside that region, our definition gives a smooth extension of $\mt U_0', \mt S_0'$ that guarantees that they are supported in a $\frac{1}{\lambda}$-neighborhood of $\mathfrak B_U$ and $\mathfrak B_S$ respectively.

Let us note that $\mc U_0'$ is zero for $\zeta$ large enough and $\mc S_0'$ is $\frac34 \delta_1$ for $\zeta$ large enough. This clearly follows from the fact that $\mw U$ and $\mw S$ decay (Lemma \ref{lemma:profiledecay}), so for $\zeta$ sufficiently large, we will have $\zeta \in \mathfrak B_U$ and $\zeta \in \mathfrak B_S$. Let us also note that $\mathfrak B_U, \mathfrak B_S \subset (3, \infty)$, because $\zeta_u > 3$ and Lemma \ref{lemma:aux_Slowbound}. Therefore, as $\mt U_0', \mt S_0'$ are supported in $\frac{1}{\lambda}$-neighborhoods of $\mathfrak B_U$ and $\mathfrak B_S$ (and $\lambda > 1$), we have that $\mt U_0', \mt S_0'$ are zero for $\zeta \leq 2$.

Now, all the estimates from \eqref{eq:initialconditions_prime} are trivial. We have $\mt U_{0, t}' = \mt S_{0, t}' = 0$ because $\mt U_0', \mt S_0'$ are zero for $\zeta \leq 2$. As $\mc U_0', \mc S_0'$ are constant for $\zeta$ large enough, we have that the integral defining $E_{2K}(\mc U_0', \mc S_0')^2$ converges, so it is less than $\frac{\bar E^2}{16}$ provided we take $\bar E$ sufficiently large. We clearly have 
\begin{equation*}
| \mt U_0' |, |\mt S_0' | \leq \frac34 \delta_1 \qquad \mbox{ and } \qquad \mw S + \mt S_0' \geq \frac34 \delta_1\,,
\end{equation*}
for $\zeta \in \mathfrak B_U$ and $\mathfrak B_S$ respectively. Therefore, taking $\lambda$ to be large enough, the third and fourth inequalities from \eqref{eq:initialconditions_prime} are satisfied, because $\mt U_0'$ and $\mt S'$ are supported on $\frac{1}{\lambda}$ neighborhoods of $\mathfrak B_U$, $\mathfrak B_S$. The two last inequalities of \eqref{eq:initialconditions_prime} follow directly from the fact that $\mc U_0', \mc S_0'$ are constant for sufficiently large $\zeta$.

Now, we apply Proposition \ref{prop:topological} to $(\mt U_0', \mt S_0')$. Note that as $\psi_{i, U}, \psi_{i, S}$ are supported in $\zeta < 2$ (Remark \ref{rem:eigen_compsup}) we have that 
\begin{equation} \label{eq:micaiah}
\mc U_0(\zeta ) = 0 \qquad \mbox{ and } \qquad \mc S_0 (\zeta) = \frac34 \delta_1 \qquad  \mbox{ for } \zeta \mbox{ large enough.}
\end{equation}
Moreover, Proposition \ref{prop:topological} gives us a global solution $(\mt U, \mt S)$ to \eqref{eq:US:tilde}, which taking $\mc U = \mt U + \mw U$ and $\mc S = \mt U + \mw U$ yields a solution $(\mc U, \mc S)$ to \eqref{eq:US:system}. Undoing the self-similar change of variables by taking 
\begin{equation} \label{eq:soren}
u(R, t) = \frac{1}{r} e^{s(r-1)} \mc U (\zeta, s ), \qquad \sigma (R, t) = \frac{1}{r} \cdot e^{s(r-1)} \mc S (\zeta, s )\,,
\end{equation}
where
 \begin{equation*}
R = \zeta e^{-s}, \qquad e^{-sr} = T - t = e^{-s_0 r} - t\,,
\end{equation*}
we obtain that $u(R, t), \sigma (R, t)$ for $t \in [0, T)$ satisfy equation \eqref{eq:ice}. We may recover $\rho$ from taking $\rho = (\alpha \sigma)^{\frac{1}{\alpha}}$ and then $(u, \rho)$ satisfies \eqref{eq:wombat}. It is clear from \eqref{eq:micaiah} and the changes performed that $u(R, 0)$ will be zero for $R$ large enough and $\rho (R, 0)$ will be constant for $R$ large enough (let us denote that constant by $\rho_c$). Then, the items \ref{item:Hagoromo1} and \ref{item:Hagoromo2} of Theorem \ref{th:stability} are satisfied.

Moreover, we see from Proposition \ref{prop:topological} and \eqref{eq:soren} that
\begin{equation*}
\lim_{t \to T}r(T-t)^{1-1/r} \sigma(0, t) =  \mw S( 0 ), 
\end{equation*}
so we see that $\sigma(0, t)$ tends to $+\infty$ since $\mw S(0) > 0$ as a consequence of Lemma \ref{lemma:aux_Slowbound}. This implies the $\rho$ limit stated in item \ref{item:Hagoromo3} of Theorem \ref{th:stability}.

In addition, we know from Proposition \ref{prop:topological} that $\mc U (1, s) \rightarrow \mw U (1)$ and $\mc S (\zeta, s) \rightarrow \mw S (\zeta )$ as $s \rightarrow \infty$ which implies item \ref{item:Hagoromo4} of Theorem \ref{th:stability}.

Finally, we show the $u$ limit  stated in item \ref{item:Hagoromo3} of Theorem \ref{th:stability}. First note that
\[\bar U(1)=\frac{W_0+Z_0}{2}=\frac{D_{W,0}+D_{Z,0}}2-1=\frac{D_{W,0}}2-1\,.\]
Then, assuming $r$ is sufficiently close to $r^\ast$, we obtain from Lemma \ref{lemma:aux_limits} that $\bar U(1)\neq 0$. Since $\mc U (\zeta, s) \rightarrow \mw U (\zeta)$, we obtain the $u$ limit  stated in item \ref{item:Hagoromo3} of Theorem \ref{th:stability}.

\begin{remark}\label{rem:Manifold} Observe that our construction of $(\mt U_0', \mt S_0')$ allows for small perturbations in all the norms considered (the norms appearing in \eqref{eq:initialconditions_prime}). The functions $(\mathcal U_0, \mathcal S_0)$ can be defined as before in \eqref{eq:US:initial:data} up to small perturbations in the coefficients $a_i$. In particular, the conclusion of Proposition \ref{prop:topological} holds for a finite codimension manifold of radial initial data.
\end{remark}

\subsection{Proof of Proposition \ref{prop:bootstrap}}\label{sec:bootstrap}

We will prove Proposition \ref{prop:bootstrap} via a bootstrap argument. Thus, we will assume equation \eqref{eq:initialdataassumptions}--\eqref{eq:bootstrap_hyp} hold for $s \in [s_0, s_1]$ and show an improvement on \eqref{eq:bootstrap_hyp}, specifically
\begin{align}
\| \mt U (\cdot , s)\|_{L^\infty}, \| \mt S (\cdot, s)\|_{L^\infty} &\leq \frac{1}{2} \delta_0, \label{eq:improvement_Linf} \\
E_{2K}(s) &\leq \frac{1}{2} \bar E, \label{eq:improvement_E}
\end{align}
for all $s \in [s_0, s_1]$. Showing the improved bounds \eqref{eq:improvement_Linf} and \eqref{eq:improvement_E} would clearly conclude the proof of Proposition \ref{prop:bootstrap} because $E_{2K}(s)$, $\| \mt U(\cdot, s) \|_{L^\infty}$ and $\| \mt S (\cdot, s) \|_{L^\infty}$ are continuous with respect to $s$.

From now on, and for the rest of this subsection, we will always assume that \eqref{eq:initialdataassumptions}--\eqref{eq:bootstrap_hyp} hold and that $s \in [s_0, s_1]$. In order to show equations \eqref{eq:improvement_Linf} and \eqref{eq:improvement_E}, we divide the proof in three steps. First, we will derive a series of consequences of the assumptions \eqref{eq:initialdataassumptions}--\eqref{eq:bootstrap_hyp}. Secondly, we will show \eqref{eq:improvement_Linf}, and thirdly, we will show \eqref{eq:improvement_E}. This subsection is organized in three different parts according to those three steps.

Before doing any of those steps, let us introduce some definitions. Due to Lemma \ref{lemma:profiledecay}, and recalling that $\delta_0^{3/2} \ll \delta_1 \ll \delta_0 \ll 1$, we know that there exists a value of $\zeta_0$ such that
\begin{equation}\label{eq:Novak:Variant}
\mw S (\zeta) \geq 2 \delta_0\quad\mbox{for all }\zeta \leq \zeta_0\quad\mbox{and}\quad
|\nabla \mw S (\zeta)|, |\nabla \mw U(\zeta) | \leq \frac{\delta_1}2 \quad\mbox{for all }\zeta \geq \zeta_0\,.
\end{equation}

In particular, from $\| \mt S \|_{L^\infty} \leq \delta_0$ in \eqref{eq:bootstrap_hyp}, we have that
\begin{equation} \label{eq:Sbiginside}
\mc S (\zeta) \geq \delta_0 \quad \mbox{for all} \; \zeta \leq \zeta_0\,,
\end{equation}

Let us also define the weight $\phi(\zeta)$ that we will use for the energy. We fix
\begin{equation*} 
\phi(\zeta ) = \begin{cases}
 1 & \mbox{ for } \zeta \leq \zeta_0\,, \\
\frac{ \zeta^{2(1-\eta_w)} }{ 2 \zeta_0^{2(1-\eta_w)} }&\mbox{ for }\zeta \geq 4 \zeta_0\,,
\end{cases}
\end{equation*}
and choose $\phi(\zeta)$ in the region $\zeta_0 \leq \zeta \leq 4 \zeta_0$ so that it is smooth and 
\begin{equation} \label{eq:phiderbound}
\frac{| \nabla \phi | \zeta}{\phi} \leq 2(1-\eta_w) \qquad \mbox{ and } \qquad \phi (\zeta ) \geq 1\,,
\end{equation} 
hold globally.

\subsubsection{Consequences of the bootstrap}

Let us stress once again that for all the results in this subsection, we are implicitly assuming that \eqref{eq:initialdataassumptions}--\eqref{eq:bootstrap_hyp} hold and that $s \in [s_0, s_1]$.

\begin{lemma} \label{lemma:mozart} We have the following inequalities for the $2K-1$ derivatives:
\begin{equation} \label{eq:cafeconleche}
\phi^{(2K-1)/2} |\nabla^{2K-1} \mc U| + \phi^{(2K-1)/2} |\nabla^{2K-1} \mc S |  \lesssim \frac{\bar{E} }{\zeta^{1/2} \phi^{1/2}}\,.
\end{equation}
Moreover, for $0 \leq j \leq 2K-2$ and $\zeta > \zeta_0$, we have:
\begin{align} \begin{split} \label{eq:cafesolo}
\phi^{j/2} |\nabla ^{j} \mc U|+\phi^{j/2} |\nabla^{j} \mc S | &\lesssim \delta_0^{(2K-1-j)/(2K-1)} \left( \frac{\bar{E} }{\zeta^{1/2} \phi^{1/2}} \right)^{j/(2K-1)}  \,.
\end{split} \end{align}
Finally, for $0 \leq j \leq 2K-2$ we also have the global inequality
\begin{equation} \label{eq:capuccino}
\| \phi^{j/2} \nabla^j\mc S \|_{L^\infty}+ \| \phi^{j/2} \nabla^j\mt S  \|_{L^\infty} + \| \phi^{j/2} \nabla^j\mc U  \|_{L^\infty}+ \| \phi^{j/2} \nabla^j\mt U  \|_{L^\infty} \lesssim \left( \zeta_0^{1/2} \bar E \right)^{\frac{j}{2K-2}}\,.
\end{equation}

\end{lemma}

\begin{proof}
We have the following bound on any $2K-1$-th derivative of $\mc S$
\begin{align} \label{eq:asdfg}
| \p^{2K-1} \mc S (\zeta) | &\leq  \int_{\zeta_1}^\infty \abs{\p^{2K}\mc S} dz \notag\\
&\leq \left( \int_{\zeta}^\infty \abs{\p^{2K}\mc S} ^2 z^2 \phi(z)^{2K} dz \right)^{1/2} \left( \int_{\zeta}^\infty \frac{1}{ z^2 \phi (z)^{2K} } dz \right)^{1/2} \notag\\
&\les \bar{E} \left( \int_{\zeta}^\infty \frac{1}{z^2 \phi (z)^{2K} } dz \right)^{1/2} \lesssim \bar{E} \left( \frac{1}{\zeta \phi (\zeta)^{2K}} \right)^{1/2}
\les \bar{E} \phi^{-K}\zeta^{-\frac12}\,.
\end{align}
This yields the  estimate on $\nabla^{2K-1} \mc S$ implied by \eqref{eq:cafeconleche}. 

Now, for the region $\zeta > \zeta_0$ we have that $|\mw S (\zeta)| \les \delta_0$, and by \eqref{eq:bootstrap_hyp}, we also have that $| \mt S (\zeta )| \leq \delta_0$. Therefore, $| \mc S (\zeta )| \les \delta_0$. Using interpolation (Lemma \ref{lemma:thefinalinterpolation}) between $| \mc S (\zeta )| \les \delta_0$ and $\phi^K | \nabla^{2K-1} \mc S | \frac{\zeta^{1/2}}{\bar E} \les 1$ in the region $[\zeta_0, +\infty)$, we conclude the estimate on $\nabla^{j}\mc S$ implied by \eqref{eq:cafesolo}.

Integrating \eqref{eq:asdfg}, we obtain
\begin{equation}
\abs{\p^{2K-2} \mc S (\zeta)}\les \bar{E}\frac{\zeta_0^{1/2}}{\phi (\zeta_0)^K } = \bar E \zeta_0^{1/2}  \, \label{eq:possum1}\,,
\end{equation}
which shows \eqref{eq:capuccino} for $\mc S$ and $j = 2K-2$. Standard $L^\infty$ interpolation (Gagliardo-Nirenberg) yields
\begin{equation} \label{eq:alohomora}
\| \phi^{j/2}  \nabla^j \mc S  \|_{L^\infty (B(0, \zeta_0))} \les \left( \zeta_0^{1/2} \bar E \right)^{\frac{j}{2K-2}},
\end{equation}
using that in $B(0, \zeta_0)$ we have that $\phi = 1$. For the region $\zeta > \zeta_0$, note that weighted interpolation (Lemma \ref{lemma:thefinalinterpolation}) between $\| \mc S \|_{L^\infty} \les 1$ and \eqref{eq:possum1} yields 
\begin{equation*}
\| \phi^{j/2}  \nabla^j \mc S \|_{L^\infty (B(0, \zeta_0)^c) }\les \left( \zeta_0^{1/2} \bar E \right)^{\frac{j}{2K-2}},
\end{equation*}
which together with \eqref{eq:alohomora} shows \eqref{eq:capuccino} for $\mc S$.

In order to obtain the bound for $\mt S$, recall that $\mc S = \mt S + \mw S$ and then note that $\| \phi^{j/2}  \nabla^j \mw S  \|_{L^\infty} \lesssim 1$ due to Lemma \ref{lemma:profiledecay}. Therefore, we conclude the desired bound also for $\mt S$. The bounds for $\mc U$ and $\mt U$ are proven in the same way as we did with $\mc S, \mt S$. 
\end{proof}

\begin{lemma} \label{lemma:suecia}  We have that
\begin{equation*}
\mc S \geq \frac{\delta_1}{4} \left\langle \frac{\zeta}{\zeta_0} \right\rangle^{-(r-1)}\,.
\end{equation*}
\end{lemma}
\begin{proof} 
The statement clearly holds for $\zeta < \zeta_0$ from \eqref{eq:Sbiginside}. Thus, let us work on the region $\mc O = [\zeta_0, \infty)$.
Let us recall that $\mc S$ solves
\begin{equation} \label{eq:voltaire}
(\p_s + r - 1) \mc S + y \cdot \nabla \mc S + \mc U \cdot \nabla \mc S + \alpha \mc S \div (\mc U) = 0\,.
\end{equation}
By \eqref{eq:cafesolo}, we have that 
\begin{equation} \label{eq:spinoza}
\begin{split}
\| \mc U\|_{L^\infty (\mc O )} +\| \mc S\|_{L^\infty (\mc O )}&\lesssim \delta_0 , \\
\| \phi^{1/2}  \nabla \mc U  \|_{L^\infty (\mc O )} +  \| \phi^{1/2}  \nabla  \mc S  \|_{L^\infty (\mc O )}& \lesssim \delta_0^{(2K-2)/(2K-1)} \bar E^{1/(2K-1)} \les \delta_0^{9/10} \bar{E}^{1/10} \ll \delta_0^{4/5}\,.
\end{split}
\end{equation}
Let us also define $\omega_{\mw \zeta, \mw s} (s) = \left( \mw \zeta e^{(s-\mw s)} \right)^{r-1} \mc S (\mw \zeta e^{s-\mw s}, s )$. Then, using \eqref{eq:spinoza} in \eqref{eq:voltaire}, we obtain that
\begin{equation} \label{eq:rousseau}
\left| \p_s \omega_{\mw \zeta, \mw s} (s) \right| \les  \delta_0^{9/5} \phi \left( \mw \zeta e^{s-\mw s} \right)^{-1/2} \left( \mw \zeta e^{(s-\mw s)} \right)^{r-1} 
\les \delta_0^{9/5} \phi \left( \mw \zeta e^{s-\mw s} \right)^{\frac{-1}{2} + \frac{r-1}{2-2\eta_w}} \zeta_0^{r-1}\,.
\end{equation}
Now, note that $\frac{-1}{2} + \frac{r-1}{2} <- \frac1{10}$ because $r < r^\ast (\gamma ) < 3-\sqrt{3} $, hence
\[{\phi \left( \mw \zeta e^{s-\mw s} \right)^{\frac{-1}{2} + \frac{r-1}{2-2\eta_w}} \les \left( \frac{\mw \zeta e^{s-\mw s} }{ \zeta_0 } \right)^{-\frac{2}{10}(1-\eta_w) }  \leq e^{ \frac{-1}{10} (s-\bar s)}}\]
where we used that $\eta_w$ is sufficiently small and assumed $\mw \zeta \geq \zeta_0$. Thus, we obtain
\begin{equation} \label{eq:rousseau_2}
\frac{1}{\zeta_0^{r-1}} \left|  \p_s \omega_{\mw \zeta, \mw s} (s) \right| \les \delta_0^{9/5}e^{ \frac{-1}{10}(s-\bar s)}\,,
\end{equation}
for any $s \geq \mw s$ with $\mw s, s \in [s_0, s_1]$ and any $\mw \zeta \geq \zeta_0$. 

Integrating \eqref{eq:rousseau_2}, we obtain
\begin{equation} \label{eq:montesquieu}
\frac{1}{\zeta_0^{r-1}} \left| \omega_{\mw \zeta, \mw s} (s) - \mw \zeta^{r-1} \mc S (\mw \zeta, \mw s) \right| \lesssim \delta_0^{9/5} \ll \delta_1 \delta_0^{1/5}  .
\end{equation}
Now, for any $\zeta \in \mc O$ and $s \in [s_0, s_1]$ there exists $\mw \zeta \geq \zeta_0, \mw s \in [s_0, s_1]$ such that $(\mw \zeta, \mw s) \in \{ \zeta_0 \} \times [s_0, s_1] \cup [\zeta_0, \infty) \times \{ s_0 \}$ and $\mw \zeta e^{s-\mw s} = \zeta$. Fixing such conditions for $\mw \zeta, \mw s$, we have
\begin{equation} \label{eq:escohotado}
\mc S(\mw \zeta, \mw s) \geq \frac{\delta_1}{2}.
\end{equation}
This is due to \eqref{eq:Sbiginside} for $\mw \zeta = \zeta_0$ and due to \eqref{eq:initialdataassumptions} for $\mw s = s_0$. From \eqref{eq:montesquieu}--\eqref{eq:escohotado}, we conclude
\begin{align*}
\frac{\omega_{\mw \zeta, \mw s} (s)}{\zeta_0^{r-1}} \geq \frac12 \frac{ \mw \zeta^{r-1} }{\zeta_0^{r-1}} \delta_1 - \delta_0^{1/5} \delta_1 \geq \frac14 \frac{ \mw \zeta^{r-1} }{\zeta_0^{r-1} } \delta_1 \geq \frac14 \delta_1 .
\end{align*}
Recalling that $\zeta = \mw \zeta e^{s-\mw s}$, we get that
\begin{equation*}
\frac{\zeta^{r-1}}{\zeta_0^{r-1}} \mc S (\zeta, s) \geq \frac14 \delta_1\,,
\end{equation*}
for $\zeta \geq \zeta_0$, and this completes the proof.
\end{proof}

\begin{lemma} \label{lemma:noruega} Assume $\zeta > \zeta_0$. Then, we have that
\begin{equation} \label{eq:noruega}
 | \nabla \mc S (\zeta) | \les_{\delta_1} \frac{1}{\zeta^r}\,.
\end{equation}
where here we use the notation $\les_{\delta_1}$ to imply that the implicit constant in the inequality  may depend on $\delta_1$.
\end{lemma}
\begin{proof} From \eqref{eq:US:system}, we have that
\begin{equation*}
\left( \p_s + r \right) \nabla \mc S + \zeta \nabla \p_\zeta \mc S + 
 \nabla \left( \mc U \cdot \nabla  \mc S \right) + \alpha \nabla \left( \div(\mc U) \mc S \right) = 0\,.
\end{equation*}
Now, let us define $\omega_{\mw \zeta, \mw s} (s) =  \mw \zeta^r e^{(s-\mw s)r} \nabla \mc S ( \mw \zeta e^{s-\mw s}, s)$. With this definition and using \eqref{eq:capuccino}, we have that 
\begin{equation} \label{eq:donato}
\abs{\p_s \omega_{\mw \zeta, \mw s} (s)}\les_{\delta_1} \left( \frac{ 1 }{ \phi \left( \mw \zeta e^{s-\mw s} \right) } \left( \mw \zeta e^{s-\mw s} \right)^r \right) \les_{\delta_1} \left( \mw \zeta e^{s-\mw s} \right)^{-2+2\eta_w+r}.
\end{equation}
Now, we assume that either $\mw \zeta = \zeta_0$ or $\mw s = s_0$. Therefore
\begin{equation} \label{eq:robertocarlos}
\omega_{\mw \zeta, \mw s} (\mw s) = \mw \zeta^r \mc S (\mw \zeta, \mw s )  \lesssim 1\,.
\end{equation}
Using both \eqref{eq:donato} and \eqref{eq:robertocarlos} we obtain that 
\begin{equation*}
\omega_{\mw \zeta, \mw s} (s) \les_{\delta_1} 1 + \int_{\mw s}^\infty e^{(-2+2\eta_w+r)(s-\mw s)} \les 1\,,
\end{equation*}
and this shows our estimate for any $(\zeta, s) = (\mw \zeta e^{s-\mw s}, s)$ such that $s \geq \mw s$ and either $\mw \zeta =  \zeta_0$ or $\mw s = s_0$. As any $(\zeta, s)$ with $\zeta \geq \zeta_0$, $s \geq s_0$ can be written in that way, this finishes our proof. 
\end{proof}

\begin{lemma} \label{lemma:derbounds} We have that
\begin{align} \begin{split} \label{eq:derbounds}
\phi^{j/2} \left| \nabla^j \left( \frac{1}{\mc S^{1/\alpha}} \right) \right| &\lesssim_{\delta_1}  \langle \zeta \rangle^{(r-1)/\alpha} \left( \langle \zeta \rangle^{r-1/2} \phi^{-1} \right)^{j/(2K-2)} \,,
\end{split} \end{align}
for any $0 \leq j \leq 2K-1$.
\end{lemma}
\begin{proof} First, let us note that
\begin{equation} \label{eq:finlandia}
\left|\nabla^j \left( \frac{1}{\mc S^{1/\alpha}} \right) \right| \lesssim_j \frac{1}{\mc S^{1/\alpha }}\sum_{j_1 + \ldots + j_\ell = j} \frac{| \nabla^{j_1} \mc S |}{\mc S } \cdot \ldots \cdot \frac{| \nabla^{j_1} \mc S |}{\mc S }\,.
\end{equation}
Equation \eqref{eq:derbounds} follows clearly for the case $\zeta \leq \zeta_0$, by uniform bounds on $\phi$ in this region and the inequality  $\frac{1}{\mc S^{1/\alpha + j}} \leq \delta_1^{-j-1/\alpha} \les_{\delta_1} 1$ implied by equation \eqref{eq:Sbiginside}. Therefore, let us assume from now that $\zeta \geq \zeta_0$.

Now, we want to analyze each factor in \eqref{eq:finlandia} for $\zeta \geq \zeta_0$. We first claim that for any $1 \leq j \leq 2K-1$, we have
\begin{equation} \label{eq:oslo}
\left| \phi (\zeta)^{j/2}  \nabla^j \mc S (\zeta) \right| \zeta^{r-1} \left( \zeta^{1/2-r} \right)^{\frac{j-1}{2K-2}}  \phi(\zeta)^{\frac{j}{2K-2}}\les_{\delta_1} 1 \,.
\end{equation}
In order to see this, we apply interpolation (as in Lemma \ref{lemma:thefinalinterpolation}) between \eqref{eq:cafeconleche} and \eqref{eq:noruega}, obtaining 
\begin{align*}
1 &\gtrsim_{\delta_1} \left|  \nabla^j \mc S   \zeta^{\left( 1-\frac{j-1}{2K-2}\right) r} \left( \phi(\zeta)^{K} \zeta^{1/2} \right)^{\frac{j-1}{2K-2}}\right| =  \left|  \nabla^j \mc S \zeta^r \phi(\zeta)^{j/2} \phi(\zeta)^{\frac{j-K}{2K-2}} \left(  \zeta^{1/2-r} \right)^{\frac{j-1}{2K-2}}\right|\,.
\end{align*}
Now, as $\frac{1}{K} \ll \eta_w \ll 1$, observe that $\zeta \phi(\zeta)^{-\frac{K}{2K-2}} \gtrsim_{\delta_1} \zeta^{1-(1-\eta_w) \left( 1+\frac{1}{K-1} \right) } \gtrsim 1$. This concludes \eqref{eq:oslo}.

Now using \eqref{eq:oslo} and the lower bound from Lemma \ref{lemma:suecia} we get that
\begin{align*}
\left| \frac{ \nabla^j \mc S }{\mc S} \right| 
&\les_{\delta_1}  \frac{1}{\mc S}\phi (\zeta)^{-j/2} \zeta^{-(r-1)} \left( \zeta^{r-1/2} \right)^{-\frac{j-1}{2K-2}} \phi (\zeta)^{- \frac{j}{2K-2} } \\
& \leq   \phi(\zeta)^{-j/2} \left( \zeta^{r-1/2} \right)^{-\frac{j-1}{2K-2}} \phi(\zeta)^{- \frac{j}{2K-2} } \,.
\end{align*}
Plugging this into \eqref{eq:finlandia}, we have that
\begin{equation*}
\left| \nabla^j \left( \frac{1}{\mc S^{1/\alpha}} \right) \right| \lesssim_{\delta_1} \frac{1}{\mc S^{1/\alpha}}\phi(\zeta)^{-j/2} \left( \zeta^{r-1/2} \phi(\zeta)^{-1} \right)^{j/(2K-2)}\,.
\end{equation*}
Using again Lemma \ref{lemma:suecia} to bound $\frac{1}{\mc S^{1/\alpha}}$, we obtain the desired result.
\end{proof}

\begin{lemma} \label{lemma:nipon} There exist values $C_1, C$ independent of all the other parameters such that
\begin{equation} \label{eq:niponsegundo}
 | \nabla \mc S | + | \nabla  \mc U | \leq \frac{C}{\zeta},
\end{equation}
and
\begin{equation} \label{eq:nipon}
  |\nabla^2 \mc S | + |\nabla^2 \mc U | \leq \frac{C}{\zeta^2}\,,
\end{equation}
for every $\zeta > C_1$.
\end{lemma}
\begin{proof} 
From \eqref{eq:US:system}, letting $W = \mc U + \mc S$ and $Z = \mc U - \mc S$, we have
\begin{align}  \begin{split} \label{eq:sanchez0}
(\partial_s +r-1)  W + (\zeta + \mc U + \alpha \mc S ) \p_\zeta  W &= \mc F_{\rm dis} - \frac{2\alpha \mc S}{\zeta } \mc U \,, \\
(\partial_s +r-1) Z + (\zeta  + \mc U - \alpha \mc S) \p_\zeta Z &= \mc F_{\rm dis} + \frac{2\alpha \mc S}{\zeta } \mc U\,.
\end{split} \end{align}
Denote $E_W = (\mc U + \alpha \mc S)$ and $E_Z = (\mc U - \alpha \mc S)$. Taking one derivative in \eqref{eq:sanchez0} we obtain
\begin{align}  \begin{split} \label{eq:sanchez1}
(\partial_s +r)  \p_\zeta W + (\zeta + E_W) \p_\zeta^2  W &= \p_\zeta \mc F_{\rm dis} - 2\alpha \p_\zeta\left( \frac{ \mc S \mc U}{\zeta } \right) - \p_\zeta E_W \p_\zeta W \,, \\
(\partial_s +r) \p_\zeta Z + (\zeta + E_Z) \p_\zeta^2 Z &=  \p_\zeta \mc F_{\rm dis} + 2 \alpha \p_\zeta \left( \frac{ \mc S \mc U}{\zeta }  \right)  - \p_\zeta E_Z \p_\zeta Z\,,
\end{split} \end{align}
and taking two derivatives in \eqref{eq:sanchez0}
\begin{align}  \begin{split} \label{eq:sanchez2}
(\partial_s +r+1)  \p_\zeta^2 W + (\zeta+ E_W) \p_\zeta^3  W &= \p_\zeta^2 \mc F_{\rm dis} - 2\alpha \p_\zeta^2 \left( \frac{ \mc S \mc U}{\zeta } \right) - 2\p_\zeta E_W \p_\zeta^2 W - \p_\zeta^2 E_W \p_\zeta W\,,  \\
(\partial_s +r+1)\p_\zeta^2 Z + (\zeta + E_Z) \p_\zeta^3 Z &=  \p_\zeta^2 \mc F_{\rm dis} + 2 \alpha  \p_\zeta^2 \left( \frac{ \mc S \mc U}{\zeta }  \right)  - 2 \p_\zeta E_Z \p_\zeta^2 Z - \p_\zeta^2 E_Z \p_\zeta Z\,.
\end{split} \end{align}

Now, we claim
\begin{equation}  \label{eq:rajoy1} 
\abs{\p_\zeta \mc F_{\rm dis}} \leq  \frac{1}{\zeta}, \qquad \mbox{ and } \qquad \abs{\p_\zeta^2 \mc F_{\rm dis} } \leq \frac{1}{\zeta^2}.
\end{equation}
We have
\begin{align} \label{eq:rajoy2} \begin{split}
\abs{ \p_\zeta \mc F_{\rm dis}} &\les e^{-\delta_{\rm dis}s_0}  \left( \frac{| \nabla^3 \mc U | }{\mc S^{1/\alpha}} + \frac{| \nabla^2 \mc U | | \nabla \mc S |}{\mc S^{1+1/\alpha}} \right)  \\
\abs{ \p_\zeta^2 \mc F_{\rm dis}} &\les e^{-\delta_{\rm dis}s_0}  \left( \frac{| \nabla^4 \mc U |}{\mc S^{1/\alpha}} + \frac{ | \nabla^3 \mc U | | \nabla \mc S | }{\mc S^{1+1/\alpha}} + \frac{| \nabla^2 \mc U | | \nabla^2 \mc S |}{\mc S^{1+1/\alpha}} + \frac{| \nabla^2 \mc U | | \nabla \mc S |^2}{\mc S^{2+1/\alpha}} \right)\,.
\end{split} \end{align}
Using \eqref{eq:Sbiginside} and \eqref{eq:capuccino} in \eqref{eq:rajoy2}, we see that for $\zeta < \zeta_0$, 
\begin{align*}
\abs{ \p_\zeta \mc F_{\rm dis}} &\les e^{-\delta_{\rm dis}s_0}  \frac{(\zeta_0^{1/2} \bar E)^{\frac{3}{2K-2}}}{\delta_0^{1+1/\alpha}} \ll e^{-\delta_{\rm dis} s_0/2} \ll \frac{1}{\zeta_0}\,, \\
\abs{ \p_\zeta^2 \mc F_{\rm dis}} &\les e^{-\delta_{\rm dis}s_0}  \frac{(\zeta_0^{1/2} \bar E)^{\frac{4}{2K-2}}}{\delta_0^{2+1/\alpha}} \ll e^{-\delta_{\rm dis} s_0/2} \ll \frac{1}{\zeta_0^2}\,,
\end{align*}
so \eqref{eq:rajoy1} holds for the region $\zeta < \zeta_0$. For the region $\zeta > \zeta_0$, from \eqref{eq:cafesolo} we have
\[\phi^{j/2} |\nabla ^{j} \mc U|+\phi^{j/2} |\nabla^{j} \mc S | \lesssim 1 \,.
\]
where we have used that $\zeta_0$ is sufficiently large, dependent on $\bar E$. Thus, from Lemma \ref{lemma:suecia} and  \eqref{eq:rajoy2}, we see that
\begin{align*}
\abs{ \p_\zeta \mc F_{\rm dis}} &\les_{\delta_1} e^{-\delta_{\rm dis}s_0} \frac{ \phi^{-3/2} }{\zeta^{-(r-1)\left( 1 + 1/\alpha \right)}} \les_{\delta_1}  e^{-\delta_{\rm dis}s_0} \zeta^{-3+4\eta_w+ (r-1) \left( 1 + \frac{1}{\alpha} \right)}\,, \\
\abs{ \p_\zeta^2 \mc F_{\rm dis}} &\les_{\delta_1} e^{-\delta_{\rm dis}s_0} \frac{ \phi^{-2} }{\zeta^{-(r-1)\left( 2 + 1/\alpha \right)}} \les_{\delta_1}  e^{-\delta_{\rm dis}s_0} \zeta^{-4+4\eta_w+(r-1) \left( 2 + \frac{1}{\alpha} \right)}\,.
\end{align*}
Thus, using that $(r-1) \left( 2 + \frac1\alpha \right)  < 2$ (due to Lemma \ref{lemma:connecticut}), $\eta_w$ is sufficiently small and that $e^{-\delta_{\rm dis} s_0}$ is sufficiently small depending on $\delta_i$, we conclude \eqref{eq:rajoy1} for $\zeta > \zeta_0$.

Using $L^\infty$-interpolation between $| \mt U |, | \mt S | \leq \delta_0$ and \eqref{eq:capuccino}, we have that for $i \in \{0, 1, 2\}$
\begin{align*}
|\p_\zeta^i \mc U| + |\p_\zeta^i  \mc S | &\les |\p_\zeta^i  \mw U | + |\p_\zeta^i  \mw S | + |\p_\zeta^i  \mt U | + |\p_\zeta^i  \mt S | \les |\p_\zeta^i  \mw U | + |\p_\zeta^i  \mw S | + \delta_0^{1/2}.
\end{align*}
In particular, using Lemma \ref{lemma:profiledecay} this shows our statement for $C_1 < \zeta < \delta_0^{-1/2}$, so from now on we will take $C_1' = \delta_0^{-1/2}$ and we can assume $\zeta > C_1'$. In particular, we have:
\begin{equation} \label{eq:zapatero1}
| \p_\zeta^i \mc U | + | \p_\zeta^i \mc S | + | \p_\zeta^i E_W | + | \p_\zeta^i E_Z | \leq \delta_0^{\frac12} + \delta_0^{ \frac{r-1}{2} } \ll 1,  \qquad \forall i \in \{0, 1, 2\}.
\end{equation}

Using the estimates
\begin{equation*}
2 \alpha \left| \p_\zeta \left( \frac{\mc U \mc S}{\zeta} \right) \right|
\leq 2\alpha | \mc U \mc S | + 2\alpha | \p_\zeta \mc U \mc S | + 2 \alpha | \mc U \p_\zeta \mc S | \leq \frac{1}{\zeta},
\end{equation*}
\eqref{eq:rajoy1} and \eqref{eq:zapatero1}, then from \eqref{eq:sanchez1}, we obtain
\begin{align} \begin{split} \label{eq:bonaparte}
\left| \left( \partial_s +r - 1 - \frac{E_W}{\zeta} \right) \zeta  \p_\zeta W + (\zeta + E_W) \p_\zeta (\zeta \p_\zeta  W ) \right| &\leq 2 + \frac{r-1}{10} \zeta \p_\zeta W\,,  \\
 \left| \left( \partial_s +r - 1 - \frac{E_Z}{\zeta} \right) \zeta \p_\zeta Z + (\zeta + E_Z) \p_\zeta (\zeta \p_\zeta Z ) \right| &\leq 2 + \frac{r-1}{10} \zeta \p_\zeta Z \,.
\end{split} \end{align}

Now, let us define the trajectories
\begin{equation} \label{eq:saez}
\Upsilon_W'(s) = \left(\zeta + E_W \right) =  \zeta + \mc U + \alpha \mc S\,, \qquad 
\Upsilon_Z'(s) = \left(\zeta + E_Z \right) = \zeta + \mc U - \alpha \mc S\,.
\end{equation}
starting at $(\bar \zeta, \bar s)$ such that either $\mw \zeta = C_1'$ or $\mw s = s_0$. Using \eqref{eq:zapatero1} and noting that $\delta_0$ is sufficiently small depending on $\eta_w$, we can assume that $\Upsilon_W, \Upsilon_Z$ are increasing and have derivatives between $\frac{9}{10}\zeta$ and $\frac{11}{10}\zeta$. Thus
\begin{equation}  
\frac{9}{10} \zeta \leq \Upsilon_W', \Upsilon_Z' \leq \frac{11}{10}\zeta 
\qquad \mbox{ and } \qquad 
\mw \zeta e^{9(s-\mw s)/10} \leq \Upsilon_W , \Upsilon_Z \leq \mw \zeta e^{11(s-\mw s)/10}\,.\label{eq:paesa2}
\end{equation}
Let us also define
\begin{equation} \label{eq:montoro}
\Theta^{(W)} =  \Upsilon_W \cdot \p_\zeta W \circ \Upsilon_W \quad\mbox{and}\quad  \Theta^{(Z)} = \Upsilon_Z \cdot \p_\zeta Z \circ \Upsilon_Z\,.
\end{equation}
Using definitions \eqref{eq:saez} and \eqref{eq:montoro} in \eqref{eq:bonaparte}, together with \eqref{eq:zapatero1}, we get
\begin{align}  \begin{split} \label{eq:gonzalez1}
\left| \partial_s  \Theta^{(\circ)} + (r-1) \Theta^{(\circ)} \right| &\leq 2  + 2\frac{r-1}{10} |\Theta^{(\circ)}| \,,
\end{split} \end{align}
for $\circ \in \{ W, Z\}$. 

Now, we claim that there exists some $C_2$, sufficiently large and independent of all the other parameters, such that
\begin{equation} \label{eq:suarez1} 
\abs{ \Theta^{(W)} (\mw s) } = \abs{ \mw \zeta \p_\zeta W (\mw \zeta, \mw s) } \leq C_2 \quad\mbox{and}\quad
\abs{ \Theta^{(Z)} (\mw s) } = \abs{ \mw \zeta \p_\zeta Z (\mw \zeta, \mw s) } \leq C_2\,.
\end{equation}
This is a consequence of
\begin{equation*}
\mw \zeta \left( | \nabla  \mc U |(\mw \zeta, \mw s) + | \nabla \mc S | (\mw \zeta, \mw s) \right) \les 1\,,
\end{equation*}
which is trivial for $\mw \zeta = C_1$ and follows from \eqref{eq:extrahypothesis} for $\mw s = s_0$. 

Taking $C_3 = C_2 + \frac{3}{r-1}$, we have that $\abs{ \Theta^{(W)} }, \abs{ \Theta^{(S)} } \leq C_3$. This is clearly true at $s = \mw s$ due to \eqref{eq:suarez1} and the inequality cannot break due to \eqref{eq:gonzalez1}. Therefore, we get that 
\begin{equation*}
\left| \p_\zeta W \circ \Upsilon_W \right| \leq \frac{C_3}{\Upsilon_W}  \qquad \mbox{ and } \qquad\left| \p_\zeta Z \circ \Upsilon_Z  \right| \leq \frac{C_3}{\Upsilon_Z}\,,
\end{equation*}
which yields
\begin{equation}
\left| \p_\zeta W \right| \leq \frac{C_3}{\zeta} \qquad \mbox{ and } \qquad \left| \p_\zeta Z  \right|  \leq \frac{C_3}{\zeta}\, . \label{eq:espartaquismo}
\end{equation}
and completes the proof of \eqref{eq:niponsegundo}.

Now, we study \eqref{eq:sanchez2}. First of all, let us note that
\begin{align} \begin{split} \label{eq:guerra}
|\p_\zeta W| \left| \p_\zeta^2 E_W \right| &= |\p_\zeta W| \left| \frac{1+\alpha}{2} \p_\zeta^2 W + \frac{1-\alpha}{2} \p_\zeta^2 Z \right| \leq  \frac{C_3}{\zeta} \left( | \p_\zeta^2 W | + | \p_\zeta^2 Z | \right),\\
|\p_\zeta Z| \left| \p_\zeta^2 E_Z \right| &= |\p_\zeta Z| \left| \frac{1-\alpha}{2} \p_\zeta^2 W + \frac{1+\alpha}{2} \p_\zeta^2 Z \right| \leq \frac{C_3}{\zeta} \left( | \p_\zeta^2 W | + | \p_\zeta^2 Z | \right),
\end{split} \end{align}
and 
\begin{equation} \label{eq:vonbismarck}
\left| 2 \alpha \p_\zeta^2 \left( \frac{\mc S \mc U}{\zeta} \right) \right| \leq \frac{1}{\zeta^2} + 4 \alpha \frac{\p_\zeta \mc S \p_\zeta \mc U}{\zeta} + \frac{2 \alpha}{\zeta} \left( |\mc S \p_\zeta^2 \mc U| + | \mc U \p_\zeta^2 \mc S | \right)
 \leq \frac{1+C_3^2}{\zeta^2} + \frac{1}{10\zeta} \left(  | \p_\zeta^2 W | + | \p_\zeta^2 Z | \right)\,,
\end{equation}
where we used \eqref{eq:zapatero1} and \eqref{eq:espartaquismo}. Let us also note from \eqref{eq:espartaquismo} that 
\begin{equation} \label{eq:cyrus}
| \p_\zeta E_W | \leq \frac{C_3}{\zeta} \qquad \mbox{ and } \qquad | \p_\zeta E_Z | \leq \frac{C_3}{\zeta}\,.
\end{equation}

Taking $C_3$ larger if needed, and using \eqref{eq:rajoy1}, \eqref{eq:zapatero1} and \eqref{eq:guerra}--\eqref{eq:cyrus} in \eqref{eq:sanchez2}, we get that
\begin{align*}
\left| \left( \partial_s +r-1 - \frac{2E_W}{\zeta} \right) \zeta^2 \p_\zeta^2 W + (\zeta+ E_W) \p_\zeta \left( \zeta^2\p_\zeta^2  W \right)   \right|
 &\leq 2+C_3^2 + \frac{4C_3}{\zeta} \left(  | \zeta^2\p_\zeta^2 W | + | \zeta^2\p_\zeta^2 Z | \right)\,, \\
\left| \left( \partial_s +r-1 - \frac{2E_Z}{\zeta} \right) \zeta^2 \p_\zeta^2 Z + (\zeta + E_Z) \p_\zeta \left( \zeta^2 \p_\zeta^2 Z \right) \right|
&\leq 2+C_3^2 + \frac{4 C_3}{\zeta} \left(  | \zeta^2\p_\zeta^2 W | + | \zeta^2\p_\zeta^2 Z | \right) .
\end{align*}

Defining 
\begin{equation} \label{eq:montoro2}
\Xi^{(W)} =  \Upsilon_W^2 \cdot( W \circ \Upsilon_W )\qquad \mbox{ and } \qquad \Xi^{(Z)} = \Upsilon_Z^2 \cdot (Z \circ \Upsilon_Z)\,,
\end{equation}
using \eqref{eq:zapatero1} and recalling that $\zeta, \Upsilon_W, \Upsilon_Z > C_1' = \delta_0^{-1/2}$ is sufficiently large, we have
\begin{align} \begin{split} \label{eq:aznar2}
\left| \p_s \Xi^{(W)} + \left( r-1 \right) \Xi^{(W)} \right| 
&\leq 2 + C_3^2  + \frac{r-1}{10 \Upsilon_W^{1/2}}\left( \abs{ \Xi^{(W)} } + \frac{\Upsilon_W^{2}}{\Upsilon_Z^2}\abs{ \Xi^{(Z)} \circ (\Upsilon_Z^{-1} \circ \Upsilon_W ) }  \right) \\
&\leq 2 + C_3^2  + \frac{r-1}{10}\left( \abs{ \Xi^{(W)} } + \abs{ \Xi^{(Z)} \circ (\Upsilon_Z^{-1} \circ \Upsilon_W ) }  \right) \,,\\
\left| \p_s \Xi^{(Z)} + \left( r-1 \right) \Xi^{(Z)} \right| 
&\leq 2 + C_3^2 +  \frac{r-1}{10 \Upsilon_Z^{1/2}}\left( \abs{ \Xi^{(Z)} } + \frac{\Upsilon_Z^{2}}{\Upsilon_W^2} \abs{ \Xi^{(W)} \circ (\Upsilon_W^{-1} \circ \Upsilon_Z )} \right) \\
&\leq 2 + C_3^2 +  \frac{r-1}{10}\left( \abs{ \Xi^{(Z)} } + \abs{ \Xi^{(W)} \circ (\Upsilon_W^{-1} \circ \Upsilon_Z )} \right)\,,
\end{split} \end{align}
where in the second and fourth lines we used \eqref{eq:paesa2}. We can pick a constant $C_4$ sufficiently large so that
\begin{align} \begin{split} \label{eq:suarez2}
\left| \Xi^{(W)} (\mw s) \right| &= \left| \mw \zeta^{2} \p_\zeta^2 \left( \mw U ( \mw \zeta, \mw s) +  \mt U ( \mw \zeta, \mw s) + \mw S (\mw \zeta, \mw s)  + \mt S (\mw \zeta, \mw s) \right) \right| \leq C_4\,, \\
\left| \Xi^{(Z)} (\mw s) \right| &= \left| \mw \zeta^{2} \p_\zeta^2 \left( \mw U ( \mw \zeta, \mw s) +  \mt U ( \mw \zeta, \mw s)- \mw S (\mw \zeta, \mw s) - \mt S (\mw \zeta, \mw s) \right) \right| \leq C_4\,.
\end{split} \end{align}
This is clear if $\mw \zeta = C_1'$ (as $C_4$ can depend on $C_1$) and follows from equation \eqref{eq:extrahypothesis} for $\mw s = s$.

Finally, taking $C_5 = C_4 + \frac{2}{r-1} (2+C_3^2)$, we conclude that $\Xi^{(W)}, \Xi^{(Z)} \leq C_5$. This follows from \eqref{eq:suarez2} for $s = \bar s$ and the inequality can not break for $s > \bar s$ due to equation \eqref{eq:aznar2} Therefore
\begin{equation*}
| \p_{\zeta}^2 W \circ \Upsilon_W | \leq \frac{C_5}{ \Upsilon_W^2  } 
\qquad \mbox{ and } \qquad 
| \p_{\zeta}^2 Z \circ \Upsilon_Z | \leq \frac{C_5}{ \Upsilon_Z^2  },
\end{equation*}
which yields
\begin{equation} \label{eq:definitiu}
| \p_\zeta^2 W | \leq  \frac{C_5}{\zeta^{2}} \qquad\mbox{and}\qquad | \p_\zeta^2 Z | \leq  \frac{C_5}{\zeta^{2}}\,.
\end{equation}
Finally, as any Cartesian second derivative of a radial function is a linear combination of $\p_\zeta^2$, $\frac{\p_\zeta}{\zeta}$ and $\frac{1}{\zeta^2}$ we conclude \eqref{eq:nipon} from \eqref{eq:zapatero1}, \eqref{eq:espartaquismo} and \eqref{eq:definitiu}.
\end{proof}

\begin{corollary} \label{cor:aka} There exists some absolute constant $\bar C$  such that 
\begin{equation} \label{eq:ao}
\| \nabla^2 \mc U  \|_{L^2} + \|  \nabla^2 \mc S  \|_{L^2} \leq \bar C\,.
\end{equation}
Moreover
\begin{equation} \label{eq:aka}
\|  \nabla^j \mc U  \|_{L^2} + \|  \nabla^j \mc S  \|_{L^2} \les_K \bar E^{\frac{j-2}{2K-2}}\,.
\end{equation}
\end{corollary}
\begin{proof}  Let $C, C_1$ be the constants from the statement of Lemma \ref{lemma:nipon}. Let $C_2 = \max \{ C_1, 1 \}$. Using Lemma \ref{lemma:nipon}, we have
\begin{equation} \label{eq:shiro}
 \|  \nabla^2 \mc U  \|_{L^2(B(0, C_2)^c)}^2 \leq C^2 \int_{B(0, C_2)^c} \frac{d\zeta}{\zeta^{4}} = 4 \pi C^2 \int_{C_2}^\infty \zeta^{-2} d\zeta = 4 \pi C^2 \frac{1}{C_2} \leq 4 \pi C^2\,.
\end{equation}
On the ball $B(0, C_2)$, we have
\begin{align} \label{eq:kuro}
\|  \nabla^2 \mc U  \|_{L^2 (B(0, C_2))}^2 &\leq C_2^3 \|  \nabla^2 \mc U  \|_{L^\infty}^2 \leq C_2^3 \left( \| \mw U \|_{W^{2, \infty}} + \|  \nabla^2 \mt U  \|_{L^\infty} \right) \notag  \\
&\leq  C_2^3 \left( \| \mw U \|_{W^{2, \infty}} + \| \mt U \|_{L^\infty}^{\frac{2K-4}{2K-2}} \|  \nabla^{2K-2} \mt U  \|_{L^\infty}^{\frac{2}{2K-2}} \right)  \notag \\
&\les 1 + \delta_0^{\frac{2K-4}{2K-2}} \left( \bar E \zeta_0^{1/2} \right)^{\frac{2}{2K-2}} \les 1\,,
\end{align}
where in the third inequality we used \eqref{eq:bootstrap_hyp}, \eqref{eq:capuccino} and Lemma \ref{lemma:profiledecay}. Combining \eqref{eq:shiro} and \eqref{eq:kuro} and taking $\bar C$ sufficiently large, we obtain equation \eqref{eq:ao} for $ \nabla^2 \mc U $. The result for $\nabla^2 \mc S $ is obtained in an analogous way.

Then, estimate \eqref{eq:aka} just follows from $L^2$-interpolation. For $\mc U$ we have
\begin{equation*}
\| \nabla^j \mc U \|_{L^2} \les \| \nabla^2 \mc U \|_{L^2}^{ \frac{2K-j}{2K-2} } \| \nabla^{2K} \mc U \|_{L^2}^{ \frac{j-2}{2K-2} }\,.
\end{equation*}
Then the required estimate on $\mc U$ follows by  \eqref{eq:ao} and 
 noting that
\begin{equation*}
\|  \nabla^{2K} \mc U  \|_{L^2} \les \| \Delta^K \mc U \|_{L^2} \leq \bar E\,.
\end{equation*}
An analogous estimate holds for $\mc S$.
\end{proof}

\begin{lemma} \label{lemma:forcingbounds} We have that
\begin{equation}\label{eq:Novak:Exemption}
\left\| \mc F_{e, \rm dis} \right\|_{L^\infty} ,
\left\| \mc F_{t, \rm dis} \right\|_X \leq \delta_1 e^{-\delta_{\rm dis} s/ 2}, \qquad 
\left\| \mc F_{e, \rm nl} \right\|_{L^\infty},
  \left\| \mc F_{t, \rm nl} \right\|_X \ll \delta_1\, .
\end{equation}
\end{lemma}
\begin{proof} Writing the expression for $\mc F_{\rm dis}$ from \eqref{eq:mattmurdock} in $(U, S)$ coordinates, we have that
\begin{equation} \label{eq:cortazar}
|\mc F_{e, \rm dis}| \les \left| e^{-\delta_{\rm dis} s} \frac{\Delta \mc U}{\mc S^{1/\alpha}}\right| \lesssim  e^{-\delta_{\rm dis} s/2} e^{-\delta_{\rm dis}s_0/2} \frac{|\Delta \mc U|}{\mc S^{1/\alpha}}\,.
\end{equation}
Lemma \ref{lemma:mozart} and Lemma \ref{lemma:suecia} give us
\begin{equation} \label{eq:neruda}
\frac{| \Delta \mc U |}{\mc S^{1/\alpha}} \lesssim_{\delta_1} \frac{\zeta^{(r-1)/\alpha }}{\phi} \les_{\delta_1} 1\,,
\end{equation}
using the definition of $\phi$ and the inequality  $\frac{r^\ast -1}{\alpha} < 2$ from Lemma \ref{lemma:connecticut}. Combining \eqref{eq:cortazar} and \eqref{eq:neruda}, and using that $s_0$ is large enough (depending on $\delta_0, \delta_1$) we obtain the bound for $\mathcal F_{e,{\rm dis}}$.

Using the embedding $L^\infty (B(0, 2)) \rightarrow L^2 (B(0, 2))$, we get
\begin{equation} \label{eq:gabo}
\| \mc F_{t, \rm dis} \|_X^2 = \| \chi_2 \mc F_{e, \rm dis} \|_X^2 =  \| \chi_2 \mc F_{e, \rm dis} \|_{L^2}^2 + \| \Delta^m ( \chi_1 \mc F_{e, \rm dis} ) \|_{L^2}^2 \les \left\| \mc F_{e, \rm dis} \right\|_{L^\infty}^2 + \| \Delta^m \mc F_{t, \rm dis} \|_{L^2}^2\,.
\end{equation}
Note for $\zeta < \frac65$ we have that $\mc S (\zeta) > \mw S (\zeta) - \delta_1 \gtrsim 1$, so $\frac{1}{\mc S} \les 1$. Therefore 
\begin{align} 
e^{-\delta_{\rm dis} s} \left| \Delta^m \left( \frac{\Delta \mc U}{\mc S^{1/\alpha}} \right) \right|
 &\lesssim_{\delta_1} e^{-\delta_{\rm dis} s} \left| \sum_{i+j=2m} | \nabla^{i+2} \mc U | \left| \nabla^{j} \left( \frac{1}{\mc S^{1/\alpha}} \right) \right| \right| \nonumber \\
&\lesssim_{\delta_1} e^{-\delta_{\rm dis} s} \left| \sum_{i+j=2m} | \nabla^{i+2} \mc U | \sum_{j_1 + \ldots + j_\ell = j} | \nabla^{j_1} \mc S | \ldots | \nabla^{j_\ell} \mc S | \right| \notag\\
&\lesssim_{\delta_1} e^{-\delta_{\rm dis} s_0/2} e^{-\delta_{\rm dis} s/2} \label{eq:singapur}\,.
\end{align}
From \eqref{eq:singapur}, noting that $s_0$ is chosen sufficiently large in terms of $\delta_1$, we get that
\begin{align} \label{eq:cambodia}
\| \Delta^m \mc F_{t, \rm dis} \|_{L^2}^2 \leq  e^{-\delta_{\rm dis}s/2} e^{-\delta_{\rm dis}s_0/4} \int_{B(0, 2)} 1 \ll \delta_1 e^{-\delta_{\rm dis}s/2},
\end{align}
Plugging \eqref{eq:cambodia} and the first bound of the statement into \eqref{eq:gabo}, we obtain the second bound of the statement.

Now, let us show the estimate for $\| \mc F_{e, \rm nl} \|_{L^\infty}$. Writing $\mc F_{\mathrm{nl}}$ from \eqref{eq:mattmurdock} in $(U, S)$ coordinates, we have that:
\begin{equation} \label{eq:bangladesh}
\mc F_{e, \rm nl} = \left(  \mc F_{\mathrm{nl}, \mt U}, \mc F_{\mathrm{nl}, \mt S} \right) = \left( -  \mt U \p_\zeta \mt U - \alpha \mt S \p_\zeta \mt S, -  \mt U \p_\zeta \mt S - \alpha \mt S \div (\mt U ) \right)\,.
\end{equation}
We know from \eqref{eq:bootstrap_hyp} that $\| \mt U \|_{L^\infty}, \| \mt S \|_{L^\infty} \leq \delta_0$ and from Lemma \ref{lemma:mozart} that $\|  \nabla^{2K-2} \mt U \|_{L^\infty}, \|  \nabla^{2K-2} \mt S \|_{L^\infty} \lesssim \bar{E}$. Using $L^\infty$-interpolation between both bounds, we obtain
\begin{equation}\label{eq:FBI:search}
\| \mt U \|_{W^{\ell, \infty}}+\| \mt S \|_{W^{\ell, \infty}} \les \delta_0^{\frac{2K-2-\ell}{2K-2}}\bar E^\frac{\ell}{2K-2}
\end{equation}
for $0\leq\ell\leq 2K-2$. Applying \eqref{eq:FBI:search} to \eqref{eq:bangladesh} we obtain
\begin{equation} \label{eq:moscu}
\|  \mc F_{\mathrm{nl}, \mt U} \|_{L^\infty} + \| \mc F_{\mathrm{nl}, \mt S} \|_{L^\infty} \les \delta_0^{2}  \left( \frac{\bar{E}}{\delta_0} \right)^{1/(2K-2)} \ll \delta_1,
\end{equation}
where the last inequality is due to the fact that $\delta_0^{3/2} \ll \delta_1$.

For the last inequality of \eqref{eq:Novak:Exemption}, using that $H^{2m}$ is an algebra, from \eqref{eq:FBI:search}. we get that
\begin{align}
\| \Delta^m (\chi_2 \mc F_{\mathrm{nl}, \mt U} ) \|_{L^2} + \| \Delta^m (\chi_2 \mc F_{\mathrm{nl}, \mt S} ) \|_{L^2} &\lesssim \| \mc F_{\mathrm{nl}, \mt U} \|_{X} + \| \mc F_{\mathrm{nl}, \mt S} \|_{X} \nonumber \\
&\lesssim_m \left( \| \mt U \|_{W^{2m, \infty}} + \| \mt S \|_{W^{2m, \infty}} \right) \left( \| \mt U \|_{W^{2m+1, \infty}} + \| \mt S \|_{W^{2m+1, \infty}}  \right) \nonumber \\ 
&\lesssim \delta_0^{(2K-2-2m)/(2K-2)} \bar{E}^{2m/(2K-2)} \delta_0^{(2K-3-2m)/(2K-2)} \bar{E}^{(2m+1)/(2K-2)}  \nonumber \\
&= \delta_0^2 \left( \frac{\bar{E}}{\delta_0} \right)^{(4m+1)/(2K-2)}  \ll \delta_1 \label{eq:cracovia}\,,
\end{align}
where in the last inequality we used $m \ll K$, $\delta_0^{3/2} \ll \delta_1$ and $\bar E\ll\frac{1}{\delta_0}$.

Equation \eqref{eq:moscu} and the embedding $L^\infty (B(0, 2)) \rightarrow L^2(B(0, 2))$ yield
\begin{equation} \label{eq:cracovia2}
\| \chi_2 \mc F_{\mathrm{nl}, \mt U} \|_{L^2} + \| \chi_2 \mc F_{\mathrm{nl}, \mt S} \|_{L^2} \les \delta_0^{9/5} \bar E^{1/5} \ll \delta_1\,,
\end{equation}
where we used $\delta_0^{3/2} \ll \delta_1$ and $\delta_0\ll\frac{1}{\bar E}$.
Combining \eqref{eq:cracovia} with \eqref{eq:cracovia2}, we conclude our bound for $\mc F_{t, \rm nl}$.
\end{proof}

\subsubsection{Proof of the bootstrap estimate \eqref{eq:improvement_Linf}}

Our strategy will be the following. First, we will show that we have $L^\infty$ estimates in a compact region $\zeta < \frac65$. As the extended and truncated solutions agree on that region (Lemma \ref{lemma:agree}), we can do that for the truncated equation, for which we have very precise information about its linearized operator (due to Section \ref{sec:linear}). Then, we will propagate those $L^\infty$ estimates for the extended equation to the region $\zeta > \frac65$ using trajectory estimates.

\begin{lemma} \label{lemma:paraguay} Under the bootstrap assumptions \eqref{eq:bootstrap_hyp}, we get the stronger bound 
\begin{equation}\label{eq:asdf}
\| \mt U \|_{L^\infty (B(0, 6/5))} + \| \mt S \|_{L^\infty (B(0, 6/5))} \lesssim \delta_1 / \delta_g\,.
\end{equation}
\end{lemma}
\begin{proof} By Lemma \ref{lemma:agree}, the truncated solution and the extended solution agree on $B \left(0, \frac65 \right)$. Then,
\begin{equation} \label{eq:kharkov}
\| \mt U_e \|_{L^\infty (B(0, 6/5))} = \| \mt U_t \|_{L^\infty(B(0, 6/5))} \lesssim \| \mt U_t \|_{H^{2m} (B(0, 6/5))} \lesssim \| \mt U_t \|_{\dot{H}^{2m}(B(0, 6/5))} \leq \int_{B(0, 2)} \left( \Delta^m \mt U_t \right)^2\,,
\end{equation}
where we used that the $H^{2m}(B(0, 2))$ norm is equivalent to the $\dot{H}^{2m}(B(0, 2))$ norm since $\mt U_t$ vanishes at the boundary. An analogous calculation shows the same bound for $\mt S_e$. Now, we claim
\begin{equation} \label{eq:turkmenistan}
\| (\mt U_t, \mt S_t ) \|_X \leq \left( 2 + \frac{4}{\delta_g} \right) \delta_1\, .
\end{equation}
It is clear that \eqref{eq:asdf} follows directly from \eqref{eq:kharkov} and \eqref{eq:turkmenistan}, so it remains to show \eqref{eq:turkmenistan}.

Clearly \eqref{eq:turkmenistan} is true at $s = s_0$ by our assumptions on the initial conditions \eqref{eq:initialdataassumptions}. Let us recall that by \eqref{eq:kappa:bnd}, the unstable part $P_{\rm uns}(\mt U, \mt S) = (1-P_{\rm sta})(\mt U, \mt S)$ will have $X$ norm at most $\delta_1$, so in order to show \eqref{eq:turkmenistan} we just need to ensure that $\| P_{\rm sta} (\mt u ,\mt \sigma) \|_X \leq \delta_1 (1+4/\delta_g)$.

Let us recall that the truncated problem \eqref{eq:CE} reads
\begin{equation*}
\p_s (\mt U_t, \mt S_t) + \mc L (\mt U_t, \mt S_t) = \mc F_{t, \rm  dis} + \mc F_{t, \rm nl} \,,
\end{equation*}
where the forcings $\mc F_{t, \rm dis}, \mc F_{t, \rm nl}$ are calculated via solving the extended equation.

Projecting the previous equation and using that $\mc L$ is invariant on $V^{\ast \, \perp}$, we get
\begin{align} \label{eq:uvas}
\p_s P_{\rm sta}(\mt U_t, \mt S_t) + \mc L P_{\rm sta} (\mt U_t, \mt S_t) = P_{\rm sta} \mc F_{t, \rm  dis} + P_{\rm sta} \mc F_{t, \rm nl}\,.
\end{align}

By Duhamel, the solution to the linear equation \eqref{eq:uvas} is given by
\begin{equation*}
P_{\rm sta} (\mt U_t(s), \mt S_t(S)) = T(s) P_{\rm sta} (\mt U_{t, 0}, \mt S_{t, 0} ) + \int_{s_0}^s T(s-\tilde s) (P_{\rm sta}\mc F_{t, \rm dis}(\tilde{s}) + P_{\rm sta} \mc F_{t, \rm nl} (\tilde{s}) ) d\tilde{s}\,,
\end{equation*}
where we recall that $T(s)$ is the contraction semigroup generated by $\mc L$. Recall also that $\| (\mt U_{t, 0}, \mt S_{t, 0}) \|_X \leq \delta_1$ due to the hypothesis of Proposition \ref{prop:bootstrap} and that the semigroup has an exponential decay in the stable space from \eqref{eq:prim}. Using those two observations together with Lemma \ref{lemma:forcingbounds}, we estimate
\begin{align*}
\left\| P_{\rm sta} (\mt U_t(s), \mt S_t(s)) \right\|_X  \leq  \delta_1 + \int_{s_0}^s 2 \delta_1 e^{-(\mt s-s_0)\delta_g/2}d \mt s  \leq \left( 1 + \frac{4}{\delta_g} \right) \delta_1\,,
\end{align*}
and we are done.
\end{proof}

Now as $\delta_1/\delta_g \ll \delta_0$, it is clear that we get \eqref{eq:improvement_Linf} in the region $\zeta < \frac65$. The objective is to extend the estimate from Lemma \ref{lemma:paraguay} to the whole space. We will chose some parameter $\zeta_1$ and divide the argument in two different regions, the region $6/5\leq \zeta \leq \zeta_1$ and the region $\zeta > \zeta_1$. The strategy is similar for both regions, since it will be based on trajectory estimates. The fundamental difference is that in the region $6/5 \leq \zeta \leq \zeta_1$ the profiles are not small, and we will not be able to extract decay for the perturbation along the trajectories. However, one can bound the amount of time the trajectory stays in this region by a constant, and therefore the profiles will only grow by a constant factor between $\zeta = 6/5$ and $\zeta = \zeta_1$. In contrast, the profiles will be small for $\zeta > \zeta_1$, and one can show that the damping part of the linearized operator dominates in this regime. This will give exponential decay for the perturbations in the region $\zeta \geq \zeta_1$.

Let us recall the equation for $\mt W$, which reads
\begin{align} \begin{split} \label{eq:uruguay}
& (\p_s +r-1+\frac{\alpha}{\zeta}\mw W+\frac{1+\alpha}2\p_{\zeta}\mw W)\mt W 
+ \left( \zeta+ \frac{1+\alpha}{2} \mw W + \frac{1-\alpha}{2} \mw Z \right) \p_\zeta \mt W
 + \left( \frac{1-\alpha}{2} \p_{\zeta}  \mw W -\frac{\alpha \mw Z}{\zeta}  \right)\mt Z
\\
&\qquad\qquad = \mc F_{\rm dis} + \mc F_{\mathrm{nl}, W},
\end{split} \end{align}
Let us define 
\begin{equation} \label{eq:manzanas}
J(\zeta) = \frac{\alpha}{\zeta} |\mw W| + \frac{1+\alpha}{2} |\p_\zeta \mw W| +  \frac{\alpha}{\zeta} |\mw Z| + \frac{1+\alpha}{2} | \p_\zeta \mw Z| \,.
\end{equation}
Let us define the trajectories $\mw \Upsilon_W^{(\zeta_\star, s_\star)}$ and $\mw \Upsilon_Z^{(\zeta_\star, s_\star)}$ solving the following ODEs
\begin{align} 
\p_s \mw \Upsilon_W^{(\zeta_\star, s_\star)}(s) &=  \left( \mw \Upsilon_W^{(\zeta_\star, s_\star)}(s) + \frac{1+\alpha}{2} \mw W( \mw \Upsilon_W^{(\zeta_\star, s_\star)}(s)) + \frac{1-\alpha}{2} \mw Z(\mw \Upsilon_W^{(\zeta_\star, s_\star)}(s) ) \right)\,, \label{eq:crisalidaW} \\
\p_s \mw \Upsilon_Z^{(\zeta_\star, s_\star)}(s) &=  \left( \mw \Upsilon_Z^{(\zeta_\star, s_\star)}(s) + \frac{1-\alpha}{2} \mw W(\mw \Upsilon_Z^{(\zeta_\star, s_\star)}(s)) + \frac{1+\alpha}{2} \mw Z(\mw \Upsilon_Z^{(\zeta_\star, s_\star)}(s) ) \right)\,. \label{eq:crisalidaZ}
\end{align}
and starting at the point $\mw \Upsilon_W^{(\zeta_\star, s_\star)} (s_\star) = \mw \Upsilon_Z^{(\zeta_\star, s_\star)}(s_\star) = \zeta_\star$. To ease the notation, we usually omit superindex $(\zeta_\star, s_\star)$.

Using \eqref{eq:manzanas}--\eqref{eq:crisalidaZ}, from \eqref{eq:uruguay} we obtain
\begin{equation} \label{eq:fresasW}
\left| (\p_s + r - 1) \mt W (\mw \Upsilon_W (s) ) \right| \leq 2\delta_1 + J(\mw \Upsilon_W (s) )  \left( | \mt W ( \mw \Upsilon_W (s)) | + | \mt Z (\mw \Upsilon_W(s) ) |\right),
\end{equation}
where we have also used Lemma \ref{lemma:forcingbounds} to bound the forcings. In an analogous way, we obtain
\begin{equation} \label{eq:fresasZ}
\left| (\p_s + r - 1) \mt Z (\mw \Upsilon_Z (s)) \right| \leq 2\delta_1 + J(\mw \Upsilon_Z (s) )  \left( | \mt W ( \mw \Upsilon_Z (s)) | + | \mt Z ( \mw  \Upsilon_Z (s)) |\right).
\end{equation} 

Let us also note that there exists some large enough constant $C_1$ such that for any $(\zeta_\star, s_\star)$ such that either $ \zeta_\star = \frac65$ or $s_\star = s_0$, we have
\begin{equation} \label{eq:frambuesas}
|\mt W (\zeta_\star,  s_\star) |, |\mt Z (\zeta_\star, s_\star) | \leq C_1 \frac{\delta_1}{\delta_g}\,.
\end{equation}
This follows from \eqref{eq:initialdataassumptions} for $s_\star = s_0$ and from Lemma \ref{lemma:paraguay} for $ \zeta_\star = \frac65$. We will always assume that our trajectories start at a point $(\zeta_\star, s_\star)$ that is either in $\left[ \frac65, \infty \right) \times \{ s_0 \}$ or in $\left\{ \frac65 \right\} \times [s_0, s_1]$, so the above condition will always be satisfied. 

Observe also that for $\zeta > \frac65$, the right hand sides of \eqref{eq:crisalidaW} and \eqref{eq:crisalidaZ} are larger than $\frac{\mw \Upsilon_W}{C_0}$ and $\frac{\mw \Upsilon_Z}{C_0}$ respectively, for some absolute constant $C_0$, due to Lemma \ref{lemma:DWDZproperties}. Therefore, as $\zeta_\star \geq \frac65$, we obtain that 
\begin{align}  \label{eq:dallas}
\p_s \mw  \Upsilon_\circ (s) \geq \frac{\mw \Upsilon_\circ ( s)}{C_0}, \qquad \mw \Upsilon_\circ^{(\zeta_\star, s_\star)} (s) \geq \mw \zeta_\star e^{(s-s_\star)/C_0}, 
 \end{align}
where $\circ \in \{ W, Z\}$ and in the second inequality we are integrating the first inequality one from $s_\star$ to $s$.

Now, we  estimate $\mt W, \mt Z$ in the region $(\zeta, s) \in \left[ \frac65, \infty \right) \times [s_0, s_1]$. As the profiles decay (Lemma \ref{lemma:profiledecay}), there exists some $\zeta_1$ such that $J(\zeta) < \frac{r-1}{4}$ for $\zeta > \zeta_1$. We treat different the cases where $\frac65 < \zeta  < \zeta_1$ and $\zeta > \zeta_1$.

\textbf{Case $\frac65 < \zeta < \zeta_1$.}
Let us fix a large enough constant $C_2$ so that $J(\zeta) + (r-1) < C_2$ for $\zeta \in \left[\frac65, \zeta_1 \right]$. Then, from \eqref{eq:fresasW}--\eqref{eq:fresasZ}, we have
\begin{align} \begin{split} \label{eq:mango}
\left| \p_s \mt W ( \mw \Upsilon_W (s), s) \right| \leq 2\delta_1  + C_2  \left( \left| \mt W ( \mw \Upsilon_W (s), s) \right| + \left| \mt Z ( \mw \Upsilon_W (s), s) \right| \right) \,,\\
\left| \p_s \mt Z ( \mw \Upsilon_Z (s), s) \right| \leq  2\delta_1 + C_2  \left( \left| \mt W ( \mw \Upsilon_Z (s), s) \right| + \left| \mt Z ( \mw \Upsilon_Z (s), s) \right| \right)\,.
\end{split} \end{align}
Now, we claim that
\begin{equation} \label{eq:limon}
| \mt W (\zeta, s) |, | \mt Z (\zeta, s) | < \frac{C_1 \delta_1}{\delta_g} e^{C_3 \zeta},
\end{equation}
for some constant $C_3$ sufficiently large. It is clear that this holds whenever $\zeta = \frac65$ or $\mw s = s_0$ due to \eqref{eq:frambuesas}. Let us show \eqref{eq:limon} by contradiction. Let $s_b$ the first time at which \eqref{eq:limon} fails. Assume that it fails at $\zeta = \zeta_b$ and without loss of generality assume that 
\begin{equation} \label{eq:limoncontradicted1}
\mt W(\zeta_b, s_b) = \frac{C_1 \delta_1}{\delta_g} e^{C_3 \zeta_b} \qquad \mbox{ and } \qquad | \mt Z (\zeta_b, s_b) | \leq \frac{C_1\delta_1}{\delta_g} e^{C_3 \zeta_b}\,,
\end{equation}
because the case where $\mt W(\zeta_b, s_b)$ is negative or the case where the bound fails for $\mt Z$ are analogous.

 As $(\zeta_b, s_b) \in \left[ \frac65, \zeta_1 \right] \times [s_0, s_1]$, and the field of the ODE of $\mw \Upsilon_W$ is positive (see \eqref{eq:dallas}), there is a unique starting point $(\zeta_\star, s_\star) \in \left[ \frac65, \zeta_1 \right] \times \{ s_0 \} \cup \left\{ \frac65 \right\} \times [s_0, s_1]$ such that the trajectory $\left( \mw \Upsilon_W^{(\zeta_\star, s_\star)}(s), s \right)$ passes through $(\zeta_b, s_b)$. That is, $\Upsilon^{(\zeta_\star, s_\star)}(s_b) = \zeta_b$. Let us fix that pair $(\zeta_\star, s_\star)$.

As equation \eqref{eq:limon} ceases to hold along $(\Upsilon_W(s), s)$ at time $s_b$, the derivative of the left-hand-side is greater than the derivative of the right-hand-side, yielding
\begin{align} 
\p_s \mt W ( \mw \Upsilon_W (s_b), s_b) &\geq \frac{C_1 \delta_1}{\delta_g} e^{C_3 \mw  \Upsilon_W(s_b)}C_3 \p_s  \mw \Upsilon_W(s_b)\notag\\& \geq \frac{C_1C_3 \delta_1}{\delta_g} e^{C_3  \mw \Upsilon_W(s_b)} \frac{ \mw \Upsilon_W (s_b)}{C_0}= \frac{C_1 C_3 \delta_1}{\delta_g} e^{C_3 \zeta_b} \frac{\zeta_b}{C_0}\,,\label{eq:naranja1}
 \end{align}
where in the second inequality we used \eqref{eq:frambuesas}.

On the other hand, plugging \eqref{eq:limoncontradicted1} into \eqref{eq:mango} at time $s_b$, we get
\begin{equation} \label{eq:naranja2}
\p_s \mt W ( \mw \Upsilon_W (s), s) \Big|_{s=s_b} \leq 2\delta_1 + C_2 \frac{2 C_1 \delta_1}{\delta_g} e^{C_3  \zeta_b} \leq  4C_2C_1 \frac{ \delta_1}{\delta_g} e^{C_3  \zeta_b},
\end{equation}
where in the last inequality we used $\delta_g \leq 1$ and $C_1, C_2 > 1$ (they are constants large enough, we can enlarge them if needed). 

Comparing \eqref{eq:naranja1} with \eqref{eq:naranja2} we get a contradiction as long as $C_3 \frac{ \zeta_b}{C_0} > 4C_2$, which can be easily enforced by taking $C_3 = \frac{4C_0C_2}{6/5}$ because $\zeta_b > \frac65$. Therefore, we conclude that \eqref{eq:limon} holds, which gives a uniform bound
\begin{equation} \label{eq:odessa}
| \mt W (\zeta, s) |, | \mt Z (\zeta, s) | \leq \frac{C_1 \delta_1}{\delta_g} e^{C_3 \zeta_1} \leq \frac{C_4}{\delta_g} \delta_1\,,
\end{equation}
for the region $\frac65 \leq \zeta \leq \zeta_1$ and some large constant $C_4$. As $\frac{C_4}{\delta_g}\delta_1 \ll \delta_0$, note that we now have \eqref{eq:improvement_Linf} in the region $\zeta \leq \zeta_1$. Finally, let us treat the region $\zeta > \zeta_1$.

\textbf{Case $ \zeta > \zeta_1$.} First of all note that the previous case yields that 
\begin{equation} \label{eq:illinois}
|\mt W(\zeta_1, s)|,  |\mt Z (\zeta_1, s)| \leq C_4 \frac{\delta_1}{\delta_g}\,,
\end{equation} for all $s \in [s_0, s_1]$. Let us recall that $\zeta_1$ was defined so that $|J(\zeta)| \leq \frac{r-1}{4}$ in this region $\zeta > \zeta_1$. Therefore, from \eqref{eq:fresasW} and \eqref{eq:fresasZ}, we have
\begin{align} \begin{split} \label{eq:dakota}
\left| (\p_s + r - 1) \mt W ( \mw \Upsilon_W (s), s) \right| &\leq 2\delta_1 + \frac{r-1}{4} \left( \| \mt W(\cdot, s) \|_{L^\infty[\zeta_1, \infty) } + \| \mt Z(\cdot, s) \|_{L^\infty[\zeta_1, \infty) } \right)\,, \\
\left| (\p_s + r - 1) \mt Z ( \mw \Upsilon_Z (s), s) \right| &\leq 2\delta_1 + \frac{r-1}{4} \left( \| \mt W(\cdot, s) \|_{L^\infty[\zeta_1, \infty) } + \| \mt Z(\cdot, s) \|_{L^\infty[\zeta_1, \infty) } \right)\,.
\end{split} \end{align}

We claim that 
\begin{equation} \label{eq:wyoming}
| \mt W (\zeta, s) |, | \mt Z (\zeta, s | < 2C_4 \frac{\delta_1}{\delta_g}\,,
\end{equation}
for all $\zeta \geq \zeta_1$ and $s \in [s_0, s_1]$. This clearly holds at $\zeta = \zeta_1$ or $s = s_0$ due to \eqref{eq:initialdataassumptions} and \eqref{eq:illinois}. Let $s_b$ be the first time at which equation \eqref{eq:wyoming} breaks down, and suppose that happens at $\zeta_b$. Without loss of generality assume that $\mt W(\zeta_b, s_b) = 2C_4 \frac{\delta_1}{\delta_g}$. Let us consider the trajectory passing through $(\zeta_b, s_b)$. Then, from \eqref{eq:dakota} we have
\begin{align*}
0 &\leq \p_s \mt W( \mw \Upsilon_W(s), s) \Big|_{s=s_b} \\
&\leq -(r-1)\left( 2C_4 \frac{\delta_1}{\delta_g} \right) + 2\delta_1 + \frac{(r-1)}{4} \left( \| \mt W(\cdot, s) \|_{L^\infty[\zeta_1, \infty) } + \| \mt Z(\cdot, s) \|_{L^\infty[\zeta_1, \infty) } \right) \\
&\leq -(r-1)\left( 2C_4 \frac{\delta_1}{\delta_g} \right) + 2\delta_1 + (r-1)\left( C_4 \frac{\delta_1}{\delta_g} \right) \\
&\leq - \frac{r-1}{2} C_4 \frac{\delta_1}{\delta_g}\,.
\end{align*}
In the third inequality, we used that \eqref{eq:wyoming} is still true at $s=s_b$ with a non-strict inequality by continuity. Thus, we get a contradiction, and this shows that such $(\zeta_b, s_b)$ cannot exist, so \eqref{eq:wyoming} holds for any $ \zeta \geq \zeta_1$ and $s \geq s_0$.

Combining equation \eqref{eq:wyoming} with Lemma \ref{lemma:paraguay} and equation \eqref{eq:odessa}, and using that $\frac{\delta_1}{\delta_g} \ll \delta_0$, we conclude equation \eqref{eq:improvement_Linf}.

\subsubsection{Proof of bootstrap estimate \eqref{eq:improvement_E}}
As in the previous subsections, we work under the assumptions \eqref{eq:initialdataassumptions}--\eqref{eq:bootstrap_hyp}. In this subsection, we will show the improved bound \eqref{eq:improvement_E}, concluding the proof of Proposition \ref{prop:bootstrap}. We will do so via high order weighted energy estimates, which we will do directly on our extended equation. 

We divide the proof in three steps. First, we will show that the dominant terms on a high derivative of a quadratic term are those where all (or all except one) derivatives fall on the same factor. The second step will be to treat the dominant terms in the energy calculation, using integration by parts in a similar way as one would do for the classical energy estimates. In our third and final step,   we treat the term coming from the dissipation, which one cannot expect to bound (since it has more derivatives than our energy). Thus, the strategy is to extract the correct sign for this term.

We take $K$ Laplacians in \eqref{eq:US:system}. Note that $\Delta^K (y F) = y\Delta^K F + 2K \div \Delta^{K-1}F$, for any vector field $F$. Thus
\begin{align} \begin{split} \label{eq:killbill}
\Delta^K (y \nabla \mc S) &= y \cdot \Delta^K \nabla \mc S + 2K \div \Delta^{K-1} \nabla \mc S = y \cdot \nabla \Delta^K \mc S + 2K \Delta^K \mc S \,,\\
\Delta^K (y \nabla \mc U_i) &= y \cdot \Delta^K \nabla \mc U_i + 2K \div \Delta^{K-1} \nabla \mc U_i = y \cdot \nabla \Delta^K \mc U_i + 2K \Delta^K \mc U_i\,.
\end{split} \end{align}
We have also used that $\Delta \nabla f = \nabla \Delta f$ and that $\Delta f = \div \nabla f$. Taking $K$ Laplacians in \eqref{eq:US:system} and using \eqref{eq:killbill}, we obtain:

\begin{align} \label{eq:HDNS} \begin{split} 
(\p_s + r -1 + 2K )\Delta^K \mc U_i + y\cdot \nabla \Delta^K \mc U_i + \Delta^K (\mc U \cdot \nabla \mc U_i) + \alpha \Delta^K (\mc S \p_i \mc S) &= \frac{r^{1+\frac{1}{\alpha}}}{\alpha^{1/\alpha}  }
e^{-\delta_{\rm dis} s}  \Delta^K \frac{\Delta \mc U_i}{ \mc S^{1/\alpha }}\,, \\
(\p_s + r - 1 + 2K)\Delta^K \mc S + y \cdot \nabla \Delta^K \mc S + \alpha \Delta^K (\mc S \div \mc U) + \Delta^K (\mc U \cdot \nabla \mc S) &= 0\,.
\end{split} \end{align}

Now, we multiply each equation by $\phi^{2K} \Delta^K \mc U_i$ or $\phi^{2K} \Delta^K \mc S$ respectively in order to do energy estimates. First of all, we claim that as a consequence of Lemma \ref{lemma:leibnitzlaplace2}
\begin{align} \begin{split} \label{eq:huawei}
\left\| \left( \Delta^K (\mc U \cdot \nabla \mc U_i ) - \mc U \nabla \Delta^K \mc U_i - 2K \p_\zeta \mc U \Delta^K \mc U_i  \right) \phi^{K} \right\|_{L^2} &= O \left( \bar{E} \right), \\
\left\|  \left( \Delta^K  (\mc S \nabla \mc S) - \mc S \nabla \Delta^K  \mc S - 2K \nabla \mc S \Delta^K \mc S \right) \phi^{K} \right\|_{L^2} &= O \left( \bar{E} \right), \\
\left\|  \left( \Delta^K  (\mc U \cdot \nabla \mc S) - \mc U \cdot \nabla \Delta^K  \mc S - 2K \p_\zeta \mc U \cdot \Delta^K \mc S \right) \phi^{K} \right\|_{L^2} &= O \left( \bar{E} \right), \\
\left\|  \left( \Delta^K  (\mc S \div (\mc U)) - \mc S \div ( \Delta^K  \mc U ) - 2K \nabla \mc S \cdot \Delta^K \mc U \right) \phi^{K} \right\|_{L^2} &= O \left( \bar{E} \right). \\
\end{split} \end{align}
However, to apply Lemma \ref{lemma:leibnitzlaplace2}, we need to verify its hypotheses. \\

\textbf{Verifying the Hypotheses of Lemma \ref{lemma:leibnitzlaplace2}. } We start checking Hypothesis \eqref{eq:tomate1}. Let $2 \leq j \leq 2K-1$, $\beta = \frac{j-2}{2K-3}$, $p = \frac{2}{\beta}$. We have that
\begin{align}
&\int_{\mathbb{R}^3} \left( | \nabla^j \mc U | + | \nabla^j \mc S | \right)^2 \left( | \nabla^{2K+1-j} \mc U | + | \nabla^{2K+1-j} \mc S | \right)^2 \notag \\
& \qquad \leq 
\left( \int_{\mathbb{R}^3} \left( | \nabla^j \mc U | + | \nabla^{j} \mc S | \right)^{2p} \right)^{1/p}
\left( \int_{\mathbb{R}^3} \left( | \nabla^{2K+1-j} \mc U | + | \nabla^{2K+1-j} \mc S | \right)^{2p/(p-1)} \right)^{1-1/p} \notag  \\
& \qquad \leq 
\left( \| \nabla^2 \mc U \|_{L^\infty} + \| \nabla^2 \mc S \|_{L^\infty} \right)^{2\beta} 
\left( \| \nabla^{2K-1} \mc U \|_{L^2} + \| \nabla^{2K-1} \mc S \|_{L^2} \right)^{2(1-\beta)} \notag \\
&\qquad \qquad \cdot \left( \| \nabla^2 \mc U \|_{L^\infty} + \| \nabla^2 \mc S \|_{L^\infty} \right)^{2(1-\beta)}
\left( \| \nabla^{2K-1} \mc U \|_{L^2} + \| \nabla^{2K-1} \mc S \|_{L^2} \right)^{2\beta} \notag \\
&\qquad \leq \left( \| \nabla^{2K-1} \mc U \|_{L^2}  + \| \nabla^{2K-1} \mc S \|_{L^2} \right)^2 
\left( \| \nabla^2 \mc U \|_{L^\infty} + \| \nabla^2 \mc S \|_{L^\infty} \right)^2  \notag \,,\\
&\qquad\les \bar{E}^{\frac{2K-3}{2K-2}} \label{eq:hidrogeno}
\end{align}
where in the first inequality we use H\"older and in the second one we use endpoint Gagliardo-Nirenberg for each integral. In the last inequality we are also using \eqref{eq:aka} and $ | \nabla^2 \mc U |+ | \nabla^2 \mc S | \les 1$, which follows from $| \nabla^2 \mt U |+ | \nabla^2 \mt S | \les 1$ which itself follows from interpolating $|\mt U|, |\mt S| \leq \delta_0$ (equation \eqref{eq:bootstrap_hyp}) with \eqref{eq:capuccino}.

Now let $\mc B_2 = \mathbb{R}^3 \setminus B(0, \zeta_0)$. Now, we let $\beta' = \frac{j-1}{2K-1}$ and $q = \frac{2}{\beta'}$. We have that
\begin{align}
& \int_{\mc B_2} \left( | \nabla^j \mc U | + | \nabla^j \mc S | \right)^2 \left( | \nabla^{2K+1-j} \mc U | + | \nabla^{2K+1-j} \mc S | \right)^2 \phi^{2K}  \notag \\
& \qquad \leq \left( \int_{\mc B_2} \left( | \nabla^j \mc U | + | \nabla^j \mc S | \right)^{2q} \phi^{2K \beta' q} \right)^{1/q} 
\left( \int_{\mc B_2} \left( | \nabla^{2K+1-j} \mc U | + | \nabla^{2K+1-j} \mc S| \right)^{2q/(q-1)} \phi^{2qK (1-\beta')/(q-1)} \right)^{1-1/q} \notag \\
& \qquad \les \left[ \left( 
\| \nabla \mc U \|_{L^\infty (\mc B_2)}| + \| \nabla \mc S \|_{L^\infty (\mc B_2)}  \right)^{2\beta'}  
\left( \|  \phi^{2K} \nabla^{2K}  \mc U \|_{L^2} + \| \nabla^{2K} \mc S \|_{L^2} \right)^{2(1-\beta')}
+ \left( \| \nabla \mc U \|_{L^\infty (\mc B_2)} + \| \nabla \mc S \|_{L^\infty (\mc B_2)} \right)^2  \right] \notag \\
&\qquad \qquad \cdot \left[ 
 \left( \| \nabla \mc U \|_{L^\infty (\mc B_2)} + \| \nabla \mc S \|_{L^\infty (\mc B_2)} \right)^{2(1-\beta')}
 \left( \| \nabla^{2K}  \mc U \phi^{2K} \|_{L^2} + \| \nabla^{2K} \mc S \phi^{2K} \|_{L^2}  \right)^{2 \beta'} 
 + \left( \|  \nabla \mc U \|_{L^\infty (\mc B_2)} + \| \nabla \mc S \|_{L^\infty (\mc B_2)} \right)^2 \right] \notag \\
& \qquad \les \left[ \delta_0^{2\beta'} \bar E^{1-\beta'}+ \delta_0 \right] \cdot \left[ \delta_0^{2(1-\beta')} \bar E^{\beta'} + \delta_0 \right] \les \delta_0 \bar E\,. \label{eq:oxigeno}
\end{align}
In the first inequality, we used H\"older. In the second inequality we used Gagliardo-Nirenberg (Lemma \ref{lemma:gn_interpolation}). In the third inequality we used that $\|  \nabla^{2K} \mc U  \|_{L^2} \les \| \Delta^K \mc U \|_{L^2}$ (and the same for $\mc S$), and
\begin{equation*}
\| \nabla \mc U \|_{L^\infty (\mc B_2)} + \| \nabla  \mc S \|_{L^\infty (\mc B_2)} \leq 
\|  \nabla \mt U  \|_{L^\infty (\mc B_2)} + \|  \nabla  \mt S  \|_{L^\infty (\mc B_2)} + \| \mw U \|_{L^\infty (\mc B_2)} + \| \mw S \|_{L^\infty (\mc B_2)} \les \delta_0.
\end{equation*}
Combining \eqref{eq:hidrogeno} with \eqref{eq:oxigeno}, we see that Hypothesis \eqref{eq:tomate1} is satisfied.

Now, let us check Hypothesis \eqref{eq:tomate2}. The first part trivially holds due to our bootstrap hypothesis \eqref{eq:bootstrap_hyp}. For the second part, note that Corollary \ref{cor:aka} yields
\begin{equation} \label{eq:helio}
 \|  \nabla^{2K-1} \mc U  \|_{L^2} \leq \bar E \frac{1}{\bar E^{1/(2K-3)}}\,,
\end{equation}
so we just need to treat the region $\zeta > \zeta_0$. Let $\mc B_2 = \mathbb{R}^3 \setminus B(0, \zeta_0)$. Using Lemma \ref{lemma:thefinalinterpolation2}, we have that
\begin{equation} \label{eq:fluor}
\|  \phi^{K-1/2} \zeta^{-2/(2K)}  \nabla^{2K-1} \mc U  \|_{L^2(\mc B_1)} \leq \| \phi^{K}  \nabla^{2K} \mc U  \|_{L^2(\mc B_1)}^{\frac{2K-1}{2K}} \| \zeta^{-2} \mc U \,,\|_{L^2(\mc B_1)}^{\frac{1}{2K}} \les \bar E^{1-1/(2K)} \,,
\end{equation}
where we have used that $\| \mc U \|_{L^\infty}, \| \zeta^{-2} \|_{L^2(\mc B_1)} \les 1$. Now, note that $\frac{\phi^{1/4} \zeta^{1/K}}{\langle \zeta \rangle^{1/2} } \leq 1$ because $\phi^{1/2} \leq \zeta^{1-\eta_w}$ and $\frac{1}{K} \ll \eta_w$. Therefore, multiplying the weight in \eqref{eq:fluor} by $ \frac{\phi^{1/4} \zeta^{1/K}}{\langle \zeta \rangle^{1/2} }$ we obtain
\begin{equation} \label{eq:fluor2}
\left\| \phi^K  \nabla^{2K-1} \mc U  \frac{1}{\phi^{1/4} \langle \zeta \rangle^{1/2}} \right\|_{L^2(\mc B_1)} \les \bar E^{1-1/(2K)}\,.
\end{equation}
Combining \eqref{eq:helio} with \eqref{eq:fluor2}, we conclude Hypothesis \eqref{eq:tomate2} holds for $\mc U$. The case of $\mc S$ is completely analogous.

Finally, let us check Hypothesis \eqref{eq:tomate3}. The second part follows directly from \eqref{eq:niponsegundo} in Lemma \ref{lemma:nipon}. For the first part note that
\begin{align*}
\| \mc U \|_{W^{3, \infty}} &\leq \| \mw U \|_{W^{3, \infty}} + \| \mt U \|_{W^{3, \infty}} \\
& \les 1 + O_K(1) \| \mt U \|_{L^\infty}^{(2K-5)/(2K-2)} \left(  \| \mt U \|_{L^\infty}^{3/(2K-2)} +  \|  \nabla^{2K-2} \mt U  \|_{L^\infty}^{3/(2K-2)} \right) \\
& \les 1 + \delta_0^{1/2} \left( \bar E \zeta_0^{1/2} \right)^{3/(2K-2)} \les 1\,.
\end{align*}
where we have used Lemma \ref{lemma:profiledecay}, $L^\infty$-interpolation, $\| \mt U \|_{L^\infty} \leq \delta_0$ (from \eqref{eq:bootstrap_hyp}) and equation \eqref{eq:capuccino}.

The proof is analogous for $\mc S$ and this concludes checking all the hypotheses of Lemma \ref{lemma:leibnitzlaplace2}. \\

\textbf{Main energy calculation.} Let us go back to \eqref{eq:HDNS}. Let $\mathcal J$ denote the dissipative term
\begin{equation*}  
\mc J = \int \frac{r^{1+\frac{1}{\alpha}} }{\alpha^{\frac1\alpha} }e^{-\delta_{\rm dis} s} \Delta^K \frac{\Delta \mc U}{ \mc S^{1/\alpha }} \Delta^K \mc U \phi^{2K}\,.
\end{equation*}
Then, using \eqref{eq:huawei}, we obtain that
\begin{align}
\frac{\p_s}{2} \left( E_{2K}^2 \right) &= -(2K+r-1) E_{2K}^2 + \frac{1}{2} \int \div(y \phi^{2K} ) \left( (\Delta^K \mc U_i)^2 + (\Delta^K \mc S )^2 \right) \nonumber \\
& \qquad - \sum_{i=1}^3 \int \mc U \cdot \nabla \Delta^K \mc U_i \Delta^K \mc U_i \phi^{2K} - 2K \int \p_\zeta \mc U (\Delta^K \mc U)^2 \phi^{2K} \nonumber \\
& \qquad - \alpha \int \mc S \nabla \Delta^K \mc S \cdot \Delta^K \mc U \phi^{2K} - 2K \alpha \int \Delta^K \mc S \nabla \mc S \cdot \Delta^K \mc U \phi^{2K} \nonumber \\
& \qquad - \int \mc U \cdot \nabla \Delta^K \mc S \Delta^K \mc S \phi^{2K} - 2K \int \p_\zeta \mc U (\Delta^K \mc S)^2 \phi^{2K} \nonumber \\
& \qquad - \alpha \int \mc S \div (\Delta^K \mc U) \Delta^K \mc S \phi^{2K} - 2K \alpha \int \nabla \mc S \cdot \Delta^K \mc U \Delta^K \mc S \phi^{2K}   + \mc J + O (\bar{E}^2 ) \nonumber \\
&= -2K E_{2K}^2 + \frac12 \int \div \left( y\phi^{2K} + \mc U \phi^{2K} \right) \left( (\Delta^K \mc U)^2 + (\Delta^K \mc S)^2 \right) \nonumber \\
& \qquad - 2K \int \p_\zeta \mc U \left(  (\Delta^K \mc U)^2  + (\Delta^K \mc S)^2 \right) \phi^{2K} \nonumber \\
& \qquad - \alpha \int \mc S \div( \Delta^K \mc S \cdot \Delta^K \mc U )\phi^{2K} - 4K \alpha \int \Delta^K \mc S \nabla \mc S \cdot \Delta^K \mc U \phi^{2K}  + \mc J + O (\bar{E}^2 )\nonumber  \\
&\leq -2K \int (1+\p_\zeta \mc U - \alpha | \nabla \mc S |) \left( (\Delta^K \mc U )^2 + (\Delta^K \mc S)^2 \right) \phi^{2K} \nonumber \\
& \qquad + \frac12 \int (\zeta + | \mc U | + \alpha \mc S) | \nabla (\phi^{2K}) | \left( (\Delta^K \mc U)^2 + (\Delta^K \mc S)^2 \right) + \mc J + O (\bar{E}^2 )\nonumber \\
& =  -2K \int \left( 1 + \p_\zeta \mc U - \alpha | \nabla \mc S | - \frac{| \nabla \phi|}{2\phi} \left( \zeta + | \mc U | + \alpha S \right)  \right) \left( (\Delta^K \mc U)^2 + (\Delta^K \mc S)^2 \right) \phi^{2K} + \mc J + O (\bar{E}^2 ) \,.
\label{eq:churros}
\end{align}
where we have used \eqref{eq:bootstrap_hyp}.
Now, we claim that 
\begin{equation} \label{eq:chocolate}
\left( 1 + \p_\zeta \mc U - \alpha | \nabla \mc S | - \frac{| \nabla \phi|}{2\phi} \left( \zeta + | \mc U | + \alpha S \right)  \right) \geq \eta\,,
\end{equation}
for some positive constant $\eta$. In the region $\zeta < \zeta_0$ we have that $\nabla \phi = 0$ and $1 + \p_\zeta \mw U - \alpha | \nabla \mw S | \geq \eta_{\rm damp} > 0$ due to Lemma \ref{lemma:Vdamping}. Therefore, taking $\delta_0$ small enough and using that $| \p_\zeta \mt U |,  | \nabla \mt S | \lesssim \delta_0^{(2K-1)/(2K-2)} \bar{E}^{1/(2K-2)}$ (equation \eqref{eq:FBI:search}), we conclude that \eqref{eq:chocolate} holds for $\eta \leq \frac12 \eta_{\rm damp}$ and $\zeta < \zeta_0$. 

In the region $\zeta > \zeta_0$, equation \eqref{eq:chocolate} reduces to show
\begin{equation*}
1 - \frac{ | \nabla  \phi |\zeta }{2 \phi} > \eta\,.
\end{equation*}
since the remaining terms are $ O (\delta_0^{(2K-1)/(2K-2)} \bar{E}^{1/(2K-2)})$ by \eqref{eq:FBI:search}.

Using that $|\nabla \phi| \zeta  \leq 2(1-\eta_w ) \phi$, we conclude that taking $\delta_0$ small enough \eqref{eq:chocolate} holds for $\eta < \frac12 \eta_w$.

Putting everything together, we have that \eqref{eq:chocolate} holds globally for $\eta = \frac12 \min \{ \eta_{\rm damp}, \eta_w \}$, and plugging this in \eqref{eq:churros}, we conclude 
\begin{equation} \label{eq:alaska}
\frac{\p_s}{2} \left( E_{2K}^2 \right) \leq - K \eta E_{2K}^2  + O (\bar{E}^2 ) + \mc J
\end{equation} 
\vspace{0.1in}

\textbf{Sign on the dissipative term.}  We proceed by estimating $\mc J$ by subtracting off the highest order term. Letting 
\begin{equation*}
\mc G^2 = \sum_{i=1}^3 \int \frac{( \nabla \Delta^K \mc U_i )^2}{\mc S^{1/\alpha}} \phi^{2K},
\end{equation*}
integrating by parts, and using Cauchy-Schwarz, we have that
\begin{align} \label{eq:maine}
\left| \frac{\alpha^{1/\alpha} }{r^{1+1/\alpha }} e^{\delta_{\rm dis} s}\mc J + \mc G^2 \right| 
&\leq \mc G \sum_{i=2}^{2K} \binom{2K-1}{i-2} \left(  \int \frac{\mc S^{1/\alpha}}{\phi} |\p^i \mc U|^2 \phi^i \left( \p^{2K+1-i} \left( \frac{1}{\mc S^{1 / \alpha}} \right) \right)^2 \phi^{2K+1-i}  \right)^{1/2} \notag \\
&\leq \mc G \Bigg(  (2K-1) \left( \int   | \p^{2K} \mc U |^2 \phi^{2K} \right)^{1/2} 
\left\| \mc S^{1/\alpha} \left| \nabla \left( \frac{1}{\mc S^{1/\alpha}} \right) \right| \right\|^{1/2}_{L^\infty} \notag \\
&\qquad+\sum_{i=2}^{2K-1} \binom{2K-1}{i-2} \left(  \int_{\mc B_1} \frac{\mc S^{1/\alpha}}{\phi} |\p^i \mc U|^2 \phi^i \left( \p^{2K+1-i} \left( \frac{1}{\mc S^{1 / \alpha}} \right) \right)^2 \phi^{2K+1-i}  \right)^{1/2} \notag \\
&\qquad+\sum_{i=2}^{2K-1} \binom{2K-1}{i-2} \left(  \int_{\mc B_2} \frac{\mc S^{1/\alpha}}{\phi} |\p^i \mc U|^2 \phi^i \left( \p^{2K+1-i} \left( \frac{1}{\mc S^{1 / \alpha}} \right) \right)^2 \phi^{2K+1-i}  \right)^{1/2} \Bigg) \notag \\
&= \mc G (\mc Q_1 + \mc Q_2 + \mc Q_3 ),
\end{align}
where we have divided the integrals in the regions $\mc B_1 = B(0, \zeta_0)$ and $\mc B_2 = \mathbb{R}^3 \setminus \mc B_1$. Using Lemma \ref{lemma:derbounds}, we have that
\begin{equation} \label{eq:boliazul}
\mc Q_1 \les_{K, \bar E, \delta_0, \zeta_0} K \bar{E}^{1/2} \left\| \mc S^{1/\alpha} \langle \zeta \rangle^{2(r-1)/\alpha} \phi(\zeta)^{-1} \right\|_{L^\infty}^{1/2} \les_{K, \bar E, \delta_0, \zeta_0} 1,
\end{equation}
where in the last equality we have used Lemma \ref{lemma:connecticut} and $\mc S \langle \zeta \rangle^{(r-1)} \les 1$ by Lemma \ref{lemma:suecia}.

In order to bound $\mc Q_2$, we use again Lemma \ref{lemma:derbounds} and note that $\phi (\zeta) \les_{\zeta_0, K} 1$ in $\mc B_1$. We obtain
\begin{align} \label{eq:bolirojo}
\mc Q_2 &\les_{K, \bar E, \delta_0, \zeta_0} \sum_{i=2}^{2K-1} 
 \left( \int_{\mc B_1} \mc S^{1/\alpha} |\p^i \mc U|^2  \langle \zeta \rangle^{2\frac{r-1}{\alpha}}  \right)^{1/2} \notag \\
 &\les_{K, \bar E, \delta_0, \zeta_0} \left( \int_{\mc B_1} \frac{1}{\zeta^{1/2}}   \right)^{1/2} \les_{K, \bar E, \delta_0, \zeta_0} 1,
\end{align}
where in the second inequality we also used \eqref{eq:cafeconleche} (for $i=2K-1$) and \eqref{eq:cafesolo} (for $i \leq 2K-2$).

Lastly, let us bound $\mc Q_3$. Note that in $\mc B_2$ we have \eqref{eq:cafeconleche}, which yields
\begin{align*}
|\nabla^i \mc U| \phi^{i/2} &\lesssim \bar{E} \left( \frac{1}{\zeta^{1/2} \phi^{1/2}} \right)^{\frac{i}{2K-1}}.
\end{align*} 

Plugging this into $\mc Q_3$ and using Lemma \ref{lemma:derbounds}, we obtain
\begin{align*} 
\mc Q_3 \lesssim_{K, \bar E, \delta_0}  \sum_{i=2}^{2K} \left(  \int \frac{\mc S^{1/\alpha}}{\phi}  \left( \frac{1}{\zeta \phi (\zeta) } \right)^{\frac{i}{2K-1}} \langle \zeta \rangle^{\frac{2(r-1)}{\alpha}} \left( \langle \zeta \rangle^{2r-1} \phi^{-2} \right)^{\frac{2K+1-i}{2K-2}}  \right)^{1/2} \,.
\end{align*}
Noting that $1 \lesssim_{\delta_0, \zeta_0} \langle \zeta \rangle^{2r} \phi (\zeta )^{-1}$, and using Lemma \ref{lemma:suecia}, we get
\begin{align}  \label{eq:bolinegro}
\mc Q_3 &\lesssim_{K, \bar E, \delta_0, \zeta_0}   \sum_{i=2}^{2K} \left(  \int \frac{\langle \zeta \rangle^{-(r-1)/\alpha}}{\phi} \langle \zeta \rangle^{\frac{2(r-1)}{\alpha}}\langle \zeta \rangle^{2r-1} \phi^{-2}   \right)^{1/2} \notag \\
&=    \sum_{i=2}^{2K} \left(  \int \frac{1}{\phi^3} \langle \zeta \rangle^{(r-1) (2+1/\alpha)}\langle \zeta \rangle    \right)^{1/2} \notag \\
&\lesssim_{K, \bar E, \delta_0, \zeta_0}  \left( \int \frac{\langle \zeta \rangle^3}{\phi^3 (\zeta )}\right)^{1/2} \notag \\
& \lesssim_{K, \bar E, \delta_0, \zeta_0} 1\,,
\end{align}
where in the third line we used Lemma \ref{lemma:connecticut}.

Using \eqref{eq:boliazul}--\eqref{eq:bolinegro} in \eqref{eq:maine}, we conclude that
\begin{equation*}
\left| \frac{\alpha^{1/\alpha} }{r^{1+1/\alpha }} e^{\delta_{\rm dis} s}\mc J + \mc G^2  \right| \lesssim_{\delta_0, K, \bar E, \zeta_0} \mc G \,.
\end{equation*}

Thus, as $s_0$ is sufficiently large in terms of $K, \bar E, \delta_0, \zeta_0$, we get
\begin{equation*} 
 \frac{\alpha^{1/\alpha} }{r^{1+1/\alpha }} e^{\delta_{\rm dis} s}\mc J +  \mc G^2  \leq  \frac{1}{2} \mc G e^{s_0 \delta_{\rm dis} /4} \leq \frac{\mc G^2}{2} + \frac{e^{\delta_{\rm dis} s_0/2}}{2}\,,
\end{equation*} 
and in particular
\begin{equation} \label{eq:massachusetts}
\mc J \leq \frac{2r^{1+1/\alpha}}{\alpha^{1/\alpha}} e^{-\delta_{\rm dis} s_0 / 2}\,.
\end{equation}
Plugging \eqref{eq:massachusetts} into \eqref{eq:alaska}, we deduce
\begin{equation} \label{eq:rhodeisland}
\frac{\p_s}{2} (E_{2K}^2 ) \leq -K \eta E_{2K}^2 + \mt C \bar{E}^2  \leq \mt C \left( -5E_{2K}^2 + \bar E^2 \right).
\end{equation}
for some universal constant $\mt C$. In the second inequality we are using that $K$ can be taken large enough in terms of $\mt C, \eta$, in particular, such that $K > \frac{5 \mt C}{\eta}$.

Finally, as $E_{2K}(s_0) \leq \bar E/2$, it is clear that $E_{2K}(s)$ can never go above $\bar E/2$, as it would contradict \eqref{eq:rhodeisland}. Therefore, $E_{2K}(s) \leq \bar E/2$ and we conclude \eqref{eq:improvement_E}. This finishes the proof of Proposition \ref{prop:bootstrap}.

\subsection{Topological argument for the initial unstable coefficients}\label{sec:top}

In this subsection, we prove Proposition \ref{prop:topological}. In particular, we will always assume \eqref{eq:initialdataassumptions} and \eqref{eq:extrahypothesis}. Whenever we have \eqref{eq:kappa:bnd}, we may apply Proposition \ref{prop:bootstrap} and thus deduce \eqref{eq:bootstrap_hyp}. In such cases, we are in the hypotheses of the previous subsection, so we may use any result of that subsection.

For this subsection, we will assume that all the variables make reference to the truncated equation, and we will use the subscript 'e', whenever we want to make reference to the extended equation.

Let us define 
\begin{equation*}
\kappa_i (s) = \left\langle (\mt U_t (\cdot, s), \mt S_t (\cdot, s) ), (\psi_{i, U}, \psi_{i, S} ) \right\rangle, \qquad \kappa (s) = \sum_{i=1}^N (\psi_{i, U}, \psi_{i, S} ) \kappa_i(s)\,.
\end{equation*}
Note that as $(\psi_{i, U}, \psi_{i, S})$ form an orthonormal base of $V$ (by Corollary \ref{cor:findimbon}) the norm $\sqrt{\sum_i \kappa_i^2}$ is just the norm inherited from $X$, which we denote by $| \kappa |_X$. We will also work with another norm on $\kappa$. We define the metric $B$ to be the canonical metric associated to the basis of item  \ref{item:stability_outwards} in Lemma \ref{lemma:abstract_result} (where we take $A = \mc L$ due to Corollary \ref{cor:A0}). In particular, equation \eqref{eq:chisinau} yields
\begin{equation} \label{eq:yoplait}
\langle \mc L \kappa (s), \kappa (s) \rangle_B \geq \frac{-6\delta_g}{10} | \kappa (s) |_B^2\,.
\end{equation}
As any two norms are equivalent on a finite dimensional vector space, we have that $| \cdot |_X$ and $| \cdot |_B$ are equivalent norms. Moreover, both depend only on the space $X$ and $\delta_g$, so we have that $| w |_X \les_m | w |_B \les_m | w (s)|_X$

It will be useful to consider the following exponentially contracting regions for the unstable modes:
\begin{align} \begin{split} \label{eq:Rdef}
\mc R(s) &= \{ w \in \mathbb{R}^N \mid | w |_X \leq \delta_1 e^{-\frac{7}{10} \delta_g (s-s_0)}\} \, ,\\
\mt{ \mc R }(s) &= \{ w \in \mathbb{R}^N \mid | w  |_B \leq \delta_1^{ \frac{11}{10} } e^{- \frac{7}{10}\delta_g (s-s_0)}\}\,,
\end{split} \end{align}
where we clearly have $\mt{\mc R} (s) \Subset \mc R(s)$ as $|w|_X \les_m |w|_B$.

\begin{lemma} \label{lemma:nle} Provided that $\kappa (s) \in \mc R (s)$ for $s \in [s_0, s_1]$ and assuming our initial data hypothesis 
\begin{equation}
\label{eq:hylia}
\| (\mt u_0, \mt \sigma_0) \|_X \leq \delta_1\,,
\end{equation}
 we have that 
\begin{align} \label{eq:maracas1}
\| (\mt u, \mt \sigma) (\cdot, s ) \|_X &< 3\delta_1 e^{- \frac12 \delta_g (s-s_0)} \,,\\
\left\| \mc F_{\mathrm{nl}} (\cdot, s) \right\|_{X} &< \delta_1^{3/2} e^{- \frac{9}{10} \delta_g (s-s_0)}\,.\label{eq:maracas2}
 \end{align}
\end{lemma}
\begin{proof}
First of all, note that $\kappa (s) \in \mc R(s)$ implies \eqref{eq:kappa:bnd}, so we have \eqref{eq:bootstrap_hyp} and we are in the hypothesis of the previous subsection.

We claim that
\begin{equation} \label{eq:maracas3}
\| P_{\rm sta}(\mt u, \mt \sigma) (\cdot, s ) \|_X < 2\delta_1 e^{- \frac12 \delta_g (s-s_0)}.
\end{equation}
It is clear that \eqref{eq:maracas3} implies \eqref{eq:maracas1}, since $\| P_{\rm uns}(\mt u, \mt \sigma) \|_X = |\kappa (s) |$, which is smaller than $\delta_1 e^{- \frac{7}{10} \delta_g (s-s_0)}$ by hypothesis. Therefore, it suffices to show \eqref{eq:maracas2}--\eqref{eq:maracas3}. Assume by contradiction that either \eqref{eq:maracas2} or \eqref{eq:maracas3} are violated at some first time $s' \leq s_1$. \\

\textbf{Case 1: \eqref{eq:maracas3} breaks at $s=s'$.} Let $s \in [s_0, s']$. By \eqref{eq:Novak:Exemption}, $ \delta_1 \ll \delta_g \ll \delta_{\rm dis}$, continuity on \eqref{eq:maracas2} and that $s_0$ is chosen sufficiently large dependent on $\delta_1$, we have
\begin{align} \begin{split} \label{eq:tambor1}
\left\| \mc F_{\rm dis} (\cdot, s) \right\|_{X} &\leq \delta_1 e^{-\delta_{\rm dis} s/2} \leq \delta_1^{3/2} e^{- \frac{9}{10}\delta_g (s-s_0)}\,, \\
\left\| \mc F_{\mathrm{nl}} (\cdot, s) \right\|_{X} &\leq \delta_1^{3/2} e^{- \frac{9}{10} \delta_g (s-s_0)}.
\end{split} \end{align}

On the other hand, as the stable space $V^{\ast \, \perp}$ is invariant under $\mc L$ (by Lemma \ref{lemma:abstract_result}) we get the commutation relation $P_{\rm sta} \mc L = \mc L P_{\rm sta}$. Therefore, taking the projection $P_{\rm sta}$ on \eqref{eq:ohare}, we get 
\begin{equation*}
\p_s P_{\rm sta}(\mt u, \mt \sigma) = \mc L P_{\rm sta}( \mt u, \mt \sigma) + P_{\rm sta} \mc F_{\mathrm{nl}} + P_{\rm sta} \mc F_{\rm dis}\,.
\end{equation*}
By Duhamel, we get
\begin{equation*}
P_{\rm sta}(\mt u (\cdot, s), \mt \sigma (\cdot, s)) = T(s-s_0)P_{\rm sta} (\mt u_0, \mt \sigma_0) + \int_{s_0}^s T(s-\mw s) P_{\rm sta}\mc F_{\rm nl} (\cdot, \mw s) d\mw s +  \int_{s_0}^s T (s-\mw s) P_{\rm sta}\mc F_{\rm dis} (\cdot, \mw s) d\mw s\,.
\end{equation*}
Now, using the bound on $T(s)$ over $V_{\rm sta}$ (equation \eqref{eq:prim}), we get
\begin{align*}
\| P_{\rm sta}(\mt u (\cdot, s), \mt \sigma (\cdot, s)) \|_X &\leq 
e^{- \frac12 \delta_g (s-s_0)} \|(\mt u_0, \mt \sigma_0)\|_X
 + \int_{s_0}^s e^{- \frac12 \delta_g (s-\mw s)} \left( \| \mc F_{\rm nl} (\cdot, \mw s)\|_X + \| \mc F_{\rm dis} (\cdot, \mw s) \|_X \right) d\mw s \\
 &\leq e^{- \frac12 \delta_g (s-s_0)}\delta_1 + 2 \delta_1^{3/2} \int_{s_0}^s e^{-\delta_g \left( \frac12 (s-\mw s)+ \frac{9}{10}(\mw s-s_0) \right) } d\mw s \\
 &\leq e^{- \frac12 \delta_g (s-s_0)}\delta_1 + 2 \delta_1^{3/2} e^{- \frac12 \delta_g(s-s_0)} \int_{s_0}^\infty e^{- \frac{4}{10} \delta_g (\mw s-s_0)/2} d\mw s \\
 &\leq \left( 1 + \frac{5\delta_1^{1/2}}{\delta_g} \right)\delta_1 e^{- \frac12 \delta_g (s-s_0)} \leq \frac32 \delta_1 e^{- \frac12 \delta_g (s-s_0)}\,.
\end{align*}
In the second inequality, we used \eqref{eq:tambor1} and \eqref{eq:hylia}. Clearly this last inequality at $s=s'$ implies \eqref{eq:maracas3}, so that we arrive to contradiction, because we supposed \eqref{eq:maracas3} broke at time $s=s'$.

\textbf{Case 2: \eqref{eq:maracas2} breaks at $s=s'$.} By assumption, we have
\begin{equation} \label{eq:cazadora0}
\| \mc F_{\mathrm{nl}} (\cdot, s') \|_X = \delta_1^{3/2} e^{ \frac{9}{10}\delta_g (s'-s)}.
\end{equation}
By continuity on \eqref{eq:maracas3} we also have
\begin{equation*}
\| P_{\rm sta}(\mt u, \mt \sigma) (\cdot, s' ) \|_X \leq \delta_1 e^{- \frac12 \delta_g (s'-s_0)}\,,
\end{equation*}
so using that $| \kappa (s') | \leq \delta_1 e^{-\frac{7}{10}\delta_g (s'-s_0)}$ by hypothesis, we get
\begin{equation}\label{eq:cazadora_trescuartos}
\| (\mt u, \mt \sigma) (\cdot, s' ) \|_X \leq 2\delta_1 e^{- \frac12 \delta_g (s'-s_0)}\,.
\end{equation}
Interpolating between \eqref{eq:capuccino} and \eqref{eq:cazadora_trescuartos}, we get that
\begin{equation} \label{eq:cazadora2}
\left( \| \mt u \|_{W^{2m+1, \infty}} +  \| \mt \sigma \|_{W^{2m+1, \infty}}  \right)^2 
\les  \left( \delta_1 e^{-\frac12 \delta_g(s'-s_0)} \right)^{2 \left( 1- \frac{2m+1}{2K-2} \right)}  \bar E ^{2\frac{2m+1}{2K-2}} 
\les \delta_1^{9/5}  e^{-\frac{9}{10}\delta_g(s'-s_0)}\,,
\end{equation}
where we used $m \ll K$ and $\delta_1\ll \frac{1}{\bar E} $ in the last inequality.

From \eqref{eq:bangladesh}, we obtain that 
\begin{equation} \label{eq:cazadora1}
\| \mc F_{\mathrm{nl}}  \|_{W^{2m, \infty}} \les \left( \| \mt u \|_{W^{2m+1, \infty}} +  \| \mt \sigma \|_{W^{2m+1, \infty}}  \right)^2  \les \delta_1^{ \frac{9}{5}} e^{- \frac{9}{10} \delta_g (s'-s_0)}\,.
\end{equation}
This clearly contradicts \eqref{eq:cazadora0}. 
\end{proof}

As the unstable space $V$ is invariant under $\mc L$ (by Lemma \ref{lemma:abstract_result}, with $A = \mc L$ from Corollary \ref{cor:A0}), we get the commutation relation $P_{\rm uns} \mc L = \mc L P_{\rm uns}$. Therefore, taking the projection $P_{\rm uns}$ on \eqref{eq:ohare}, we get that $\kappa$ satisfies the ODE
\begin{equation} \begin{cases}  \label{eq:trieste}
\kappa'(s) &= \mc L |_{V} \kappa (s) + P_{\rm uns} \mc F_{\rm dis} + P_{\rm uns} \mc F_{\mathrm{nl}}\,, \\
\kappa_i (s_0) &= a_i\,.
\end{cases} \end{equation}
for some initial values $a_i$. 

For the solution corresponding to the initial data given in \eqref{eq:US:initial:data}, for $|a| \leq \delta_1$, we define the stopping time
\[s_a = \inf_{s>s_0} \{ s: \kappa (s)\notin \mc R \}\,.\]

\begin{lemma}[Outgoing property]  \label{lemma:outgoing} Let us suppose $|a_i| \leq \delta_0$ and that $\kappa \in \mt {\mc R}$ at times $s\in [s_0, s_1]$ and that at time $s_1 \geq s_0$ we have 
\begin{equation} \label{eq:alewife}
| \kappa (s_1) |_B = \delta_1^{\frac{11}{10}} e^{- \frac{7}{10} \delta_g (s_1-s_0)},
\end{equation}
that is, $\kappa (s_1) \in \p ( \mt{\mc R} (s_1) )$. Then, we have that $\kappa (s) \notin \mt{\mc R} (s)$ for $s$ close enough to $s_1$ from above. That is, $\kappa (s)$ exits $\mt{\mc R}(s)$ at $s = s_1$.
\end{lemma}
\begin{proof} 
Given that $\kappa (s_1) \in \p (\mt{\mc R}(s_1))$, we have that $\kappa (s)$ will be exiting $\mt{\mc R} (s)$ at time $s = s_1$ if and only if
\begin{equation} \label{eq:kendall}
  \langle \kappa'(s_1) , \kappa (s_1) \rangle_B  > - \frac{7}{10} \delta_g | \kappa (s_1)|_B^2,
\end{equation}
which using \eqref{eq:trieste} can be written as
\begin{equation} \label{eq:kendall2}
 \left( \langle  \mc L|_{V} \kappa (s_1), \kappa (s_1) \rangle_B +  \langle P_{\rm uns} \mc F_{\rm dis}, \kappa (s_1) \rangle_B + \langle P_{\rm uns} \mc F_{\mathrm{nl}}, \kappa (s_1) \rangle_B \right)  > - \frac{7}{10} \delta_g | \kappa (s_1) |_B^2\,.
\end{equation}

On the one hand, we have from equation \eqref{eq:yoplait} that
\begin{equation} \label{eq:southstation}
 \langle \mc L|_{V} \kappa (s_1), \kappa (s_1) \rangle_B \geq \frac{-6\delta_g}{10} | \kappa |_B^2\,.
\end{equation}
                                  
On the other hand, we have that
\begin{align}
 \langle P_{\rm uns} \mc F_{\rm dis}, \kappa (s_1) \rangle + \langle P_{\rm uns} \mc F_{\mathrm{nl}} (\cdot, s_1), \kappa (s_1) \rangle
& \leq  \left\| P_{\rm uns} \mc F_{\rm dis}  (\cdot, s_1) \right\|_B |\kappa (s_1)|_B  + \| P_{\rm uns} \mc F_{\mathrm{nl}} \|_B |\kappa (s_1)|_B \notag \\
& \les_m  \left( \left\| P_{\rm uns} \mc F_{\rm dis}  (\cdot, s_1) \right\|_X   + \| P_{\rm uns} \mc F_{\mathrm{nl}}  (\cdot, s_1) \|_X \right) \frac{| \kappa (s_1) |_B^2}{\delta_1^{\frac{11}{10}} e^{- \frac{7}{10}\delta_g (s_1-s_0)}} \notag \\
& \leq  \left( \left\| \mc F_{\rm dis}  (\cdot, s_1) \right\|_X   + \| \mc F_{\mathrm{nl}}  (\cdot, s_1) \|_X \right) \frac{| \kappa (s_1) |_B^2}{\delta_1^{\frac{11}{10}} e^{-\frac{7}{10}\delta_g (s_1-s_0)}} \notag \\
& \leq \delta_1^{\frac32} e^{- \frac{9}{10} \delta_g (s_1-s_0)} \frac{| \kappa (s_1) |_B^2}{ \delta_1^{\frac{11}{10}} e^{- \frac{7}{10} \delta_g (s_1-s_0)} } \notag \\
&= \delta_1^{\frac{4}{10}} e^{- \frac{2}{10}\delta_g (s_1-s_0)} | \kappa (s_1) |_B^2\,,  \label{eq:ashmont}
\end{align}
where we used equation \eqref{eq:alewife} in the second line and Lemma \ref{lemma:forcingbounds} in the fourth one.

Combining equations \eqref{eq:southstation} and \eqref{eq:ashmont}, and choosing  $\delta_1$ sufficiently small dependent on  $\delta_g$, we conclude equation \eqref{eq:kendall2}.
\end{proof}

\begin{proposition}  \label{prop:topological_2} 
There exists specific initial conditions $a_i$ such that $| a |_X \leq \delta_1$ and the $\kappa (s)$ defined by the ODE \eqref{eq:trieste} satisfies that $\kappa (s) \in \mc R$ for all $s \geq s_0$. 

Moreover, for such initial conditions $a_i = \kappa_i(s_0)$, we have that 
\[\| (\mt u, \mt \sigma) \|_X \lesssim \delta_1 e^{- \frac{7}{10} \delta_g (s-s_0)}\,.\]
\end{proposition}
\begin{proof}
We argue by contradiction, so let us suppose that for any such initial conditions $a_i$, there exists a time of exit $s_a$ such that $\kappa_i$ exits $\mt{\mc R}$ after $s = s_a$ and let us define $b_i(a) = \kappa_i (s_a)$. Due to Lemma \ref{lemma:outgoing} we can equivalently define $(b_i, s_a)$ letting 
\[s_a = \inf_{s} \{ s: \kappa (s)\notin \mt{\mc R} (s) \}\quad\mbox{and}\quad b_i=\kappa_i(s_a)\,.\]

First, let us argue that the mapping $a\rightarrow s_a$ is continuous (in the domain where $s_a \neq +\infty$). Let $a, \bar a \in \mt{\mc R} (s_0)$ so that $s_a, s_{\bar a}$ are finite, and let $\kappa (s)$ and $\bar \kappa (s)$ be the unstable modes corresponding to the solutions with initial conditions dictated by $a$ and $\bar a$. By local existence of solutions \cite{Da07}, we can extend $\kappa (s)$ up to some time $s_a + \eps$ for some $\eps$ sufficiently small. Also, by Lemma \ref{lemma:outgoing} (and again taking $\eps$ sufficiently small), we can ensure that $\kappa (s_a + \eps) \notin \mt{\mc R} (s_a + \eps)$. Let $\eps' = \frac12 \mathrm{d}_B (\kappa (s_a + \eps), \mt{\mc R}(s_a + \eps))$. By stability \cite{Da07}, there exists $\delta$ sufficiently small such that $|a-\bar a |_B \leq \delta$ guarantees $|\kappa (s_a + \eps) - \bar \kappa (s_a + \eps)|_B \leq \frac12 \eps'$. In that case, $\bar \kappa (s_a + \eps) \notin \mt{\mc R}(s_a + \eps)$ and we deduce $s_{\bar a} < s_a + \eps$. In a symmetric way, we can find $\bar \delta$ such that if $|a - \bar a |_B \leq \bar \delta$, then $s_a < s_{\bar a} + \eps$. Thus, we conclude that $a \rightarrow s_a$ is continuous. As a consequence, $b_i = \kappa_i (s_a)$ is also continuous with respect to $a$, as it is the composition of continuous functions.

Now, we define the mapping $H : B_B (1) \rightarrow \p B_B (1)$ as follows. For each $\vartheta \in V$ with $| \vartheta |_B \leq 1$ consider $a = \delta_1^{\frac{11}{10}} \vartheta$ and evolve \eqref{eq:trieste} with $\kappa(s_0) = a$. As we are supposing every initial data exits $\mt{\mc R}(s)$ at some $s$, we have a time $s_a$ and the corresponding values $b_i (a)$ such that $|b(a)|_B = \delta_1^{\frac{11}{10}} e^{-\delta_g (s_a-s_0)} $. We now consider $H(\vartheta)$ to be $\frac{b(a)}{|b(a)|_B}$, which is trivially on the boundary of $B_{B} (1)$. As $b(a)$ is continuous, we also get that $H$ is continuous. Moreover, note that $H$ is just the identity on the boundary of the ball, as $|a|_B = \delta_1^{\frac{11}{10}}$ implies $s_a = s_0$ and $b(a) = a$ due to Lemma \ref{lemma:outgoing}.

Therefore, we have constructed a mapping $H$ from the the unit ball on $V$ to its boundary which is continuous and is the identity restricted to its boundary. The map $-H$ would therefore have no fixed point, contradicting Brouwer's fixed point theorem. Therefore, there must exist at least a $N$-tuple of values for $a_i$ such that $\kappa (s) \in \mt{\mc R}(s)$ for all $s \geq s_0$.

Now, with that value for $a_i$, we conclude the unstable projection of $\mt u, \mt \sigma$ in $X$ are always bounded by $\delta_1 e^{- \frac{7}{10}\delta_g (s-s_0)}$. Therefore, for such $a_i$, we have \eqref{eq:kappa:bnd}, so we can apply Proposition \ref{prop:bootstrap} and obtain \eqref{eq:bootstrap_hyp}, which gives us the desired bounds. 
\end{proof}

From Proposition \ref{prop:topological_2}, we conclude the bound \ref{eq:roy} on Proposition \ref{prop:topological}. Moreover, by Proposition \ref{prop:bootstrap}, we see that $E_{2K}$, $\| \mt U \|_{L^\infty}$ and $\| \mt S \|_{L^\infty}$ remain bounded uniformly in time. Thus, the local-wellposedness results \cite{Da07}  imply that the solution to \eqref{eq:US:tilde} can be continued for all times. As \eqref{eq:CE2} is a linear equation, and the forcing remains bounded due to Lemma \ref{lemma:forcingbounds}, it can also be continued for all times.

Thus, the only thing that remains to show from Proposition \ref{prop:topological} is equation \eqref{eq:ike}. Let us show it. In the region $\zeta \leq \frac65$, Lemma \ref{lemma:agree} gives us that $\mt U_t = \mt U_e$ and $\mt S_t = \mt S_e$, so
\begin{equation*}
\| \mt U_e \|_{L^\infty [0, 6/5]} + \| \mt S_e \|_{L^\infty [0, 6/5]}
\leq \| \mt U_t \|_{L^\infty } + \| \mt S_t \|_{L^\infty}
\les \| (\mt U_t, \mt S_t) \|_X\,.
\end{equation*}
Therefore, for $\zeta \leq \frac65$, we obtain equation \eqref{eq:ike} from \eqref{eq:roy}.

Now let us show equation \eqref{eq:ike} for $\zeta > \frac65$. From now on, we will only refer to the extended equation, so we drop the subindex `e'. The approach will be similar to the one we followed for closing the $L^\infty$ estimates. Let us change to $W, Z$ variables and recall the definitions of $J(\zeta), \mw \Upsilon_W, \mw \Upsilon_Z$ from \eqref{eq:manzanas}--\eqref{eq:crisalidaZ}. Whenever we want to indicate the initial condition of $\Upsilon_\circ (s)$, we will use the notation $\Upsilon_\circ^{(\zeta_\star, s_\star)}(s)$ which is the only trajectory such that $\Upsilon_\circ (s_\star) = \zeta_\star$. Let us also recall that there exists some $\zeta_1$ such that $| J(\zeta) | \leq \frac{r-1}{4}$ for $\zeta > \zeta_1$.

For the region $\frac65 \leq \zeta \leq \zeta_1$, we argue by contradiction. Let us note that in this region equation \eqref{eq:mango} holds for a sufficiently large $C_2$. Let $C_1, C_3$ be sufficiently large constants. We claim
\begin{equation} \label{eq:naruhodo}
| \mt W (\zeta, s ) |, | \mt Z (\zeta, s) | < C_1\delta_1 e^{- \frac{7}{10} \delta_g (s-s_0)} e^{C_3 \zeta} ,
\end{equation}
in the region $\frac65 \leq \zeta \leq \zeta_1$. Note that if $C_1$ is taken sufficiently large, this is satisfied at $s=s_0$ because of our initial conditions hypothesis and at $\zeta = \frac65$, because we already have \eqref{eq:ike} in the region $\zeta \leq \frac65$. In order to show it for other $\zeta$ and $s$, we argue by contradiction and suppose that \eqref{eq:naruhodo} holds until time $s_b> s_0$, and is broken at some point $\zeta_b \in \left( \frac65 , \zeta_1 \right]$. Without loss of generality, we have that
\begin{equation} \label{eq:asogi}
\mt W( \zeta_b, s_b ) = C_1 \delta_1 e^{C_3 \zeta_b} e^{- \frac{7}{10}  \delta_g (s_b-s_0)} \qquad \mbox{ and } \qquad | \mt Z (\zeta_b, s_b) | \leq C_1 \delta_1 e^{C_3 \zeta_b} e^{- \frac{7}{10}  \delta_g (s_b-s_0)},
\end{equation}
because the cases where \eqref{eq:naruhodo} breaks with $\mt W$ being negative or where it breaks for $\mt Z$ are analogous. Let $(\zeta_\star, s_\star)$ such that either $\zeta_\star = \frac65 $ or $s_\star = s_0$ and $\mw \Upsilon_W^{(\zeta_\star, s_\star)} (s_b) = \zeta_b$. As $s_b$ is the first time at which equation \eqref{eq:naruhodo} breaks, we have that
\begin{align} 
\p_s \left( \mt W (\mw \Upsilon_W^{(\zeta_\star, s_\star)}, s) \right) \Big|_{s = s_b} &\geq \p_s \left( C_1 \delta_1 e^{-\frac{7}{10}  \delta_g  (s-s_0)} e^{C_3 \mw \Upsilon_W^{(\zeta_\star, s_\star)} }\right) \Big|_{s=s_b} \notag \\
&\geq C_1 \delta_1 e^{- \frac{7}{10}  \delta_g (s_b-s_0)} e^{C_3 \mw \Upsilon_W^{(\zeta_\star, s_\star)} (s_b)} \left( \frac{7}{10}  \delta_g + C_3 \frac{\mw \Upsilon_W^{(\zeta_\star, s_\star)}(s_b)  }{C_0} \right), \notag\\
&\geq C_1 \delta_1 e^{-\frac{7}{10}  \delta_g (s_b-s_0)} e^{C_3 \mw \Upsilon_W^{(\zeta_\star, s_\star)} (s_b)}  \cdot \frac{C_3 }{C_0}  \cdot \frac65 \,,\label{eq:susato}
\end{align}
where in the second inequality we used equation \eqref{eq:dallas} and in the third one $\mw \Upsilon_W^{(\zeta_\star, s_\star)} \geq \frac65$ (by \eqref{eq:dallas}). 

On the other hand, combining \eqref{eq:mango} with \eqref{eq:asogi}, we see that
\begin{equation} \label{eq:natsume}
\left| \p_s |_{s=s_b}\mt W \left( \mw \Upsilon_W^{(\zeta_\star, s_\star)} (s), s) \right)  \right| \leq (C_1 + 2)\delta_1  e^{ \frac{7}{10} \delta_g (s_b-s_0)} e^{C_3 \mw \Upsilon_W^{(\zeta_\star, s_\star)} (s_b)}\,.
\end{equation}
Now, comparing \eqref{eq:susato} and \eqref{eq:natsume}, and taking $C_3$ sufficiently large with respect to $C_1$, we arrive to contradiction. Therefore, we conclude \eqref{eq:naruhodo} for some constants $C_1, C_3$. As $C_3, \zeta_1$ are constants, $e^{C_3 \zeta_1} \les 1$, so this shows \eqref{eq:ike} in the region $\frac65 \leq \zeta \leq \zeta_1$.

Finally, we need to show \eqref{eq:ike} in the region $\zeta > \zeta_1$. We will do it by bootstrap. Let $C_4$ be a sufficiently large constant and let us assume that
\begin{equation} \label{eq:laxe}
\| \mt W \|_{L^\infty [\zeta_1, \infty)} + \| \mt Z \|_{L^\infty [\zeta_1, \infty)} \leq 2C_4 \delta_1 e^{- \frac{7}{10} \delta_g (s-s_0)}\,.
\end{equation} 
We will show the reinforced estimate
\begin{equation} \label{eq:corme}
\| \mt W \|_{L^\infty [\zeta_1, \infty)} + \| \mt Z \|_{L^\infty [\zeta_1, \infty)} \leq C_4 \delta_1 e^{- \frac{7}{10} \delta_g (s-s_0)}\,,
\end{equation}
which implies that \eqref{eq:laxe} cannot be broken. From equations \eqref{eq:uruguay}--\eqref{eq:manzanas}, we get that
\begin{equation} \label{eq:cambados}
\left|  (\p_s +r-1)\mt W 
+ \left( \zeta+ \frac{1+\alpha}{2} \mw W + \frac{1-\alpha}{2} \mw Z \right) \p_\zeta \mt W  \right| \leq J(\zeta) \left( | \mt  W | + | \mt Z | \right) + \mc F_{\rm dis} + \mc F_{\mathrm{nl}, W}\,.
\end{equation}
Using equation \eqref{eq:laxe} in \eqref{eq:cambados}, and recalling that $J(\zeta) \leq \frac{r-1}{4}$ in the region $\zeta > \zeta_1$, we obtain that
\begin{equation} \label{eq:rianxo}
\left|  (\p_s +r-1)\mt W 
+ \left( \zeta+ \frac{1+\alpha}{2} \mw W + \frac{1-\alpha}{2} \mw Z \right) \p_\zeta \mt W  \right| \leq \frac{r-1}{2} \delta_1 C_4 e^{- \frac{7}{10} \delta_g (s-s_0)} + \mc F_{\rm dis} + \mc F_{\mathrm{nl}, W}\,.
\end{equation}

We recall from \eqref{eq:Novak:Exemption} that
\begin{equation} \label{eq:padron}
| \mc F_{\rm dis} | \les \delta_1 e^{-\delta_{\rm dis} s / 2} \leq \delta_1^2 e^{-\delta_{\rm dis} (s-s_0)/2}\,.
\end{equation}
In addition, as $\zeta \geq \zeta_1$, we have from \eqref{eq:mattmurdock} that
\begin{align}
| \mc F_{\mathrm{nl}, W} | & \les \left( | \mt W | + | \mt Z | \right) \left(  | \mt W | + | \p_\zeta \mt W | \right) \notag \\
& \leq \delta_1 e^{- \frac{7}{10} \delta_g (s-s_0)}  \cdot \left(  | \mt W | + | \p_\zeta \mt W | \right), \label{eq:noia}
\end{align}
where we bounded the first factor by \eqref{eq:laxe}. Now from \eqref{eq:cafeconleche} we get that $|\p_\zeta^{2K-1} \mt W| \les \bar E$ and from \eqref{eq:laxe} we get that $| \mt W | \leq \delta_1$. Applying $L^\infty$ interpolation between those two bounds, we get 
\[ | \p_\zeta \mt W | \les \delta_1^{\frac{2K-2}{2K-1}} \bar E^{\frac{1}{2K-1}} \ll \delta_1^{1/2}\,. \]
Plugging this into \eqref{eq:noia}, we obtain that
\begin{equation} \label{eq:baiona}
| \mc F_{\mathrm{nl}, W} | \les \delta_1^{3/2} e^{- \frac{7}{10} \delta_g (s-s_0)}\,.
\end{equation}
Using equations \eqref{eq:padron} and \eqref{eq:baiona} in \eqref{eq:rianxo}, we obtain that
\begin{equation*}
\left|  (\p_s +r-1)\mt W 
+ \left( \zeta+ \frac{1+\alpha}{2} \mw W + \frac{1-\alpha}{2} \mw Z \right) \p_\zeta \mt W  \right| \leq 
\frac{r-1}{2} \delta_1 C_4 e^{- \frac{7}{10} \delta_g (s-s_0)} + \delta_1^{3/2} e^{- \frac{7}{10} \delta_g (s-s_0)}.
\end{equation*}
Recalling the definition of $\mw \Upsilon_W$ from \eqref{eq:crisalidaW}, we can rewrite this equation as
\begin{equation} \label{eq:ribeira}
\left|  (\p_s +r-1)\mt W (\mw \Upsilon_W (s), s)  \right| \leq \left( \frac{r-1}{2} + \frac{\delta_1^{1/2}}{C_4} \right) C_4 \delta_1 e^{- \frac{7}{10} \delta_g (s-s_0)} \leq \frac{3(r-1)}{5}  C_4 \delta_1 e^{- \frac{7}{10} \delta_g (s-s_0)}\,.
\end{equation}
Now, we conclude \eqref{eq:corme} by contradiction. Equation \eqref{eq:corme} clearly holds for $s = s_0$, and also holds for $\zeta = \zeta_1$ because we know \eqref{eq:ike} in the region $\frac65 \leq \zeta \leq \zeta_1$. Moreover, if \eqref{eq:corme} breaks for some trajectory at $s_b$, without loss of generality, we would have
\begin{equation*}
\p_s \mt W (\mw \Upsilon_W (s)) \Big|_{s=s_b} \geq \p_s \left( C_4 \delta_1 e^{- \frac{7}{10} \delta_g(s-s_0)} \right) \Big|_{s=s_b} \geq -\frac{r-1}{4} \delta_1 C_4 e^{- \frac{7}{10} \delta_g(s_b-s_0)}\,.
\end{equation*}
On the other hand, \eqref{eq:ribeira} implies
\begin{equation*}
\p_s \mt W (\mw \Upsilon_W (s)) \Big|_{s=s_b} \leq -(r-1) C_4 \delta_1 e^{- \frac{7}{10} \delta_g(s_b-s_0)} +\frac{3(r-1)}{5} C_4 \delta_1 e^{- \frac{7}{10} \delta_g (s-s_0)} = - \frac{2(r-1)}{5} C_4 \delta_1 e^{- \frac{7}{10} \delta_g (s-s_0)}\,,
\end{equation*}
so we obtain the desired contradiction. The cases where \eqref{eq:corme} breaks with $\mt W$ being negative, or where it breaks for $\mt Z$ are handled completely analogously.

Therefore, we have shown \eqref{eq:corme} under the assumption of \eqref{eq:laxe}, which allows us to conclude \eqref{eq:corme} by bootstrap. This proves \eqref{eq:ike} for $\zeta \geq \zeta_1$, so it finalizes the proof of Proposition \ref{prop:topological}.

\appendix

\section{Auxiliary Lemmas} \label{sec:aux}

\subsection{Proof of Proposition \ref{prop:existence}}\label{sec:rapid:antigen}
\begin{proof}[Proof of Proposition \ref{prop:existence}] Let us start by proving that there are no non-degenerate closed curves $C$ (not crossing the nullsets of the denominators) such that the field $\left( \frac{N_W}{D_W} ,  \frac{N_Z}{D_Z} \right)$ is tangent with constant direction to the curve at each point of $C$. If there was such a curve, it would also be tangent to the polynomial field $\tilde F = (N_W D_Z , N_Z D_W )$.

By Poincar\'e-Hopf theorem, there has to be some equilibrium point $P$ in the interior of $C$. Moreover, as $C$ does not cross $D_W = 0$ or $D_Z = 0$, any equilibrium point inside $C$ has to be a solution of $N_W = N_Z = 0$. Thus $C$ has to be in the region $W > Z$, by Lemma \ref{lemma:equilibrium_points}, the only such point is $P_\eye$ and it is a saddle point for all $\ga > 1$, $r < \frac{3\ga-1}{2+\sqrt{3}(\ga-1)}$. By Poincar\'e-Hopf theorem, a closed orbit cannot contain just a saddle equilibrium point, so we conclude there is no such  curve $C$ for all $\ga > 1$, $r < \frac{3\ga-1}{2+\sqrt{3}(\ga-1)}$. Note that by Lemma \ref{lemma:aux_otromas} this covers the case $\ga > 1$, $r \in (r_3, r_4)$.

Now, let us show that such curve $C$ does not exist also in the case $\ga = 7/5$ and $r$ close enough to $r^\ast$. By the same reasoning as above, $C$ has to encircle the point $P_\eye$ (which is no longer a saddle point). By Lemma \ref{lemma:equilibrium_points}, we have that $P_\eye$ lies in the region $D_W > 0, D_Z < 0$, and as $C$ does not intersect $D_W = 0$ or $D_Z = 0$, we have that $C$ is also contained in the region $\Omega$ (that is, the region where $D_W > 0, D_Z < 0$). Now, we let $P_\eye'$ to be the intersection of the branch of $N_W = 0$ passing through $P_\eye$ with the nullset $D_Z = 0$. We also let $P_\eye''$ to be the point in the same horizontal of $P_\eye$ which lies over $D_Z = 0$. We define the region $\mc T$ to be the triangular region enclosed by $N_W = 0$, the horizontal segment from $P_\eye$ to $P_\eye''$ and $D_Z = 0$. We call those parts of $\p \mc T$ by $S_1, S_2, S_3$ respectively. As $\mc T$ is a region from $P_\eye$ to $D_Z = 0$, and our curve $C$ encloses $P_\eye$ and stays in $D_Z < 0$, it necessarily has to pass through $\mc T$ and has to cross either $S_1$ or $S_2$ in the outwards direction. We get our final contradiction by Lemma \ref{lemma:lemon} which asserts that the field $(N_W D_Z, N_Z D_W)$ points inwards to $\mc T$ both in $S_1$ and $S_2$ for $\ga = 7/5$ and $r$ sufficiently close to $r^\ast$.

By our previous reasoning, the orbits of our system are the orbits of the modified field $(N_W D_Z, N_Z D_W )$, as long as the orbit does not intersect the nullset of $D_W$ or $D_Z$. By Picard-Lindel\"of's theorem, those trajectories exists locally. Moreover, by Poincar\'e-Bedixson (and the previous fact that there are no periodic orbits), every bounded semitrajectory converges to an equilibrium point.

However, those trajectories may intersect $D_W = 0$ or $D_Z = 0$, which give singularities for the change of variables between the fields $(N_W/D_W, N_Z/D_Z)$ and $(N_W D_Z, N_Z D_W)$. Therefore, we add the possibility that the trajectory of the original system intersects those nullsets. We have thus seen that the trajectories (from left or right) either are unbounded, or converge to an equilibrium point, or converge to a point of the nullsets $D_W = 0, D_Z = 0$. 

In the case of reaching an equilibrium, as all equilibria are hyperbolic, Hartman-Grobman ensures that the rate of convergence is exponentially fast, reaching the equilibrium in infinite time.
\end{proof}

\subsection{Interpolation Lemmas}

\begin{lemma} \label{lemma:thefinalinterpolation} Let $a \geq 1$ and consider $\mc O = [a, \infty)$. Let $\beta_1 \in (0, 1)$.
Then, for any $0\leq i\leq n$, $\beta_2\in\mathbb R$, 
 we have the interpolation inequality
\begin{equation} \label{eq:murcia}
\left\| x^{i \beta_1+\beta_2} \p^i_x f(x) \right\|_{L^\infty (\mc O)} \lesssim_n \left\| x^{\beta_2} f(x) \right\|_{L^\infty (\mc O) }^{1-i/n} \left\| x^{n \beta_1+\beta_2} \p^n_x f(x) \right\|_{L^\infty (\mc O) }^{i/n} + \left\| x^{\beta_2} f(x) \right\|_{L^\infty (\mc O) }\,.
\end{equation}
\end{lemma}
\begin{proof} First, let us define 
\begin{equation*}
\Phi (x) = \int_0^x y^\alpha_{1} dy = \frac{1}{\alpha_{1}+1}  x^{\alpha_{1}+1}\,,
\end{equation*}
and $g(x) = x^{\alpha_2} f(\Phi(x))$ for some $\alpha_1 > 0, \alpha_2$ to be fixed later. Using Faa di Bruno for the $m$-th derivative of the composition and identifying the only term where $m$ derivatives fall on $f$, we have
\begin{equation} \label{eq:andalucia}
\left| \p^m_x g(x) - x^{\alpha_{1} m + \alpha_2} f^{(m)}(\Phi (x)) \right| \lesssim_m \sum_{i=0}^{m-1} f^{(i)} (\Phi(x)) x^{\alpha_{1} i + \alpha_2 - (m-i)},
\end{equation}
where $i$ corresponds to the quantity of derivatives falling on $f$.

Let $\mt O = [\Phi^{-1} (a), +\infty ) \subset [1, +\infty)$. By the Kolmogorov-Landau inequality on the half-line, we have
\begin{equation} \label{eq:extremadura}
\| \p^m_x g(x) \|_{L^\infty (\mt O)} \lesssim_n \| g(x) \|_{L^\infty (\mt O)}^{1-m/n} \| \p^n_x g(x) \|_{L^\infty (\mt O )}^{m/n}\,.
\end{equation}

Using \eqref{eq:andalucia} and \eqref{eq:extremadura} we obtain
\begin{equation} \label{eq:castillalm}
\left\| x^{\alpha_{1} m + \alpha_2} f^{(m)}(\Phi (x)) \right\|_{L^\infty (\mt O) } \lesssim_n 
\sum_{j=0}^{m-1} \left\| f^{(j)} (\Phi(x)) x^{\alpha_{1} j + \alpha_2 - (m-j)} \right\|_{L^\infty (\mt O)}
+ \| g(x) \|_{L^\infty (\mt O)}^{1-m/n} \| \p^n_x g(x) \|_{L^\infty (\mt O )}^{m/n}\,.
\end{equation}

Applying \eqref{eq:castillalm} $i$ times, we conclude that for any $0 \leq i \leq n$: 
\begin{equation} \label{eq:valencia}
\left\| x^{\alpha_{1} i + \alpha_2} f^{(i)}(\Phi (x)) \right\|_{L^\infty (\mt O) } \lesssim_n
\left\| g(x) \right\|_{L^\infty (\mt O)}
+ \| g(x) \|_{L^\infty (\mt O)}^{1-i/n} \| \p^n_x g(x) \|_{L^\infty (\mt O )}^{i/n}\,.
\end{equation}
%
%

On the other hand, taking $m=n$ in \eqref{eq:andalucia} and using \eqref{eq:valencia}, we see that
\begin{align*}
\left\| \p_x^n g(x) - x^{\alpha_{1} n + \alpha_2} f^{(n)}(\Phi (x)) \right\|_{L^\infty (\mc O)} &\lesssim_n \| g(x) \|_{L^\infty (\mt O)} + \sum_{j=0}^{n-1}  \| g(x) \|_{L^\infty (\mt O)}^{1-j/n} \| \p^n_x g(x) \|_{L^\infty (\mt O )}^{j/n} \\
&\lesssim_n \| g(x) \|_{L^\infty (\mt O )} +  \| g(x) \|_{L^\infty (\mt O)}^{1/n} \| \p^n_x g(x) \|_{L^\infty (\mt O )}^{(n-1)/n}\,,
\end{align*}
so there exists some constant $C_n$ such that
\begin{align*}
\| \p_x^n g(x) \|_{L^\infty (\mt O) } &\leq \| x^{\alpha_{1} n + \alpha_2} f^{(n)} (\Phi (x)) \|_{L^\infty (\mt O)} +  C_n \| g(x) \|_{L^\infty (\mt O )} +  C_n \| g(x) \|_{L^\infty (\mt O)}^{1/n} \| \p^n_x g(x) \|_{L^\infty (\mt O )}^{(n-1)/n} \\
& \leq \| x^{\alpha_{1} n + \alpha_2} f^{(n)} (\Phi (x)) \|_{L^\infty (\mt O)} +  C_n \| g(x) \|_{L^\infty (\mt O )} +  C_n \frac{\| g(x) \|_{L^\infty (\mt O)} \mt C^n }{n}  + \frac{C_n}{\mt C^{n/(n-1)}} \frac{\| \p^n_x g(x) \|_{L^\infty (\mt O )}}{n/(n-1)} \\
& \leq \| x^{\alpha_{1} n + \alpha_2} f^{(n)} (\Phi (x)) \|_{L^\infty (\mt O)} + \frac{1}{2} \| \p^n_x g(x) \|_{L^\infty (\mt O )} + O_n \left( \| g(x) \|_{L^\infty (\mt O )} \right),
\end{align*}
where the second inequality holds for any $\mt C > 0$ due to Young's inequality and for the third inequality we chose $\mt C = 2 C_n$. Therefore, we get 
\begin{align} \label{eq:madrid}
\| \p^n_x g(x) \|_{L^\infty (\mt O )}
&\les_n   \| g(x) \|_{L^\infty (\mt O )} + \| x^{\alpha_{1} n + \alpha_2} f^{(n)} (\Phi (x)) \|_{L^\infty (\mt O)}\,.
\end{align}

Now, plugging in \eqref{eq:madrid} into \eqref{eq:valencia}, we obtain
\begin{align} \label{eq:castillaleon}
\left\| x^{\alpha_{1} i + \alpha_2} f^{(i)} (\Phi (x)) \right\|_{L^\infty (\mt O) } &\lesssim_n
\left\| g(x) \right\|_{L^\infty (\mt O)}
+ \| g(x) \|_{L^\infty (\mt O)}^{1-i/n} \|  x^{\alpha_{1} n + \alpha_2} f^{(n)} (\Phi (x)) \|_{L^\infty (\mt O )}^{i/n} \notag  \\
&= \left\| x^{\alpha_2} f(\Phi (x) ) \right\|_{L^\infty (\mt O)}
+ \| x^{\alpha_2} f(\Phi (x)) \|_{L^\infty (\mt O)}^{1-i/n} \|  x^{\alpha_{1} n + \alpha_2} f^{(n)} (\Phi (x)) \|_{L^\infty (\mt O )}^{i/n} \,.
\end{align}
Letting $z = \Phi(x) = \frac{1}{\alpha_{1}+1} x^{\alpha_{1}+1} $ and writing \eqref{eq:castillaleon} in terms of $z$, we obtain
\begin{equation*}
   \left\| z^{\frac{\alpha_{1}i + \alpha_2}{\alpha_{1}+1}} f^{(i)} (z) \right\|_{L^\infty (\mc O) } 
   \lesssim_{n}  \frac{\left\| z^{\frac{\alpha_2}{\alpha_1+1}}f(z ) \right\|_{L^\infty (\mc  O)}}{(\alpha_{1}+1)^{\frac{\alpha_{1}}{\alpha_{1}+1}i}}
+ \left\| z^{\frac{\alpha_2}{\alpha_1+1}} f(z) \right\|_{L^\infty (\mc O)}^{1-i/n} \left\|  z^{\frac{\alpha_{1}n+\alpha_2}{\alpha_{1}+1} } f^{(n)} (z) \right\|_{L^\infty (\mc O )}^{i/n} \,.
\end{equation*}
Taking $\alpha_{1} > 0$ such that $\beta_1 = \frac{\alpha_{1}}{\alpha_{1}+1}$ and taking $\alpha_2 = (1+\alpha_1) \beta_2$, we obtain the desired inequality. It is clear that this imposes $\beta_1 \in (0, 1)$.

\end{proof}

\begin{lemma} \label{lemma:gn_interpolation} Let $a \geq 1$ and consider $\mc O = \mathbb{R}^3 \setminus B(0, a)$. Let $\kappa \in (0, 1-2/n)$. We have that for any $0 \leq i \leq n$:
\begin{equation} \label{eq:brandon}
\left( \int_{\mc O} |x|^{2 \kappa n} \left( | \nabla^i f | (x) \right)^p   \right)^{1/p}
\les_{n} \left( \int_{\mc O} | x |^{2 \kappa n} \left( | \nabla^n f |(x) \right)^2 \right)^{\beta/2} \left\| f (x) \right\|_{L^\infty (\mc O )}^{1-\beta} + \| f (x)\|_{L^\infty (\mc O)}\,,
\end{equation}
where $\beta = i/n$ and $\frac{1}{p} = \frac{\beta}{2}$. 
\end{lemma}
\begin{proof}
First, due to Gagliardo-Nirenberg inequality for bounded domains, we have that there exist absolute constants $C_1, C_2$ such that we have
\begin{equation} \label{eq:avada1}
\left\|  \nabla^i f  \right\|_{L^p(A)}^p \leq C_1 \left\|  \nabla^n  f \right\|_{L^2(A)}^2\| f \|_{L^\infty(A)}^{p-2} + C_2 \| f \|_{L^\infty(A)}^p,
\end{equation}
where $A$ is the annulus where $1 \leq | x | \leq 4$. Now, for any $\lambda$ let $\lambda A$ be the annulus $\lambda \leq | x | \leq 2\lambda$. We can write \eqref{eq:avada1} in terms of $f_\lambda (x) = f(\lambda x)$ which is a function defined on $\lambda A$. We obtain
\begin{equation} \label{eq:avada2}
\lambda^{ip} \lambda^{-3} \left\|  \nabla^i f_\lambda  \right\|_{L^p(\lambda A)}^p \leq C_1 \lambda^{2n} \lambda^{-3} \left\|  \nabla^n f_\lambda  \right\|_{L^2(\lambda A)}^2  \| f \|_{L^\infty(\lambda A)}^{p-2} + C_2 \| f_\lambda \|_{L^\infty (\lambda A)}^p\,.
\end{equation} 
Noting that $ip = n \beta p = 2n$, we have that 
\begin{align} \label{eq:kedavra}
\lambda^{2\kappa n} \left\|  \nabla^i f_{\lambda}  \right\|_{L^p (\lambda A)}^p  &\leq C_1 \lambda^{2\kappa n} \left\|  \nabla^n f_\lambda  \right\|_{L^2(\lambda A)}^2  \| f \|_{L^\infty(\lambda A)}^{p-2} + C_2 \lambda^{3-2n(1-\kappa)} \| f_\lambda \|_{L^\infty (\lambda A)}^p \notag \\
&\leq C_1 \lambda^{2\kappa n} \left\|  \nabla^n f_\lambda  \right\|_{L^2(\lambda A)}^2  \| f \|_{L^\infty(\lambda A)}^{p-2} + C_2 \lambda^{-1} \| f_\lambda \|_{L^\infty (\lambda A)}^p\,,
\end{align}
where in the second inequality we used $1-\kappa \geq 2/n$. Now, note that any function $g$ defined over $\lambda A$ has a corresponding $f$ defined over $A$ such that $g = f_\lambda$. Therefore, \eqref{eq:kedavra} holds for all $g$ over $\lambda A$.

Finally, we combine all those estimates at different scales. Set $\lambda_j=a2^{j-\frac12}$.
Let us consider functions $g_j$ such that $g_j = g$ for $a2^j \leq | x | \leq a 2^{j+1}$, $g_j$ is supported on some $\lambda_j A$, and moreover we have that
\[ \|  \nabla^j g_j  \|_{L^2(\lambda_j A)}\leq C_3 \|  \nabla^j g_j  \|_{L^2(B_j)}\quad\mbox{and}\quad  \| g_j \|_{L^\infty (\lambda_j A)}\leq C_3  \| g_j \|_{L^\infty (B_j)}
\,,\]
for $B_j = \{ x: a 2^j \leq | x | \leq a 2^{j+1} \}\subset \lambda_j A$ and some constant $C_3>0$.
 Using \eqref{eq:kedavra} for $g_j$, we have that
\begin{align*}
\int_{\mc O} | x |^{2 \kappa n} \left( | \nabla^i g|(x) \right)^p &\leq \sum_{j=0}^\infty (4\lambda_j)^{2\kappa n} \left\| \nabla^i g_j \right\|_{L^p(\lambda_j A)}^p \\
&\leq C_1 \sum_{j \geq 0} \lambda_j^{2\kappa n} \left\|  \nabla^n g_j \right\|_{L^2(\lambda_j A)}^2  \| g_j \|_{L^\infty(\lambda_j A)}^{p-2} + C_2 \sum_{j \geq 0} \lambda^{-1} \| g_j \|_{L^\infty (\lambda_j A)} \\
&\leq 16^n C_1 C_3^p \| g \|_{L^\infty (\mc O )}^{p-2} \sum_{j \geq 0} \lambda_j^{2\kappa n} \left\|  \nabla^n g_j  \right\|_{L^2(B_j)}^2   + 16^n C_2 C_3^p  \| g \|_{L^\infty (\mc O)}^p \sum_{j \geq 0} \lambda_j^{-1} \\
&= 16^n C_1 C_3^p \| g \|_{L^\infty (\mc O )}^{p-2} \sum_{j \geq 0} \lambda_j^{2\kappa n} \left\|  \nabla^n g_j  \right\|_{L^2(B_j)}^2   + 16^n C_2 C_3^p  \| g \|_{L^\infty (\mc O)}^p \frac{2}{a}\,.
\end{align*}
As $a \geq 1$, we have that $\frac{2}{a} \leq 2$ and this concludes our proof. Note that $C_1, C_2, C_3, p$ are independent of $
\kappa, n, a$.
\end{proof}

\begin{lemma} \label{lemma:thefinalinterpolation2} Let $a \geq 1$ and consider $\mc O = \mathbb{R}^3 \setminus B(0, a)$. Let $\kappa \in (0, 1-2/n )$ and $\kappa'\in\mathbb R$. We have that for any $0 \leq i \leq n$:
\begin{equation} \label{eq:aragorn}
\| |x|^{\kappa i+\kappa'} \nabla^i f \|_{L^2 (\mc O)} \lesssim_{n,\kappa,\kappa'} \left\| \abs{x}^{\kappa'} f(x) \right\|_{L^2 (\mc O) }^{1-\beta} \left\| \abs{x}^{n \kappa+\kappa'} \nabla^n_x f(x) \right\|_{L^2 (\mc O) }^{\beta } + \left\| \abs{x}^{\kappa'} f(x) \right\|_{L^2 (\mc O) }\,,
\end{equation}
where  $\beta = i/n$.
\end{lemma}
\begin{proof}
We will follow an analogous strategy to the proof of Lemma \ref{lemma:thefinalinterpolation2}. The principal difference is that in place of \eqref{eq:avada1}, we use the following Gagliardo-Nirenberg inequality
\begin{equation} \label{eq:messi}
\left\|  \nabla^i f  \right\|_{L^2(A)}^2 \leq C_1 \left\|  \nabla^n f \right\|_{L^2(A)}^{2\beta} \| f \|_{L^2(A)}^{2(1-\beta)} + C_2 \| f \|_{L^2(A)}^2\,,
\end{equation}
where $A$ is the annulus where $1 \leq | x | \leq 4$ and $C_1, C_2$ are some absolute constants.  Given \eqref{eq:messi}, one argues in a completely analogous manner as for Lemma \ref{lemma:thefinalinterpolation2}.
\end{proof}

\begin{lemma} \label{lemma:leibnitzlaplace} Let $f, g$ be radially symmetric scalar funcitons over $\mathbb{R}^3$ and $F, G$ to be radially symmetric vector fields over $\mathbb{R}^3$. Let us assume that $f, F_i \in W^{2m, \infty}$ and $g, G_i \in H^{2m}$. We have the following inequalities:
\begin{align}
\left\| \Delta^m (F \nabla G_i )  - F \nabla \Delta^m G_i - 2m \p_\zeta F \Delta^m G_i \right\|_{L^2} &\lesssim_m \| F \|_{W^{2m, \infty}} \|G \|_{H^{2m-1}} \,,\label{eq:FG} \\
\left\| \Delta^m (f \nabla g) - f \nabla \Delta^m g - 2m \nabla f \Delta^{m} g   \right\|_{L^2} &\lesssim_m \| f \|_{W^{2m, \infty}} \| g \|_{H^{2m-1}}   \label{eq:fg}\,,\\
\left\| \Delta^m (F \nabla g) - F \nabla \Delta^m g -2m \p_\zeta F \Delta^m g  \right\|_{L^2} &\lesssim_m \| F \|_{W^{2m, \infty}} \| g \|_{H^{2m-1}} \label{eq:Fg}\,, \\
\left\| \Delta^m (f \div(G) )- f \div (\Delta^m G) - 2m \nabla f \Delta^m G \right\|_{L^2} &\lesssim_m \|f \|_{W^{2m, \infty}} \| G \|_{H^{2m-1}} \label{eq:fG}\,.
\end{align}
\end{lemma}
\begin{proof} Equation \eqref{eq:fg} and  equation \eqref{eq:fG} are clear by examination because we are substracting exactly the terms where $m$ or $m+1$ derivatives fall on $g$ or $G$. For equation \eqref{eq:fG} note also that $\div(\Delta^m G) = \Delta^m \div (G)$ because $\Delta G = \nabla \div (G) $ for radial $G$.

Let us consider the equation \eqref{eq:Fg} by expanding $\Delta^m(F \nabla g)$. We clearly have
\begin{align} \label{eq:albondingas}
\Delta^m (F \nabla g ) &= \left( \sum_i \p_i^2 \right)^m \sum_j F_j \p_j g = 
\sum_j \left( F_j  \p_j  \left( \sum_i \p_i^2 \right)^mg \right) \nonumber \\
&\qquad + \sum_j \left( 2m \sum_{k} \p_k F_j \p_k \p_j \left( \sum_i \p_i^2 \right)^{m-1}  g \right)
+ O \left( \| F \|_{W^{2m , \infty}} \| g \|_{H^{2m-1}} \right) \nonumber \\
&= F \nabla \Delta^m g  + 2m \sum_{k, j} \p_k F_j \p_{k} \p_j \Delta^{m-1} g  + O_m \left( \| F \|_{W^{2m , \infty}} \| g \|_{H^{2m-1}} \right)\,.
\end{align}

Now, we claim that for a radially symmetric field $F$,
\begin{equation} \label{eq:macarrones}
\p_a F_b = \frac{y_a y_b}{\zeta^2} \left( \p_\zeta F - \frac{F}{\zeta} \right) + \frac{\delta_{a, b} F}{\zeta}\,,
\end{equation}
which follows just from writting $F_b = F \frac{x_b}{\zeta}$, where $F$ is the radial component of $F$ and expanding $\p_a \left( F \frac{x_b}{\zeta} \right)$. Using \eqref{eq:macarrones}, we see that
\begin{align} \label{eq:pulpo}
\sum_{k, j} \p_k F_j \p_k \p_j \Delta^{m-1} g &= \left( \p_\zeta F - \frac{F}{\zeta} \right) \sum_{k, j} \frac{y_k y_j}{\zeta^2} \p_k \p_j \Delta^{m-1} g + \frac{F}{\zeta} \Delta^m g \nonumber \\
&= \left( \p_\zeta F - \frac{F}{\zeta} \right) \p_{\zeta}^2 \Delta^{m-1} g  + \frac{F}{\zeta} \Delta^m g \nonumber \\
&= \left( \p_\zeta F - \frac{F}{\zeta} \right) \left( \Delta^m g - \frac{2}{\zeta} \p_\zeta \Delta^{m-1}g \right) + \frac{F}{\zeta} \Delta^m \nonumber \\
&= \p_\zeta F \Delta^m g + \frac{\p_\zeta F - F/\zeta}{\zeta} \p_\zeta \Delta^{m-1} g\,.
\end{align}
In the second equality, we have used that $\sum_{k, j} \frac{y_k y_j}{\zeta^2} \p_k \p_j = \p_{\zeta}^2$ which follows from the fact that $\left[ \p_\zeta, \frac{y_k}{\zeta} \right]$ and
\begin{equation} \label{eq:caldo}
\p_\zeta^2 = \p_\zeta \sum_{j} \frac{y_j}{\zeta} \p_j = \sum_j \frac{y_j}{\zeta} \p_\zeta \p_j = \sum_{j, k} \frac{y_j y_k}{\zeta^2} \p_{j, k}\,.
\end{equation}
Substituting \eqref{eq:pulpo} in \eqref{eq:albondingas}, we have
\begin{equation} \label{eq:ropavieja}
\Delta^m (F \nabla g) = F \nabla \Delta^m g + 2m \left( \p_\zeta F \Delta^m g + \frac{\p_\zeta F - F/\zeta}{\zeta} \p_\zeta \Delta^{m-1} g \right) + O_m \left( \| F \|_{W^{2m , \infty}} \| g \|_{H^{2m-1}} \right)\,.
\end{equation}
Finally, note that $\left\| \frac{\p_\zeta F - F/\zeta}{\zeta} \right\|_{\infty} \lesssim \left\| F \right\|_{W^{2m, \infty}}$, because of the radial symmetry of $F$. This completes the proof of \eqref{eq:Fg}.

Finally, let us show \eqref{eq:FG}. Analysing $\Delta^m (F\nabla G_i)$ we have that
\begin{align}
\Delta^m (F \nabla G_i ) &= F \cdot \nabla \Delta^m G_i + 2m \sum_{j, k} \p_j F_k \p_{j} \p_k \Delta^{m-1} G_i + O \left( \| F \|_{W^{2m, \infty}} \| G \|_{H^{2m-1}} \right) \notag \\
&= F \cdot \nabla \Delta^m G_i + 2m \left( \p_\zeta F - \frac{F}{\zeta}  \right) \p_\zeta^2 \left( \frac{y_i}{\zeta} \Delta^{m-1} G \right) \notag \\
&\qquad + 2m  \frac{F}{\zeta} \Delta^m G_i + O \left( \| F \|_{W^{2m, \infty}} \| G \|_{H^{2m-1}} \right) \notag \\
&= F \cdot \nabla \Delta^m G_i + 2m \left( \p_\zeta F - \frac{F}{\zeta}  \right) \frac{y_i}{\zeta} \left(  \Delta^{m} G - \frac{2 \zeta \p_\zeta - 2}{\zeta^2} \Delta^{m-1} G \right) \notag \\
&\qquad + 2m  \frac{F}{\zeta} \Delta^m G_i + O \left( \| F \|_{W^{2m, \infty}} \| G \|_{H^{2m-1}} \right) \notag \\
&=
F \cdot \nabla \Delta^m G_i + 2m \p_\zeta F \Delta^m G_i
- 2m \frac{ \p_\zeta F - \frac{F}{\zeta} }{\zeta} \frac{y_i}{\zeta} \frac{2 \zeta \p_\zeta - 2 }{\zeta} \Delta^{m-1} G \notag \\
& \qquad + O \left( \| F \|_{W^{2m, \infty}} \| G \|_{H^{2m-1}} \right) \label{eq:marmitako}\,.
\end{align}
Noting as before that $\frac{\p_\zeta F - F/\zeta}{\zeta}$ is bounded in $L^\infty$ because of the symmetry and noting that 
\begin{equation} \label{eq:marmitako2}
\frac{2 \zeta \p_\zeta - 2}{\zeta} \Delta^{m-1} G = \left( 3 \p_\zeta - \div \right) \Delta^{m-1} G,
\end{equation}
we conclude \eqref{eq:FG}. 
\end{proof}

\begin{lemma} \label{lemma:leibnitzlaplace2} Let us assume that $\mc U$ is a radially symmetric vector field and $\mc S$ is a radially symmetric scalar field. Let us denote the radial variable by $\zeta$. Let $\phi$ be some radially symmetric weight with $\phi \geq 1$, $\phi = 1$ on $B(0, 1)$ and $\phi (\zeta)^{1/2} \leq \zeta$ for $\zeta > 1$.

Moreover, we assume that for any $0 \leq i \leq 2K$ we have:
\begin{equation} \label{eq:tomate1}
\int_{\mathbb{R}^3} \left(  | \nabla^i \mc U |  + | \nabla^i \mc S | \right)^2 \left( | \nabla^{2K+1-i} \mc U |  + | \nabla^{2K+1-i} \mc S | \right)^2 \phi^{2K} \leq \bar \eps \bar E^2 .
\end{equation}
for some $\bar \eps \ll 2^{-4K}$. Let us also assume
\begin{equation} \label{eq:tomate2}
\| \Delta^K \mc U \phi^K \|_{L^2}^2, \| \Delta^K \mc S \phi^K \|_{L^2}^2 \leq \bar{E}^2, \qquad \| \nabla^{2K-1} \mc U   \phi^K \frac{1}{\phi^{1/4} \langle \zeta \rangle^{1/2}}  \|_{L^2}^2, \|  \nabla^{2K-1} \mc S  \phi^K \frac{1}{\phi^{1/4} \langle \zeta \rangle^{1/2}} \|_{L^2}^2 \leq \bar \eps \bar{E}^2\,,
\end{equation}
and
\begin{equation} \label{eq:tomate3}
\| \mc U \|_{W^{3, \infty}}, \|  \mc S \|_{W^{3, \infty}} \les 1 \qquad \|  \nabla \mc U \phi^{1/2} \|_{L^\infty}, \| \nabla \mc S \phi^{1/2} \|_{L^\infty} \les 1\,.
\end{equation}

Then, we have
\begin{align}
\left\| \left( \Delta^K (\mc U \cdot \nabla \mc U_i ) - \mc U \nabla \Delta^K \mc U_i - 2K \p_\zeta \mc U \Delta^K \mc U_i  \right) \phi^{K} \right\|_{L^2} &= O \left( \bar{E} \right), \label{eq:lechuga1} \\
\left\|  \left( \Delta^K  (\mc S \nabla \mc S) - \mc S \nabla \Delta^K  \mc S - 2K \nabla \mc S \Delta^K \mc S \right) \phi^{K} \right\|_{L^2} &= O \left( \bar{E} \right), \label{eq:lechuga2} \\
\left\|  \left( \Delta^K  (\mc U \cdot \nabla \mc S) - \mc U \cdot \nabla \Delta^K  \mc S - 2K \p_\zeta \mc U \cdot \Delta^K \mc S \right) \phi^{K} \right\|_{L^2} &= O \left( \bar{E} \right),\label{eq:lechuga3} \\
\left\|  \left( \Delta^K  (\mc S \div (\mc U)) - \mc S \div ( \Delta^K  \mc U ) - 2K \nabla \mc S \cdot \Delta^K \mc U \right) \phi^{K} \right\|_{L^2} &= O \left( \bar{E} \right). \label{eq:lechuga4} 
\end{align}
\end{lemma}

\begin{proof} Equations \eqref{eq:lechuga2} and \eqref{eq:lechuga4} just follow from distributing the derivatives in $\Delta^K (\mc S \nabla \mc S)$ or $\Delta^K (\mc S \div(U))$ respectively. For example, equation \eqref{eq:lechuga2} follows from
\begin{equation*}
\phi^K \left| \Delta^K (\mc S \nabla \mc S) - \mc S \nabla \Delta^K \mc S - 2K \nabla \mc S \Delta^K \mc S \right|   
\leq  \sum_{i=0}^{2K-2} \binom{2K}{i}  | \nabla^{2K-i} \mc S  |  | \nabla^{i+1} \mc S | \phi^K\,,
\end{equation*}
so that
\begin{align*}
\left\|  \Delta^K (\mc S \nabla \mc S) - \mc S \nabla \Delta^K \mc S - 2K \nabla \mc S \Delta^K \mc S\right\|_{L^2} \leq \| \mc S \|_{L^\infty} \|  \nabla^{2K} \mc S \phi^K \|_{L^2} + 2^{2K} \bar \eps^{\frac12} \bar E = O \left( \bar E \right) \,.
\end{align*}
Equation \eqref{eq:lechuga4} is shown in a completely analogous way. 

Let us now show \eqref{eq:lechuga3}. Reasoning in the exact same way as we did to obtain \eqref{eq:ropavieja}, we have that
\begin{align*}
&\left| \Delta^K ( \mc U \nabla \mc S) - \mc U \nabla \Delta^K \mc S - 2K \left( \p_\zeta \mc U \Delta^K \mc S + \frac{\p_\zeta \mc U - \mc U/\zeta}{\zeta} \p_\zeta \Delta^{K-1} \mc S \right) \right| \\
& \qquad \leq | \nabla^{2K} \mc U | | \nabla  \mc S | + 2^{2K}\sum_{i=1}^{2K-2}  | \nabla^{2K-i} \mc U | | \nabla^{i+1} \mc S |\,.
\end{align*}
Using \eqref{eq:tomate1} and \eqref{eq:tomate2}, we obtain that
\begin{align}
\left\|  \left( \Delta^K  (\mc U \cdot \nabla \mc S) - \mc U \cdot \nabla \Delta^K  \mc S - 2K \p_\zeta \mc U \cdot \Delta^K \mc S \right) \phi^{K} \right\|_{L^2}
&\leq 2K \left\| \frac{\p_\zeta \mc U - \mc U/\zeta}{\zeta} \langle \zeta \rangle^{1/2} \phi^{1/4} \right\|_{L^\infty} \left\| \frac{\phi^{K}}{\zeta^{1/2} \phi^{1/4}} \nabla^{2K-1} \mc S  \right\|_{L^2} + 2^{2K} \bar \eps^{1/2} \bar E \notag \\
&\les \bar E + 2K \bar \eps^{1/2} \bar E \left\| \frac{ \p_\zeta \mc U - \mc U / \zeta }{\zeta} \langle \zeta \rangle^{1/2} \phi^{1/4} \right\|_{L^\infty} \notag \\
&\les \bar E \left( 1 +  \left\| \frac{ \p_\zeta \mc U - \frac{\mc U}{\zeta} }{\zeta} \langle \zeta \rangle^{1/2} \phi^{1/4} \right\|_{L^\infty}  \right)\,. \label{eq:cebolla}
\end{align}
Consider the regions $\mc B_1 = B(0, 1)$ and $\mc B_2 = \mathbb{R}^3 \setminus \mc B_1$. We have that
\begin{align}
\left\| \frac{ \p_\zeta \mc U - \mc U/\zeta }{\zeta} \langle \zeta \rangle^{1/2} \phi^{1/4} \right\|_{L^\infty} 
&\leq \left\| \frac{ \p_\zeta \mc U - \mc U / \zeta}{\zeta}   \right\|_{L^\infty(\mc B_2)} + \left\| \p_\zeta \mc U \right\|_{L^\infty(\mc B_2)}+ \left\| \mc U \right\|_{L^\infty(\mc B_1)} \notag \\
&\leq \left\| \mc U \right\|_{W^{3, \infty}} + \left\| \p_\zeta \mc U \right\|_{L^\infty(\mc B_2)}+ \left\| \mc U \right\|_{L^\infty(\mc B_2)} \notag \\
&\les 1, \label{eq:ajo}
\end{align}
where in the first inequality we used $\phi (\zeta)^{1/4} \leq \zeta^{1/2}$ and in the last inequality we used \eqref{eq:tomate3}. Plugging \eqref{eq:ajo} into \eqref{eq:cebolla} we conclude \eqref{eq:lechuga3}.

Lastly, let us show \eqref{eq:lechuga1}. In the same way as we obtained \eqref{eq:marmitako}, we have that
\begin{align*}
&\left| \Delta^K (\mc U \cdot \nabla \mc U_i) - \mc U \cdot \nabla \Delta^K \mc U_i - 2K \p_\zeta \mc U \Delta^K \mc U_i
+ 2K \frac{ \p_\zeta \mc U - \frac{\mc U}{\zeta} }{\zeta} \frac{y_i}{\zeta} \frac{2 \zeta \p_\zeta - 2 }{\zeta} \Delta^{m-1} \mc U  \right| \notag \\
& \qquad \leq | \nabla^{2K} \mc U | | \nabla \mc U | +  2^{2K} \sum_{i=1}^{2K-2} | \nabla^{2K-i} \mc U | | \nabla^{i+1} \mc U_i |\,.
\end{align*}
Using \eqref{eq:marmitako2}, \eqref{eq:tomate1} and \eqref{eq:tomate2} we get that
\begin{align*}
&\left\|\left( \Delta^K (\mc U \cdot \nabla \mc U_i) - \mc U \cdot \nabla \Delta^K \mc U_i - 2K \p_\zeta \mc U \Delta^K \mc U_i \right) \phi^K \right\|_{L^2} \\
&\qquad \leq 2K \left\| \frac{ \p_\zeta \mc U - \frac{\mc U}{\zeta} }{\zeta} \langle \zeta \rangle^{1/2} \phi^{1/4} \right\|_{L^\infty} \left\| \frac{\phi^{K}}{\langle \zeta \rangle^{1/2} \phi^{1/4}} (3 \p_\zeta - \div ) \Delta^{K-1} \mc U \right\|_{L^2} + \bar E + 2^{2K} \bar \eps^{1/2} \bar E \\
&\qquad \leq  2K \bar \eps^{1/2} \bar E \left\| \frac{ \p_\zeta \mc U - \frac{\mc U}{\zeta} }{\zeta} \langle \zeta \rangle^{1/2} \phi^{1/4} \right\|_{L^\infty}  + 2 \bar E\,.
\end{align*}
Finally, using \eqref{eq:ajo} we conclude \eqref{eq:lechuga1}.
\end{proof}

\subsection*{Explicit computations}

\begin{lemma} \label{lemma:connecticut} For every $\gamma >1 $ we have that $r^\ast (\ga ) < 2-\frac{1}{\ga}$. Equivalently,
\begin{equation}
2 > (r^\ast (\gamma) - 1)\left( 2+ \frac{1}{\alpha} \right)\,.\notag
\end{equation}
We also have $r^\ast (\gamma) < \gamma$.
\end{lemma}
\begin{proof} For $\gamma \leq \frac{5}{3}$ we have that
\begin{equation*}
2 - \frac{1}{\gamma} - r^\ast  (\gamma) = \frac{ \left(2 \sqrt{2}-\sqrt{\gamma -1}\right) \sqrt{\gamma -1}  }{ \left(\sqrt{2} \sqrt{\frac{1}{\gamma -1}}+1\right)^2 \gamma  }\,,
\end{equation*}
which is positive, as $\gamma - 1 < 2\sqrt{2}$. On the other hand, for $\gamma > \frac{5}{3}$, 
\begin{equation*}
2 - \frac{1}{\gamma} - r^\ast  (\gamma) =  \frac{   (\gamma -1) \left(\left(2 \sqrt{3}-3\right) \gamma -\sqrt{3}+2\right)    }{\left(\sqrt{3} (\gamma -1)+2\right) \gamma} > 0\,.
\end{equation*}
This concludes the proof of $r^\ast (\ga ) < 2-\frac{1}{\ga}$. As a consequence, we get 
\begin{equation*}
r^\ast (\ga) - \ga < -\ga + 2 - \frac{1}{\ga} = \frac{-(\ga - 1)^2}{\ga} < 0\,,
\end{equation*}
so we also get $\ga > r^\ast (\ga)$. Finally, note that
\begin{equation*}
2 > (r^\ast (\gamma) - 1)\left( 2+ \frac{1}{\alpha} \right) \Leftrightarrow 1 > (r^\ast (\gamma) - 1) \frac{\ga}{\ga - 1} \Leftrightarrow \frac{\ga}{\ga - 1} > r^\ast (\ga ) - 1 \Leftrightarrow 2 - \frac{1}{\ga} < r^\ast (\ga )\,.
\end{equation*}
\end{proof}

\begin{lemma} \label{lemma:aux_limits} For $\gamma = 7/5$, we have that:
\begin{align*}
\lim_{r\rightarrow r^\ast }D_{Z, 2} &= \frac{19 - 9 \sqrt 5}{132}  < 0, &\qquad
\lim_{r\rightarrow r^\ast}W_1 &= \frac{5-3 \sqrt{5}}{4} < 0, \\
\lim_{r\rightarrow r^\ast} D_{W,0} &= \frac{\sqrt{5} - 1}{2} > 0, &\qquad
\lim_{r\rightarrow r^\ast}Z_1 &= \frac{3 \sqrt{5} - 5}{6} > 0, \\
\lim_{r\rightarrow r^\ast} N_{W,0} &= -\frac52 + \sqrt{5} < 0 &\qquad
\lim_{r \rightarrow r^\ast } \p_Z N_Z (P_s) &= \frac{3(-5+7\sqrt{5})}{20} > 0, \\
\lim_{r \rightarrow r^\ast} D_{Z, 1} k &= \frac{3 \sqrt{5} - 1}{4}  > 0, &\qquad
\lim_{r\rightarrow r^\ast} k N_{Z, 1} &= \frac{25-9\sqrt{5}}{12} > 0 , \\
\lim_{r \rightarrow r^\ast}(\p_Z N_W (P_s) - W_1 \p_Z D_W ) &= \frac{1-\sqrt{5}}{2} < 0, &\qquad
\lim_{r\rightarrow r^\ast}  \frac{D_{Z, 2}}{2 k D_{Z, 1}} &= \frac{-29 + 12 \sqrt 5 }{726} < 0, \\
\lim_{r \rightarrow r^\ast} \p_Z B^{\rm fl}_{7/5}(W_0, Z_0) &= 625 \frac{-58951 + 26559 \sqrt{5} }{1254528}>0, &
\lim_{r\rightarrow r^\ast}  Z_0 &= -\sqrt{5} < 0.
\end{align*}
\end{lemma}
\begin{proof}
We compute the limits using their formulas and obtain the results above.
\end{proof}

\begin{lemma} \label{lemma:aux_bounds7o5} For $\ga = 7/5$ and $n$ sufficiently large with $r \in (r_n, r_{n+1})$, we have:
\begin{align*}
D_{W, 0} &\leq 2|Z_0|, &\qquad
|Z_1| &\leq 3/10, &\qquad
|W_1| &\leq 1/2, \\
|\p_i N_\circ (P_s)| &\leq 2, &\qquad
|\p_i D_\circ | &\leq 3/5, &\qquad
|\p_i  \p_j N_\circ (P_s) | &\leq 7/5,
\end{align*}
for any $i, j \in \left\{ W, Z \right\}$. 
\end{lemma} 
\begin{proof} The first three items follow from the limits of $D_{W, 0}$, $Z_0$, $Z_1$ and $W_1$ as $r \rightarrow r^\ast$ in Lemma \ref{lemma:aux_limits}. For $|\p_i N_\circ (P_s)| \leq 2$, note that
\begin{align*}
\lim_{r \to r^\ast} \nabla N_W (P_s) = \left( \frac{7}{4}-\frac{29}{4 \sqrt{5}},1-\frac{4}{\sqrt{5}}\right), \qquad
\lim_{r \to r^\ast} \nabla N_Z (P_s) = \left( \frac{1}{10} \left(7 \sqrt{5}-5\right),\frac{3}{20} \left(7 \sqrt{5}-5\right)\right),
\end{align*}
and all the components on those limits are smaller than $2$ in absolute value. Lastly, for the last two items, let us write the expressions of $D_\circ, N_\circ$ for $\ga = 7/5$, which are
\begin{align*} \begin{split}
D_W(W, Z) &= 1 + \frac{3}{5}W + \frac{2}{5} Z, \quad \quad N_W (W, Z) = -r W-\frac{7 W^2}{10}-\frac{2 WZ}{5}+\frac{Z^2}{10}, \\
D_Z(W, Z) &= 1 + \frac{2}{5}W + \frac{3}{5} Z, \quad \quad  N_Z (W, Z) = -r Z-\frac{7 Z^2}{10}-\frac{2 WZ}{5}+\frac{W^2}{10}. \\
\end{split} \end{align*}
It is clear that any first derivative of $D_\circ$ is at most $3/5$ in absolute value and any second derivative of $N_\circ$ is at most $7/5$ in absolute value. 
\end{proof}

\begin{lemma} \label{lemma:aux_DZ1_cancellation} Let $\gamma \in (1, +\infty)$. We have that $D_{Z, 1} = 0$ for $r = r^\ast (\gamma )$.
\end{lemma}
\begin{proof} We separate in two cases: $1 < \ga \leq \frac{5}{3}$ and $\ga > \frac53$. For each case, we compute the limit 
\begin{equation*}
\lim_{r \to r^\ast } D_{Z, 1} = \lim_{r \to r^\ast } \left( 1 + \frac{3-\ga}{4} W_1 + \frac{1 + \ga}{4} Z_1 \right) = 0,
\end{equation*}
using equation \eqref{eq:W1Z1}.
\end{proof}

\begin{lemma} \label{lemma:aux_DW0} Let us recall $D_{W, 0} = D_W (W_0, Z_0)$. For every $\gamma \in (1, +\infty)$ and $r\in (1, r^\ast(\gamma ))$ we have $D_{W, 0} > 0$.
\end{lemma}
\begin{proof} We have that
\begin{equation*}
D_{W, 0} = D_W(W_0, Z_0) = \frac{\gamma+(\gamma -3) r+1 +\mc R_1}{2 (\gamma -1)} > \frac{\gamma+(\gamma -3) r+1 }{2 (\gamma -1)}\,.
\end{equation*}
Now, if $\ga \geq 3$, this is clearly positive. If $\ga < 3$, we see that $1+\ga - (3 - \ga)r$ decreases with $r$. Therefore, it suffices to check that $1+\ga - (3-\ga)r^\ast (\ga) >0$ to conclude that $D_{W, 0} > 0$. For $1 < \ga < \frac53$, we obtain
\begin{equation*}
1+\ga - (3-\ga)r^\ast (\ga) = \frac{4 (\gamma -1)}{\frac{\sqrt{2}}{\sqrt{\gamma -1}}+1} > 0\,.
\end{equation*}
For $\gamma \geq \frac53$, we obtain 
\begin{equation*}
1+\ga - (3-\ga)r^\ast (\ga)  = \frac{\gamma -1 }{(\sqrt{3} (\gamma -1)+2)} \left(\left(3+\sqrt{3}\right) \gamma +\sqrt{3}-5\right),
\end{equation*}
where both the numerator and denominator on the fraction are clearly positive since $\ga \geq \frac53$. The last factor is also positive since $\frac{5-\sqrt{3}}{3+\sqrt{3}} < \frac53$.
\end{proof}

\begin{lemma} \label{lemma:aux_DZ1} Let us recall $D_{Z, 1} = \nabla D_Z (P_s) (W_1, Z_1)$. For every $\gamma \in (1, +\infty)$ and $r\in [1, r^\ast(\gamma ))$ we have $D_{Z, 1} > 0$. \end{lemma}
\begin{proof} First of all, $D_{Z, 1} = -\frac{3-5\ga+(1+\ga)r + \mc R_2 (\ga - 1)}{4(\ga - 1)}$, so it suffices to show that the numerator is negative. Using Lemma \ref{lemma:connecticut}, we have that
\begin{equation} \label{eq:hades}
3-5\ga+(1+\ga)r \leq 3-5\ga+(1+\ga) \left( 2 - \frac{1}{\ga }\right) = \frac{-(\ga - 1) (3\ga - 1)}{\ga} < 0.
\end{equation}
Therefore, the proof would follow if we show that
\begin{equation} \label{eq:atenea}
0<\left( 3+r-5\ga + r\ga \right)^2 - (\ga - 1)^2 \mc R_2^2= (3 \ga - 1) \mc R_1^2 + (3\ga - 5)(1-3\ga+(\ga+1)r)\mc R_1\,.
\end{equation}
Now, using Lemma \ref{lemma:connecticut} again, we have
\begin{equation*}
1-3\ga+(\ga+1)r < 1-3\ga+(\ga+1)\left( 2 - \frac{1}{\ga} \right) = \frac{-(\ga - 1)^2}{\ga} < 0\,.
\end{equation*}
Therefore, if $\ga \leq \frac53$, both summands in \eqref{eq:atenea} are positive and we are done. Thus, we just need to show that \eqref{eq:atenea} is positive for $\gamma > \frac53$. It suffices to show that
\begin{equation} \label{eq:afrodita}
A = (3 \ga - 1)^2 \mc R_1^2 - (3\ga - 5)^2(1+r-3\ga+r\ga)^2 > 0.
\end{equation}
For $\ga \geq \frac53$, we have that
\begin{equation*}
\frac{dA}{dr} = -32 (\gamma -1) (6 \gamma - 2 +  (\ga - 1)(3 \ga + 1) r) < 0, \qquad \mbox{ and } \qquad A\Big|_{r=r^\ast} = 0,
\end{equation*}
so we conclude that the inequality in \eqref{eq:afrodita} for all $r \in (1, r^\ast (\gamma ))$.
\end{proof}

\begin{lemma} \label{lemma:aux_DZ1check} Let us recall $\check D_{Z, 1} = \nabla D_Z (P_s) (W_1, \check Z_1)$. For every $\gamma \in (1, +\infty)$ and $r \in [1, r^\ast (\gamma )]$ we have $\check D_{Z, 1} > 0$. \end{lemma}
\begin{proof} We have that $\check D_{Z, 1} = -\frac{3+r-5\ga + r\ga - (\ga - 1)\mc R_2}{4(\ga - 1)}$, so we just need to show that the numerator is negative. This follows from $-5\ga+3+r+r\ga < 0$ which was justified in \eqref{eq:hades}.
\end{proof}

\begin{lemma} \label{lemma:aux_sinf} Let us recall $s_\infty^{\rm fr} = F_0 - W_0 - Z_0 = \frac{-4(r-1)}{3(\gamma - 1)} - W_0 - Z_0$. Let either $\gamma > 1$ and $r = r_3$ or $\gamma = 7/5$ and $r = r^\ast(7/5)$. We have that $s_\infty^{\rm fr} > 0$. 
\end{lemma}
\begin{proof}
We have that
\begin{equation} \label{eq:aimar}
6 (\ga - 1) s_\infty^{\rm fr} = 9 \gamma -3 \gamma  r+r-3 \mc R_1 -7\,.
\end{equation}
Lemma \ref{lemma:aux_otromas} yields
\begin{equation*}
9 \gamma - 7 +(1-3\ga)r > 9 \gamma - 7 +(1-3\ga) \left( 2 - \frac{1}{\ga} \right) = 3\ga - 2 - \frac{1}{\ga} > 0,
\end{equation*}
so \eqref{eq:aimar} is positive because
\begin{equation*}
\left( 9 \gamma -3 \gamma  r+r -7 \right)^2 - 9\mc R_1^2 = 16 (r-1) (2 - 5r + 3\ga r) > 0\,,
\end{equation*}
where in the last inequality we used \eqref{eq:cosilla2}.
\end{proof}

\subsection*{Properties of the phase portrait}

\begin{lemma} \label{lemma:lemon} Let $\ga = 7/5$ and $r$ sufficiently close to $r^\ast$.
Let us recall that the region $\mc T$ is the triangular region enclosed by $N_W = 0$, the horizontal segment from $P_\eye$ to $P_\eye''$ and $D_Z = 0$. We call those parts of $\p \mc T$ by $S_1, S_2, S_3$ respectively. Then, the field $(N_W D_Z, N_Z D_W)$ points inwards to $\mc T$ both in $S_1$ and $S_2$.
\end{lemma}
\begin{proof}
The field over $N_W=0$ can be simply written as $(0, N_Z D_W)$. The branch of $N_W = 0$ passing through $P_\eye$ can be parametrized as 
\begin{equation} \label{eq:musktwitter}
Z = 2W - \sqrt{10rW + 11W^2}\,, \qquad W > 0\,,
\end{equation}
in the region $W \geq Z$. Its derivative satisfies
\begin{align*}
Z ' &= \frac{2 \sqrt{W (10 r+11 W)}-5 r-11 W}{\sqrt{W (10 r+11 W)}}\\& \leq \frac{7W + 7\sqrt{rW}-5 r-11 W}{\sqrt{W (10 r+11 W)}} \leq \frac{ 8\sqrt{rW}-4 r-4 W}{\sqrt{W (10 r+11 W)}} =  \frac{-4(\sqrt{W} - \sqrt{r})^2}{\sqrt{W (10 r+11 W)}} < 0\,,
\end{align*}
where we used that $\sqrt{11}<7/2$ in the first inequality. Therefore, as $Z$ decreases with $W$, the field $(0, N_ZD_W)$ will point inwards to $\mc T$ on $S_1$ if $N_ZD_W > 0$. We have that $D_W > 0$ as $\mc T$ is in $\Omega$. With respect to $N_Z$, note that it intersects our branch of $N_W$ at two points: $(0, 0)$ and $P_\eye$ (the other two intersections of $N_W=0$ and $N_Z = 0$ from Lemma \ref{lemma:equilibrium_points} correspond to the other branch). Therefore, the sign of $N_Z$ at $S_1$ is the opposite one to the one after $P_\eye$. Taking asymptotics of $N_Z$ over the branch \eqref{eq:musktwitter}, we see
\begin{align*}
 N_Z(W, 2W - \sqrt{10rW + 11W^2} ) &= \frac{8}{5} \left(2 W^{3/2} \sqrt{10 r+11 W}-7 W^2\right)+r \left(\sqrt{W} \sqrt{10 r+11 W}-9 W\right) \\
&= \frac85 (-7 + 2 \sqrt{11}) W^2 + O\left( W^{3/2} \right),
\end{align*}
and as $2\sqrt{11} < 7$, we get that $N_Z$ is negative between $P_\eye$ and infinity over the branch \eqref{eq:musktwitter}, so it is positive over $S_1$.

Now let us evaluate the field over $S_2$. As $S_2$ is horizontal, the field $(N_W D_Z, N_Z D_W)$ will points inwards if $N_Z D_W > 0$. Similarly as before, $D_W > 0$ because $\mc T$ lies on $\Omega$. With respect to $N_Z$, we have that
\begin{equation*}
N_Z (P_\eye + (t, 0)) = \frac{1}{40} t \left(5 \left(1+3 \sqrt{3}\right) r+4 t\right),
\end{equation*}
which is positive for $t > 0$, so we also get $N_Z > 0$.
\end{proof}

\begin{lemma} \label{lemma:aux_patatillas3} For $\ga = \frac75$, and $r$ close enough to $r^\ast (7/5)$, we have that
\begin{equation}  \label{eq:golob}
b^{\rm fl}_{7/5, W} \left( \frac{- D_{Z, 2} - \sqrt{D_{Z, 2}^2 - 8 D_{Z, 1}D_{Z, 3}^{\rm fl}/3 }}{2 D_{Z, 3}^{\rm fl} / 3} \right) < \mw W_0\,.
\end{equation}
\end{lemma}
\begin{proof} As $\ga = \frac75$ is fixed, Proposition \ref{prop:taylor} together with equations \eqref{eq:Ps}, \eqref{eq:Psbar}, \eqref{eq:W1Z1} give formulas for all the coefficients, depending on $r$ and its radicals $\mc R_1, \mc R_2$ (given in equations \eqref{eq:R1}, \eqref{eq:def_R2}). As $b^{\rm fl}$ only depends on those coefficients, we get expressions for $D_{Z, 1}, D_{Z, 2}, D_{Z, 3}^{\rm fl}$, depending on them. 

Using that for $\ga = 7/5$, we have
\begin{align}
r &= \frac{7}{4}-\frac{1}{8} \sqrt{5} \sqrt{5 \mc R_1^2+4}\,, \label{eq:golob2} \\
\mc R_2 &= \frac{1}{4} \sqrt{\mc R_1 \left(-24 \sqrt{5} \sqrt{5 \mc R_1^2+4}-95 \mc R_1+80\right)-12 \sqrt{5} \sqrt{5 \mc R_1^2+4}+184} \notag\,,
\end{align}
we can express both sides of equation \eqref{eq:golob} just in terms of $\mc R_1$. Taylor expanding $\mc R_1$, we obtain
\begin{align*}
b^{\rm fl}_{7/5, W} \left( \frac{- D_{Z, 2} - \sqrt{D_{Z, 2}^2 - 8 D_{Z, 1}D_{Z, 3}^{\rm fl}/3 }}{2 D_{Z, 3}^{\rm fl} / 3} \right) &= 
\left(\frac{3 \sqrt{5}}{2}-\frac{5}{2}\right)-\frac{15 \mc R_1}{4} \\
&\qquad -\frac{15 \left(2001480+886943 \sqrt{5}\right) \mc R_1^2}{1936} +O\left(\mc R_1^3\right)\\
\mw W_0 &= \frac{1}{2} \left(3 \sqrt{5}-5\right)-\frac{15 \mc R_1}{4}+\frac{15 \sqrt{5} \mc R_1^2}{16} +O\left(\mc R_1^3\right),
\end{align*}
so we conclude that \eqref{eq:golob} holds if $\mc R_1$ is sufficiently small. On the other hand, for $\ga= \frac75$, we have that $r^\ast = \frac14 (7 - \sqrt{5})$. From \eqref{eq:golob2}, we see that $\mc R_1 \rightarrow 0$ as $r \rightarrow r^\ast$, so we are done.
\end{proof}

\begin{lemma} \label{lemma:aux_NWtriangle}  Let us recall that $\mc T^{(M)}$ is the triangle with vertices $P_s$, $P_s + (M, -M)$ and $P_s + (M, h)$ where $h$ is such that this third point falls on the line $D_Z = 0$. We have that $N_W < 0$ on $\mc T^{(M)}$ for any $M > 0$.
\end{lemma}
\begin{proof}
A generic point of $\mc T^{(M)}$ can be written as $P = P_s + t(1, s)$, where $s$ is between $-1$ and $-\frac{1-\alpha}{1+\alpha}$. Thus
\begin{align*}
N_W (P) &= N_W(P_s) +  t\nabla N_W(P_s) \cdot (1, s) + t^2(1, s) \frac{HN_W}{2} (1, s)  \\
&= N_{W, 0} + \frac{1}{4} t \Big( (-4 r-4 \gamma  W_0+\gamma  Z_0-3 Z_0)+ s(\gamma  W_0-3 W_0+2 \gamma  Z_0-2 Z_0) \Big) \\
&\qquad +\frac{1}{4}t^2 \Big( (\gamma -1) s^2+ (\gamma -3) s - 2\gamma   \Big) \\
&= N_{W, 0} + \frac{t}{4} A + \frac{t^2}{4} B\,.
\end{align*}
We have that $N_{W, 0} < 0$ due to Lemmas \ref{lemma:aux_34bounds} and \ref{lemma:aux_limits}. We will conclude the proof by showing that $A < 0$ and $B < 0$ for $s \in \left[ -1, -\frac{1-\alpha}{1+\alpha} \right]$.

Let us start showing $B < 0$. As $B$ is a second degree polynomial in $s$ with positive second derivative, it will be negative for $s \in \left[ -1, -\frac{1-\alpha}{1+\alpha} \right]$ as long as it is negative in both extrema. We have that
\begin{align*}
B|_{s= -1 } = 2 - 2\ga < 0, \qquad \mbox{ and } \qquad B|_{s= -\frac{1-\alpha}{1+\alpha} } = - \frac{16 (\ga - 1)\ga}{(1+\ga)^2} < 0\,,
\end{align*}
and this concludes $B < 0$.

Now, we show $A < 0$. As $A$ is an affine function of $s$, in order to show that it is negative for $s \in \left[ -1, -\frac{1-\alpha}{1+\alpha} \right]$, it suffices to show it at both extrema. We have that
\begin{equation*}
(\ga^2-1)A\Big|_{s = -\frac{1-\alpha}{1+\alpha} } = -8\mc R_1 \ga + 4 \underbrace{\left( 3 \gamma ^2-10 \gamma  +3+\left(-3 \gamma ^2+6 \gamma +1\right) r \right)}_C,
\end{equation*}
so it suffices to show $C < 0$. As $C$ is an affine function of $r$, it suffices to check its sign at $r=1$ and $r=r^\ast$:
\begin{align*}
C\Big|_{r=1} &= 4-4\ga < 0,  \\
C \Big|_{r=r^\ast} &= -\left(\sqrt{3}-1\right) (\gamma -1) (3 \gamma -1) < 0, \quad \mbox{ for } \ga>\frac53\,,   \\
 C \Big|_{r=r^\ast} &= \frac{-2 (\gamma -1) \left(3 \gamma-1 +4 \sqrt{2} \sqrt{\frac{1}{\gamma -1}}\right)}{ \left( 1 + \sqrt{\frac{2}{\ga - 1}} \right)^2 } < 0 \quad \mbox{ for } 1 < \ga \leq \frac53\,.
\end{align*}

Finally, we need to show that $A |_{s=-1}$ is also negative. We have that
\begin{equation} \label{eq:buerovallejo}
A \Big|_{s=-1} = -3(\gamma +1) \mc R_1 + \underbrace{9 \gamma ^2-22 \gamma +1 +\left(-3 \gamma ^2-2 \gamma +17\right) r}_D\,.
\end{equation}
We again split into two cases, $1 < \ga < 3$ and $\ga \geq 3$. 

For the case $1 < \ga < 3$, we will show $D < 0$, which trivially gives $A|_{s=-1} < 0$ from \eqref{eq:buerovallejo}. As $D$ is an affine function of $r$, it suffices to show that it is negative at $r=1$ and at $r = 1+\frac{2}{\left(1+\sqrt{\frac{2}{\gamma -1}}\right)^2} \geq r^\ast$. We have that
\begin{equation*}
D\Big|_{r=1} =  6 (\ga-3)(\ga-1) < 0, \quad \mbox{ and } \quad 
D\Big|_{r =  1+\frac{2}{\left(1+\sqrt{\frac{2}{\gamma -1}}\right)^2}} 
=  3(\ga-3)\ga  \sqrt{\frac{2}{\gamma -1}} -4 < 0\,.
\end{equation*}

For the case $\ga \geq 3$, we will show that $9(\ga+1)^2 \mc R_1^2 > D^2$. This clearly implies that $A|_{s=-1} < 0$ from \eqref{eq:buerovallejo}. We have that
\begin{equation*}
\frac{9(\ga+1)^2 \mc R_1^2 - D^2}{16(\ga - 1)} = 27 \ga^2 - 10 \ga -5 + \left(-6 \ga^2 - 28 \ga +10\right) r + \left( -3 \ga^2 + 2\ga + 13 \right) r^2.
\end{equation*}
For $\ga > 3$, both the terms in $r$ and $r^2$ in the previous equation are negative, thus, the expression is decreasing. In particular, we can lower bound it by its value at $r=r^\ast$, that is
\begin{align*}
\frac{9(\ga+1)^2 \mc R_1^2 - D^2}{16(\ga - 1)} &\geq 27 \ga^2 - 10 \ga -5 + \left(-6 \ga^2 - 28 \ga +10\right) r^\ast + \left( -3 \ga^2 + 2\ga + 13 \right) \left( r^\ast \right)^2 \\
&= \frac{6 (\gamma -1)^2 \left(3 \left(3-\sqrt{3}\right) \gamma ^2-2 \left(7-\sqrt{3}\right) \gamma +5 \sqrt{3}-7\right)}{\left(\sqrt{3} (\ga - 1) + 2\right)^2}.
\end{align*}
Noting that the second-degree polynomial $\left(3 \left(3-\sqrt{3}\right) \gamma ^2-2 \left(7-\sqrt{3}\right) \gamma +5 \sqrt{3}-7\right)$ is positive for $\ga > 3$, we are done.
\end{proof}

\begin{lemma} \label{lemma:equilibrium_points} There are three points in our phase portrait $W - Z \geq 0$ at which $N_W = N_Z = 0$, which are $(0, 0)$, $(-r, -r)$, $\Peye$. \begin{itemize}
\item For $\ga > 1$ and $r \in (r_3, r_4)$, no equilibrium point is in the region $\Omega$. Moreover, the point $P_\eye$ is a saddle point of the field $(N_W D_Z, N_Z D_W)$.
\item For $\ga = 7/5$ and $r$ sufficiently close to $r^\ast$, only the point $\Peye$ is in $\Omega$ and it lies in the region $Z > Z_0$.
\end{itemize}
\end{lemma}
\begin{proof} From B\'ezout's theorem, there are at most four solutions to $N_W = N_Z = 0$. By direct substitution in the expressions of $N_W, N_Z$, it is clear that $(0, 0)$, $(-r, -r)$,
\begin{equation*}
\Peye = \left( \frac{ 2(\sqrt{3}-1)r }{3\ga - 1}, - \frac{ 2(1+\sqrt{3})r }{ 3\ga - 1 }\right) \qquad \mbox{ and } \qquad 
 \left(  - \frac{ 2(1+\sqrt{3})r }{ 3\ga - 1 }, \frac{2(\sqrt{3}-1)r}{3\ga - 1}\right)\,,
\end{equation*}
are those four solutions. It is also clear that the last one lies in $W-Z<0$. Now, recall $\Omega$ is the region of $W-Z\geq 0$ where $D_W > 0$ and $D_Z < 0$. Note that $D_Z(0, 0) = 1$, $D_W(-r,-r) = 1-r<0$, so the only equilibrium point that can possibly lie on $\Omega$ is $P_\eye$.

For the case $r \in (r_3, r_4)$, a direct calculation gives us
\begin{equation*}
(3 \ga - 1) D_Z (\Peye) = 3 \gamma - 1 -\left(\sqrt{3} (\gamma -1)+2\right) r 
> 3 \gamma - 1 -\left(\sqrt{3} (\gamma -1)+2\right) r_4 (\gamma)\,.
\end{equation*}
From  Lemma \ref{lemma:aux_otromas}, we have $r_4 < \frac{3\gamma - 1}{2+\sqrt{3}(\gamma-1)}$. Therefore, in the case $r \in (r_3, r_4)$ we get that
\begin{equation*}
(3 \ga - 1) D_Z (\Peye)  >   3 \gamma - 1 -\left(\sqrt{3} (\gamma -1)+2\right) \frac{3\gamma - 1}{2+\sqrt{3}(\gamma-1)} = 0\,.
\end{equation*}

For $\ga = 7/5$ and $r = r^\ast (\ga)$, we get
\begin{equation}
D_Z (P_\eye) = \frac{-3-7 \sqrt{3}+5 \sqrt{5}+\sqrt{15}}{32} < 0, \quad \mbox{ and } \quad D_W (P_\eye) = \frac{-3+7 \sqrt{3}+5 \sqrt{5}-\sqrt{15}}{32} > 0\,.\notag
\end{equation}
Recall that $Y_0$ is the $Z$ coordinate of $P_\eye$. At $\ga = 7/5$ and $r = r^\ast$, we have
\begin{equation}
Y_0 = -\frac{5 \left(1+\sqrt{3}\right) \left(4+\sqrt{5}\right)}{4 \left(1+\sqrt{5}\right)^2} > -\sqrt{5} = Z_0\,.\notag
\end{equation}

Lastly, we need to show that for the case $\ga>1$, $r\in (r_3, r_4)$ we have that $P_\eye$ is a saddle point of $(N_W D_Z, N_Z D_W)$. We calculate the Jacobian at $P_\eye$, and its eigenvalues are given by
\begin{align*}
\lambda_\pm &= 
\frac{r}{2 (1-3 \gamma )^2} \left( (-9 \ga^2 +18 \gamma - 5) +(15 \gamma^2 -18\ga -1) r \pm \sqrt{B} \right)
= \frac{r}{2 (1-3 \gamma )^2} \left(A \pm \sqrt{B} \right), \\
B &= \left(27 \gamma ^2-30 \gamma +7\right)^2+(3 \gamma  (\gamma  (3 \gamma  (\gamma +20)-94)+28)+25) r^2+\left(106-6 \gamma 
   \left(45 \gamma ^3-110 \gamma +88\right)\right) r.
\end{align*}

Let $r' = \frac{3 \gamma -1}{\sqrt{3} (\gamma -1)+2}$. We will show that $B>0, A>0$ and $A^2-B > 0$ for all $r \in (1, r')$. This directly gives that $\lambda_+ > 0, \lambda_-<0$ for $r \in (1, r')$. Assume $r' > r_4$ by Lemma \ref{lemma:aux_otromas}, we would be done. 

We start with $A > 0$. As $A$ is affine with respect to $r$, it suffices to check that $A$ is positive for $r=1$ and for $r=r'$. This follows from
\begin{equation*}
A \Big|_{r=1} = 6(\ga^2-1) >0, \qquad A \Big|_{r=r'} = \frac{((6+2\sqrt{3})\ga + (9 - 5\sqrt{3})(\ga - 1) ) (3\ga - 1)(\ga - 1)}{2 + (\ga - 1)\sqrt{3} } > 0\,.
\end{equation*}

Now, we show $B > 0$. As an auxiliary step, we will show that $\frac{d}{dr} B < 0$. As $\frac{d}{dr} B$ is an affine function of $r$ it suffices to show that for $r=1$ and $r=r'$. We have:
\begin{align*}
\frac{d}{dr}\Big|_{r=1} B &= -12 (\gamma -1) (\gamma +1) (21 \gamma^2- 30\ga +13)  < 0, \\
\frac{d}{dr}\Big|_{r=r'} B &= -\frac{2 (\gamma -1) (3 \gamma -1) \left( 9 \left(5 \sqrt{3}-1\right) \gamma ^3+\left(15 \sqrt{3}-99\right) \gamma ^2+\left(213-105 \sqrt{3}\right) \gamma +53 \sqrt{3}-81\right)}{\sqrt{3} (\gamma - 1)+2} < 0\,.
\end{align*}
Therefore, $B \geq B|_{r=r'}$, so it suffices to show that $B|_{r=r'} > 0$. We have that
\begin{equation*}
B\Big|_{r=r'} = 
\frac{6  (3\ga-1)^2(\ga-1)^2 
\left( (19\sqrt{3}-18) \ga^2 + (60-34\sqrt{3}) \ga (\ga-1) + 26-15 \sqrt{3}\right)
}{\left(\sqrt{3} \gamma -\sqrt{3}+2\right)^2} > 0\,.
\end{equation*}

Finally, let us show that $A^2 - B > 0$. We have that
\begin{align*}
\frac{A^2 - B}{24(\ga-1)(3\ga-1)}
&= \underbrace{9 \gamma ^2-6 \gamma +1 +(4-12 \gamma ) r +\left(-3 \gamma ^2+6 \gamma +1\right) r^2}_{C}\,.  
\end{align*}
Now note that $C$ is a second order polynomial of $r$ and 
\begin{equation*}
C \Big|_{r=1} = 6(\ga^2-1) > 0, \qquad C \Big|_{r=r'} = 0 \qquad \mbox{ and } \qquad \frac{d}{dr} \Big|_{r=r'} C = -2 \sqrt{3} (\ga - 1) (3\ga - 1) < 0\,.
\end{equation*}
In particular, for any $\eps$ sufficiently small, $C$ is positive at $r = r'-\eps$. As $C$ is a polynomial of $r$, positive at $r=1$ and $r = r'-\eps$, it has an even quantity of roots (counted with multiplicity) in the interval $(1, r'-\eps)$. Therefore, as $C$ is a second-degree polynomial of $r$ and it has a root at $r=r'$, there are no roots in the interval $[1, r'-\eps)$. As we can take $\eps$ sufficiently small, we conclude that $C$ is positive for all $r \in [1, r')$.
\end{proof}

\begin{lemma} \label{lemma:aux_snes} Let $\ga = 7/5$ and $r$ sufficiently close to $r^\ast (7/5)$. Let $P' = \left( W_0 - T_{\rm DW}, Z_0 \right)$ for $T_{\rm DW} = \frac{5}{3} \left(\sqrt{4 r^2-14 r+11}-2 r+3\right)$. We have that $D_W(P') = 0$ and that along the horizontal segment $[P', P_s]$, the field $(N_W D_W, N_Z D_W)$ points downwards.
\end{lemma}
\begin{proof} Using the formulas for $W_0, Z_0$ in \eqref{eq:Ps}, we get that $D_W(P') = 0$. Then, we need to show that the sign of the third-degree polynomial $N_Z (W_0 - t, Z_0) D_W(W_0 - t, Z_0)$ is negative for $t \in (0, T_{\rm DW})$. We clearly have that $N_Z (W_0 - t, Z_0) D_W(W_0 - t, Z_0)$ vanishes at $0$ (because $N_Z(P_s) = 0$) and $T_{\rm DW}$ (because $D_W(P') = 0$). At $r = r^\ast$, we have
\begin{equation*}
\frac{N_Z (W_0 - t, Z_0) D_Z(W_0 - t, Z_0)}{t(T_{\rm DW}-t)} = \frac{-3}{50} (7\sqrt{5} - 5 - t) < 0\,,
\end{equation*}
so the first degree polynomial above is negative also for $r$ sufficiently close to $r^\ast$.
\end{proof}

\begin{lemma}
\label{lemma:aux_bextra} Let us recall that for $\gamma \in (1, +\infty)$ and $r \in (r_3, r_4)$ we define
\begin{equation*}
b^{\rm extra}(t) = (X_0 - t, Y_0 + t)\,,
\end{equation*}
where $X_0, Y_0$ are defined in \eqref{eq:Peye}. Let us also define $t_f^{\rm extra} = \frac12 (X_0 - Y_0)$. We have that $D_W(b^{\rm extra}(t)) > 0$ and $D_Z (b^{\rm extra}(t)) > 0$ for all $t \in (0, t_f^{\rm extra}]$. Moreover, if we let
\begin{equation*}
P^{\rm extra}(t) = (1, 1) \cdot \left( N_W (b^{\rm extra}(t)) D_Z (b^{\rm extra}(t),   N_Z (b^{\rm extra}(t)) D_W (b^{\rm extra}(t) \right)\,,
\end{equation*}
we have that $P^{\rm extra}(t) > 0$ for all $t \in (0, t_f^{\rm extra})$.
\end{lemma}
\begin{proof}
By direct calculation, we get that
\begin{align*}
t_f^{\rm extra} &= \frac{2 \sqrt{3} r}{3 \gamma -1}\,, \\
(6\ga - 2) D_W(b^{\rm extra}(t)) &= 
2 (3 \gamma -1)+2 \left(\sqrt{3} \gamma -\sqrt{3}-2\right) r-(\gamma -1 ) (3 \gamma -1) t \,,\\
 (6\ga - 2) D_Z(b^{\rm extra}(t)) &= 
 2 (3 \gamma -1)-2 \left(\sqrt{3} \gamma -\sqrt{3}+2\right) r+(\gamma -1) (3 \gamma -1) t \,,\\
 \frac{(3\ga - 1)^2}{t (\ga - 1) }P^{\rm extra}(t) &=
-4 \sqrt{3} r (-3 \gamma +2 r+1)+(3 \gamma -1 )  (2 r+1-3 \gamma )  t\, .
\end{align*}
The second expression is clearly decreasing and the third one increasing. With respect to the fourth one, note that $-3\ga + 2r + 1 \leq -\ga+1 < 0$ due to Lemma \ref{lemma:connecticut}, so it is also decreasing. Thus, we just need to show the following quantities are positive:
\begin{align*}
(6\ga - 2) D_W(b^{\rm extra}(t_f^{\rm extra})) &= 
-2-4r+6\ga > 0 \,,\\
 (6\ga - 2) D_Z(b^{\rm extra}(0)) &= 
-2 + 6\ga + 2r(-2+\sqrt{3}-\sqrt{3}\ga)\,, \\
 \frac{(3\ga - 1)^2}{t (\ga - 1) }P^{\rm extra}(t_f^{\rm extra}) &=
-2 \sqrt{3} r (1+2r-3\ga) > 0\, .
\end{align*}
The first and the third expressions are trivially positive (recall $r^\ast < \ga$ from Lemma \ref{lemma:connecticut}). The second expression is clearly decreasing with $r$, and it vanishes at $r = \frac{3\ga - 1}{2+\sqrt{3}(\ga - 1)}$. Therefore, we just need to show that $r_4 < \frac{3\ga - 1}{2+\sqrt{3}(\ga - 1)}$. This is shown in Lemma \ref{lemma:aux_otromas}.
\end{proof}

\begin{lemma} \label{lemma:aux_extralemma} Let $S$ be the vertical segment between $(W_0, Z_0)$ and $(W_0, W_0)$. Let either $(\gamma, r) = (7/5, r^\ast)$ or $(\gamma, r) \in (1, +\infty) \times (r_3, r_4)$. We have that $N_W(P) < 0$ for any $P \in S$.
\end{lemma}
\begin{proof}
The second-degree polynomial $N_W(W_0, y)$ (in $y$) has second derivative $\frac{\gamma - 1}{2} > 0$. As we want to show that the polynomial is negative for $y \in [Z_0, W_0]$, it suffices to show it at the endpoints. Clearly $N_W(W_0, Z_0) = N_{W, 0} < 0$ in our range of $(\gamma, r)$ because of Lemmas \ref{lemma:aux_34bounds} and \ref{lemma:aux_limits}. Observe that
\begin{align*}
16(\ga - 1)^4 N_W(W_0, W_0) &=  A\cdot B   \\
&=\left( 3(\ga-3)\left(\ga-\frac13 \right) +  ( -5\ga^2 +10\ga - 1 )r   -(1+\ga) \mc R_1 \right) \\
&\qquad \cdot \left( -3(\ga-3)\left( \ga - \frac13 \right) + (\gamma -3)(\gamma +1) r + (\gamma +1) \mc R_1 \right).
\end{align*}
In the case $\gamma = 7/5$, $r = r^\ast (7/5)$, we obtain $A \cdot B = \frac{64}{625} (-45+17\sqrt{5}) < 0$. From now on, let us assume $r \in (r_3, r_4)$.

We first show $B > 0$. If $\gamma \in (1, 3)$, as $(\ga-3)(\ga+1)r$ decreases with $r$, and $r < r_4 < \frac{3\ga - 1}{2+\sqrt{3}(\ga - 1)}$ we get
\begin{align*}
 B &> \left( 3(\ga-3)\left(\ga-\frac13 \right) +  ( -5\ga^2 +10\ga - 1 ) \frac{3\ga - 1}{2+\sqrt{3}(\ga - 1)} \right)  \\
&= \frac{-\left(\sqrt{3}-1\right) (\gamma -3) (\gamma -1) (3 \gamma -1)}{(2+\sqrt{3} (\ga-1))} > 0\,.
\end{align*}
Now, let us assume that $\gamma \geq 3$. We get that 
\begin{equation*}
\mc R_1 \frac{dB}{dr} = (\gamma +1) \left(-3 (\gamma -2) \gamma  +(\gamma -3)^2 r-7 +(\gamma -3) \mc R_1\right).
\end{equation*}
We claim $\frac{dB}{dr} < 0$. As $|-3 (\gamma -2) \gamma| >(\ga-3)^2 r$ (because $\ga, \ga-2 > \ga-3$ and $3>r$), the claim $\frac{dB}{dr} < 0$ follows from
\begin{equation*}
\left(-3 (\gamma -2) \gamma  +(\gamma -3)^2 r-7 \right)^2 - (\gamma -3)^2 \mc R_1^2 = 32 (\ga-1)^3 > 0\,.
\end{equation*} 
Therefore, to prove $B > 0$, we just need to show $B > 0$ at $r = r^\ast$. We have that
\begin{equation*}
B|_{r = r^\ast} = \frac{8 \left(\sqrt{3}-1\right) (\gamma -1)^2}{\sqrt{3} \gamma -\sqrt{3}+2} > 0\,.
\end{equation*}

Now, let us prove $A < 0$. If we have $-5\ga^2 + 10\ga - 1 \leq 0$
\begin{align*}
A &= \left( 3(\ga-3)\left(\ga-\frac13 \right) +  ( -5\ga^2 +10\ga - 1 )r   -(1+\ga) \mc R_1 \right)  \\
&\leq 
\left( 3(\ga-3)\left(\ga-\frac13 \right) +  ( -5\ga^2 +10\ga - 1 )   \right) = 2-2\ga^2 < 0\,.
\end{align*}
Thus, we just need to consider the case where $-5\ga^2 + 10\ga - 1 > 0$, which is $\ga < 1 + \frac{2}{\sqrt{5}}$. From now on, let us assume $\ga < 1 + \frac{2}{\sqrt{5}}$. In that case, we have that $r < \frac{2}{\left(\sqrt{2} \sqrt{\frac{1}{\gamma -1}}+1\right)^2}+1$. Therefore, we have that
\begin{align*}
A &\leq  3(\ga-3)\left(\ga-\frac13 \right) +  ( -5\ga^2 +10\ga - 1 ) \left( \frac{2}{\left(\sqrt{2} \sqrt{\frac{1}{\gamma -1}}+1\right)^2}+1 \right) \\
&\leq -\frac{4 \sqrt{\gamma -1} 
\left(\left(3 \sqrt{\gamma -1}+\sqrt{2}\right) \gamma -\sqrt{\gamma -1}+\sqrt{2}\right)
}{\left(\sqrt{2} \sqrt{\frac{1}{\gamma  -1}}+1\right)^2}\,.
\end{align*}
Noting $\sqrt{2} > \sqrt{\gamma - 1}$, we see that all the terms in parenthesis in the last expression are positive, so we conclude $A < 0$ and we are done.
\end{proof}

\begin{lemma} \label{lemma:aux_diagonal_from_P2} Let $\ga > 1$. Consider the $S_1$ the diagonal halfline of slope $-1$ from $P_s$, that is, the halfline given by $(W_0 + t, Z_0 - t)$ for $t > 0$. We have that 
\begin{equation*}
\left( N_W D_Z ,  N_Z D_W\right) \cdot (-1, -1) > 0 \qquad \mbox{ on } S_1\,.
\end{equation*}
\end{lemma}
\begin{proof} 
We have that
\begin{align*}
\frac{2}{t} \left( N_W D_Z + N_Z D_W \right) \Big|_{(W_0+t, Z_0-t)} &=
 t \left(2 (\gamma -1) r + 4 (\gamma +2)+6 W_0+3 \left(\gamma ^2+1\right) Z_0\right) \\
 &\qquad +4 r (W_0+4 Z_0+2) +4 \left(W_0 (3 Z_0+2)+3 \gamma Z_0^2+(3 \gamma +2)Z_0+4\right) \\
 &= At + B\,.
\end{align*}
We will just prove that $A$ and $B$ are both negative, which yields the desired statement.

For $B$, we have that
\begin{equation*}
\frac{2(\ga - 1)}{r-1} B = -9 \gamma ^2+22 \gamma  -17 + (\gamma -3) (3 \gamma -5) r+(3 \gamma -5) \mc R_1 \,.
\end{equation*}
As $-9 \gamma ^2+22 \gamma  -17$ is negative for all real $\gamma$, $B < 0$ just follows if we show the following quantity is positive:
\begin{align*}
\left(  -9 \gamma ^2+22 \gamma  -17 \right) - \left(  (\gamma -3) (3 \gamma -5) r+(3 \gamma -5) \mc R_1 \right)^2 = 32 (\gamma -1)^2 ((3 \gamma -5) r+2)\,.
\end{align*}
Thus, we just need to show $(3\ga - 5)r+2 > 0$. This is trivial for $\ga \geq \frac53$, and for $\ga < \frac53$ Lemma \ref{lemma:connecticut} yields
\begin{equation} \label{eq:cosilla2}
(3\ga - 5)r+2 > \left( 3\ga - 5 \right) \left( 2 - \frac{1}{\ga} \right) + 2 = \frac{(6\ga - 5)(\ga - 1)}{\ga} > 0\,.
\end{equation}

Now, let us show that $A < 0$. We have that
\begin{equation} \label{eq:cosilla}
\frac{4}{\ga - 1}A = -9 \gamma +(3 \gamma -1) r + 7 +3 \mc R_1\,.
\end{equation}
By Lemma \ref{lemma:connecticut}, we have
\begin{equation*}
 -9 \gamma +(3 \gamma -1) r + 7 \leq -9 \gamma +(3 \gamma -1) \left( 2 - \frac{1}{\ga} \right) + 7 = 2 + \frac{1}{\ga} - 3\ga < 0\,.
\end{equation*}
Therefore, in order to show \eqref{eq:cosilla}, it suffices to note
\begin{align*}
(-9 \gamma +(3 \gamma -1) r + 7 )^2 - (3 \mc R_1)^2 = 
16 (r-1) ((3 \gamma -5) r+2) > 0\,,
\end{align*}
where $(3\ga-5)r+2 > 0$ was already shown in \eqref{eq:cosilla2}.
\end{proof}

\begin{lemma} \label{lemma:aux_DZ=0_repels} Let $H$ be the halfline of $D_Z = 0$ which is to the right of $P_s$. We have that $\nabla D_Z  \cdot (N_W D_Z, N_Z D_W) \geq 0$ on $H$. Moreover, $\nabla D_Z  \cdot (N_W D_Z, N_Z D_W)$ is negative between $\bar P_s$ and $P_s$, and positive again to the left of $\bar P_s$. \end{lemma}
\begin{proof} The intersection of $D_W= 0$ and $D_Z = 0$ is $(-1, -1)$ which on the diagonal $W = Z$, thus it is the leftmost point of $D_Z$ in our region of interest $W > Z$. On $H$, we have
\begin{equation*}
\nabla D_Z \cdot (N_W D_Z, N_Z D_W)  = N_Z D_W \p_Z D_Z = \frac{\gamma+1}{4} D_W N_Z\,.
\end{equation*}
Therefore, we just need to analyze the sign $N_Z > 0$ . The system $N_Z = D_Z = 0$ has two solutions (by B\'ezout), which are $P_s$ and $\bar P_s$. From their formulas \eqref{eq:Ps} and \eqref{eq:Psbar}, we see that $P_s$ is the rightmost. Thus, $N_Z$ will have constant sign on $H$, so we just need to check the sign on $N_Z$ at infinity along the direction of $H$, given by $\nabla D_Z^\perp = \left( \frac{1+\alpha}{2}, -\frac{1-\alpha}{2}  \right)$. We get
\begin{equation*}
\lim_{t \to \infty} \frac{1}{t^2} N_Z \left( \frac{1+\ga}{4}t, -\frac{3-\ga}{4} t \right) = \frac{1}{8}  \left(\ga^2 - 2\ga + 1 \right) = \frac{1}{8} (\ga - 1)^2 > 0\,.
\end{equation*}
\end{proof}

\subsection*{Computer assisted Lemmas}

\begin{lemma} \label{lemma:aux_otromas} We have that $r_4 < \frac{3\ga - 1}{2+\sqrt{3}(\ga - 1)}$ for all $\ga >1$
\end{lemma}
\begin{proof} If $\ga \geq \frac53$, the proof follows trivially because $r^\ast (\ga ) = \frac{3\ga - 1}{2+\sqrt{3}(\ga - 1)}$. If $1<\ga<\frac53$, the proof is computer-assisted and we refer to Appendix B for details about the implementation.
\end{proof}

\begin{lemma} \label{lemma:aux_Peyeout} Let $r = r^\ast$ and $\gamma = 7/5$. Let us recall \eqref{eq:bfl7o5} 
\begin{equation*}
b^{\rm fl}_{7/5}(t) = \left( W_0 + W_1 t + \frac{W_2}{2} t^2 - \left( W_0 + W_1 + \frac{W_2}{2} \right)t^3 ,
 Z_0 + Z_1 t + \frac{Z_2}{2} t^2 - \left( Z_0 + Z_1 + \frac{Z_2}{2} \right)t^3 \right)\,,
\end{equation*}
and \eqref{eq:Peye}
\begin{equation*} 
P_\eye=(X_0, Y_0)= \left( \frac{2 \left(\sqrt{3}-1\right) r}{3 \ga-1},-\frac{2 \left(1+\sqrt{3}\right) r}{3 \ga-1}\right).
\end{equation*}
There exists some $t_W \in (0, 1)$ such that $b^{\rm fl}_{7/5, W} (t_W) = X_0$ and  $b^{\rm fl}_{7/5, Z}(t_W) > Y_0$. That is, $P_\eye$ is below $b^{\rm fl}_{7/5}(t)$. Moreover, we have that $b^{\rm fl}_{7/5, W}(t)$ is decreasing for $t\in (0, t_W)$.
\end{lemma}
\begin{proof} The first part of the proof is computer-assisted and we refer to Appendix B for details about the implementation. In order to show that $b^{\rm fl}_{7/5, W}(t)$ is decreasing, we calculate the polynomial at $\gamma = 7/5$, $r = r^\ast$. We obtain
\begin{equation*}
b^{\rm fl}_{7/5, W}(t) = -\frac{1}{12} \left(3 \sqrt{5}-5\right) (t (15 t-4)+3)\,,
\end{equation*}
which is globally negative.
\end{proof}

\begin{lemma} \label{lemma:paella} Let $\gamma > 1$ and $r \in (r_3, r_4)$. Let $\theta_{\rm extra} = \frac34 \pi$ and $\theta_{\rm fl} \in [-\pi, \pi)$ to be the angle
\begin{equation*}
\theta_{\rm fl} = - \arctan \left( \frac{ b_Z^{\rm fl \; \prime}(1) }{ b_W^{\rm fl \; \prime}(1) } \right),
\end{equation*}
that is, the angle formed by $b^{\rm fl}(t)$ for $t\in [0, 1]$ when arriving at $P_\eye$ at $t=1$. 

Then, we have that $\theta_{\rm fl} < \theta_{\rm extra}$ and $\theta_{\rm fl} > \theta_{\rm extra} - \pi$.
\end{lemma}
\begin{proof}
We need to show that $-\frac{\pi}{4} \leq \theta_{\rm fl} \leq \frac{3\pi}{4}$, or equivalently, that the vector $-b^{\rm fl\; \prime}(1)$ has angle between $-\frac{\pi}{4}$ and $\frac{3\pi}{4}$. That is equivalent to $-b^{\rm fl\; \prime}(1) \cdot (1, 1) > 0$. We show that $b^{\rm fl\; \prime}_W(1) + b^{\rm fl\; \prime}_Z(1) < 0$ with a computer-assisted proof. Details about the implementation can be found in Appendix B.
\end{proof}

\begin{lemma}   \label{lemma:aux_34bounds}   Let $\gamma \in (1, + \infty)$ and $r \in (r_3, r_4)$. We have that: 
\begin{align*}
D_{W, 1}  &> 0, &\qquad
N_{W, 0} &< 0, &\qquad
N_{Z, 1} &> 0, &\qquad
W_1 &< 0, \\
Z_4 - \frac{W_4 Z_1}{W_1} &> 0, &\qquad
\p_Z N_Z (P_s) &> 0, &\qquad
D_Z (P_\eye) &> 0,  &\qquad
\frac{X_0 + Y_0}{2} &> -r.
\end{align*}
\end{lemma}
\begin{proof} The first five inequalities are done with a computer-assisted proof. We refer to Appendix B for details about the implementation. Let us start with $\p_Z N_Z (P_s)$. We have that
\begin{equation*}
16(\ga - 1)^2 \p_Z N_Z (P_s) = 9 \gamma ^3-5 \gamma ^2-5 \gamma -\left(3 (\gamma -1) \gamma ^2+\gamma +7\right) r+9
+ (-3\ga+1 )(\ga - 3)\mc R_1 = A+B\mc R_1\,.
\end{equation*}
We have that
\begin{equation*}
\frac{d}{dr}A = -(\ga + 1)(3 \ga^2 - 7\ga + 6) < 0\,,
\end{equation*}
because the polynomial $3\ga^2 - 7\ga + 6$ does not have real roots. For $1 < \ga \leq \frac53$, we get that
\begin{equation*}
A \geq A \Big|_{r=r^\ast} = \frac{ 4 (\gamma +1) \left(2 \sqrt{2} \ga \sqrt{\gamma -1}  + \left( \sqrt{2} \sqrt{\gamma -1}+4\right) (\gamma - 1)\right) }{ \left(\frac{\sqrt{2}}{\sqrt{\gamma -1}}+1\right)^2} > 0\,,
\end{equation*}
and for $\ga \geq \frac53$,
\begin{align*}
A &\geq A \Big|_{r=r^\ast} = \frac{ 9 \left(\sqrt{3}-1\right) \gamma ^2+\left(30-14 \sqrt{3}\right) \gamma +9 \sqrt{3}-25 }{ \sqrt{3} (\gamma -1)+2 } \\
& \geq \frac{ 9 \left(\sqrt{3}-1\right)+\left(30-14 \sqrt{3}\right)+9 \sqrt{3}-25 }{ \sqrt{3} (\gamma -1)+2 } =  \frac{ 4 (\sqrt{3}-1) }{ \sqrt{3} (\gamma -1)+2 } > 0\,.
\end{align*}
Therefore, we get that $A > 0$ globally. If $\ga \leq 3$, we get that $B = (-3\ga + 1)(\ga - 3) \geq 0$ so we are done. From now on, let us assume $\ga > 3$. We will show that $A^2 - B^2 \mc R_1^2 > 0$ which will conclude the proof of the sign of $\p_Z N_Z (P_s)$. We have
\begin{equation*}
A^2 - B^2\mc R_1^2 = 32 (\gamma -1)^2 \left(2 \gamma  (\gamma  (9 \gamma -14)+9)+(\gamma  (\gamma  (3 \gamma -11)+17)-1) r^2-4 \gamma  (3 (\gamma -2) \gamma +7) r\right) = 32 (\gamma -1)^2 C\,,
\end{equation*}
so we need to show $C > 0$. We have
\begin{align*}
\frac{d^2}{dr^2} C &= - 2 + 34\ga - 22\ga^2 + 6\ga^3 = 6\ga (\ga - 2)^2 + 2\ga^2 + 10\ga - 2 > 0\,,
\end{align*}
so
\begin{align*}
\frac{d}{dr} C &\leq \frac{d}{dr} \Big|_{r=r^\ast} C = \frac{-2(\ga - 1)}{2 + \sqrt{3} (\ga - 1)}
 \left( 3 \left(2 \sqrt{3}-3\right) \gamma ^3+\left(39-12 \sqrt{3}\right) \gamma ^2+\left(14 \sqrt{3}-47\right) \gamma +1 \right) \\
 &\leq \frac{-2(\ga - 1)}{2 + \sqrt{3} (\ga - 1)}
  \left( 3 \left(2 \sqrt{3}-3\right) \gamma \cdot 9+\left(39-12 \sqrt{3}\right) \gamma \cdot 3 +\left(14 \sqrt{3}-47\right) \gamma +1 \right) \\
  & \leq \frac{-2(\ga - 1)}{2 + \sqrt{3} (\ga - 1)}
  \left( 1 + (32\sqrt{3} - 11) \ga \right) < 0\,.
\end{align*}
Finally, we have
\begin{align*}
 \frac{(2 + \sqrt{3} (\ga - 1))^2}{ -(\gamma -1)} C &\leq  \frac{(2 + \sqrt{3} (\ga - 1))^2}{ -(\gamma -1)} C \Big|_{r=r^\ast} \\
 &=
 \left(9 \left(4 \sqrt{3}-9\right) \gamma ^4+\left(300-156 \sqrt{3}\right) \gamma ^3+\left(220 \sqrt{3}-438\right) \gamma ^2+\left(204-100 \sqrt{3}\right) \gamma -1\right) \\
 &\leq  \left(9 \left(4 \sqrt{3}-9\right) \gamma ^3 \cdot 3+\left(300-156 \sqrt{3}\right) \gamma ^3+\left(220 \sqrt{3}-438\right) \gamma \cdot 3+\left(204-100 \sqrt{3}\right) \gamma -1\right) \\
 &= \left(57-48 \sqrt{3}\right) \gamma ^3+10 \left(56 \sqrt{3}-111\right) \gamma -1 < 0\,.
\end{align*}

Now, we show $D_Z (P_\eye) > 0$, we get that
\begin{equation*}
(3\ga - 1)D_Z (P_\eye) = 3\ga - 1 - r \left( 2 + \sqrt{3}(\ga - 1)\right)\,.
\end{equation*}
Now, this quantity vanishes at $r = \frac{3\gamma - 1}{2+\sqrt{3}(\gamma-1)}$ and by Lemma \ref{lemma:aux_otromas}, $r_4 < \frac{3\gamma - 1}{2+\sqrt{3}(\gamma-1)}$.

Lastly, from \eqref{eq:Peye}
\begin{equation*}
\frac{X_0 + Y_0}{2} - r = \frac{2r}{1-3\ga} + r = \frac{ 3r (\ga - 1)}{3\ga - 1} > 0\,.
\end{equation*}
\end{proof}

\begin{lemma}  \label{lemma:aux_WioverZi_7o5} For $\ga = 7/5$ and $r = r^\ast (7/5)$ we have that $|W_i/Z_i|< 2$ for all $0 \leq i \leq 160$. 
\end{lemma}
\begin{proof}  We prove the statement via a computer-assisted proof. The code can be found in the supplementary material and we refer to Appendix \ref{sec:computer} for details about the implementation. \end{proof}

\begin{lemma} \label{lemma:tenthousand_7o5} Let $\gamma = 7/5$ and $r = r^\ast(7/5)$. 
\item For any $i \leq 160$, we have 
\begin{equation*}
\left| \bar C_\ast (i+1)^2 \right| < \left| \frac{Z_{i+1}}{Z_i} \right| < \left| \bar C_\ast (i+1)^2 \frac{500}{i+1} \right|\,.
\end{equation*} 
\item  For $160 \leq i \leq 10000$, we have the further refinement 
\begin{equation*}
 \left| \bar C_\ast (i+1)^2  \right| < \left| \frac{Z_{i+1}}{Z_i} \right| < \left| 3\bar C_\ast (i+1)^2 \right|\,. 
 \end{equation*}
\item Moreover,
\begin{equation} \label{eq:Z10000}
\abs{Z_{10000}+6\cdot 10^{46770}}\leq 10^{46770}\,.
\end{equation}

\end{lemma}
\begin{proof} We prove the statement via a computer-assisted proof. The code can be found in the supplementary material and we refer to Appendix \ref{sec:computer} for details about the implementation. \end{proof}

\begin{lemma} \label{lemma:aux_F1} Let either $\gamma > 1 $ and $r = r_3$ or $\gamma = 7/5$ and $r = r^\ast(7/5)$. We have that $\frac{Z_1/2 - W_1}{W_1 + Z_1} \leq -1$.
\end{lemma}
\begin{proof}
The proof for $r = r_3$ is computer assisted. The code can be found in the supplementary material and details about the implementation can be found in Appendix \ref{sec:computer}. For the case $\ga = 7/5$, $r = r^\ast ( 7/5 )$, we have that
\begin{equation}
\frac{Z_1 / 2 - W_1}{W_1 + Z_1} = -4 < -1\,.\notag
\end{equation}
\end{proof}

\begin{lemma} \label{lemma:aux_a2nr} Let us recall
\begin{equation*}
a_2^{\rm nr} = \frac{s_\infty^{\rm fr}}{2} \left( -W_2 + Z_2/2 - (W_2+Z_2)\frac{Z_1/2-W_1}{W_1+Z_1}\right) + W_1^2 + W_1Z_1/2 - Z_1^2/2\,.
\end{equation*}
Let either $\gamma > 1$ and $r = r_3$ or $\gamma = 7/5$ and $r = r^\ast(7/5)$. We have that $a_2^{\rm nr} > 0$.
\end{lemma}
\begin{proof}
The proof is computer assisted. The code can be found in the supplementary material and details about the implementation can be found in Appendix \ref{sec:computer}.
\end{proof}

\begin{lemma} \label{lemma:aux_signsZ3Z4} Let $\gamma > 1$. We have that $Z_3 > 0$ for $r$ close enough to $r_3$ from above and $Z_4 > 0$ for $r$ close enough to $r_4$ from below. In other words, we have that:
\begin{align*}
Z_3(3-k) < 0 \quad \quad \text{for} \; \; r = r_3\,,\\
Z_4(4-k) > 0 \quad \quad \text{for} \; \; r = r_4\,.
\end{align*}
\end{lemma}
\begin{proof}
The proof is computer-assisted and we refer to Appendix B for details about the implementation.
\end{proof}

\subsection*{Additional properties of the profiles}
Let us recall that $\mw W(\zeta ) = \zeta W^E (\xi)$ and $\mw Z (\xi) = \zeta Z^E (\xi)$, where $\xi=\log \zeta$ and $(W^E, Z^E)$  is a solution to the ODE \eqref{eq:wz:ODE}. We use the notation $D_W^E$ to denote the function
\[D_W^E(\xi)=D_W(W^E(\xi), Z^E(\xi))\,.\]
We define $D_Z^E$, $N_W^E$ and $N_Z^E$ in an analogous fashion.

\begin{lemma} \label{lemma:aux_Slowbound} We have that $\mw S > 0$.
\end{lemma}
Our profile for Euler is in the region $W^E-Z^E > 0$ and cannot cross $W^E = Z^E$ because it is an invariant manifold of the ODE \eqref{eq:wz:ODE}. Therefore, $\mw S = \frac{\mw W - \mw Z}{2} > 0$.

\begin{lemma} \label{lemma:DWDZproperties} We have that $\zeta + \mw U - \alpha \mw S$ is uniformly bounded and strictly positive for $\zeta > 1$. Moreover, $\zeta + \mw U - \alpha \mw S > \zeta \eps$ for all $\zeta > \frac65$ and some $\eps > 0$ sufficiently small. We also have that $\zeta + \mw U + \alpha \mw S > \zeta \eps$ for all $\zeta > \frac65$ and some $\eps' > 0$.
\end{lemma}
\begin{proof} Passing to $W^E, Z^E, \xi$ coordinates the stated inequality reads
\begin{equation}
e^{\xi} \left(1 + W^E(\xi) \frac{1 - \al}{2} + Z^E (\xi) \frac{1 + \al}{2} \right) = e^{\xi} D_Z^E (\xi)< 0\,,\notag
\end{equation}
for $\xi > 0$. Now, let us recall that our profile $(W(\xi), Z(\xi))$ is given by Proposition \ref{prop:left_main} for $\xi > 0$. In particular, $D_Z^E > 0$.

With respect to the second claim, just note that the solution converges to $P_\infty = (0, 0)$ (with $D_Z (P_\infty) = 1$), so in particular, we will have that $D_Z^E(\xi) > \eps$ for $\xi > C$, where $C$ is a sufficiently large constant. Then, as $[\log (6/5), C]$ is a compact interval where $D_Z^E > 0$, we can find an $\eps > 0$ sufficiently small that bounds $D_Z$ from below. Therefore, as $e^{\xi} = \zeta$, we conclude that $e^{\xi} D_Z^E(\xi) > \eps \zeta$.

With respect to the third claim, just note that $\zeta + \mw U + \alpha \mw S$ corresponds to $\xi D_W^E(\log \zeta)$. By Proposition \ref{prop:left_main}, we have that $D_W^E > 0$ for $\xi > \log (6/5)$, and using the same argument as in the paragraph above for $D_Z^E$, we conclude that $e^{\xi} D_W^E(\xi) > \zeta \eps'$ for every $\xi > \log (6/5)$ and some $\eps' > 0$ sufficiently small. 
\end{proof}

\begin{lemma} \label{lemma:Sincrease} For $\ga = 7/5$, $r$ sufficiently close to $r^\ast$ and $\zeta<1$ we have that $\p_\zeta \mw S > 0$.
\end{lemma}
\begin{proof} Writting this in $(W, Z)$ variables, we need to show that $W^E - Z^E +\frac{N_W^E}{D_W^E} - \frac{N_Z^E}{D_Z^E} > 0$. As for $\zeta < 1$ (that is, $\xi < 0$) we have $D_W > 0, D_Z < 0$, we can reduce to show negativity of:
\begin{align*}
& D_W^E D_Z^E (W^E - Z^E) + N_W^E D_Z^E - N_Z^E D_W^E \\
& \qquad = \frac{-1}{10}(W-Z) \left(-\sqrt{4 r^2-14 r+11}+2 r+W+Z-1\right) \left(\sqrt{4 r^2-14 r+11}+2 r+W+Z-1\right).
\end{align*}
As $W > Z$, it suffices to show that for $\xi > 0$ the $Z^E$ stays above $Z_\Delta (W) = \sqrt{4 r^2-14 r+11}-2 r-W+1$. $Z_\Delta (W)$ is a diagonal line of slope $-1$ passing through $P_s$, and from the proof of Proposition \ref{prop:solnearxi0} (see also Lemma \ref{lemma:aux_diagonal_from_P2}), we have that our solution stays above. 
\end{proof}

\begin{lemma} \label{lemma:Wsuffices} For $\ga = 7/5$, $r$ sufficiently close to $r^\ast$ and $\zeta > 1$ we have that
\begin{equation} \label{eq:sanxenxo}
 M:=- \frac{1-\alpha}{1+\alpha} D_Z^E  + D_W^E + \left(  \frac{1+\alpha}{2} - \frac{(1-\alpha)^2}{2(1+\alpha)} \right) \frac{N_W^E}{D_W^E} > 0\,.
\end{equation}
\end{lemma}
\begin{proof} Given that the solution is in the region $D_W > 0$, it suffices to show positivity for
\begin{equation} \label{eq:corinna}
30 D_W^E M = -W (10 r+W-16)+Z (Z+4)+10\,.
\end{equation}
Let us recall from the proof of Proposition \ref{prop:mainleft7o5} that $(W^E, Z^E)$ is contained in the triangle $T$ delimited by the lines $W = W_0$, $W=Z$ and $D_Z = 0$. As both $T$ and \eqref{eq:corinna} depend continuously on $r$, we reduce to show that every point of $T$ satisfies \eqref{eq:corinna} for $r=r^\ast (7/5)$. Thus, let us fix $r=r^\ast$ from now on.

Solving \eqref{eq:corinna} in $W$, we see that the previous quantity is positive if $W_{(-)}(Z) \leq W \leq W_{(+)}(Z)$, where $W_{(i)}(Z)$ are given by
\begin{equation}
W_{(\pm)}(Z) = \frac{1}{4} \left(+5 \sqrt{5}-3 \pm \sqrt{16 Z^2+64 Z-30 \sqrt{5}+294}\right),\notag
\end{equation}
and the radical is positive for all $Z$. Now, the triangle $T$ is contained in $-1 \leq W \leq W_0$, because the rightmost side of $T$ is $W = W_0$ and the leftmost point is $(-1, -1)$ (the intersection of $D_Z = 0$ and $W=Z$). Thus, it suffices to show that $W_{(-)}(Z) < -1$ and $W_0 < W_{(+)}(Z)$. 

We have that
\begin{align*}
4 (-1-W_{(-)}(Z)) &= -5 \sqrt{5}-1 + \sqrt{16 Z (Z+4)-30 \sqrt{5}+294}\,, \\
4 ( W_{(+)}(Z) - W_0 ) &= 7-\sqrt{5} + \sqrt{16 Z (Z+4)-30 \sqrt{5}+294}\,,
\end{align*}
where we recall that the radical is positive for all $Z$. The second expression is clearly positive and the first one is also positive because 
\begin{equation*}
16 Z (Z+4)-30 \sqrt{5}+294 - \left( -5 \sqrt{5}-1 \right)^2 = 8 (13 - 5 \sqrt{5} + (Z+2)^2) > 0\,.
\end{equation*}
\end{proof}

\begin{lemma} \label{lemma:Vdamping} For $\ga = 7/5$ and $r$ sufficiently close to $r^\ast$, there exists a value $\eta_{\rm damp} > 0$ such that $1+\p_\zeta \mw U - \alpha | \p_\zeta \mw S |  > \eta_{\rm damp}$ globally. 
\end{lemma}
\begin{proof} First, observe
\begin{align}
1+\p_\zeta \mw U -  \alpha \p_\zeta \mw S &= \p_\zeta \left( \zeta + \zeta \frac{W^E + Z^E}{2} - \alpha \zeta \frac{W^E - Z^E}{2} \right) = \p_\zeta ( \zeta D_Z^E)\notag \\
&= D_Z^E + \zeta \p_\zeta D_Z^E = D_Z^E + \frac{1-\alpha}{2} \frac{N_W^E}{D_W^E} + \frac{1+\alpha}{2} \frac{N_Z^E}{D_Z^E} \,,\label{eq:logitechZ}
\end{align}
and similarly
\begin{equation} \label{eq:logitechW}
1 + \p_\zeta \mw U + \alpha \p_\zeta \mw S = D_W^E + \zeta \p_\zeta D_W^E = D_W^E + \frac{1+\alpha}{2} \frac{N_W^E}{D_W^E} + \frac{1-\alpha}{2} \frac{N_Z^E}{D_Z^E}\,.
\end{equation}

Let us first reduce to show that \eqref{eq:logitechZ} is greater than $\eta_{\rm damp}$. In the region $\zeta \leq 1$, Lemma \ref{lemma:Sincrease} yields that $\p_\zeta \mw S > 0$, so it is clear that $1+\p_\zeta \mw U - \alpha | \p_\zeta \mw S |$ is given by \eqref{eq:logitechZ}. With respect to the region $\zeta > 1$, note that $\eqref{eq:logitechW} = \frac{1-\alpha}{1+\alpha}\eqref{eq:logitechZ} + \eqref{eq:sanxenxo}$. Therefore, as \eqref{eq:sanxenxo} is positive by Lemma \ref{lemma:Wsuffices}, is also suffices to show that \eqref{eq:logitechZ} is lower bounded by some $\eta_{\rm damp} > 0$. In order to show this, we divide in two cases: $\xi \leq 0$ and $\xi > 0$.

\textbf{Case $\xi < 0$.} We start showing \eqref{eq:logitechZ} is greater than $\eta_{\rm damp} > 0$ for $\zeta \leq 1$ (that is, $\xi \leq 0$). As we will show that this is strictly positive at $\xi = 0$ and as $\xi \rightarrow -\infty$, by compactness (say, reparametrising the domain), we can reduce to showing that \eqref{eq:logitechZ} is positive. 

Let us start noting that the statement is true at $\xi = 0$ because $D_{Z, 1} > 0$ due to Lemma \ref{lemma:aux_limits}. Now, we show it for $\xi < 0$. As $D_W > 0, D_Z < 0$ in this region, we need to show negativity for:
\begin{align} 
A = 50 D_W D_Z \eqref{eq:logitechZ} &= Z^2 (-(12 r+W-23))-30 r (W+1) Z \notag \\
&\qquad +W (-4 r (2 W+5) +W (W+21)+70) +56 W Z+80 Z+50\,.\label{eq:A}
\end{align}
Solving in $Z$, we get that we need to show $Z$ is below
\begin{align*}
f(W) &=\frac{\sqrt{g(W)}+30 r W+30 r-56 W-80}{2 (-12 r-W+23)}, \qquad \mbox{ where }\\
g(W) &= (-30 r W-30 r+56 W+80)^2-4 (-12 r-W+23) \left(-8 r W^2-20 r W+W^3+21 W^2+70 W+50\right).
\end{align*}
For $W > W_0$ and $r$ close to $r^\ast$, we have $g(W) > 0$. Note that our solution can also be parametrized as some $Z_\ast (W)$ because is decreasing in $W$ (by Remark \ref{rem:horizontal}). We know that $Z_\ast (W) < f(W)$ for $W$ close enough (from above) to $W_0$, because we already checked the sign for $\xi = 0$. Moreover, we have $f(W) = -W + (6-5r) + O(1/W)$, while $Z_\ast (W) = -W - \frac{4(r-1)}{3(\ga - 1)} + O(1/W)$ from the proof of Proposition \ref{prop:solnearxi0}. Thus, as $6-5r^\ast > \frac{-4(r^\ast-1)}{6/5}$, we also have that $Z_\ast(W) < f(W)$ for $W$ sufficiently large. Thus, by continuity, if $Z_\ast (W)$ crosses $f(W)$ at some $W \in (W_0, +\infty)$, is has to do so in both directions. We show that this is impossible by checking that the field $(N_W D_Z, N_Z D_W)$ points always to the same side of the curve $(W, f(W))$. Indeed, defining
\begin{equation*}
P_f(W) = f'(W) N_W(W, f(W)) D_Z(W, f(W)) - N_Z (W, f(W)) D_W(W, f(W))\,.
\end{equation*} 
we have that $P_f(W) < 0$ for $r$ sufficiently close to $r^\ast$. Clearly, we have that $P_f(W_0) = 0$ because $N_Z(P_s) = D_Z (P_s) = 0$. Moreover, we have that 
\begin{equation*}
\frac{\p}{\p r}\Big|_{r=r^\ast} ( P_f'(W_0) ) = \frac{2}{363} \left(17 \sqrt{5}-31\right) > 0\,
\end{equation*}
so we get that $P_f'(W_0)$ is negative for $r$ close enough (from below) to $r^\ast$. Therefore, it suffices to check $\frac{P_f(W)}{(W-W_0)^3} \Big|_{r=r^\ast} > 0$. Let us fix $r = r^\ast$ for the rest of this case. We have that
\begin{align*}
P_f(W) &= \frac{ P_{f, 1}(W) + P_{f, 2}(W) \sqrt{ g(W) } }{ 100 (2 + 3\sqrt{5} - W)^5 \sqrt{ g(W)} } \\
P_{f, 1}(W) &=  880 \sqrt{5} W^9+872 W^9+467 \sqrt{5} W^8-5115 W^8-15117 \sqrt{5} W^7-74695 W^7 -447170 \sqrt{5} W^6 \\
&-864890 W^6-3634655 \sqrt{5} W^5-4535775 W^5-14090850 \sqrt{5}
   W^4-19856150 W^4-47811625 \sqrt{5} W^3-23786875 W^3 \\
   &-72742000 \sqrt{5} W^2-32280000 W^2-104707875 \sqrt{5} W+80830625 W-59254375 \sqrt{5}+84001875  \\
P_{f, 2}(W) &= -435 \sqrt{5} W^7-439 W^7+664 \sqrt{5} W^6+2140 W^6+21074 \sqrt{5} W^5+41470 W^5+167470 \sqrt{5} W^4+329040 W^4 \\
&+887560 \sqrt{5} W^3 +853650 W^3+1794425 \sqrt{5} W^2+2377375
   W^2+3598525 \sqrt{5} W-570125 W+2233125 \sqrt{5}-1711375
   \end{align*}
We obtain that
\begin{align*}
Q_f(W) &= P_f(W) \left( P_{f, 1}(W) - P_{f, 2}(W)\sqrt{g(W)} \right) \sqrt{P_{f, 0}(W)} = \frac{P_{f, 1}(W)^2 - P_{f, 2}(W)^2 P_{f, 0}(W) }{ 100 (2 + 3\sqrt{5} - W)^5} \\
&= P_{f, 4}(W) (W + \sqrt{5})^2 (W+1) (W-W_0)^3\,,
\end{align*}
where
\begin{align*}
P_{f, 4}(W) &= 70 \left(11+\sqrt{5}\right) W^7+\left(-1520-1340 \sqrt{5}\right) W^6+\left(-45150-5730 \sqrt{5}\right) W^5 \\
&+\left(-176050-45630 \sqrt{5}\right) W^4+\left(-259900-123200
   \sqrt{5}\right) W^3 \\
   &+\left(216300 \sqrt{5}-795000\right) W^2+\left(991250 \sqrt{5}-2159250\right) W+624750 \sqrt{5}-1336250\,.
\end{align*}
$P_{f, 4}(W)$ only has a single root for $W > W_0$. Let us call $x$ to that root, we have that $P_{f, 4}(W) < 0$ for $W \in [W_0, x)$ and $P_{f, 4}(W) > 0$ for $W \in (x, \infty)$. Therefore, $Q_f(W)$ is also negative for $W \in (W_0, x)$ and positive for $W \in (x, \infty)$. Letting $P_{f, 5}(W) = P_{f, 1}(W) - P_{f, 2}(W)\sqrt{g(W)}$ we have that $P_f(W) P_{f, 5}(W) = Q_f(W)$. Moreover, we have that
\begin{equation} \begin{cases}
P_{f, 5}(10) &= -625 \left(3087362507+371529853 \sqrt{5}+\sqrt{9412199900256492638+2656241362034147990 \sqrt{5}}\right) < 0\, \\
P_{f, 5}(15) &= 20000 \left(-336818691+1184891137 \sqrt{5}+\sqrt{6 \left(557-33 \sqrt{5}\right)} \left(19411872 \sqrt{5}-461659\right)\right) > 0\,,
\end{cases} \notag\end{equation}
so the only root of $Q_f(W)$ for $W > W_0$ (which is simple) is a zero of $P_{f, 5}(W)$. In particular, $P_f (W)$ does not vanish for $W > W_0$.

\textbf{Case $\xi > 0$.} First of all, as $(W^E, Z^E)$ goes from $P_s$ to $P_\infty = (0, 0)$, it can be parametrized by some compact domain, therefore, we just need to show \eqref{eq:logitechZ} is positive for $\xi > 0$, and we will automatically get a positive lower bound by compactness. As in this region we have $D_W^E, D_Z^E > 0$, this corresponds to showing that \eqref{eq:A} is positive. Solving \eqref{eq:A} as in the previous case, this is true as long as $W > f(W)$.

Let us recall that for $\xi > 0$ the solution $(W^E, Z^E)$ is inside a triangle $T$ formed by the lines $D_Z = 0$, $W = Z$ and $W = W_0$, in particular, $W^E$ is between $-1$ and $W_0$ (left and right extrema of the triangle) for $\xi > 0$. There are three simple roots of $D_Z (W, f(W))$, given by $-1$, $\bar W_0$ and $W_0$. Using that $D_Z(W, f(W)) < 0$ for $W > W_0$ from the previous case, we conclude that $D_Z(W, f(W))$ is positive for $W \in (\bar W_0, W_0)$ and negative for $W \in (-1, \bar W_0)$. As $D_Z^E > 0$, this automatically yields that $(W^E, Z^E)$ lies in the region $W > f(W)$ for $W \in (-1, \bar W_0)$ (for a given fixed $W$, the points with $D_Z > 0$ lie above those with $D_Z < 0$). Thus, we just need to deal with the region $W \in (\bar W_0, W_0)$.

From our proof of Proposition \ref{prop:mainleft7o5}, we have that $(W^E, Z^E)$ lies in the region $\mc T \subset T$, which has lower boundary $b^{\rm nl}_{7/5}(t)$ or $b^{\rm fl}_{7/5}(t)$ for $W \in (\bar W_0, W_0)$. Thus, it suffices to show that the sign of \eqref{eq:A} is positive at $b^{\rm nl}_{7/5}(t)$ and $b^{\rm fl}_{7/5}(t)$. 

Let us start with $b^{\rm fl}_{7/5}(t)$. We define the quantity
\begin{equation*}
S_{7/5}^{\rm fl} (t) =   D_W (b^{\rm fl}_{7/5} (t)) D_Z (b^{\rm fl}_{7/5} (t))^2  + \frac{1-\alpha}{2} N_W(b^{\rm fl}_{7/5} (t)) D_Z(b^{\rm fl}_{7/5} (t)) 
+ \frac{1+\alpha}{2} N_Z (b^{\rm fl}_{7/5} (t)) D_W (b^{\rm fl}_{7/5} (t))\,.
\end{equation*}
It is clear that this is a polynomial in $t$, multiple of $t$ because $b^{\rm fl}_{7/5}(0) = P_s$ and $D_Z(P_s) = N_Z(P_s) = 0$. We have that
\begin{equation*}
\lim_{r\rightarrow \left( r^\ast \right)^- } \frac{S_{7/5}^{\rm fl}(0)}{r^\ast - r} = \frac{1}{605} \left(205-89 \sqrt{5}\right) > 0\,,
\end{equation*} 
so we can reduce to show $ \frac{S_{7/5}^{\rm fl}(t)}{t^3} \Big|_{r=r^\ast} > 0$. For the rest of the treatment of $b^{\rm fl}(t)$ let us assume $r= r^\ast$. We have that
\begin{align*}
 \frac{S_{7/5}^{\rm fl}(t)}{t^3} \Big|_{r=r^\ast} &= \frac{1}{ 15552 \left(14661 \sqrt{5}-32783\right) } \Big( 
 454726725 \sqrt{5} t^6-1016799875 t^6-479229660 \sqrt{5} t^5 \\
 & +1071590100 t^5+899397288 \sqrt{5} t^4-2011113480 t^4-1391938278 \sqrt{5} t^3+3112468598 t^3 \\
 &+833713587 \sqrt{5}t^2-1864240269 t^2-834718230 \sqrt{5} t+1866486726 t+746056440 \sqrt{5}-1668233016 \Big)\,,
\end{align*}
and this polynomial is positive for all $t \in \mathbb{R}$.

Lastly, we need to show that \eqref{eq:A} is positive at $b^{\rm nl}_{7/5}(s)$ for $s \leq s_{7/5, \rm int}$ defined in the proof of Proposition \ref{prop:mainleft7o5}. Let us recall $ s_{7/5, \rm int}^{n-3} \lesssim \frac{n!}{Z_n}$ also from the proof of Proposition \ref{prop:mainleft7o5}. We define
\begin{equation*}
S_{7/5}^{\rm nl} (s) =   D_W (b^{\rm nl}_{7/5} (s)) D_Z (b^{\rm nl}_{7/5} (s))^2  + \frac{1-\alpha}{2} N_W(b^{\rm nl}_{7/5} (s)) D_Z(b^{\rm nl}_{7/5} (s)) 
+ \frac{1+\alpha}{2} N_Z (b^{\rm nl}_{7/5} (s)) D_W (b^{\rm nl}_{7/5} (s)),
\end{equation*}
which is a $3n$-th degree polynomial. Following the same proof as in Lemma \ref{lemma:nixon}, we get that for $2 \leq i \leq 3n$
\begin{equation*}
\frac{\p^i S_{7/5}^{\rm nl} (s)}{i!} \les \left( \frac{|Z_n|}{n!} \right)^{i/n} \frac{(k-n)^{i/n - \lfloor i/n \rfloor}}{n^{\min \{ 7, \ell (i) \} }}\,.
\end{equation*}
Thus, we get
\begin{align*}
S_{7/5}^{\rm nl} (s) - S_{7/5}^{\rm nl\; \prime } (0)s &\les \sum_{i=2}^{n-1} \left( \frac{|Z_n| (k-n)}{n!}s^n \right)^{i/n} \frac{1}{n^{\min \{ 7, \ell (i) \} }} +   \sum_{i=n}^{3n} \left( \frac{|Z_n|}{n!}s^n \right)^{i/n} \frac{1}{n^{\min \{ 7, \ell (i) \} }}\,.
\end{align*}
 Now let $s_0 = \left( \frac{n!}{|Z_n|} \right)^{1/(n-2)}$. By Corollary \ref{cor:Znsign} we have that $s_0 \les \frac{1}{n}$. Assuming $s \leq s_0$, we have
\begin{align} \label{eq:une}
S_{7/5}^{\rm nl} (s) - S_{7/5}^{\rm nl\; \prime } (0)s  &\les  s^2 \sum_{i=2}^{3n}  \frac{1}{n^{\min \{ 7, \ell (i) \} }} \les s^2\,.
\end{align}
Moreover
\begin{equation} \label{eq:deux}
 S_{7/5}^{\rm nl\; \prime } (0) = \frac{1-\alpha}{2} N_{W, 0} D_{Z, 1} + \frac{1+\alpha}{2} D_{W, 0} N_{Z, 1} = \frac{ \frac{1-\alpha}{2} W_1 + \frac{1+\alpha}{2} Z_1   }{D_{W, 0} D_{Z, 1}} = \frac{1}{D_{W, 0}} > 0,
\end{equation}
where the last inequality is due to Lemma \ref{lemma:aux_limits}. We conclude from \eqref{eq:une}--\eqref{eq:deux} and from the fact that $s_0 \les \frac{1}{n}$ that for $0 \leq s \leq s_0$:
\begin{equation*}
S_{7/5}^{\rm nl} (s) = \frac{s}{D_{W, 0}} \left( 1 + O \left( \frac{1}{n} \right) \right),
\end{equation*}
and in particular it is positive. As we have that $s_{7/5, \rm int}^{n-3} \les \frac{n!}{Z_n} = s_0^{n-2}$ we get that $s_{7/5, \rm int} < s_0$ for $n$ sufficiently large ($r$ sufficiently close to $r^\ast$), so we are done.
\end{proof}

\begin{lemma} \label{lemma:venecia} We have that for all $\zeta \in (0, 1)$
\begin{equation} 
\frac{\alpha \mw Z}{\zeta} - \frac{1-\alpha}{2} \p_\zeta \mw W < \frac{-1}{100}\,.\notag
\end{equation}
\end{lemma}
\begin{proof} Note that on the variables used for the Euler profile, this is just saying that our profile is in the region
\begin{equation*}
A = \alpha Z^E - \frac{1-\alpha}{2} \left( W^E + \frac{N_W^E }{D_W^E}\right) + \frac{1}{100} < 0
\end{equation*}
for $\xi < 0$. From our proof of Theorem \ref{th:mainr3}, we know that $(W^E(\xi), Z^E(\xi)) \in \Omega$ for $\xi < 0$, and we recall that $\Omega$ is defined to be the region where $D_W > 0, D_Z < 0$. In particular, it suffices to show that
\begin{equation} \label{eq:cni}
500AD_W = W (200 r+60 Z-197)+20 W^2+2 Z (10 Z+51)+5\,,
\end{equation}
is negative in $\Omega$. As $500AD_W$ is continuous with respect to $r$ and $\Omega$ is independent of $r$, we may just show this at $r=r^+$ and the result will hold true in a neighbourhood of $r^+$ by continuity.

We get that $500AD_W$ is negative for $Z \in (Z_{(1)}(W), Z_{(2)}(W))$, where $Z_{(i)}(W)$ are the two branches of the hyperbola implicitly defined by \eqref{eq:cni} and they are given at $r=r^+$ by
\begin{align*}
Z_{(1)}(W) &= \frac{1}{20} \left(-30 W-51-\sqrt{500 W^2 + 1000\sqrt{5} W+2501}\right), \\
Z_{(2)}(W) &= \frac{1}{20} \left(-30 W-51+\sqrt{500 W^2 + 1000\sqrt{5} W+2501}\right),
\end{align*}
where the second degree polynomials inside the square root are positive for all $W \in \mathbb R$. On the other hand, $\Omega$ is given by $\frac{-5-3W}{2} < Z < \frac{-5-2W}{3}$ for $W > -1$, so we just need to show $Z_{(1)}(W) < \frac{-5-3W}{2}$ and $\frac{-5-2W}{3} < Z_{(2)}(W)$ for all $W > -1$.

Let us start with $Z_{(1)}(W) < \frac{-5-3W}{2}$. We have that
\begin{equation} \label{eq:cni2}
20 \left( \frac{-5-3W}{2} - Z_{(1)}(W) \right) =  1+\sqrt{500 W^2 + 1000\sqrt{5} W+2501} > 0\,.
\end{equation}

With respect to $\frac{-5-2W}{3} < Z_{(2)}(W)$, we have
\begin{equation} \label{eq:cni3}
60\left( Z_{(2)}(W) + \frac{5+2W}{3}  \right) =   -53-50 W+3 \sqrt{500 (W + \sqrt{5})^2 + 1}\,.
\end{equation}
To show that \eqref{eq:cni3} is positive, it suffices to show that the term with the square root dominates. That is the case as
\begin{align*}
4500(W + \sqrt{5})^2 + 9 - (50W+53)^2 &= 100 \left(20 W^2+\left(90 \sqrt{5}-53\right) W+197\right) > 0\,,
\end{align*}
where we used $W \geq -1$ to conclude the last inequality.
\end{proof}

\begin{lemma} \label{lemma:florencia} We have that for every $\zeta \in (0, 1)$
\begin{equation} 
  \frac{\alpha \mw W}{\zeta} - \frac{1-\alpha}{2} \p_\zeta \mw Z  > \frac{1}{100}\,.\notag
\end{equation}
\end{lemma}
\begin{proof} This is equivalent to show positivity for
\begin{equation} \notag
B = \alpha W^E - \frac{1-\alpha}{2} \left( Z^E + \frac{N_Z^E}{D_Z^E} \right) - \frac{1}{100}\,,
\end{equation}
when $\xi < 0$. From our proof of Theorem \ref{th:mainr3}, we know that $(W^E (\xi), Z^E(\xi)) \in \Omega$ for $\xi < 0$. In particular, $D_Z < 0$, so it suffices to show negativity for
\begin{equation} \label{eq:nsa2}
500 B D_Z = Z (200 r-203) + 20Z^2 +20 W^2+60WZ +98W -5\,.
\end{equation}
Solving the polynomial in \eqref{eq:nsa2} in $Z$, one finds that \eqref{eq:nsa2} is negative for $Z \in (Z_{(-)}, Z_{(+)})$, where 
\begin{equation} \label{eq:sp499}
Z_{(\pm)} (W) = \frac{1}{40} \left(\pm \sqrt{40000 r^2+24000 r W-81200 r+2000 W^2-32200 W+41609}-200 r-60 W+203\right)\,.
\end{equation}
Let us recall from Remark \ref{rem:horizontal} that our solution is decreasing in $W$, so we can parametrize it as $Z_\ast(W)$. Thus, we have to show $Z_{(-)}(W) < Z_\ast (W) < Z_{(+)}(W)$ for $W > W_0$.

Let us start with $Z_\ast (W) < Z_{(+)}(W)$. As we know that the solution is in $\Omega$ (where $D_Z < 0$) for $\xi < 0$, it suffices to show that $Z_{(+)}(W) > \frac{-5-2W}{3}$, as $Z = \frac{-5-2W}{3}$ is the line at which $D_Z = 0$. By continuity, we may take $r = r^\ast$, and note
\begin{equation} \label{eq:sp500}
120 \left( Z_{(+)}(W) + \frac{5+2W}{3} \right) = -100 W+150 \sqrt{5}-241 +3 \sqrt{200 W \left(10 W-30 \sqrt{5}+49\right)-14700 \sqrt{5}+34509}\,.
\end{equation}
The second-degree polynomial inside the square root is always positive and moreover it dominates the expression, since
\begin{align*}
&9 \left(200 W \left(10 W-30 \sqrt{5}+49\right)-14700 \sqrt{5}+34509\right)-\left(-100 W+150 \sqrt{5}-241\right)^2 \\
&\qquad = 8000 \left(W+\frac{1}{2} \left(5-3 \sqrt{5}\right)\right)^2 > 0,
\end{align*}
so \eqref{eq:sp500} is positive.

Now, let us show $Z_\ast (W) > Z_{(-)}(W)$. First of all, note from Proposition \ref{prop:solnearxi0} that $Z_\ast (W) = -W + O(1)$, while doing series in \eqref{eq:sp499}, we get that $Z_{(-)}(W) = \frac{-3-\sqrt{5}}{2} W + O(1)$, so the inequality is clearly true for $W$ sufficiently large. For $W$ sufficiently close (from above) to $W_0$, we also have that \eqref{eq:nsa2} is negative, because it is zero for $W = W_0$ (as $D_Z = 0$) and 
\begin{align*}
& (-203 + 200r) (-Z_1) + 40 Z_0 (-Z_1) + 40 W_0 (-W_1) + 60 W_0 (-Z_1) + 60 Z_0 (-W_1) + 98 (-W_1)  \\
& \qquad = \sqrt{r^\ast - r} \left( \frac{-1}{33} \sqrt{42979610 \sqrt{5}-92240400} + o(1) \right)\,,
\end{align*}
as $r \rightarrow r^\ast$ and $\sqrt{42979610 \sqrt{5}-92240400} > 0$.

Therefore, as $Z_\ast (W)$ is above $Z_{(-)}(W)$ for $W$ sufficiently large and for $W$ sufficiently close to $W_0$, we just need to discard the case that $Z_\ast (W)$ crosses $Z_{(-)}(W)$ in both directions at some intermediate points. This is impossible because the field $(N_W D_Z, N_Z D_W)$ points always to the left of $(W, Z_{(-)}(W))$ for $W \in (W_0, +\infty)$. Concretely, if we define
\begin{equation*}
P_{(-)}(W) = Z_{(-)}'(W) N_W(W, Z_{(-)}(W))D_Z(W, Z_{(-)}(W)) -  N_Z(W, Z_{(-)}(W))D_W(W, Z_{(-)}(W)),
\end{equation*}
we will show that $P_{(-)}(W) > 0$ for all $W > W_0$. As we have $P_{(-)}(W_0) = 0$ (because $N_Z(P_s) = D_Z(P_s) = 0$) and
\begin{equation*}
\lim_{r \to (r^\ast)^-} \frac{P_{(-)}'(W_0)}{ \sqrt{ r^\ast - r } } = \frac{1}{9} \sqrt{425427 \sqrt{5}-815045} > 0,
\end{equation*}
we can reduce to show that $\frac{P_{(-)}(W)}{(W-W_0)^2} \Big|_{r=r^\ast} > 0$. Defining 
\begin{align*}
P_{(-)}^{(0)}(W) &= 200 W \left(10 W-30 \sqrt{5}+49\right)-14700 \sqrt{5}+34509, \\
P_{(-)}^{(1)}(W) &= 28000000 W^4-114000000 \sqrt{5} W^3+205000000 W^3-509000000 \sqrt{5} W^2 +1158252000 W^2 \\
&-464135000 \sqrt{5} W+1041517700 W+1596837600 \sqrt{5}-3572154667 \\
P_{(-)}^{(2)}(W) &= 600000 W^3-1600000 \sqrt{5} W^2+2920000 W^2-2862000 \sqrt{5} W+6371200 W+6167650 \sqrt{5}-13689861,
\end{align*}
we have that $P_{(-)}^{(0)}(W) > 0$ for all $W \in \mathbb{R}$. We also have
we have that
\begin{align*}
Q_{(-)}(W) &= P_{(-)}(W) \frac{ \sqrt{P_{(-)}^{(0)}(W)} \left( P_{(-)}^{(1)}(W) - \sqrt{ P_{(-)}^{(0)}(W)}  P_{(-)}^{(2)}(W) \right)}{ (W-W_0)^2 } \Big|_{r=r^\ast} \\
&= 80000000 W^6+\left(780000000-240000000 \sqrt{5}\right) W^5+\left(3051240000-1554000000 \sqrt{5}\right) W^4 \\
&+\left(2544792000-1196400000 \sqrt{5}\right) W^3+\left(5195893000
   \sqrt{5}-11358448560\right) W^2 \\
   &+\left(193380463-85410820 \sqrt{5}\right) W
   -1096899300 \sqrt{5}+2427799463\,,
\end{align*}
and the polynomial $Q_{(-)}(W)$ is positive for $W \geq W_0$ at $r=r^\ast$. Therefore, we have that $\frac{P_{(-)}(W)}{(W-W_0)^2} \Big|_{r=r^\ast}$ does not change signs for $W \geq W_0$. Its sign is positive because
\begin{equation*}
P_{(-)}(2) = \frac{15532539-5956350 \sqrt{5}+\frac{1743962911-759810800 \sqrt{5}}{\sqrt{6901-\frac{8900 \sqrt{5}}{3}}}}{800000} > 0\,.
\end{equation*}
\end{proof}

\begin{lemma} \label{lemma:profiledecay} We have that our smooth self-similar profiles have the following asymptotics as $\zeta \rightarrow \infty$:
\begin{equation}\label{eq:Embiid}
| \p_\zeta^i \mw W | + | \p_\zeta^i \mw Z | = O \left( \zeta^{1-r-i} \right).
\end{equation}
Moreover, we have that for $\delta$ sufficiently small, there exists $\zeta_0 > 0$ such that 
\begin{align}\label{eq:SU:up:low:bnd}
|\nabla \mw S (\zeta )|+ | \nabla \mw U (\zeta )| \leq \delta^{3/2} \quad \mbox{and}\quad \mw S (\zeta')  \geq \delta
\end{align}
for every $\zeta' \leq \zeta_0 \leq \zeta$.
\end{lemma}
\begin{proof}
To show $\mw W,\mw Z=O(\zeta^{1-r})$ and $\partial_\zeta \mw W,\partial_\zeta \mw Z=O(\zeta^{-r})$. Near $(W_E,Z_E)=0$ we have
\[\partial_\xi W=-rW+O(W^2+Z^2)\quad \partial_\xi Z=-rZ+O(W^2+Z^2)\,,\]
which implies $W_E,Z_E=O(e^{-r\xi})=O(\zeta^{-r})$, which translates to $\mw W,\mw Z=O(\zeta^{1-r})$.
\[\partial_\zeta \mw W_W= W_E+\partial_\xi W=O(e^{-r\xi})=O(\zeta^{-r})\,,\]
and we obtain an analogous bound for $\partial_\zeta \mw Z$.

Assuming \eqref{eq:Embiid} holds inductively for $i=0,\ldots m$, then by the Leibniz rule
\begin{align*}
   \begin{split} 
(r-1+m+O(\zeta^{-r}))\p_{\zeta}^m\mw W+(\zeta+O(\zeta^{1-r}))\p_{\zeta}^{m+1}  \mw W &=O(\zeta^{1-m-2r})\,,\\
 (r-1+m+O(\zeta^{-r}))\p_{\zeta}^m\mw Z+(\zeta+O(\zeta^{1-r}))\p_{\zeta}^{m+1}  \mw Z  &=O(\zeta^{1-m-2r})\,.
 \end{split}
  \end{align*}  
  Thus, by Gr\"onwall, we obtain \eqref{eq:Embiid} for $i=m+1$.
  \end{proof}
\section{Implementation details of the computer-assisted part} \label{sec:computer}

In this appendix, we discuss the technical details about the
implementation of the different rigorous
numerical computations that appear in the proofs throughout the paper.
We performed the rigorous computations using the Arb library
\cite{Johansson:Arb} and specifically its C implementation. We attach the code as supplementary material. See Table \ref{tablecompi} for the specific programs/commands to run each Lemma/Proposition, with more details in the Supplementary Material. Since the code is long, as an extra step in guaranteeing correctness, we further verified the C implementation of functions  against a numerical implementation in Mathematica. We only attach here the C version since it is the only mathematically rigorous implementation. We have also sacrificed efficiency by readability and some parts of the code could be optimized (e.g. the splitting between the regimes $\gamma \sim 1$ and $\gamma \sim \infty$ could be optimized function by function, or the calculation of the more complicated barriers could also be optimized, as well as the aspect ratio -- see below for a precise definition). Instead, we decided to write a much more modular design with many small functions performing simple tasks, at the price of sometimes duplicating code. Other times, we found empirically that the gain in precision from a higher order method vs a lower order method (for example using $f\left(\frac{a+b}{2}\right) + \frac{b-a}{2}f'([a,b])$ instead of $f([a,b])$ as an enclosure of $f$) was beaten by the computational cost of the former and the net execution time was comparable for both methods. In such a case, we decided to keep the lower order method to gain in readability.

The implementation is split into several files dealing with the basic functions (such as $W_0$, $Z_0$ for example), more complicated functions needed for the barriers (e.g.\ $P^{nl}$) and an additional general utility file.

There are two versions of the basic and barriers' files depending on whether $\gamma \in [1,3]$ or $\gamma \in [3,\infty)$, and an extra file with additional functions for the case $\gamma = \frac75$. The rationale behind the splitting is that different desingularizations of the functions are required for the respective cases. In the former case, we will work with the variables $\tilde{\gamma} = \gamma-1, \tilde{r} = \frac{r-1}{\gamma-1}$ due to the singular behaviour of the functions as $\gamma \to 1$. In the latter, we will work with the variables $\gamma_{\rm inv} = \frac{1}{\gamma} \in [0,\frac13]$, and $\beta$, where $r = \frac{13}{10} - \gamma_{\rm inv}\left(\frac{5}{12}\right) + \frac{3}{20}\beta$. This change of variables is used to map the region $\Omega = \{\gamma_{\rm inv} \in [0,\frac13], r \in (r_3(\gamma), r_4(\gamma)\}$ into a rectangular-like region to avoid recalculating or bounding $r_3(\gamma)$ and $r_4(\gamma)$ every time, leading to smaller errors. For performance reasons and because of Lemma \ref{lemma:aux_signsZ3Z4} or Proposition \ref{prop:right_f} we computed an enclosure of $r_3(\gamma_{\rm inv}),r_4(\gamma_{\rm inv})$ and $\tilde{r}_3(\tilde{\gamma}), \tilde{r}_4(\tilde{\gamma})$ via the following Lemmas:

\begin{lemma}
\label{lemma:enclosure_r3}
Let $\gamma \geq 3 $. 
Then $\beta_3 \in \bar{\beta_3}$, where
\begin{align*}
\bar{\beta}_3(\gamma_{\rm inv}) & = (-0.12274496668801302 \gamma_{\rm inv}^8 + 0.42078810964241387 \gamma_{\rm inv}^7 - 0.623996430280739 \gamma_{\rm inv}^6 \\
& + 0.4105016227331515 \gamma_{\rm inv}^5 + 
 0.20672452719140819 \gamma_{\rm inv}^4 - 1.0572166089549326 \gamma_{\rm inv}^3 \\
 & + 
 1.804198700610401 \gamma_{\rm inv}^2 - 0.35416295479734694 \gamma_{\rm inv} + 
 0.09216512413383933) + 10^{-7}[-1,1]\,,
 \end{align*}
 and
 \begin{align*}
 r_3 = \frac{13}{10} - \gamma_{\rm inv}\frac{5}{12} + \frac{3}{20}\beta_3\,.
 \end{align*}
\end{lemma}

\begin{lemma}
\label{lemma:enclosure_r4}
Let $\gamma \geq 3$. 
Then $\beta_4 \in \bar{\beta_4}$, where
\begin{align*}
\bar{\beta}_4(\gamma_{\rm inv}) & = 
 (-0.04469537027555534\gamma_{\rm inv}^8 + 0.27333057184133175\gamma_{\rm inv}^7 - 
   0.7172811883027264\gamma_{\rm inv}^6 \\
&   + 0.9255018926764634*\gamma_{\rm inv}^5 - 
   0.4952968302717332\gamma_{\rm inv}^4 - 0.6817068021865448\gamma_{\rm inv}^3 \\
   & + 
   1.8794062687026156\gamma_{\rm inv}^2 - 1.0362653478653305\gamma_{\rm inv} + 
   0.6762522531779247) + 10^{-7}[-1,1]\,,
\end{align*}
 and
 \begin{align*}
 r_4 = \frac{13}{10} - \gamma_{\rm inv}\frac{5}{12} + \frac{3}{20}\beta_4\,.
 \end{align*}
\end{lemma}
%
%
%
%

\begin{lemma}
\label{lemma:enclosure_r3_small}
Let $1 < \gamma \leq 3$. Then $\tilde{r}_3 \in \bar{\tilde{r}}_3$, where
\begin{align*}
\bar{\tilde{r}}_3(\tilde{\gamma}) & =
 0.12958483718253389 -
 0.055797750679595685(\tilde{\gamma} - 1) + 
 0.025268384121421402(\tilde{\gamma} - 1)^2 \\
 & - 
 0.012079846976628505(\tilde{\gamma} - 1)^3 + 
 0.006116771307938418(\tilde{\gamma} - 1)^4 - 
 0.0032535214532335432(\tilde{\gamma} - 1)^5 \\
 & + 
 0.0016116474810902726(\tilde{\gamma} - 1)^6 - 
 0.0008337203606963439(\tilde{\gamma} - 1)^7 + 
 0.001012190680858338(\tilde{\gamma} - 1)^8 \\
 & - 
 0.0007409251358921898(\tilde{\gamma} - 1)^9 - 
 0.00036102326965520293(\tilde{\gamma} - 1)^{10} + 
 0.0003377160636686555(\tilde{\gamma} - 1)^{11} \\
 & + 
 0.0004026573999596568(\tilde{\gamma} - 1)^{12} - 
 0.00028839031633197794(\tilde{\gamma} - 1)^{13}
 + 10^{-6}[-1,1]\,.
 \end{align*}
\end{lemma}

\begin{lemma}
\label{lemma:enclosure_r4_small}
Let $1 < \gamma \leq 3$. Then $\tilde{r}_4 \in \bar{\tilde{r}}_4$, where
\begin{align*}
\bar{\tilde{r}}_4(\tilde{\gamma}) & =
0.17138836639778826 - 
 0.07719915367902941(\tilde{\gamma} - 1) + 
 0.037195168499089215(\tilde{\gamma} - 1)^2 \\ 
 & - 
 0.01925242261821647(\tilde{\gamma} - 1)^3  + 
 0.010950870775304766(\tilde{\gamma} - 1)^4 - 
 0.0066817396642915305(\tilde{\gamma} - 1)^5 \\
 & + 
 0.0023871873304486257(\tilde{\gamma} - 1)^6 - 
 0.0005906689017045608(\tilde{\gamma} - 1)^7 + 
 0.005507190023795072(\tilde{\gamma} - 1)^8 \\
 & - 
 0.00526607727745243(\tilde{\gamma} - 1)^9 - 
 0.004296434160444562(\tilde{\gamma} - 1)^{10} + 
 0.0042797575355271456(\tilde{\gamma} - 1)^{11} \\
 & + 
 0.002982168551811326(\tilde{\gamma} - 1)^{12} - 
 0.0025067232778351023(\tilde{\gamma} - 1)^{13} +10^{-6}[-2,2]\,.
\end{align*}
\end{lemma}

Any statement that has to be proved in $\Omega$ for $\gamma \geq  3$ will be proved in the following region $\Omega' = \{\gamma_{\rm inv} \in [0, \frac13], \beta \in [\bar \beta_{3}(\gamma_{\rm inv}),\bar \beta_{4}(\gamma_{\rm inv})]\}$ or conversely, in $\Omega' = \{\tilde{\gamma} \in [0, 2], \tilde{r} \in [\bar{ \tilde{r}}_{3}(\tilde{\gamma}),\bar{ \tilde{r}}_{4}(\tilde{\gamma})]\}$ in the case $\gamma \leq  3$, which will imply the correctness of the statement in $\Omega$ thanks to the monotonicity of $r(\beta)$ with $\beta$, $r(\tilde{r})$ with $\tilde{r}$ as well as the monotonicity of $k(r)$ (cf. Lemma \ref{lemma:k}):

Throughout the code, we will also desingularize the different variables in such a way that there is a finite limit whenever $\gamma$ tends to the singular point (either 1 or $\infty$). For example, instead of calculating $W_0$ or $Z_0$, we will calculate $\frac{W_0}{\gamma_{\rm inv}}$ and $\frac{Z_0}{\gamma_{\rm inv}}$ respectively, to be able to reach the corresponding limits as $\gamma_{\rm inv} \rightarrow 0$.

An important desingularization in the case $1 < \gamma \leq 3$ is the following. If one expands $W_k$ or $Z_k$ in powers of $\tilde{\gamma} = \gamma-1$ it is easy to obtain that $W_k = \frac{W_k^{\rm s}}{\tilde{\gamma}} + W_k^{\rm ns}$, where $W_k^{\rm s}, W_k^{\rm ns}$ are $O(1)$. However, $W_k + Z_k = O(1)$ introduces an extra cancellation and this appears at many levels. In contrast with the case $\gamma \geq 3$ a homogeneous desingularization is not possible anymore. To remedy this situation we will perform two steps. The first one is to split the recurrence for $W_k$ and $Z_k$ into $W_k^{\rm s}, Z_k^{\rm s}, W_k^{\rm ns}, Z_k^{\rm ns}$. In particular, this yields:
\begin{align*}
D_{W,i} & = \frac12 W_i^{\rm s} + \frac{1+\alpha}{2}W_i^{\rm ns} + \frac{1-\alpha}{2}Z_i^{\rm ns} \,,\\
D_{Z,i} & = \frac12 Z_i^{\rm s} + \frac{1+\alpha}{2}Z_i^{\rm ns} + \frac{1-\alpha}{2}W_i^{\rm ns}\,,
\end{align*}
and (for $W_n^{\rm ns}$)
\begin{align*}
D_{W,0} W_n^{\rm ns} & = -\sum_{j=0}^{n-2}\binom{n-1}{j}DW_{n+1-j}W_{j+1}^{\rm ns} + \frac14 \sum_{j=0}^{n-1}\binom{n-1}{j}(W_{j}^{\rm ns}-Z_{j}^{\rm ns})W_{n-1-j}^{\rm s} + \tilde{N}_{W,n-1}^{\rm ns} + R_{W,n-1} + W_n^{\rm s} \frac{r-1}{\gamma-1}\,, \\
\tilde{N}_{W,n-1}^{\rm ns} & = \nabla N_W^{\rm ns} \cdot (W_{n-1}^{\rm ns},Z_{n-1}^{\rm ns}) + \frac12 \sum_{j=1}^{n-1} \binom{n}{j}(W_j^{\rm ns},Z_j^{\rm ns}) HN_W (W_{n-1-j}^{\rm ns},Z_{n-1-j}^{\rm ns})^{\top}\,, \\
\nabla N_W^{\rm ns} & = \left(-r-W_0^{\rm ns}-\frac12 Z_0^{\rm ns} - \frac54 W_0^s - (\gamma-1)W_0^{\rm ns} + \frac{\gamma-1}{4}Z_0^{\rm ns},-\frac12 W_0^{\rm ns} - \frac14 W_0^{\rm s} + \frac{\gamma-1}{4}W_0^{\rm ns} + \frac{\gamma-1}{2} Z_0^{\rm ns}\right) \,,\\
R_{W,i} & = \sum_{j=1}^{i} \binom{i}{j}W_j^s\left(-\frac54 W_{i-j}^{\rm ns} - \frac14 Z_{i-j}^{\rm ns}\right)\,,
\end{align*}
and (for $Z_n^{\rm ns}$)
\begin{align*}
(n-k)D_{Z,1} Z_n^{\rm ns} & = 
-\sum_{j=1}^{n-2}\binom{n}{j}DZ_{n-j}Z_{j+1}^{\rm ns} - \frac14 \sum_{j=1}^{n-1}\binom{n}{j}(Z_{n-j}^{\rm ns}-W_{n-j}^{\rm ns})W_{j}^{\rm s} + W_n^{\rm s} \left(\frac{r-1}{\gamma-1} + \frac14(W_0^{\rm ns} - Z_0^{\rm ns})\right) \,,\\
& + \frac12 W_n^{\rm s} Z_1^{\rm ns} - \frac{1-\alpha}{2} Z_1^{\rm ns} W_n^{\rm ns} + \frac14 W_0^{\rm s} W_n^{\rm ns} + \left(-\frac12 Z_0^{\rm ns} + \frac14 W_0^s + \frac{\gamma-1}{2}W_0^{\rm ns} + \frac{\gamma-1}{4}Z_0^{\rm ns}\right)W_n^{\rm ns}\,, \\
& + \tilde{Q}_{Z,n}^{\rm ns} + R_{Z,n}  \\
\tilde{Q}_{Z,n}^{\rm ns} & = \frac12 \sum_{j=1}^{n-1} \binom{n}{j}(W_j^{\rm ns},Z_j^{\rm ns}) HN_Z (W_{n-1-j}^{\rm ns},Z_{n-1-j}^{\rm ns})^{\top} \,,\\
R_{Z,i} & = \sum_{j=1}^{i} \binom{i}{j}Z_j^s\left(-\frac54 Z_{i-j}^{\rm ns} - \frac14 W_{i-j}^{\rm ns}\right)\,,
\end{align*}
as well as $Z_k^{\rm s} = -W_k^{\rm s} = -Z_{k-1}^{\rm s}$ for all $k \geq 1$. Moreover, we will propagate estimates of the form ``singular'' and ``non-singular'' into some of the building blocks of the barriers. The second step is related to this phenomenon and concerns the observation of the following cancellation (we write it for a generic barrier though it applies everywhere):
\begin{align*}
P \tilde{\gamma} = (b_Z')^{\rm ns} (N_W \tilde \gamma) D_Z - (b_W')^{\rm ns} (N_Z \tilde\gamma) D_W + (b_Z')^{\rm s} (N_W D_Z + N_Z D_W)\,,
\end{align*}
where the barrier $b = (b_W,b_Z)$ is split into the singular and non-singular parts $b^{\rm s} = (b_W^{\rm s}, b_Z^{\rm s}), b^{\rm ns} = (b_W^{\rm ns},b_Z^{\rm ns})$ respectively and we have used the fact that $(b_Z')^{\rm s} = -(b_W')^{\rm s}$ and we exploit an extra cancellation in $N_W D_Z + N_Z D_W$ writing it in terms of $W + Z$ and $W - Z$.

The general philosophy is to run a branch and bound algorithm for all the open conditions that have to be checked throughout the paper. We will first enclose the condition at a given box in parameter space (which is at most 2 dimensional). For instance, in the $\gamma \geq 3$ case, starting from a subset of $\Omega'$, where we picked either $\{\gamma_{\rm inv} \in \gamma_{\rm inv}^{K,N} = [\frac{K-1}{N}  \frac13,\frac{K}{N}  \frac13], \quad \beta \in [\bar{\beta}_3(\gamma_{\rm inv}^{K,N}),\bar{\beta}_4(\gamma_{\rm inv}^{K,N})], K, N \in \mathbb{Z}\}$ (for the most demanding calculations) or the full set $\{\gamma_{\rm inv} \in [0, \frac13],\beta \in [0,0.7]\}$ for the least demanding ones. In the case $\gamma \leq 3$ the least demanding intervals are taken to be $\{(\tilde{\gamma},\tilde{r}) \in [0,1] \times [0.125,0.335] \} \cup \{(\tilde{\gamma},\tilde{r}) \in [1,2] \times [0.085,0.195] \}$. For the specific values of $K$ and $N$ used please see Table \ref{tableruntime}. If the enclosure gives a definite sign, we accept (or reject) it, depending on whether the sign is the desired one or not. If the enclosure does not give a sign, we split the box in 2 accross one of the dimensions and call this procedure recursively. The program keeps dividing unless a certain tolerance ($10^{-10}$) in the maximum length in any dimension of the box is reached, in which case the program fails. In our case, this tolerance was never met. In order to select which direction to split by, a reasonable criterion should be to keep an aspect ratio proportional to the gradient of the function to be evaluated. Instead, due to the costly evaluation or estimation of that gradient, we determined empirically that keeping an aspect ratio of around 10 optimized the running time for $\gamma \geq 3$ (in $\beta,\gamma_{\rm inv}$ variables) and an aspect ratio of $\frac15$ (in $\tilde{r},\tilde{\gamma}$ variables) in the case $\gamma \leq  3$. For example, this meant that we split along the $\beta$ direction if the width in the $\beta$ direction was bigger than 10 times the width in the $\gamma_{\rm inv}$ direction, otherwise along the $\gamma_{\rm inv}$ one. For the cases where the problem is 1 dimensional, we treat is as a 2 dimensional one with width 0 in one of the dimensions.

In Table \ref{tableruntime} we presented the maximum times (per run) of the different parts of the code. In total, our computations ran for at most about 5000 CPU hours, although a more realistic estimate is between 3000 and 4000 CPU hours. We have also included the logs from the cluster runs as supplementary material to provide a more detailed estimate of the runtime.

We now move on to the specific details of the corresponding lemmas and propositions, in the order in which they appear on the paper:

\begin{detailssth}{Lemma \ref{lemma:k}}
We start by using the formulation \eqref{eq:k_asquotientDZ1} writing
\begin{align*}
k(r) = \frac{A(r) - B \mathcal{R}_2}{A(r) + B \mathcal{R}_2}, \quad A(r) = -5 + \frac{3}{\gamma} + r(1+\frac{1}{\gamma}), \quad B = 1 - \frac{1}{\gamma}\,.
\end{align*}
It is enough to check that 
\begin{align} \label{cond1k}
(-1+\gamma) \gamma^2 (A(r) \partial_r (\mathcal{R}_{2}^2) - 2 A'(r) \mathcal{R}_{2}^2) < 0\,.
\end{align}
 We can write condition  \eqref{cond1k} as $T_1(r,\gamma) + T_2(r,\gamma) \mathcal{R}_1 < 0$, where
\begin{align*}
T_1(r,\gamma) & = 4 (\ga -1) (3 \ga - 1) \left(\ga^2 r-3 \ga^2-14 \ga r+14 \ga+17 r-15\right)\,, \\
T_2(r,\gamma) & = (\gamma +1)^2 (3 \gamma -5) (r-1)\,.
\end{align*}

We first show that $T_1(r, \gamma) < 0$. Clearly, we just need to show negativity for $\left(\ga^2 r-3 \ga^2-14 \ga r+14 \ga+17 r-15\right)$. As this expression is affine in $r$ and $r^\ast (\ga) < 2 - \frac{1}{\ga}$ (by Lemma \ref{lemma:connecticut}) it suffices to show negativity for the endpoints $r = 1$ and $r=2 - \frac{1}{\ga}$. We have that:
\begin{align*}
\left(\ga^2 r-3 \ga^2-14 \ga r+14 \ga+17 r-15\right) \Big|_{r=1} &= 2 - 2\ga^2 < 0\,, \\
\left(\ga^2 r-3 \ga^2-14 \ga r+14 \ga+17 r-15\right) \Big|_{r=2-1/\ga} &= - \frac{(17 + \ga)(\ga - 1)^2}{\ga} < 0\,.
\end{align*}
Therefore, $T_1(r, \ga) < 0$. As we trivially have $T_2(r, \ga) < 0$ for $\ga \leq \frac53$, this concludes the case $\ga \leq \frac53$. For the case $\ga > \frac53$, it suffices to show that $\frac{1}{\gamma^8}(T_1^2 - T_2^2 \mc R_1^2) > 0$. We check with a computer-assisted proof that this is positive for $\ga \geq 5/3$ and $1 \leq r \leq 2$. As $r^\ast < 2-\frac{1}{\ga} < 2$, this ends our proof.

\end{detailssth}

\begin{detailssth}{Proposition \ref{prop:left_global}} Our choice of $b^{\rm fl}(t)$ ensures that $P^{\rm fl}(t)$ is a $7^{\rm th}$-degree polynomial multiple of $t^2(1-t)^2$. Thus, it suffices to check the positivity of $Q^{\rm fl}(t) = P^{\rm fl}(t)/(t^2(1-t)^2)$, which is a $3^{rd}$-degree polynomial. We validate the condition that $Q^{\mathrm{fl} \,\prime}(t)$ is increasing, either by validating $Q^{\mathrm{fl} \,\prime\prime\prime}(1) = Q^{\mathrm{fl} \,\prime\prime\prime}(t) < 0$ and $Q^{\mathrm{fl} \,\prime\prime}(1) > 0$ (in the case $\gamma \geq  3$) or by validating $Q^{\mathrm{fl} \,\prime\prime}(0) > 0$ and $Q^{\mathrm{fl} \,\prime\prime}(1) > 0$, which is enough since $Q^{\mathrm{fl} \,\prime\prime}(t)$ is linear (in the case $\gamma \leq  3$). Therefore, in both cases $Q^{\mathrm{fl} \,\prime}(t) < Q^{\mathrm{fl} \,\prime}(1)$ and $Q^{\mathrm{fl}}(t) > Q^{\mathrm{fl}}(1) - \max\{Q^{\mathrm{fl}\,\prime}(1),0\}$. We validate $Q^{\mathrm{fl}}(1) - \max\{Q^{\mathrm{fl}\,\prime}(1),0\}>0$. In order to optimize the code, we do not perform divisions by $1-t$ or $t$ and read the coefficients of $Q^{\rm fl}$ off the coefficients of $P^{\rm fl}$.

The case $\gamma \leq  3$ presents an extra complication since $\tilde{\gamma}Q^{\rm fl}(\tilde{\gamma},t) \sim \tilde{\gamma}F_1(\tilde{\gamma},t) + (t-1)F_2(t)$ for some smooth functions $F_1$ and $F_2$ and thus it is impossible to determine the sign out of a uniform bound on the evaluation of $Q^{\rm fl} \tilde{\gamma}$ (which is what one can compute with the previously described desingularization) due to its vanishing at $\tilde{\gamma} = 0, t = 1$. Instead, we have to work harder and further extract the leading, subleading and the rest of the terms out of expanding $Z_k = \frac{Z_k^{\rm s}}{\tilde{\gamma}} + Z_k^c + \tilde{\gamma}Z_k^{{\rm desing},2}$ where $Z_k^{\rm s}, Z_k^c, Z_k^{{\rm desing},2}$ are $O(1)$ and analogously, $W_k$ and all the barriers. In order to extract a sign out of $Q^{\mathrm{fl}}(\tilde{\gamma},1)$ we will extract a sign out of $F_1(\tilde{\gamma},1)$.

\end{detailssth}
\begin{detailssth}{Proposition \ref{prop:left_local_3}}
 Throughout this proposition, we renormalize in $\gamma$ as explained above in order to have meaningful limits of the relevant quantities as $\gamma_{\rm inv} \to 0$. We discuss in detail the case $\gamma \geq   3$: the case $\gamma \leq 3$ is done in an analogous way, considering the desingularization and splitting into the singular and non-singular parts of the relevant quantities outlined above.
For steps 1, 2, 4, 5, we additionally consider the polynomials under the following change of variables: $\tilde{s} = s(k-3)$ to make the validation region constant. In steps 1 and 2 we use the following formula:
\begin{align*}
\frac{B^{\rm fl}}{\gamma_{\rm inv}^4[s^2(k-3)^2]} & = \frac14\left(\frac{B_2}{\gamma_{\rm inv}}\right)^2\left(\frac{(Z_0-Z)}{\gamma_{\rm inv}[s(k-3)]}\right)^2  + \frac14\left(\frac{B_3}{\gamma_{\rm inv}}\right)^2\left(\frac{(W_0-W)}{\gamma_{\rm inv}[s(k-3)]}\right)^2 \\
                          & - \frac12\left(\frac{B_2}{\gamma_{\rm inv}}\right)\left(\frac{B_3}{\gamma_{\rm inv}}\right)\left(\frac{(Z_0-Z)}{\gamma_{\rm inv}[s(k-3)]}\right)\left(\frac{(W_0-W)}{\gamma_{\rm inv}[s(k-3)]}\right) \\
                          & + \frac12 B_1^2 \frac{(B_2 Z_1 - B_3 W_1)}{\gamma_{\rm inv}^2}\frac{(Z_1(W_0-W) - W_1(Z_0-Z))}{\gamma_{\rm inv}^2[s^2(k-3)^2]}\,.
\end{align*}              
                           
Note that there is an extra cancellation in the last parenthesis of the last term, yielding
\begin{gather*}
\frac{Z_1(W_0-W) - W_1(Z_0-Z)}{\gamma_{\rm inv}^2s^2(k-3)^2} = \frac12\frac{(Z_2 W_1-W_2 Z_1)}{\gamma_{\rm inv}^2} + \frac16 s \left(Z_{3}^{\rm norm,1} \frac{W_1}{\gamma_{\rm inv}} - \frac{W_3 Z_1}{\gamma_{\rm inv}^2} (k-3)\right),\\Z_{3}^{\rm norm,1} = Z_3(k-3)\,.
\end{gather*}

Steps 4 and 5: Here $\tilde{s}$ is split into the two cases $[0,0.175]$ and $[0.175,0.35]$ and we simply evaluate at the whole interval in $\tilde{s}$. We further desingularize $D_Z$ and prove the sign condition for $\frac{D_Z}{s(k-3)}$ instead to ensure strict inequality for $\tilde{s} \sim 0$.

Step 3: Throughout this part we will compute $P^{\rm nl}(t(k-3))$. A natural desingularization (in $t$) would be to consider $\tilde{s} = t(k-3)$ and desingularize as in the previous steps. This is problematic, however, since for example $W_3t^3(k-3)^3$ and $Z_3 t^3 (k-3)^3$ do not desingularize in the same way (in the latter case, to desingularize, $Z_3$ should be paired with a factor of $k-3$). In fact, the natural desingularization (in $t$ and $k-3$) should be $\frac{P^{\rm nl}(t(k-3))}{t^4(k-3)^3}$. In order to overcome this complication, we will divide every summand in $t$ by the highest possible power of $k-3$ in $W$ or $Z$, keeping track of it, multiply out to compose the power series of the products and finally multiply by powers of $k-3$ if needed. The reason for doing it this way (as opposed to dividing by $k-3$ whenever it is needed) is that $k-3$ may potentially be 0, so division by $k-3$ will not yield any meaningful results. In particular, our methods will return the following vectors, from which we will construct the functions $N, D$ and $b$: (in parenthesis the terms corresponding to the power series for the different degrees of $t$)
 \begin{align*}
 W & = (W_0, W_1, \frac12 W_2, \frac16 W_3 (k-3)), \quad Z = (Z_0, Z_1, \frac12 Z_2, \frac16 Z_3^{\rm norm,1})\,, \\
 b_W^{\prime} & = (W_1,W_2, \frac12 W_3(k-3), \quad b_Z^{\prime} = (Z_1, Z_2, \frac12 Z_3^{\rm norm,1}) \,,\\
 WZ & = (W_0Z_0, W_1Z_0 + W_0Z_1, W_1Z_1 + \frac12 W_2 Z_0 + \frac12 W_0 Z_2, \\
 & \frac16 W_0 Z_3^{\rm norm,1} + \frac16 W_3 Z_0 (k-3) + \frac12 W_2 Z_1 (k-3) + \frac12 W_1 Z_2 (k-3),  \\
 & \frac16 W_1 Z_3^{\rm norm,1} + \frac16 W_3 Z_1 (k-3) + \frac14 W_2 Z_2 (k-3), \frac{1}{12} W_2 Z_3^{\rm norm,1} + \frac{1}{12} W_3 Z_2(k-3), \frac{1}{36} Z_3^{\rm norm,1} W_3 (k-3)) \\
W^2 & = (W_0^2, 2 W_0 W_1, W_1^2+W_0 W_2, W_1 W_2 (k-3) + \frac13 W_0 W_3 (k-3),
  \frac14 W_2^2 (k-3) + \frac13 W_1 W_3 (k-3), \\
  & \frac16 W_2 W_3 (k-3), \frac{1}{36} W_3^2 (k-3)^2) \,,\\
Z^2 & = (Z_0^2, 2 Z_0 Z_1, Z_1^2+Z_0 Z_2, Z_1 Z_2(k-3) + \frac{1}{3} Z_0 Z_3^{\rm norm,1},
\frac14 Z_2^2 (k-3) + \frac13 Z_1 Z_3^{\rm norm,1}, \frac16 Z_2 Z_3^{\rm norm,1}, \frac{1}{36} (Z_3^{\rm norm,1})^2)\,,
 \end{align*}
 and the powers of $k-3$ we have divided by (as functions of the powers of $t$) are given by:
 \begin{align*}
 b_{Z}^{'}, b_{W}^{'} = (0,1,1); \quad D_{Z}, D_{W} = (0,1,2,2); \quad N_{Z}, N_{W} = (0,1,2,2,3,4,4),
 \end{align*}
In the end, we validate $P^{\rm nl}(\tilde{s}_{-}) > 0$ and $P^{\rm nl,\prime}([0,\tilde{s}_{-}]) < 0$.
\end{detailssth}

\begin{detailssth}{Proposition \ref{prop:right_f}} In principle, it is clear that $P^{\rm fr} (s)(s-s_\infty^{\rm fr})^5 $ is a polynomial in $s$, as we have up to five times the denominator $(s-s_\infty^{\rm fr})$ in \eqref{eq:Pfr}. However, the two cancellations at $P_0$ (because $b^{\rm fr}(0) = P_s$ and the choice of $F_1$) give us that $P^{\rm fr}(s)(s-s_\infty^{\rm fr})^5$ is multiple of $s^2$. 

On the other hand, we have that $B^{\rm fr}(W, Z)$ is bounded if $(W, Z)$ is the solution starting at $P_0$ (instead of growing quadratically with $|W|+|Z|$). This is because $W+Z-W_0-Z_0$ is bounded over this solution and $W+Z-F_0 \lesssim \frac{1}{|W|+|Z|}$ over this solution (due to the choice of $F_0$). Those two cancellations imply that $P^{\rm fr}(s)(s-s_\infty^{\rm fr})^5$ is in fact multiple of $(s-s_\infty^{\rm fr})^2$. Therefore, it suffices to check that sign of the polynomial $Q^{\rm fr}(s) = P^{\rm fr}(s)(s-s_\infty)^3/s^2$ is positive. 

In order to reduce the dimension of the problem, we use Lemma \ref{lemma:enclosure_r3} (so that we will evaluate at $\beta = \bar{\beta}_3$ for $\gamma \geq 3$ and $\tilde{r} = \bar{\tilde{r}}_3$ for $\gamma \leq 3$) and hence deal with a 1-dimensional problem. In the former case, we also renormalize $s$ via $\tilde{s} = \frac{s}{\gamma_{\rm inv}}$ to ensure convergence to a finite value as $\gamma_{\rm inv} \to 0$ (this includes the right scaling with respect to $\gamma_{\rm inv}$ in $s_{\infty}^{\rm fr}$ as well). In the latter, we apply the desingularization scheme described above without desingularizing $s$.

In order to bound $Q^{\rm fr}(\tilde{s}) = a_0 + a_1 \tilde{s} + a_2 \tilde{s}^2 + a_3 \tilde{s}^3 + a_4 \tilde{s}^4$ or $Q^{\rm fr}(s) = a_0 + a_1 s + a_2 s^2 + a_3 s^3 + a_4 s^4$, we validate the conditions $a_2 < 0$ and: 
\begin{align*}
\left\{
\begin{array}{cc}
a_0 - \frac{a_1^2}{4 a_2} + |a_3|\left(\frac{s^{\rm fr}_{\infty}}{\gamma_{\rm inv}}\right)^3 + |a_4|\left(\frac{s^{\rm fr}_{\infty}}{\gamma_{\rm inv}}\right)^4 < 0 & \text{ for } \gamma \geq \frac53 \,,\\
a_0 - \frac{a_1^2}{4 a_2} + |a_3|\left(s^{\rm fr}_{\infty}\right)^3 + |a_4|\left(s^{\rm fr}_{\infty}\right)^4 < 0 & \text{ for } \gamma \leq \frac53\,,
\end{array}
\right.
\end{align*} where we are bounding the first 3 terms of the polynomial by its maximum, given that they correspond to a negative parabola.

The case $\gamma = \frac75, r = r^{\ast}$ is done directly using the framework of $\gamma \geq  3$ since the desingularization is simpler.
\end{detailssth}

\begin{detailssth}{Lemma \ref{lemma:biglebowski}} Instead of computing the $6 \times 6$ determinant directly, we write it in block form, apply the formula
$\text{det}\left(\begin{array}{cc}
A & B \\
C & D
\end{array}\right) = \text{det}(A)\text{det}(D - CA^{-1}B)$
and compute the determinants on the right-hand-side. The above formula holds so long as $A$ is invertible (which we ensure along the calculation).  The computations of the determinants on the right-hand-side of the equation are comparatively more efficient since $A$ is triangular and the expression for $A^{-1}$ is simple.
\end{detailssth}

\begin{detailssth}{Lemma \ref{lemma:aux_otromas}} The condition is equivalent to $\frac{3 - \sqrt{3}}{2 + \sqrt{3}\tilde{\gamma}} - \tilde{r}_4 > 0$, which is what we actually validate.
\end{detailssth}

\begin{detailssth}{Lemma \ref{lemma:aux_Peyeout}} In order to prove the first part we will take the full interval $T_W = (0.6019, 0.6021)$ and prove that on the one hand $b^{\rm fl}_{7/5, Z} (T_W) > Y_0$ and on the other $b^{\rm fl}_{7/5, W} (0.6019) < X_0, b^{\rm fl}_{7/5, W} (0.6021) > X_0$.
\end{detailssth}
\begin{detailssth}{Lemma \ref{lemma:paella}}
We apply the same renormalization/desingularization as in Proposition \ref{prop:left_global}.
\end{detailssth}

\begin{detailssth}{Lemma \ref{lemma:aux_34bounds} (fifth inequality)} The inequality is problematic due to $W_4,Z_4$ blowing up at $r_3$ and $Z_4$ blowing up at $r_4$ as well. Instead, in the case $\gamma \geq  3$ we implemented a method that returned $W_4^{\rm norm,2} = \frac{1}{\gamma_{\rm inv}}W_4(k-3)$ and $Z_4^{norm,2} = \frac{1}{\gamma_{\rm inv}}Z_4(k-3)(k-4)$ adapting equations \eqref{eq:Wn}--\eqref{eq:Zn} accordingly and proved instead:
\begin{align*}
Z_4^{\rm norm,2} - (k-4) W_4^{\rm norm,2} \frac{Z_1}{W_1} > 0\,.
\end{align*}

In the case $\gamma \leq  3$ we split $W_k = \frac{W_k^{\rm s}}{\tilde{\gamma}} + W_k^{\rm ns}, Z_k = \frac{Z_k^{\rm s}}{\tilde{\gamma}} + Z_k^{\rm ns}$ and respectively $W_4^{\rm norm,2,*} = W_4^{*}(k-3), Z_4^{\rm norm,2,*} = Z_4^{*}(k-3)(k-4)$ for $* = \{s,ns\}$. Using that $Z_k^{\rm s} = -W_k^{\rm s}$ for all $k$, it is enough to validate the conditions
\begin{align*}
(Z_4^{\rm norm,2,ns} W_1^{\rm ns} - W_4^{\rm norm,2,ns} Z_1^{\rm ns} (k-4)) < 0\,, \\
-(Z_4^{\rm norm,2,ns} + (k-4) W_4^{\rm norm,2,ns} + (k-3)(k-4)(W_1^{\rm ns} + Z_1^{\rm ns}) > 0\,.
\end{align*}

\end{detailssth}
\begin{detailssth}{Lemmas \ref{lemma:aux_34bounds} (top row inequalities)} In the case $\gamma \geq  3$ we desingularize by computing $\frac{N_{W,0}}{\gamma_{\rm inv}}$ and $\frac{N_{Z,1}}{\gamma_{\rm inv}}$ in order for them to have finite limits as $\gamma_{\rm inv} \to 0$. Similarly, in the case $\gamma \leq 3$ we desingularize by computing $N_{W,0}\tilde{\gamma}$ and $N_{Z,1}\tilde{\gamma}$ in order for them to have finite limits as $\tilde{\gamma} \to 0$.
\end{detailssth}

\begin{detailssth}{Lemmas \ref{lemma:aux_WioverZi_7o5}, \ref{lemma:tenthousand_7o5}} The implementation is straightforward; however, due to the large numbers that appear throughout the process ($Z_{10000} \sim 10^{46770}$), ultra-high precision is required to avoid overestimation and to be able to extract the signs out of the relevant quantities. We used 2000 bits to accomplish this.
\end{detailssth}

\begin{detailssth}{Lemmas \ref{lemma:aux_F1}, \ref{lemma:aux_a2nr}} In the case $\gamma \geq  3$, we renormalize $s_{\infty}^{\rm fr}$ by considering $\frac{s_{\infty}^{\rm fr}}{\gamma_{\rm inv}}$ and proceed as in previous Lemmas. In the case $\gamma \leq  3$ we don't renormalize with respect to $\gamma$. We also remark that in the case $\gamma = \frac75$ it is enough to validate the condition $W_1 + Z_1 < 0$ thanks to Lemma \ref{lemma:aux_WioverZi_7o5}.
\end{detailssth}

\begin{detailssth}{Lemma \ref{lemma:aux_signsZ3Z4}} The inequality is also problematic due to $Z_3$ blowing up at $r_3$ and $Z_4$ blowing up at $r_4$. We generalized the implementation of $Z_n^{norm,1}$ in  \ref{prop:left_local_3} to return $Z_n^{norm,1} = \frac{1}{\gamma_{\rm inv}}Z_n(k-n)$ for any given $n$ in the case $\gamma \geq  3$ and to return $Z_n^{norm,1,*} = Z_n^{*}(k-n), * = \{s,ns\}$ in the case $\gamma \leq  3$. In the latter case, since $Z_n^s$ is uniformly bounded (in $\tilde{\gamma}$) at $r=r_n$ it is enough to check the conditions $Z_3^{norm,1,ns} < 0$ and $Z_4^{norm,1,ns} > 0$ at $\tilde{r}_3$ and $\tilde{r}_4$ respectively.
\end{detailssth}

\begin{detailssth}{Lemmas \ref{lemma:enclosure_r3}, \ref{lemma:enclosure_r4}, 
\ref{lemma:enclosure_r3_small}, 
\ref{lemma:enclosure_r4_small}}
 Straightforward since the quantities are not singular.
\end{detailssth}

\begin{longtable}{|c|c|c|c|c|}
\hline
Lemma / Proposition & $\gamma$ & $K$ & $N$ & Time (longest $K$, HH:MM:SS) \\
\hline
Lemma \ref{lemma:k} & $\gamma \geq 5/3$ & N/A & N/A & $\sim$ 00:00:00\\
\hline
Proposition \ref{prop:left_global} & $\gamma \leq 3$ & [1,68] & 100 & 23:33:55\\
\hline
Proposition \ref{prop:left_global} & $\gamma \leq 3$ & [681,1000] & 1000 & 05:13:58\\
\hline
Proposition \ref{prop:left_global} & $\gamma \geq 3$ & [1,100] & 100 & 00:46:33\\
\hline
Proposition \ref{prop:left_local_3} (Step 1) & $\gamma \leq 3$ & [1,100] & 100 & $\sim$ 00:00:00\\
\hline
Proposition \ref{prop:left_local_3} (Step 2) & $\gamma \leq 3$ & [1,100] & 100 & 00:11:17\\
\hline
Proposition \ref{prop:left_local_3} (Step 3: $P^{\rm nl}(s_{-}(k-3))$) & $\gamma \leq 3$ & [1,100] & 100 & 01:33:31\\
\hline
Proposition \ref{prop:left_local_3} (Step 3: $\frac{d}{ds}P^{\rm nl}(s)$) & $\gamma \leq 3$ & [1,100] & 100 & 00:10:04\\
\hline
Proposition \ref{prop:left_local_3} (Step 4) & $\gamma \leq 3$ & [1,10] & 10 & $\sim$ 00:00:00\\
\hline
Proposition \ref{prop:left_local_3} (Step 5) & $\gamma \leq 3$ & [1,10] & 10 & $\sim$ 00:00:00\\
\hline
Proposition \ref{prop:left_local_3} (Step 1) & $\gamma \geq 3$ & [1,100] & 100 & 00:00:16\\
\hline
Proposition \ref{prop:left_local_3} (Step 2) & $\gamma \geq 3$ & [1,100] & 100 & 01:12:38\\
\hline
Proposition \ref{prop:left_local_3} (Step 3: $P^{\rm nl}(s_{-}(k-3))$) & $\gamma \geq 3$ & [1,100] & 100 & 14:22:36\\ 
\hline
Proposition \ref{prop:left_local_3} (Step 3: $\frac{d}{ds}P^{\rm nl}(s)$) & $\gamma \geq 3$ & [1,100] & 100 & 00:45:36\\
\hline
Proposition \ref{prop:left_local_3} (Step 4) & $\gamma \geq 3$ & [1,10] & 10 & $\sim$ 00:00:00\\
\hline
Proposition \ref{prop:left_local_3} (Step 5) & $\gamma \geq 3$ & [1,10] & 10 & $\sim$ 00:00:00\\
\hline
Proposition \ref{prop:right_f} & $\gamma \leq 3$ & N/A & N/A & 00:00:24\\
\hline
Proposition \ref{prop:right_f} & $\gamma \geq 3$ & N/A & N/A & 00:01:52\\
\hline
Lemma \ref{lemma:biglebowski} & 7/5 & N/A & N/A & $\sim$ 00:00:00\\
\hline
Lemma \ref{lemma:aux_otromas} & $\gamma \leq 3$ & N/A & N/A & $\sim$ 00:00:00\\
\hline
Lemma \ref{lemma:aux_Peyeout} & 7/5 & N/A & N/A & $\sim$ 00:00:00\\
\hline
Lemma \ref{lemma:paella}  & $\gamma \leq 3$ & N/A & N/A & 00:08:22\\
\hline
Lemma \ref{lemma:paella}  & $\gamma \geq 3$ & N/A & N/A & 00:01:08\\
\hline
Lemma \ref{lemma:aux_34bounds} (fifth inequality) & $\gamma \leq 3$ & [1,100] & 100 & $\sim$ 00:00:00\\
\hline
Lemma \ref{lemma:aux_34bounds} (fifth inequality) & $\gamma \geq 3$ & [1,100] & 100 & 00:00:15\\
\hline
Lemmas \ref{lemma:aux_34bounds} (top row) & $\gamma > 1$ & N/A & N/A & $\sim$ 00:00:00\\
\hline
Lemmas \ref{lemma:aux_WioverZi_7o5}, \ref{lemma:tenthousand_7o5} & $7/5$ & N/A & N/A & 13:36:38\\
\hline
Lemma \ref{lemma:aux_F1} & $\gamma > 1$ & N/A & N/A & $\sim$ 00:00:00\\
\hline
Lemma \ref{lemma:aux_a2nr} & $\gamma > 1$ & N/A & N/A & $\sim$ 00:00:00\\
\hline
Lemma \ref{lemma:aux_signsZ3Z4} & $\gamma > 1$ & N/A & N/A & $\sim$ 00:00:00\\
\hline
Lemma \ref{lemma:enclosure_r3} (top enclosure) & $\gamma \geq 3$ & N/A & N/A & 02:34:19\\
\hline
Lemma \ref{lemma:enclosure_r3} (bottom enclosure) & $\gamma \geq 3$ & N/A & N/A & 02:37:38\\
\hline
Lemma \ref{lemma:enclosure_r4} (top enclosure) & $\gamma \geq 3$ & N/A & N/A & 02:39:16\\
\hline
Lemma \ref{lemma:enclosure_r4} (bottom enclosure) & $\gamma \geq 3$ & N/A & N/A & 02:37:59\\
\hline
Lemma \ref{lemma:enclosure_r3_small} (top enclosure) & $\gamma \leq 3$ & N/A & N/A & 00:00:23\\
\hline
Lemma \ref{lemma:enclosure_r3_small} (bottom enclosure) & $\gamma \leq 3$ & N/A & N/A & 00:00:24\\
\hline
Lemma \ref{lemma:enclosure_r4_small} (top enclosure) & $\gamma \leq 3$ & N/A & N/A & 00:00:11\\
\hline
Lemma \ref{lemma:enclosure_r4_small} (bottom enclosure) & $\gamma \leq 3$ & N/A & N/A & 00:00:12\\
\hline
\caption{Performance of the code in the different Lemmas/Propositions and regions.}
\label{tableruntime}
\end{longtable}

\begin{longtable}{|c|c|c|c|c|}
\hline
Lemma / Proposition & $\gamma$ & Compilation Command & Executable  \\
\hline
Lemma \ref{lemma:k} & $\gamma \geq 5/3$ & make check\_gamma\_high\_fast & 
      check\_gamma\_high\_fast \\
\hline
Proposition \ref{prop:left_global} & $\gamma \leq 3$ & make check\_gamma\_low\_slow & 
	    check\_gamma\_low\_Qfl\\
\hline
Proposition \ref{prop:left_global} & $\gamma \geq 3$ & make check\_gamma\_high\_slow & 
	    check\_gamma\_high\_Qfl\\
\hline
Proposition \ref{prop:left_local_3} (Step 1) & $\gamma \leq 3$ & make check\_gamma\_low\_slow & 
	    check\_gamma\_low\_Bfl\_zero\\
\hline
Proposition \ref{prop:left_local_3} (Step 2) & $\gamma \leq 3$ & make check\_gamma\_low\_slow &  check\_gamma\_low\_Bfl\_sminus\\
\hline
Proposition \ref{prop:left_local_3} (Step 3: $P^{\rm nl}(s_{-}(k-3))$) & $\gamma \leq 3$ & make check\_gamma\_low\_slow &  
	    check\_gamma\_low\_Pnl\_sminus\\
\hline
Proposition \ref{prop:left_local_3} (Step 3: $\frac{d}{ds}P^{\rm nl}(s)$) & $\gamma \leq 3$ & make check\_gamma\_low\_slow & 	    check\_gamma\_low\_dPnl\\
\hline
Proposition \ref{prop:left_local_3} (Step 4) & $\gamma \leq 3$ & make check\_gamma\_low\_slow &  check\_gamma\_low\_DW\_bnl\\
\hline
Proposition \ref{prop:left_local_3} (Step 5) & $\gamma \leq 3$ & make check\_gamma\_low\_slow &  check\_gamma\_low\_DZ\_bnl\\
\hline
Proposition \ref{prop:left_local_3} (Step 1) & $\gamma \geq 3$ & make check\_gamma\_high\_slow &  check\_gamma\_high\_Bfl\_zero\\
\hline
Proposition \ref{prop:left_local_3} (Step 2) & $\gamma \geq 3$ & make check\_gamma\_high\_slow & check\_gamma\_high\_Bfl\_sminus\\
\hline
Proposition \ref{prop:left_local_3} (Step 3: $P^{\rm nl}(s_{-}(k-3))$) & $\gamma \geq 3$ & make check\_gamma\_high\_slow &  check\_gamma\_high\_Pnl\_sminus\\
\hline
Proposition \ref{prop:left_local_3} (Step 3: $\frac{d}{ds}P^{\rm nl}(s)$) & $\gamma \geq 3$ & make check\_gamma\_high\_slow & check\_gamma\_high\_dPnl\\
\hline
Proposition \ref{prop:left_local_3} (Step 4) & $\gamma \geq 3$ & make check\_gamma\_high\_slow & check\_gamma\_high\_DW\_bnl\\
\hline
Proposition \ref{prop:left_local_3} (Step 5) & $\gamma \geq 3$ & make check\_gamma\_high\_slow & check\_gamma\_high\_DZ\_bnl\\
\hline
Proposition \ref{prop:right_f} & $\gamma \geq 3$ & make check\_gamma\_high\_slow & 
	    check\_gamma\_high\_Pfr  \\
\hline
Proposition \ref{prop:right_f} & $\gamma \leq 3$ & make check\_gamma\_low\_slow & 
	    check\_gamma\_low\_Pfr  \\
\hline
Lemma \ref{lemma:biglebowski} & 7/5 & make check\_75\_fast & check\_75\_fast\\
\hline
Lemma \ref{lemma:aux_otromas} & $\gamma \leq 3$ & make check\_gamma\_low\_fast & check\_gamma\_low\_fast \\
\hline
Lemma \ref{lemma:aux_Peyeout} & 7/5 & make check\_75\_fast & check\_75\_fast\\
\hline
Lemma \ref{lemma:paella}  & $\gamma \leq 3$ & make check\_gamma\_low\_fast & check\_gamma\_low\_fast \\
\hline
Lemma \ref{lemma:paella}  & $\gamma \geq 3$ & make check\_gamma\_high\_fast & check\_gamma\_high\_fast \\
\hline
Lemma \ref{lemma:aux_34bounds} (fifth inequality) & $\gamma \leq 3$ & make check\_gamma\_low\_slow & check\_gamma\_low\_W4Z4 \\
\hline
Lemma \ref{lemma:aux_34bounds} (fifth inequality) & $\gamma \geq 3$ & make check\_gamma\_high\_slow & check\_gamma\_high\_W4Z4 \\
\hline
Lemmas \ref{lemma:aux_34bounds} (top row) & $\gamma \leq 3$ & make check\_gamma\_low\_fast & check\_gamma\_low\_fast\\
\hline
Lemmas \ref{lemma:aux_34bounds} (top row) & $\gamma \geq 3$ & make check\_gamma\_high\_fast &       check\_gamma\_high\_fast\\
\hline
Lemmas \ref{lemma:aux_WioverZi_7o5}, \ref{lemma:tenthousand_7o5} & $7/5$ & make check\_75\_slow & check\_75\_slow\\
\hline
Lemma \ref{lemma:aux_F1} & $\gamma \leq 3$ & make check\_gamma\_low\_fast & check\_gamma\_low\_fast\\
\hline
Lemma \ref{lemma:aux_F1} & $\gamma \geq 3$ & make check\_gamma\_high\_fast &       check\_gamma\_high\_fast\\
\hline
Lemma \ref{lemma:aux_a2nr} & $\gamma \leq 3$ & make check\_gamma\_low\_fast & check\_gamma\_low\_fast\\
\hline
Lemma \ref{lemma:aux_a2nr} & $\gamma \geq 3$ & make check\_gamma\_high\_fast &       check\_gamma\_high\_fast\\
\hline
Lemma \ref{lemma:aux_signsZ3Z4} & $\gamma \leq 3$ & make check\_gamma\_low\_fast & check\_gamma\_low\_fast\\
\hline
Lemma \ref{lemma:aux_signsZ3Z4} & $\gamma \geq 3$ & make check\_gamma\_high\_fast &       check\_gamma\_high\_fast\\
\hline
Lemma \ref{lemma:enclosure_r3} (top enclosure) & $\gamma \geq 3$ & make check\_gamma\_high\_slow & check\_gamma\_high\_r3\_top \\
\hline
Lemma \ref{lemma:enclosure_r3} (bottom enclosure) & $\gamma \geq 3$ & make check\_gamma\_high\_slow & check\_gamma\_high\_r3\_bottom\\
\hline
Lemma \ref{lemma:enclosure_r4} (top enclosure) & $\gamma \geq 3$ & make check\_gamma\_high\_slow & check\_gamma\_high\_r4\_top \\
\hline
Lemma \ref{lemma:enclosure_r4} (bottom enclosure) & $\gamma \geq 3$ & make check\_gamma\_high\_slow & check\_gamma\_high\_r4\_bottom \\
\hline
Lemma \ref{lemma:enclosure_r3_small} (top enclosure) & $\gamma \leq 3$ & make check\_gamma\_low\_slow & check\_gamma\_low\_r3\_top\\
\hline
Lemma \ref{lemma:enclosure_r3_small} (bottom enclosure) & $\gamma \leq 3$ & make check\_gamma\_low\_slow & check\_gamma\_low\_r3\_bottom\\
\hline
Lemma \ref{lemma:enclosure_r4_small} (top enclosure) & $\gamma \leq 3$ & make check\_gamma\_low\_slow & check\_gamma\_low\_r4\_top \\
\hline
Lemma \ref{lemma:enclosure_r4_small} (bottom enclosure) & $\gamma \leq 3$ & make check\_gamma\_low\_slow & check\_gamma\_low\_r4\_bottom\\
\hline
\caption{Executables and compilation commands for the different Lemmas}
\label{tablecompi}
\end{longtable}

\section*{Acknowledgements}
T.B.\ was supported by the NSF grant DMS-1900149, a Simons Foundation Mathematical and Physical
Sciences Collaborative Grant and a grant from the Institute for Advanced Study.  G.C.-L.\ was supported by a grant from the Centre de Formaci\'o Interdisciplin\`aria Superior, a MOBINT-MIF grant from the Generalitat de Catalunya and a Praecis Presidential Fellowship from the Massachusetts Institute of Technology. G.C.-L.\ would also like to thank the Department of Mathematics at Princeton University for partially supporting him during his stay at Princeton and for their warm hospitality.
This project has received funding from the European Research Council (ERC) under the European Union's Horizon 2020 research and innovation program through the grant agreement 852741 (G-C.-L.\ and J.G.-S.). J.G.-S.\ was partially supported by NSF through Grant DMS 1763356. 
This material is based upon work supported by the National Science Foundation under Grant No. DMS-1929284 while J.G.-S.\ was in residence at the Institute for Computational and Experimental Research in Mathematics in Providence, RI, during the
program ``Hamiltonian Methods in Dispersive and Wave Evolution Equations''. 
This work is supported by the Spanish State Research Agency, through the Severo Ochoa and Mar\'ia de Maeztu Program for Centers and Units of Excellence in R\&D (CEX2020-001084-M).
We thank Princeton University and the Institute for Advanced Study for computing facilities via the Princeton Research Computing resources and the School of Natural Sciences Computing resources respectively.
Part of this research was conducted using computational
resources and services at the Center for Computation and Visualization, Brown University. The programs ran on the \textit{batch} queue: for more specific details about the hardware please check \url{https://docs.ccv.brown.edu/oscar/system-overview}.

\bibliographystyle{plain}

\end{document}